\documentclass{memo-l}

\usepackage{amssymb}
\usepackage{mathtools}
\usepackage{enumitem}
\usepackage{array,booktabs,longtable}
\usepackage{graphicx}
\usepackage{pdflscape}
\usepackage{rotating}
\usepackage[hidelinks]{hyperref}
\usepackage{xcolor}
\setlist[enumerate,1]{label=\upshape(\arabic*),ref=\upshape(\arabic*),leftmargin=2.5em}
\setlist[enumerate,2]{label=\upshape(\alph*),ref=\upshape(\alph*),leftmargin=3em}
\theoremstyle{plain}

\newtheorem{theorem}{Theorem}[chapter]

\newtheorem{lemma}[theorem]{Lemma}

\newtheorem{proposition}[theorem]{Proposition}

\newtheorem{corollary}[theorem]{Corollary}

\theoremstyle{definition}

\newtheorem{definition}[theorem]{Definition}

\newtheorem{remark}[theorem]{Remark}

\newtheorem{example}[theorem]{Example}

\newcommand{\Ocal}{\mathcal{O}}

\newcommand{\Pcal}{\mathcal{P}}

\newcommand{\Fix}{\operatorname{Fix}}

\newcommand{\bcal}{\mathcal{B}}

\newcommand{\Tcal}{\mathcal{T}}

\newcommand{\F}{\mathbb{F}}

\newcommand{\Fq}{\mathbb{F}_q}

\newcommand{\Z}{\mathbb{Z}}

\newcommand{\Aut}{\operatorname{Aut}}

\newcommand{\Inn}{\operatorname{Inn}}

\newcommand{\Out}{\operatorname{Out}}

\newcommand{\GL}{\operatorname{GL}}

\newcommand{\SL}{\operatorname{SL}}

\newcommand{\GU}{\operatorname{GU}}

\newcommand{\SU}{\operatorname{SU}}

\newcommand{\PGU}{\operatorname{PGU}}

\newcommand{\PSU}{\operatorname{PSU}}

\newcommand{\PGL}{\operatorname{PGL}}

\newcommand{\PSL}{\operatorname{PSL}}

\newcommand{\PpGL}{\operatorname{P\Gamma L}}

\newcommand{\PpGU}{\operatorname{P\Gamma U}}

\newcommand{\AGL}{\operatorname{AGL}}

\newcommand{\AaL}{\operatorname{A\Gamma L}}

\newcommand{\GammaL}{\operatorname{\Gamma L}}

\newcommand{\Suz}{\operatorname{Suz}}

\newcommand{\Sz}{\operatorname{Sz}}

\newcommand{\Ree}{\operatorname{Ree}}

\newcommand{\ASL}{\operatorname{ASL}}

\newcommand{\ASp}{\operatorname{ASp}}

\newcommand{\PG}{\operatorname{PG}}

\newcommand{\Sp}{\operatorname{Sp}}

\newcommand{\Sym}{\operatorname{Sym}}

\newcommand{\Alt}{\operatorname{Alt}}

\newcommand{\soc}{\operatorname{soc}}

\newcommand{\Stab}{\operatorname{Stab}}

\newcommand{\rank}{\operatorname{rank}}

\newcommand{\dotcup}{\mathbin{\dot\cup}}

\newcommand{\Orth}{\operatorname{O}}

\newcommand{\Arf}{\operatorname{Arf}}

\newcommand{\supp}{\operatorname{supp}}

\newcommand{\srg}{\operatorname{srg}}

\newcommand{\PAut}{\operatorname{PAut}}

\begin{document}

\title{Three-point orbits of $2$-transitive groups and their combinatorial applications}

\author{Zhilin Zhang}
\address{School of Statistics and Data Science, Guangdong University of Finance and Economics, Guangzhou, Guangdong 510320, China}
\email{20241032@gdufe.edu.cn}


\begin{abstract}
Let $G$ be a finite $2$-transitive permutation group.  We determine the orbits of $G$
on the unordered triples of its point set, together with the three-point stabilizer and
its image inside $S_3$, for every affine and almost simple family in the finite
$2$-transitive classification.  These data classify the $G$-invariant triple systems and
the simple $2$-\mbox{$(v,3,\lambda)$} designs admitting $G$, and yield in a uniform way
the standard objects attached to them:  the incidence and block graphs (with
arc-transitivity and strong regularity in the Steiner case), the constant-weight and
design codes, and the Schurian coherent configurations on triples.  
\end{abstract}

\maketitle

\clearpage

\tableofcontents

\chapter{Introduction}
% ============================================================

% ---------- A. Motivation and history ----------

\noindent\textbf{Prerequisites and notation.}  The monograph is addressed to readers
with a basic training in algebra;  the only background used systematically
is the following, and full definitions, statements and proofs are gathered
in Chapter \ref{sec:prelim} (group-theoretic classification) and Chapter
\ref{sec:design} (applications).

\noindent\textit{Group theory.}  We use the standard elementary language
of group theory:  subgroups, normal subgroups, quotients, simple groups,
composition series, and the classical linear groups over finite fields.

\noindent\textit{Permutation groups.}  A permutation group is a subgroup
of the symmetric group on a finite set;  group actions, orbits,
stabilizers, the orbit--stabilizer theorem, transitivity, primitivity and
$k$-homogeneity are used throughout, and $2$-transitivity is our standing
hypothesis.

\noindent\textit{Finite fields and number theory.}  The finite field
$\Fq$ ($q$ a prime power) underlies the classical groups that form the
main source of $2$-transitive groups;  divisibility and congruence
arguments enter through the integrality of the orbit parameters.

\noindent\textit{Graph theory.}  Incidence graphs, block graphs and
strongly regular graphs occur in the combinatorial applications.

\noindent\textit{Coding theory.}  Linear and constant-weight codes over
finite fields --- length, dimension, weight distribution, minimum
distance --- are used in the coding applications.

\noindent\textit{Association schemes.}  The orbit structure on triples
is analysed through association schemes in the final application.

A \emph{permutation group} is a subgroup $G\le\Sym(\Pcal)$ of the symmetric group on a finite set $\Pcal$ of $v$ points; equivalently, it is a finite group equipped with a faithful action on $\Pcal$. The action is \emph{transitive} if one point, hence every point, can be carried to any other. It is \emph{$k$-transitive} if $G$ is transitive on the ordered $k$-tuples of pairwise distinct points, and \emph{$k$-homogeneous} if it is transitive on the unordered $k$-subsets $\binom{\Pcal}{k}$; every $k$-transitive group is $k$-homogeneous, and the converse holds for $k=2$ (and, apart from small exceptions, for $k\ge3$ by Livingstone--Wagner). Our standing hypothesis throughout is \emph{$2$-transitivity}: for any two ordered pairs $(\alpha,\beta)$ and $(\alpha',\beta')$ of distinct points, some $g\in G$ satisfies $g\alpha=\alpha'$, $g\beta=\beta'$, equivalently the point stabilizer $G_\alpha$ is transitive on $\Pcal\setminus\{\alpha\}$.

Permutation groups are the oldest form in which groups were studied: they occur in Galois's analysis of equations, and Jordan's \emph{Trait\'e des substitutions} (1870) gave the first systematic account of multiply transitive groups. Finite multiply transitive groups are rare, and their complete classification was one of the principal consequences of the classification of finite simple groups. Burnside proved that a finite $2$-transitive group has a unique minimal normal subgroup, which is either elementary abelian (the \emph{affine type}) or nonabelian simple (the \emph{almost simple type}); the affine groups with a regular normal subgroup were classified by Hering, and the almost simple $2$-transitive actions by Curtis, Kantor, Seitz and their successors. The resulting finite $2$-transitive groups, listed precisely in Chapter \ref{sec:prelim}, form a short explicit catalogue. They occupy a privileged position: they are the most symmetric of all nontrivial actions, and they are the natural automorphism groups of highly regular combinatorial structures.

Group actions and combinatorial structures meet most directly in the problem of orbits on subsets, and under the hypothesis of $2$-transitivity the first two orders of this action carry no information: there is a single point-orbit and a single orbital, the unique orbital graph is the complete graph $K_v$, and the Schurian permutation scheme on $\Pcal$ is the trivial scheme of class $1$. Genuinely non-trivial structure begins exactly at the \emph{third} order, namely the orbits of $G$ on the unordered triples $\binom{\Pcal}{3}$, together with the stabilizer of a triple. These three-point orbits are the single unifying object of this monograph.

The third order has been a recurring theme of combinatorics and group theory ever since Kirkman \cite{kirkman1847} and Steiner \cite{steiner1853} introduced triple systems in the 1840s--50s. Witt \cite{witt1938} built the Mathieu groups around Steiner systems, and their three-point actions were among the first triple orbits to be studied. Livingstone and Wagner \cite{livingstone1965} proved the fundamental monotonicity result, that a transitive group has at least as many orbits on $k$-sets as on $(k-1)$-sets; consequently, for a $2$-transitive group a single triple orbit occurs exactly in the $3$-homogeneous case. In more recent decades triple orbits have become a standard construction kit. Key and Moori \cite{keymoori1992} derive designs, codes and graphs from the sporadic groups through their triple actions (see also Assmus and Key \cite{assmuskey1992}); Mesner and Bhattacharya \cite{mesner1990} organized the triple-orbit data into association schemes on triples, and these schemes were computed for all known $2$-transitive families by Balmaceda and Briones \cite{balmaceda2021}. Triple-orbit counts are the raw input of the flag-transitive Steiner-design classifications of Huber \cite{huber2005,huber2007,huber2009book} \textup{(Steiner $3$-designs in \cite{huber2005}, $4$-designs in \cite{huber2007}, and the survey book \cite{huber2009book})}, and the flag-transitive $2$-\mbox{$(v,3,\lambda)$} designs with $\lambda>1$ were classified in a parameter-restricted range by Francot and Montinaro \cite{francot2024} \textup{($v\equiv1,3\pmod6$ and $v\equiv6\pmod\lambda$, excluding subgroups of $A\Gamma L(1,v)$)}.  The $2$-$(v,3,2)$ case was classified by Devillers, Liang, Praeger and Xia \cite{devillers2020}, whose examples are the $\PSL(n,3)$ families and the complement of the Fano plane, and the flag-transitive $2$-designs with $\gcd(r,\lambda)=1$ and unitary socle $\PSU(n,q)$ were classified by Alavi, Bayat and Daneshkhah \cite{alavi2019}, the Hermitian unital being one of the two infinite families of that classification.  The invariant designs treated here are not required to be flag-transitive, so these classifications are complementary:  the orbit-unions of a $2$-transitive group are all recorded, flag-transitive or not.  The modern spreading/separating hierarchy of Bamberg, Freedman and Giudici \cite{bamberg2023} evaluates permutation groups by how far their triple action is from being transitive. What runs through this history is a single observation: designs, graphs, codes and association schemes all read their data from the orbit parameters $(\lambda_i)$, the stabilizer orders, and the two-point stabilizer suborbits, and nothing more. Once the three-point orbit classification is known, the graph, code and scheme data follow without further group-theoretic work; the purpose of this monograph is to carry out that classification and then to make each of these transfers explicit and complete.

% ---------- B. General theory (Chapter 2) ----------

% ---------- (moved to Chapter 3) ----------

Every automorphism group $G\le\Sym(\Pcal)$ of such a design decomposes $\binom{\Pcal}{3}$ into orbits, and each orbit is a triple system whose automorphism group contains $G$. When $G$ is $2$-transitive the orbit decomposition is governed entirely by the setwise stabilizer of an arbitrary triple,
\begin{equation}
G_{\Ocal} = \{g\in G : g(\Ocal)=\Ocal\},
\qquad \Ocal = \{x,y,z\}\in \binom{\Pcal}{3}.
\end{equation}
Indeed the number of blocks in the orbit of $\Ocal$ is $|G|/|G_{\Ocal}|$, and if this orbit is a $2$-design with parameters $(v,3,\lambda)$ then
\begin{equation}
\lambda = \frac{6|G|}{|G_{\Ocal}|\, v(v-1)}, \qquad v=|\Pcal|.
\label{eq:lambda}
\end{equation}
The stabilizer $G_{\Ocal}$ sits in the canonical exact sequence
\begin{equation}
1\longrightarrow G_{(\Ocal)}\longrightarrow G_{\Ocal}\longrightarrow \overline{G}_{\Ocal}\le \Sym(\Ocal)\cong S_3\longrightarrow 1,
\label{eq:ses}
\end{equation}
where $G_{(\Ocal)}$ is the \emph{pointwise} stabilizer of the three points and $\overline{G}_{\Ocal}$ is the \emph{stabilizer image}:  the image $\overline{G}_{\Ocal}=\rho_{\Ocal}(G_{\Ocal})$ of the setwise stabilizer $G_{\Ocal}$ under the permutation action on the three points, a subgroup of $S_3$, hence one of $1, C_2, C_3, S_3$. Thus the orbit data are determined by two numbers for each orbit: the order of the pointwise stabilizer (the kernel) and the isomorphism type of the stabilizer image in $S_3$. Throughout, ``stabilizer image'' means $\overline{G}_{\Ocal}$ (also written $\overline{G}_X$ for a triple $X$), to be distinguished from the full setwise stabilizer $G_{\Ocal}$, which also carries the kernel $G_{(\Ocal)}$;  the two are never conflated.

The individual orbits are not the only invariant designs. Suppose $\binom{\Pcal}{3}=\Tcal_1\dotcup\cdots\dotcup\Tcal_r$ is the orbit decomposition. Each $\Tcal_i$ is a $2$-design (\textup{Lemma \ref{lem:lambda}(2)}); because the orbits are disjoint, every nonempty union
\[
\Tcal_I=\bigcup_{i\in I}\Tcal_i,\qquad \varnothing
\ne I\subseteq\{1,\ldots,r\},
\]
is again a simple $2$-design with $\lambda_I=\sum_{i\in I}\lambda_i$ (\textup{Proposition \ref{prop:union-orbits}}). Conversely, every nonempty $G$-invariant simple triple system is a union of these orbits (\textup{Corollary \ref{cor:all-invariant}}). The classification therefore has two outputs: the three-point orbit decomposition with stabilizer data, and the complete family of invariant triple systems obtained from nonempty unions. This two-layer structure is the organizing principle of the monograph and is part of \textup{Theorem \ref{thm:main}}.

The input to the whole analysis is the classification of finite $2$-transitive groups (see \textup{Chapter \ref{sec:prelim}} for the precise statement, assembled from Hering \cite{hering1985}, Huppert \cite{huppert1967}, Curtis--Kantor--Seitz \cite{curtis1976} and the surveys \cite{cameron1981,cameron1999}):
\begin{enumerate}
\item \textbf{Affine type.}  $\soc(G)=V$ is elementary abelian of order $p^d$, and $G=V\rtimes G_0$ with $G_0\le\GL(d,p)$ transitive on the nonzero vectors (Hering's theorem \cite{hering1985}).
\item \textbf{Almost simple type.}  $\soc(G)=T$ is a nonabelian simple group with $T\le G\le\Aut(T)$ acting $2$-transitively: $T=\PSL(n,q)$ ($n=2$ on the projective line, $n\ge3$ on projective space), $\Sp(2n,2)$ ($n\ge3$, on the two halves of the nonsingular quadric), $\PSU(3,q)$, $\Sz(q)$, $\Ree(q)$ (i.e.\ $^2G_2(q)$), $\Alt(n)$, or a sporadic group (the Mathieu groups together with $HS$ and $Co_3$), with $G$ contained in the corresponding automorphism group.
\end{enumerate}
Chapter \ref{sec:prelim} also develops the three computational devices used throughout: the suborbit reduction (a surjection from the suborbits of the two-point setwise stabilizer onto the triple orbits, which is a bijection only after the injectivity condition (equivalently that each triple stabilizer contain a $3$-cycle) is checked), the two-point-stabilizer orbit count, and the Burnside (class-function) count, which gives the number of three-point orbits as the average, over the group, of the number of triples fixed by an element:  a triple is $g$-invariant exactly when it is a union of $g$-cycles of total size $3$, whence $n(g)=\binom{c_1(g)}{3}+c_1(g)c_2(g)+c_3(g)$ and $r=\frac{1}{|G|}\sum_{g\in G}n(g)$.

% ---------- C. Principal application: triple systems (Chapter 3) ----------

The first and principal application is the classification of $2$-transitive triple systems. The classical starting point is the Steiner case $\lambda=1$, settled by M.~Hall \cite{hall1985} (the affine case being due to J.~I.~Hall, with independent results by Kantor and Shult; see \cite{hall1985}): a Steiner triple system with a $2$-transitive automorphism group is an affine space $\mathrm{AG}(r,3)$ or a projective space $\mathrm{PG}(s,2)$, apart from the single exception of $\Alt(7)$ on the $15$ points of $\mathrm{PG}(3,2)$. The systems occurring within our families are recorded in \textup{Remark \ref{rem:sts}}. The present monograph goes beyond the Steiner case in two directions: it treats the non-Steiner systems $\lambda\ge2$, and it records the full orbit and stabilizer data for every family rather than the designs alone.

The purpose of the monograph is thus to give the orbit--design framework and the family-specific results, stated with their proof status, for the specified $2$-transitive groups $G$, together with the structure of $G_{\Ocal}$ for every triple in the relevant families. This is achieved by combining the classification of finite $2$-transitive groups with a detailed case analysis; the key point at every step is the structure of the stabilizer $G_{\Ocal}$ and the solution set of the equations defining it. After translating a triple to $\{0,x,y\}$ (affine type) or normalizing it on a line, a unital or an ovoid (almost simple type), the stabilizer is described by a small number of linear conditions in the point stabilizer, and the answer depends only on elementary field arithmetic (squares in $\Fq$, primitive cube roots, whether $-1$ is a square, and so on). The family-specific data are complete in the parameter ranges explicitly identified as proved or computed; the proof status of every family (proved for all parameters, computed for an explicit range, conditional, or not treated) is recorded in \textup{Remark \ref{rem:open}}, and the finite computations are collected in \textup{Chapter \ref{sec:verify}}.

Our main result is the following. The general assertions (1)--(3) hold for every finite $2$-transitive permutation group; the family-specific data of (3) are asserted only in the ranges covered by a proof or by a stated computation.

\begin{theorem}[Main classification]
\label{thm:main}
Let $G$ be a finite $2$-transitive permutation group on a set $\Pcal$ of $v$ points, and write
\[
\binom{\Pcal}{3}=\Tcal_1\dotcup\cdots\dotcup\Tcal_r
\]
for the orbit decomposition of $G$ on three-point subsets. For each $i$ choose $\Ocal_i\in\Tcal_i$. The general assertions (1)--(2) below hold for every such $G$; the family-specific data of (3) are asserted only in the ranges explicitly covered by a proof or by a stated computation \textup{(Remarks \ref{rem:proofstatus} and \ref{rem:open})}.
\begin{enumerate}
\item The orbit lengths are $|\Tcal_i|=|G|/|G_{\Ocal_i}|$, and the orbit parameter $\lambda_i=6|G|/\bigl(|G_{\Ocal_i}|\,v(v-1)\bigr)$ is a positive integer \textup{(Lemma \ref{lem:lambda})}.
\item The setwise stabilizers fit into the exact sequence $1\to G_{({\Ocal_i})}\to G_{\Ocal_i}\to\overline{G}_{\Ocal_i}\le S_3\to1$ \textup{(\eqref{eq:ses})}; the stabilizer image is one of $1,C_2,C_3,S_3$.
\item \textbf{The family-by-family data.}  For the families and parameter ranges covered by a proof or by an explicit computation, the orbit decomposition, the stabilizer structure and the orbit parameters $\lambda_i$ are collected in \textup{Table \ref{tab:main}}, each row being proved analytically or verified computationally in its stated range and tagged accordingly; the proof status of each row is recorded in \textup{Remarks \ref{rem:proofstatus} and \ref{rem:open}}.  Rows with a single orbit are precisely the $3$-homogeneous groups \textup{(Remark \ref{rem:3hom})}.
\end{enumerate}
\end{theorem}

The combinatorial applications of this classification are developed in \textup{Chapter \ref{sec:design}}:  the invariant triple systems (every nonempty union of orbits, \textup{Propositions \ref{prop:union-orbits} and \ref{cor:all-invariant}}), the graph-theoretic consequences (the incidence graph and its arc-transitivity, \textup{Theorems \ref{thm:incidence} and \ref{thm:arc}};  the block graph, regular in general and strongly regular in the Steiner case, \textup{Theorem \ref{thm:stblock}}), the coding applications (the constant-weight orbit codes meeting the Johnson bound, \textup{Theorems \ref{thm:cwdist} and \ref{thm:johnson}};  the rank and minimum weight of the design codes, \textup{Theorems \ref{thm:descode} and \ref{thm:module}}), and the association schemes on triples (\textup{\S\ref{sec:assoc}}).  In each case the data are read directly from the orbit parameters and the two-point stabilizer suborbits of \textup{Theorem \ref{thm:main}}, with no further group-theoretic input.

\begin{landscape}
\begin{table}[p]
\centering
\caption{Orbit data of \textup{Theorem \ref{thm:main}}: for each family the number $r$ of three-point orbits, the image of the three-point stabilizer in $S_3$ (symbols $S_3,C_3,C_2,1$), and the source.  The replication parameters $\lambda$ of each orbit are collected in \textup{Table \ref{tab:main-lambda}}.  Explicit pointwise-stabilizer orders, orbit-length formulas and the overgroup/field-extension refinements are given in the cited results; rows with $r=1$ are exactly the $3$-homogeneous groups, giving the complete design $2$-$(v,3,v-2)$ (\textup{Remark \ref{rem:3hom}}).  Every row is tagged \emph{proved} (analytically) or \emph{verified} (computationally) in its stated range; the proof status of every row, together with the exact parameter ranges, the software (GAP/Magma) and the reproducible verification scripts, is recorded in \textup{Remark \ref{rem:proofstatus}} and in the verification chapter.}\label{tab:main}
\scriptsize
\renewcommand{\arraystretch}{1.0}
\setlength{\tabcolsep}{5pt}
\begin{tabular}{@{}>{\raggedright\arraybackslash}p{4.4cm}>{\raggedright\arraybackslash}p{2.5cm}>{\raggedright\arraybackslash}p{3.3cm}>{\raggedright\arraybackslash}p{5.5cm}@{}}
\toprule
Group (degree $v$) & $r$ & Image in $S_3$ & Source \\
\midrule
\multicolumn{4}{@{}l}{\textbf{\emph{Affine type}}}\\
\addlinespace[1pt]
$\AGL(1,q)$ $(v{=}q)$ & 1--4 & $1,C_2,C_3$ & \textup{Thm.\ \ref{thm:agl1}} \\
$\AaL(1,q)$ $(v{=}q)$ & 1--4 & per \textup{Thm.\ \ref{thm:agamL1}} & \textup{Thm.\ \ref{thm:agamL1}} \\
$\ASL(n,q)$, $\AGL(n,q)$ $(v{=}q^n)$ & $\frac{q-1}{2}$, $+1$ & $S_3$ & \textup{Thm.\ \ref{thm:affine-collinear}} \\
$\ASp(2a,q)$ $(v{=}q^{2a})$ & 2 or 3 & $S_3$ & \textup{Thm.\ \ref{thm:affine-classical}} \\
$G_2'(q)$ $(v{=}q^6,\ p{=}2)$ & $q{+}1{+}c(q)$ & $C_3,\,1,\,S_3$ & \textup{Thm.\ \ref{thm:g2-classical}} \\
Hering exceptions & computed for the explicitly listed representations & per representation & \textup{Thm.\ \ref{thm:affine-exc}} \\
\addlinespace[2pt]
\midrule
\multicolumn{4}{@{}l}{\textbf{\emph{Almost simple type}}}\\
\addlinespace[1pt]
$\PSL(2,q)$, $q{\ge}4$ $(v{=}q{+}1)$ & 1 or 2 & $S_3$ or $C_3$ & \textup{Thm.\ \ref{thm:psl2}} \\
$\PSL(n,q)$, $n{\ge}3$ $(v{=}(q^n{-}1)/(q{-}1))$ & 2 & $S_3$ & \textup{Thm.\ \ref{thm:psln}} \\
$\Sp(2n,2)$ $(v{=}2^{2n-1}\pm2^{n-1})$ & 2 & $S_3$ & \textup{Thm.\ \ref{thm:sp2n-triples}} \\
$\PSU(3,q)$ $(v{=}q^3{+}1)$ & \scriptsize $(q{+}3)/2\ \textup{proved}$\\[2pt] $(5q{+}11)/6\ \textup{proved}\ (q\ \textup{odd},\ d{=}3)$\\[2pt] $1{+}q/2\ \textup{proved}$\\[2pt] $(5q{+}2)/6\ \textup{proved}\ (q{=}2^{2m{+}1})$ & $S_3,\,C_3,\,1$ & \textup{Thm.~\ref{thm:main-psu} (theoretical ranges); Prop.~\ref{prop:psu-odd} (odd-$d{=}3$); Prop.~\ref{prop:psu-even} (even-$d{=}3$)} \\
$\Sz(q)$ $(v{=}q^2{+}1)$ & $\frac{q+4}{6}$ & $C_2,\,1$ & \textup{Thm.\ \ref{thm:sz-orbits}} \\
$\Ree(q)$ $(v{=}q^3{+}1)$ & $\frac{(q+3)^2}{6}$ & $S_3,\,C_3,\,C_2,\,1$ & \textup{Thm.\ \ref{thm:ree-structure}} \\
$\Alt(n)$, $\Sym(n)$ $(v{=}n)$ & 1 & $S_3$ & \textup{\S\ref{sec:alt}} \\
Mathieu groups & 1 & $S_3$ & \textup{\S\ref{sec:mathieu}} \\
$HS$ $(v{=}176)$, $Co_3$ $(v{=}276)$ & 3, 2 & $S_3$ in every orbit & $HS$: \textup{Thm.\ \ref{thm:spor-fw}, Rem.\ \ref{rem:spor-correction}, Prop.\ \ref{prop:hs2}}; $Co_3$: \textup{Thm.\ \ref{thm:spor-fw}, Prop.\ \ref{prop:co3}} \\
\bottomrule
\end{tabular}
\end{table}

\begin{table}[p]
\centering
\caption{Replication parameters $\lambda$ of the three-point orbits of \textup{Table \ref{tab:main}}, by family and orbit; notation as in \textup{Table \ref{tab:main}}.  For $\PSU(3,q)$ the odd-$d=3$ rows are proved for every $q$ with $3\mid q+1$ (\textup{Proposition \ref{prop:psu-odd}}), as are the even-$d=3$ rows for every $q=2^{2m+1}$ (\textup{Proposition \ref{prop:psu-even}});  the stabilizer refinement of \textup{Proposition \ref{prop:psu-computed}} is computed for $q\le49$.}\label{tab:main-lambda}
\scriptsize
\renewcommand{\arraystretch}{1.0}
\setlength{\tabcolsep}{5pt}
\begin{tabular}{@{}>{\raggedright\arraybackslash}p{4.4cm}>{\raggedright\arraybackslash}p{13.2cm}@{}}
\toprule
Group (degree $v$) & Replication parameters $\lambda$ by orbit \\
\midrule
\multicolumn{2}{@{}l}{\textbf{\emph{Affine type}}}\\
\addlinespace[1pt]
$\AGL(1,q)$ $(v{=}q)$ & $\lambda=6/|G_{\Ocal}|\in\{6,3,2,1\}$; STS$(q)$ iff $q\equiv0\ (3)$ \\
$\AaL(1,q)$ $(v{=}q)$ & $\lambda=6f/|G_{\Ocal}|$ (kernel/image per cited result) \\
$\ASL(n,q)$, $\AGL(n,q)$ $(v{=}q^n)$ & $n{=}2$: $\lambda=q$ (char.\ $2$) or $2q$; $n{\ge}3$: $\lambda=q(q^{n-1}-1)$ \\
$\ASp(2a,q)$ $(v{=}q^{2a})$ & nondeg.: $\lambda=q^{2a-1}$ (char.\ $2$) or $2q^{2a-1}$; isotr.: $\lambda=q^{2a-1}-2$ ($q=2$) \\
$G_2'(q)$ $(v{=}q^6,\ p{=}2)$ & collinear: one $\lambda=2$ (if $3\mid q-1$), remaining $\lambda=6$; hyperbolic: $q-1$ orbits $\lambda=q^5$; isotropic: $\lambda=q(q^2-1),\,q^3(q^2-1)$ \\
Hering exceptions & $\lambda\in\{6,3,2\},\{8,6\},14,\{12,1\},\{32,24,6\},\{18,6,1\}$; nearfields $p=7,19,23,29,59$ \\
\addlinespace[2pt]
\midrule
\multicolumn{2}{@{}l}{\textbf{\emph{Almost simple type}}}\\
\addlinespace[1pt]
$\PSL(2,q)$ $(v{=}q{+}1)$ & $\lambda=q-1$ ($q$ even or $q\equiv3\ (4)$); $\lambda=(q-1)/2$ ($q\equiv1\ (4)$) \textup{(for $q\equiv3\pmod4$ the stabilizer image is $C_3$, cf.\ Table \ref{tab:main})} \\
$\PSL(n,q)$, $n{\ge}3$ $(v{=}(q^n{-}1)/(q{-}1))$ & collinear $\lambda=q-1$; non-collinear $\lambda=q^2$ ($n=3$) or $v-1-q$ ($n\ge4$) \\
$\Sp(2n,2)$ $(v{=}2^{2n-1}\pm2^{n-1})$ & $\lambda=2^{2n-2}+s2^{n-1}-2$ and $\lambda=2^{2n-2}$ ($s=\pm1$) \\
$\PSU(3,q)$ $(v{=}q^3{+}1)$ & collinear $\lambda=q-1$; non-collinear: $q$ odd, $d{=}1$: $1$ orbit $\lambda=(q^2-1)$ \textup{($S_3$)} and $\frac{q-1}{2}$ orbits $\lambda=2(q^2-1)$ \textup{($C_3$)} \textup{(for $q{=}3$: the three orbits are $\lambda{=}2$ (collinear, $S_3$), $8$ ($S_3$) and $16$ ($C_3$), with $\sum\lambda=v-2=26$)}; $q$ odd, $d{=}3$: $3$ orbits $\lambda=(q^2-1)/3$ \textup{($S_3$)}, $\frac{q-5}{2}$ orbits $\lambda=2(q^2-1)/3$ \textup{($C_3$)} and $\frac{q+1}{3}$ orbits $\lambda=2(q^2-1)$ \textup{(image $1$)}; $q$ even, $d{=}1$: $\frac{q}{2}$ merged orbits $\lambda=2(q^2-1)$ \textup{($C_3$)}; $q$ even, $d{=}3$: $\frac{q-2}{2}$ orbits $\lambda=2(q^2-1)/3$ and $\frac{q+1}{3}$ orbits $\lambda=2(q^2-1)$ \textup{(proved)} \\
$\Sz(q)$ $(v{=}q^2{+}1)$ & circle $\lambda=3(q-1)$; generic $\frac{q-2}{6}$ orbits $\lambda=6(q-1)$ \\
$\Ree(q)$ $(v{=}q^3{+}1)$ & $\lambda=q-1$ (2 orbits), $2(q-1)$ ($\frac{q-1}{2}$), $3(q-1)$ ($q$), $6(q-1)$ ($\frac{q(q-3)}6$) \\
$\Alt(n)$, $\Sym(n)$ $(v{=}n)$ & $\lambda=n-2$ (3-hom., $n\ge5$) \\
Mathieu groups & $\lambda=v-2$ (3-hom.) \\
$HS$ $(v{=}176)$, $Co_3$ $(v{=}276)$ & $HS$: $\lambda=72,90,12$; $Co_3$: $\lambda=162,112$ \\
\bottomrule
\end{tabular}
\end{table}

\end{landscape}

% ---------- D. Other applications (Chapter 4) ----------

The same orbit data support three further families of combinatorial structures, taken up in \textup{Chapter \ref{sec:other}}. First, the orbit data determine two canonical graphs: the incidence graph of the design (points versus blocks), whose bipartite spectrum we compute and whose arc-transitivity we analyse in the flag-transitive case, and the block graph on the blocks, which is regular in general and strongly regular in the Steiner ($\lambda=1$) case, with the edge-layer decomposition for $\lambda\ge2$. Second, the block sets give codes: the constant-weight code of weight $3$ carried by the blocks (with its distance distribution), and the design code spanned by the point rows of the incidence matrix (with its $2$-rank, its Gram-matrix identities, and its natural $G$-invariance under the induced action on the block coordinates, \textup{Theorem \ref{thm:module}}\textup{)}. Third, the orbit decomposition induces the Schurian coherent configurations of $G$ on triples, whose class numbers are the numbers of three-point orbits;  in the homogeneous cases these yield association schemes \textup{(for a Schurian coherent configuration, homogeneity is the only additional condition;  see \textup{Section \ref{sec:assoc}})}, and we give the valencies (equivalently the suborbit arrays) and, in those cases, a complete association-scheme description;  the main table is extended with the association-scheme class numbers over the families for which they are known. In each case the graph, code or scheme is a direct consequence of the orbit data, with no further group-theoretic input.

% ---------- E. Computational verification (Chapter 5) ----------

The finite computations of \textup{Chapter \ref{sec:verify}} serve two purposes: they check internal identities such as the orbit--stabilizer relation and the kernel/image orders, and they compare computed orbit data with the formulas stated in the monograph. They are consistency checks and do not by themselves prove the family-wide formulas; every claim supported only by finite computation is stated with its exact degree and group range. The verification covers the unitary suborbit census for $q\le27$, the independent Ree class computation for $^2G_2(27)$ (which reads the cycle types from the Atlas representation rather than enumerating triples), the Suzuki orbit-merge computations for $\Sz(32)$ and $\Sz(32){:}5$, the dedicated suborbit computations for $HS$ and $Co_3$, the degree-$729$ row $3^6{:}\SL(2,13)$, and the $G_2(4)$ stabilizer computation for the affine row $2^{12}{:}G_2(4)$; the detailed contents and the exact degree and group range of every computational claim are recorded in \textup{Chapter \ref{sec:verify}}.

\noindent\textbf{Global consistency checks.}  The following two identities serve as global consistency checks linking the abstract framework of this chapter with the family-by-family classification of \textup{Chapter \ref{sec:prelim}} and the computations of \textup{Chapter \ref{sec:verify}}:
\[
\sum_{i=1}^{r}|\Tcal_i|=\binom{v}{3}, \qquad \sum_{i=1}^{r}\lambda_i=v-2,
\]
the first being the orbit decomposition itself and the second the double count of point pairs over all three-point subsets;  both are verified computationally, row by row, in \textup{Chapter \ref{sec:verify}}.

% ---------- F. Roadmap ----------

The monograph is organized as follows. \textup{Chapter \ref{sec:prelim}} develops the classification of three-point orbits of $2$-transitive groups: the notation and the kernel--image exact sequence, the classification of finite $2$-transitive groups, the suborbit, two-point-stabilizer and class-function (Burnside) counts, and the complete family-by-family orbit and stabilizer data across the affine and almost simple types, together with the parallel analysis of the $2$-homogeneous groups that are not $2$-transitive (\textup{\S\ref{sec:2homog}}). \textup{Chapter \ref{sec:design}} develops the combinatorial applications: the invariant triple systems and their parameters, the incidence and block graphs (\textup{\S\ref{sec:graphs}}), the constant-weight and design codes (\textup{\S\ref{sec:coding}}), and the association schemes on triples (\textup{\S\ref{sec:assoc}}); the main classification is the \textup{Theorem \ref{thm:main}} stated in this chapter. \textup{Chapter \ref{sec:verify}} reports the computational verification, and \textup{Chapter \ref{sec:concl}} concludes with a summary of the results and of the transfer to the applications."

\chapter{Three-point orbits of \texorpdfstring{$2$}{2}-transitive groups: classification}

\label{sec:prelim}
% ============================================================

\section{Notation and the basic exact sequence}

Throughout, $\mathcal{D}=(\Pcal,\bcal)$ is a triple system (simple $2$-design with block size $3$) and $G\le\Sym(\Pcal)$ its automorphism group, acting $2$-transitively on $\Pcal$ with $v=|\Pcal|\ge 4$.  For a subset $X\subseteq\Pcal$ we write
\begin{enumerate}
\item $G_X=\{g\in G: g(X)=X\}$ (setwise stabilizer),
\item $G_{(X)}=\{g\in G: g(x)=x\ \text{for all }x\in X\}$ (pointwise stabilizer),
\item $G^X$ for the image of $G_X$ in $\Sym(X)$.
\end{enumerate}
For $\Ocal=\{x,y,z\}$ a three-point subset, the sequence \eqref{eq:ses} reads
$1\to G_{(\Ocal)}\to G_{\Ocal}\to \overline{G}_{\Ocal}\to 1$ with $\overline{G}_{\Ocal}\le S_3$.

\begin{lemma}[Basic counting]
\label{lem:lambda}
Let $G$ be $2$-transitive on $\Pcal$, $|\Pcal|=v$, and let $\mathcal{T}$ be an orbit of $G$ on $\binom{\Pcal}{3}$, $\Ocal\in\mathcal{T}$.
\begin{enumerate}
\item $|\mathcal{T}|=|G|/|G_{\Ocal}|$.
\item The pair $\{x,y\}\subseteq\Pcal$ lies in exactly $\lambda=6|G|/\bigl(|G_{\Ocal}|\,v(v-1)\bigr)$ blocks of $\mathcal{T}$;  since $G$ is $2$-transitive on pairs, this number is independent of the pair, and $(\Pcal,\mathcal{T})$ is automatically a simple $2$-$(v,3,\lambda)$ design whose automorphism group contains $G$;  in particular the displayed expression for $\lambda$ is an integer;  indeed $\lambda$ is by definition the number of blocks of $\mathcal{T}$ through a given pair, a cardinality, hence an integer, and the displayed formula merely re-expresses it.
\item Distinct $G$-orbits yield distinct block sets, although the resulting designs may be abstractly isomorphic.
\end{enumerate}
\end{lemma}
\begin{proof}
Fix an unordered pair $P\in\binom{\Pcal}{2}$ and set $b(P)=|\{B\in\mathcal{T}: P\subseteq B\}|$.  Since $G$ is $2$-transitive, for every $Q\in\binom{\Pcal}{2}$ there is $g\in G$ with $g(P)=Q$;  as $\mathcal{T}$ is $G$-invariant, $B\mapsto g(B)$ is a bijection between the blocks of $\mathcal{T}$ through $P$ and those through $Q$, so $b(P)=b(Q)=:\lambda$ is independent of the pair.
(1) is the orbit--stabilizer theorem.
(2) Count the incidences $(P,B)$ with $P\in\binom{\Pcal}{2}$, $B\in\mathcal{T}$, $P\subseteq B$ in two ways:  each block contains $\binom{3}{2}=3$ pairs, and each of the $\binom{v}{2}$ pairs lies in exactly $\lambda$ blocks, hence $\lambda\binom{v}{2}=3|\mathcal{T}|$, i.e.\ $\lambda=3|\mathcal{T}|/\binom{v}{2}=6|\mathcal{T}|/(v(v-1))=6|G|/\bigl(|G_{\Ocal}|\,v(v-1)\bigr)$, an integer.  Simplicity is immediate from $\mathcal{T}\subseteq\binom{\Pcal}{3}$.
(3) Distinct $G$-orbits are disjoint as block sets, although the resulting designs may be abstractly isomorphic.
\end{proof}

\begin{definition}[Orbit triple system]
\label{def:orbit-ts}
Let $G\le\Sym(\Pcal)$ be $2$-transitive on $\Pcal$.  For $\Ocal\in\binom{\Pcal}{3}$ we call the $G$-orbit $\Tcal(\Ocal)=\Ocal^G$ an \emph{orbit triple system} and write $\mathcal{D}(\Ocal)=(\Pcal,\Tcal(\Ocal))$ for the pair consisting of the point set and the orbit as block set.
\end{definition}

\begin{proposition}[Every orbit is a triple system]
\label{prop:each-orbit}
Let $G$ be $2$-transitive on $\Pcal$, $v=|\Pcal|$, and $\Tcal$ a $G$-orbit on $\binom{\Pcal}{3}$ with representative $\Ocal$.  Then $(\Pcal,\Tcal)$ is an orbit triple system \textup{(}\textup{Definition \ref{def:orbit-ts}}\textup{)}, a simple $2$-$(v,3,\lambda_\Tcal)$ design with
\[
\lambda_\Tcal=\frac{6|\Tcal|}{v(v-1)}=\frac{6|G|}{|G_{\Ocal}|\,v(v-1)}
\]
\textup{(}\textup{Lemma \ref{lem:lambda}\textup{(}2\textup{)}}\textup{)}.
\end{proposition}

\begin{proposition}[Unions of orbits are triple systems]
\label{prop:union-orbits}
Let $\binom{\Pcal}{3}=\Tcal_1\dotcup\cdots\dotcup\Tcal_r$ be the orbit decomposition of $G$ on three-point subsets, and let $\lambda_i$ be the parameter of $\Tcal_i$ \textup{(Proposition \ref{prop:each-orbit})}.  For every nonempty $I\subseteq\{1,\ldots,r\}$ the union $\Tcal_I=\bigcup_{i\in I}\Tcal_i$ defines a simple $G$-invariant $2$-$(v,3,\lambda_I)$ design with
\[
\lambda_I=\sum_{i\in I}\lambda_i.
\]
Conversely, every nonempty $G$-invariant family of three-point subsets is a union of $G$-orbits, hence is uniquely of the form $\Tcal_I$.
\end{proposition}
\begin{proof}
Each $\Tcal_i$ is a $2$-design by \textup{Proposition \ref{prop:each-orbit}}.  For a point pair $\{x,y\}$ the number of blocks of $\Tcal_I$ through it is $\sum_{i\in I}\lambda_i$, independent of the pair, because the orbits are disjoint;  simplicity is immediate.  For the converse, if $\mathcal B\subseteq\binom{\Pcal}{3}$ is nonempty and $G$-invariant, then for every $B\in\mathcal B$ its orbit $B^G$ is contained in $\mathcal B$, so $\mathcal B$ is a union of $G$-orbits;  the orbit decomposition is unique, hence the index set $I$ is unique.
\end{proof}

\begin{corollary}[All $G$-invariant triple systems]
\label{cor:all-invariant}
Let $G$ be $2$-transitive on $\Pcal$ with $r$ orbits on $\binom{\Pcal}{3}$.  The nonempty $G$-invariant simple triple systems on $\Pcal$ are exactly $\mathcal{D}_I=(\Pcal,\Tcal_I)$, $\varnothing
\ne I\subseteq\{1,\ldots,r\}$, with parameter $\lambda_I=\sum_{i\in I}\lambda_i$;  there are exactly $2^r-1$ of them.  Distinct unions are distinct block sets even when their parameters coincide;  $G\le\Aut(\mathcal{D}_I)$ for every $I$, with equality not asserted in general;  two different groups may yield the same block set, and two different choices of $I$ may yield isomorphic designs.  The classification is thus of $G$-invariant block sets together with their parameters, not necessarily a classification up to abstract design isomorphism.
\end{corollary}

\begin{remark}[The $3$-homogeneous case]
\label{rem:3hom}
If $G$ is $3$-homogeneous on $\Pcal$, then $r=1$:  the unique orbit $\mathcal{T}$ is $\binom{\Pcal}{3}$ itself, and by \textup{Proposition \ref{prop:each-orbit}} the only $G$-invariant triple system on $\Pcal$ is the complete design $2$-$(v,3,v-2)$, whose blocks are all three-point subsets and whose parameter is therefore fixed, $\lambda=v-2$.  Nontrivial orbit structures, that is, all triple systems other than the complete one, occur only for $2$-transitive groups that are not $3$-homogeneous, and the family-by-family analysis below concentrates on those.  The $3$-homogeneous groups occurring in \textup{Table \ref{tab:main}} are exactly its single-orbit rows:  $\Alt(n)$ and $\Sym(n)$ \textup{(}$\lambda=n-2$\textup{)}, the Mathieu groups \textup{(}$\lambda=v-2$\textup{)}, $\PGL(2,q)$ and $\PpGL(2,q)$ \textup{(}$\lambda=q-1$\textup{)}, and $\PSL(2,q)$ with $q$ even or $q\equiv3\pmod4$ \textup{(}$\lambda=q-1$;  the case $q\equiv3\pmod4$ is $3$-homogeneous but not $3$-transitive, \textup{Theorem \ref{thm:psl2}\textup{(}2\textup{)}}\textup{)}.  In addition there are three exceptional affine $1$-dimensional groups, each with a single orbit:  $\AGL(1,8)$ and $\AaL(1,8)$ on $8$ points, and $\AaL(1,32)$ on $32$ points \textup{(}verified directly in \textup{Table \ref{tab:verify}}:  the setwise stabilizer of every triple is trivial, so the single orbit has the size $\binom{8}{3}=56$ resp.\ $\binom{32}{3}=4960$\textup{)}.  Thus the complete list of $3$-homogeneous groups in the classification is the one just given;  in particular no other $\AGL(1,q)$ or $\AaL(1,q)$ is $3$-homogeneous (e.g.\ $\AGL(1,7)$ has two orbits, \textup{Table \ref{tab:verify}}).  The computation of \textup{Chapter \ref{sec:verify}} confirms that among the primitive groups tested in the stated degree range $3\le v\le200$ the $2$-transitive groups with a single three-point orbit are precisely these.
\end{remark}

\begin{remark}
When $G$ is sharply $2$-transitive, $|G_{(\Ocal)}|=1$ and $G_{\Ocal}\cong\overline{G}_{\Ocal}$.  In general the kernel $G_{(\Ocal)}$ can be large; determining both $G_{(\Ocal)}$ and $\overline{G}_{\Ocal}$ \textup{(}hence whether $G_{\Ocal}$ contains elements realizing transpositions, $3$-cycles, or all of $S_3$\textup{)} is the main technical content of this monograph.
\end{remark}

\section{The classification of finite \texorpdfstring{$2$}{2}-transitive groups}

We state the classification in the form we shall use.  Part (a) is due to Hering \cite{hering1985} (building on Huppert \cite{huppert1967}); part (b) to Curtis--Kantor--Seitz \cite{curtis1976} (completing earlier work of Maillet and others); the sporadic list is the classical one, cf.\ \cite{cameron1999,dixon1996}.

\begin{theorem}[Classification of finite $2$-transitive groups]
\label{thm:2trans}
Let $G$ be a finite $2$-transitive permutation group of degree $v\ge4$ with socle $S=\soc(G)$.  The list below is a list of permutation representations; accidental isomorphisms between socles (such as $\PSL(3,2)\cong\PSL(2,7)$) do not constitute additional actions.
\begin{enumerate}
\item \textbf{\textup{(}Affine type\textup{)}} $S=V$ is elementary abelian of order $p^d$ \textup{(}$p$ prime\textup{)}, $G=V\rtimes G_0$, $G_0\le\GL(d,p)$ transitive on $V\setminus\{0\}$.  By Hering's theorem \cite{hering1985} one of the following holds, with $p^d=q^a$ etc.\ as indicated:
  \begin{enumerate}
  \item $G_0\le \GammaL(1,p^d)$ \textup{(}soluble type; $G$ contains a sharply $2$-transitive subgroup; the Huppert classification applies\textup{)};
  \item $\SL(a,q)\trianglelefteq G_0\le\GammaL(a,q)$ with $p^d=q^a$, $a\ge2$;
  \item $\Sp(2a,q)\trianglelefteq G_0\le\Gamma\Sp(2a,q)$ with $p^d=q^{2a}$, $a\ge2$;
  \item $G_2'(q)\trianglelefteq G_0$ with $p=2$, $p^d=q^6$ \textup{(}where $G_2'(2)\cong\PSU(3,3)$, and $G_2'(q)=G_2(q)$ for $q>2$\textup{)};
  \item the exceptional cases with $G_0$ one of the groups in \textup{Table \ref{tab:hering}}, acting on $V=\Fq^d$.
  \end{enumerate}
\item \textbf{\textup{(}Almost simple type\textup{)}} $S=T$ is nonabelian simple and $T\le G\le\Aut(T)$.  The possibilities are:
  \begin{enumerate}
  \item $T=\PSL(2,q)$ \textup{(}$q\ge4$\textup{)}, acting on the $q+1$ points of $\mathrm{PG}(1,q)$, and $G\le\PpGL(2,q)$;
  \item $T=\PSL(n,q)$ \textup{(}$n\ge3$, $q$ a prime power\textup{)}, acting on the $v=(q^n-1)/(q-1)$ points of the projective space $\mathrm{PG}(n-1,q)$, and $G\le\PpGL(n,q)$ \textup{(}the extension by the diagonal and field automorphisms\textup{)};  this action is $2$-transitive since the stabilizer of a point is \textup{(}the projective image of\textup{)} a parabolic subgroup, which is transitive on the remaining points;
  \item $T=\Sp(2n,2)=\Omega(2n+1,2)$ \textup{(}$n\ge3$\textup{)}, acting $2$-transitively on the $v=2^{2n-1}\pm 2^{n-1}$ points of each of the two halves of the nonsingular quadric \textup{(}the two orbits of $\Omega(2n+1,2)$ on the nonsingular points of $\mathrm{PG}(2n,2)$\textup{)};
  \item $T=\PSU(3,q)$ \textup{(}$q\ge3$\textup{)}, acting on the $q^3+1$ points of the Hermitian unital, and $G\le\PpGU(3,q)$;
  \item $T=\Sz(q)$ \textup{(}$q=2^{2m+1}>2$\textup{)}, acting on the $q^2+1$ points of the Suzuki ovoid, and $G\le\Sz(q)\cdot(2m+1)$;
  \item $T=\Ree(q)=^2G_2(q)$ \textup{(}$q=3^{2m+1}>3$\textup{)}, acting on the $q^3+1$ points of the Ree unital, and $G\le\Ree(q)\cdot(2m+1)$;
  \item $T=\Alt(n)$ \textup{(}$n\ge5$\textup{)}, natural action, $G=\Alt(n)$ or $G=\Sym(n)$ \textup{(}$n=6$ also allowed:  the exceptional outer automorphism of $\Alt(6)$ does not give a $2$-transitive degree-$6$ action with socle $\Alt(6)$ different from the natural one\textup{)};
  \item $T$ sporadic:  $T\in\{M_{11},M_{12},M_{22},M_{23},M_{24}\}$ \textup{(}with the standard actions of degrees $11,12,22,23,24$\textup{)}, or $T\in\{HS,Co_3\}$ with the $2$-transitive actions of degrees $176,276$ respectively; for the Mathieu groups $M_{12},M_{22}$ the group $G$ may also be $M_{12}.2$, $M_{22}.2$.  \textup{(}The groups $\Suz$, $\mathrm{Fi}_{22}$ and $J_1$ are \emph{not} $2$-transitive in any of their usual actions:  indeed $|\Suz|/(1782\cdot1781)$, $|\mathrm{Fi}_{22}|/(14080\cdot14079)$ and $|J_1|/(266\cdot265)$ are not integers, so the two-point stabilizer arithmetic excludes them.  For $\Suz$ of degree $1782$ \textup{(}point stabilizer $G_2(4)$\textup{)} and $\mathrm{Fi}_{22}$ of degree $14080$ the actions are rank $3$.\textup{)}
  \end{enumerate}
  \textup{(}The group $\PSU(3,2)\cong 3^2{:}Q_8$ of degree $9$ has elementary abelian socle and belongs to the affine type:  it is the sharply $2$-transitive group attached to the nearfield of order $9$, treated in \textup{Section \ref{sec:affine}} and in \textup{Table \ref{tab:verify}}.\textup{)}
\end{enumerate}
\end{theorem}

\begin{table}[htbp]
\centering
\caption{Exceptional affine cases \textup{(}Hering's list\textup{)}.}\label{tab:hering}
\small
\begin{tabular}{ccl}
\toprule
$p^d$ & $G_0$ & remark \\
\midrule
$5^2,7^2,11^2,23^2$ & $\SL(2,3)$ & nearfield cases \\
$11^2,19^2,29^2,59^2$ & $\SL(2,5)$ & nearfield cases \\
$2^4$ & $\Alt(6)$ & $2^4:\Alt(6)\le\Alt(8)$-type \\
$2^4$ & $\Alt(7)$ & \\
$3^4$ & $\SL(2,5)$ & \\
$3^4$ & $2_-^{1+4}$ & extraspecial \\
$2^6$ & $\PSU(3,3)$ & $G_2'(2)=\PSU(3,3)\le\Sp(6,2)$ \\
$3^6$ & $\SL(2,13)$ & \\
\bottomrule
\end{tabular}
\end{table}

\begin{proof}
The classification is classical; see \cite{hering1985,huppert1967,curtis1976,cameron1999,dixon1996}.  We use it as a black box; the case analysis that follows is self-contained.
\end{proof}

\begin{remark}
For the exceptional affine cases in \textup{Table \ref{tab:hering}}, $G_0$ is transitive on the nonzero vectors, and in the nearfield cases \textup{(}i.e.\ $G_0$ is $\SL(2,3)$ or $\SL(2,5)$, embedded in $\GL(2,p)$ as the multiplicative group of a non-Desarguesian nearfield\textup{)} the groups are the well-known sharply $2$-transitive groups associated with nearfields.  The degree $v=p^d$.
\end{remark}

\section{Automorphism group structures}

\begin{remark}
For the almost simple analysis in the following sections we repeatedly need the outer automorphism groups of the classical groups involved.  These are standard facts; we record here the ones used in the analysis below:
\begin{enumerate}
\item $\Out(\PSL(2,q))\cong C_f$ for $q$ even, and $\cong C_2\times C_f$ for $q=p^f$ odd \textup{(}the diagonal part $D$ is trivial for $q$ even, then $\PGL(2,q)=\PSL(2,q)$, and is $C_2$ for $q$ odd\textup{)};
\item $\Out(\PSU(3,q))\cong C_{(3,q+1)}\times C_{2f}$, the diagonal \textup{(}graph\textup{)} part of order $(3,q+1)$ and the field part of order $2f$;  then $\PGU(3,q)/\PSU(3,q)\cong C_{(3,q+1)}$ and $\PpGU(3,q)/\PSU(3,q)\cong C_{(3,q+1)}\times C_{2f}$;
\item $\Out(\Sz(q))\cong C_{2m+1}$ and $\Out(\Ree(q))\cong C_{2m+1}$ \textup{(}field automorphisms\textup{)};
\item $\Out(\Alt(n))\cong C_2$ for $n
\ne6$ \textup{(}with $\Sym(n)=\Alt(n).2$\textup{)}, and $\Out(\Alt(6))\cong C_2\times C_2$.
\end{enumerate}
These data will be used throughout \textup{Sections \ref{sec:psl2}--\ref{sec:ree}} and \textup{Section \ref{sec:as}} without further comment.
\end{remark}

% ============================================================

\section{Suborbit reduction}
\label{sec:suborbit}

We use the two-point stabilizer $G_{\{\alpha,\beta\}}$ of a pair of distinct points $\alpha,\beta$.  The following reduction is standard and valid for every $2$-transitive group.

\begin{lemma}[Suborbit reduction]
\label{lem:suborbit}
Let $G$ be $2$-transitive on $\Omega$ and fix $P=\{\alpha,\beta\}$. The map
\[
x\longmapsto\{\alpha,\beta,x\}
\]
induces a surjection from the $G_P$-orbits on $\Omega\setminus P$ onto the $G$-orbits on $\binom{\Omega}{3}$: every triple contains a pair, and by $2$-transitivity on points $G$ is transitive on unordered pairs, so every $G$-orbit of triples contains a triple of the form $\{\alpha,\beta,x\}$.  For a triple orbit $X$ let $H=\overline{G}_X\le S_3$ be the induced image of the setwise stabilizer of a representative triple.  Then the number of $G_P$-suborbits lying over $X$, that is, the number of $G_P$-orbits on the set $\{T\in X^{G}: P\subseteq T\}$ of triples of the orbit containing $P$, equals the number of $H$-orbits on the three two-subsets of $X$.  Consequently this number is $3,2,1,1$ for $H=1,C_2,C_3,S_3$, respectively:  a trivial image gives three suborbits over $X$, a transposition-only image $C_2$ gives two, and an image containing a $3$-cycle ($H=C_3$ or $S_3$) gives exactly one.  In particular the suborbit map is a bijection if and only if every triple orbit has image containing a $3$-cycle.
\end{lemma}
\begin{proof}
The triples in $X^G$ containing $P$ are parametrized by the three two-subsets of a representative triple $X=\{\alpha,\beta,\gamma\}$:  a triple $T\in X^G$ with $P\subseteq T$ has the form $T=g(X)$ with $g^{-1}(P)\subseteq X$, and $g^{-1}(P)$ is one of $\{\alpha,\beta\},\{\alpha,\gamma\},\{\beta,\gamma\}$.  Two such triples $g_1(X),g_2(X)$ lie in the same $G_P$-orbit precisely when some $h\in G_P$ satisfies $hg_1(X)=g_2(X)$;  writing $k=g_2^{-1}hg_1\in G_X$, this holds exactly when $k(g_1^{-1}P)=g_2^{-1}P$ \textup{(since $h(P)=P$)}\textup{)}.  Hence two triples lie in the same $G_P$-orbit precisely when the corresponding two-subsets lie in the same orbit of $\overline{G}_X$ \textup{(the action of $G_X$ on the three two-subsets of $X$ is isomorphic to its action on the three points, via complementation}\textup{)}.  The values $3,2,1,1$ follow by inspecting $H\le S_3$ on the three two-subsets:  a transposition fixes the two-subset complementary to its fixed point and swaps the other two, a $3$-cycle moves all three two-subsets in one cycle, and $S_3$ is transitive.  Surjectivity of the point-to-triple map is recorded above;  the precise bijection criterion and the double-coset form are stated in \textup{Proposition \ref{prop:pair-orbits}}, and injectivity must be checked in each family below.
\end{proof}

\section{Orbits on triples from the two-point stabilizer}
\label{sec:pair}

\begin{proposition}[Orbits on triples from the two-point stabilizer]
\label{prop:pair-orbits}
Let $G$ be $2$-transitive on $\Pcal$, $v\ge4$, and fix an unordered pair $P=\{\alpha,\beta\}$ with setwise stabilizer $G_P=G_{\{\alpha,\beta\}}$.  The assignment
\[
\gamma^{G_P}\longmapsto \{\alpha,\beta,\gamma\}^{G}
\]
is a well-defined surjection from the set of $G_P$-orbits on $\Pcal\setminus P$ onto the set of $G$-orbits on $\binom{\Pcal}{3}$.  The fiber over the orbit of $X=\{\alpha,\beta,\gamma\}$ is in bijection with the $G_P$-orbits on the set $\{T\in X^{G}: P\subseteq T\}$ of triples of the orbit containing the pair $P$;  in particular the map is a bijection if and only if $G_P$ is transitive on the triples of each $G$-orbit that contain $P$.  In double-coset terms, the $G_P$-orbits on $X^{G}\cong G/G_X$ correspond to the double cosets $G_P\backslash G/G_X$, and the triples containing $P$ correspond to those double cosets whose representative $g$ satisfies $g^{-1}(P)\subseteq X$.
\end{proposition}
\begin{proof}
Well defined:  if $\gamma'=h(\gamma)$ with $h\in G_P$, then $\{\alpha,\beta,\gamma'\}=h\{\alpha,\beta,\gamma\}$ lies in the same $G$-orbit.  Surjective:  every $3$-subset $T$ contains a pair $\{x,y\}$;  by $2$-transitivity on points $G$ is transitive on unordered pairs, so some $g\in G$ maps $\{x,y\}$ onto $P$, and then $g(T)=\{\alpha,\beta,\gamma\}$ for $\gamma=g(z)$, $z=T\setminus\{x,y\}$;  hence $T^{G}=\{\alpha,\beta,\gamma\}^{G}$.  The fiber statement:  the triples of $X^{G}$ containing $P$ are exactly the sets $\{\alpha,\beta,\gamma'\}$ with $\{\alpha,\beta,\gamma'\}^{G}=X^{G}$, and $G_P$ acts on this set;  its orbits are precisely the $G_P$-orbits of such $\gamma'$ on $\Pcal\setminus P$.  The double-coset statement is the standard orbit correspondence for the action of $G_P$ on $G/G_X$ (orbit--stabilizer), with $T=g(X)$ corresponding to the coset $gG_X$ and $P\subseteq g(X)$ equivalent to $g^{-1}(P)\subseteq X$.
\end{proof}

\begin{remark}
\label{rem:pair-orbits}
The bijection of \textup{Proposition \ref{prop:pair-orbits}} is \emph{not} automatic:  for $G=\AGL(1,5)$ on five points and $P=\{0,1\}$, the setwise stabilizer $G_P\cong C_2$ has two orbits on $\Pcal\setminus P$ \textup{(}$\{2,4\}$ and $\{3\}$\textup{)}, while $G$ has a single orbit $X$ on $\binom{\Pcal}{3}$ \textup{(Table \ref{tab:verify}}\textup{)}.  Consistently with \textup{Lemma \ref{lem:suborbit}}, the induced image of $G_X$ on the representative triple is a transposition, so the fiber over $X$ consists of exactly $2$ $G_P$-suborbits:  the two suborbits are fused by $G$ into the single triple orbit.  One must therefore check injectivity separately in each application;  for $HS$ and $Co_3$ it is confirmed by the suborbit computation of \textup{Chapter \ref{sec:verify}}.  Note also that all stabilizers below are \emph{setwise} stabilizers of unordered pairs:  the ordered-pair stabilizer $G_{\alpha,\beta}$ is a subgroup of index $1$ or $2$ in $G_{\{\alpha,\beta\}}$, and the two should not be conflated.
\end{remark}

\section{Counting the orbits from class functions}
\label{sec:class-fn}

\begin{remark}[Counting the orbits from class functions]
\label{rem:char}
The number $r$ of three-point orbits can be obtained without enumerating triples.  For $g\in G$ let $n(g)$ denote the number of $g$-invariant three-point subsets.  A three-point subset is $g$-invariant precisely when it is a union of $g$-cycles of total size $3$;  writing $c_k(g)$ for the number of $k$-cycles of $g$ on $\Pcal$,
\[
n(g)=\binom{c_1(g)}{3}+c_1(g)c_2(g)+c_3(g),
\]
and Burnside's lemma gives $r=\frac{1}{|G|}\sum_{g\in G}n(g)$.  The summands depend only on the cycle type of $g$, hence are constant on conjugacy classes;  the fixed-point count $c_1(g)$ is the value of the point permutation character at $g$, and $c_2(g),c_3(g)$ refine it within the same permutation representation \textup{(}available in GAP directly from the cycle types of class representatives\textup{)}.  The formula does not by itself give the structure of $G_{\Ocal}$ (the kernel $G_{(\Ocal)}$ and the induced image $\overline{G}_{\Ocal}$ still require the explicit stabilizer analysis of the family sections below), but it is an independent check of the orbit counts, and it is far cheaper than enumerating $\binom{v}{3}$ triples.  Concrete applications of the count are reported in \textup{Chapter \ref{sec:verify}}.
\end{remark}

\medskip
The following elementary observation is used repeatedly in the Suzuki and Ree
extension sections \textup{(Sections \ref{sec:sz} and \ref{sec:ree})}:  passing
from a socle to an overgroup that contains it as a normal subgroup can only
\emph{fuse} the socle orbits and never creates new ones.

\begin{lemma}[Orbit fusion under a normal subgroup]
\label{lem:norm-orbit-fusion}
Let $S\trianglelefteq G$ act on a finite set $\Omega$, and let $\Ocal_S(\Omega)$ denote the set of $S$-orbits.  Then $G/S$ acts naturally on $\Ocal_S(\Omega)$, and the $G$-orbits on $\Omega$ are precisely the unions of the $S$-orbits belonging to one orbit of $G/S$ on $\Ocal_S(\Omega)$.
\end{lemma}
\begin{proof}
Let $\Delta\in\Ocal_S(\Omega)$ and choose $x\in\Delta$.  For $g\in G$, the set $\Delta^g=\{y^g:y\in\Delta\}$ is again an $S$-orbit, since $S$ is normal in $G$;  and if $g_1S=g_2S$, then $g_2=sg_1$ for some $s\in S$, so $\Delta^{g_2}=\Delta^{sg_1}=(\Delta^s)^{g_1}=\Delta^{g_1}$.  Hence the action of $G$ on $\Ocal_S(\Omega)$ factors through $G/S$.  Let $\mathcal C=\Delta^{G/S}$ be the orbit of $\Delta$ under the quotient action, and put $U_{\mathcal C}=\bigcup_{\Gamma\in\mathcal C}\Gamma$.  The set $U_{\mathcal C}$ is $G$-invariant and contains the $G$-orbit of every point of $\Delta$.  Conversely, if $y\in U_{\mathcal C}$, then $y=x^g$ for some $x\in\Delta$, $g\in G$;  hence $y\in\Delta^G$.  Thus $\Delta^G=U_{\mathcal C}$.
\end{proof}

% ============================================================

\section{The groups \texorpdfstring{$\PSL(2,q)$}{PSL(2,q)} and their almost simple overgroups}
\label{sec:psl2}
% ============================================================

Throughout this section $G\le\PpGL(2,q)=\PGL(2,q)\rtimes C_f$ acts on the projective line $\Pcal=\mathrm{PG}(1,q)$ with $v=q+1$; $S=\PSL(2,q)\le G$.

\subsection{Three-point stabilizers inside \texorpdfstring{$\PSL(2,q)$}{PSL(2,q)}}

We recall the classical facts (standard sources, e.g.\ \cite{huppert1967,dixon1996}).  Write points of $\Pcal$ as $[1:t]$ ($t\in\Fq$) and $\infty=[0:1]$; identify $[1:t]$ with $t$ and $\infty$ with a symbol $\infty$.  Fix the reference triple $\{0,1,\infty\}$.  $\PGL(2,q)$ is sharply $3$-transitive, so for every $3$-subset $\Ocal$ there is $g\in\PGL(2,q)$ with $g(\{0,1,\infty\})=\Ocal$;  the stabilizer of $\Ocal$ inside $\PSL(2,q)$ depends on the \emph{square class} of such a $g$ in $\PGL(2,q)/\PSL(2,q)\cong\Fq^{\times}/\Fq^{\times 2}$ (for $q$ odd);  the well-definedness of this invariant is proved in \textup{Lemma \ref{lem:sqclass}}.

\begin{lemma}[Well-definedness of the square-class invariant]
\label{lem:sqclass}
Let $q$ be odd and let $s\in\PGL(2,q)$ be a permutation of the reference triple $\{0,1,\infty\}$.  Then $\det(s)\in\{1,-1\}$:  the identity and the two $3$-cycles have determinant $1$, and the three transpositions have determinant $-1$.  Consequently, if $g\in\PGL(2,q)$ maps $\{0,1,\infty\}$ to $\Ocal$, replacing $g$ by $g\circ s$ \textup{(}where $s$ permutes the reference triple\textup{)} changes the square class of $g$ by the factor $\det(s)$;  for $q\equiv1\pmod4$ ($-1$ a square) the square class is independent of the choice of $s$, while for $q\equiv3\pmod4$ the transpositions $s$ interchange the two classes.
\end{lemma}
\begin{proof}
The six elements of $\PGL(2,q)_{\{0,1,\infty\}}$ are
\[
\mathrm{id}\ (t\mapsto t),\quad
(0\,1)\ (t\mapsto 1-t),\quad
(0\,\infty)\ (t\mapsto 1/t),\quad
(1\,\infty)\ (t\mapsto t/(t-1)),
\]
\[
(0\,1\,\infty)\ (t\mapsto 1/(1-t)),\quad
(0\,\infty\,1)\ (t\mapsto (t-1)/t),
\]
with matrices $\begin{pmatrix}1&0\\0&1\end{pmatrix}$, $\begin{pmatrix}-1&1\\0&1\end{pmatrix}$, $\begin{pmatrix}0&1\\1&0\end{pmatrix}$, $\begin{pmatrix}1&0\\1&-1\end{pmatrix}$, $\begin{pmatrix}0&1\\-1&1\end{pmatrix}$, $\begin{pmatrix}1&-1\\1&0\end{pmatrix}$, of determinants $1,-1,-1,-1,1,1$, respectively.  The claim follows:  $g\circ s$ has determinant $\det(g)\det(s)$, and the square class in $\Fq^{\times}/\Fq^{\times2}$ is determined by whether that determinant is a square;  $-1$ is a square exactly for $q\equiv1\pmod4$.
\end{proof}

\begin{theorem}[$\PSL(2,q)$: three-point stabilizers]
\label{thm:psl2}
Let $S=\PSL(2,q)$, $q=p^f$ with $q\ge4$ (the exceptional small cases $q=2,3$ are excluded; the case $q=4$ is covered, with $\PSL(2,4)\cong A_5$, see Corollary \ref{cor:psl2-designs}), acting on $\Pcal=\mathrm{PG}(1,q)$, and let $\Ocal$ be a $3$-subset.  Then $S_{(\Ocal)}=1$ for every $\Ocal$, and:
\begin{enumerate}
\item If $q$ is even, $\PGL(2,q)=\PSL(2,q)$ is sharply $3$-transitive on $\Pcal$:  all $3$-subsets form a single orbit and
\[
S_{\Ocal}\cong S_3,\qquad |S_{\Ocal}|=6.
\]
\item If $q\equiv3\pmod4$, $\PSL(2,q)$ is $3$-\emph{homogeneous} \textup{(}not $3$-transitive\textup{)} on $\Pcal$:  all $3$-subsets form a single orbit and
\[
S_{\Ocal}\cong C_3,\qquad |S_{\Ocal}|=3.
\]
\item If $q\equiv1\pmod4$, $\PSL(2,q)$ has exactly \emph{two} orbits on $\binom{\Pcal}{3}$, distinguished by the square class of an element of $\PGL(2,q)$ mapping a fixed reference triple to $\Ocal$ \textup{(}i.e.\ by the class in $\PGL(2,q)/\PSL(2,q)\cong\Fq^{\times}/\Fq^{\times2}$\textup{)}, and on both orbits
\[
S_{\Ocal}\cong S_3,\qquad |S_{\Ocal}|=6.
\]
\end{enumerate}
\end{theorem}
\begin{proof}
$\PGL(2,q)$ is sharply $3$-transitive, so the image of $S_{\Ocal}$ in $S_3$ consists of those permutations of $\Ocal$ realizable by an element of $\PSL(2,q)$;  the kernel is trivial (a projective transformation fixing three points is the identity).  Normalize $\Ocal=\{0,1,\infty\}$ (possible inside $\PGL$).  The transposition fixing $\infty$ and swapping $0,1$ is the involution $t\mapsto1-t$, represented by the matrix $\begin{pmatrix}-1&1\\0&1\end{pmatrix}$ of determinant $-1$;  hence a transposition lies in $\PSL(2,q)$ iff $-1$ is a square in $\Fq^{\times}$, i.e.\ iff $q$ is even or $q\equiv1\pmod4$.  A $3$-cycle on the reference triple is realized by $t\mapsto1/(1-t)$, an element of determinant $1$ (\textup{Lemma \ref{lem:sqclass}}), which lies in $\PSL(2,q)$ for every $q\ge3$ (for $q=3$, $\PSL(2,3)\cong A_4$ on $4$ points and $S_{\Ocal}\cong C_3$).
\begin{enumerate}
\item $q$ even:  $-1=1$ is a square, transpositions and $3$-cycles exist:  image $S_3$.
\item $q\equiv3\pmod4$:  $-1$ is a non-square, so no transposition exists on any $\Ocal$ (a transposition would require $-1$ to be a square, \textup{Lemma \ref{lem:sqclass}});  the image of $S_{\Ocal}$ in $S_3$ is therefore contained in $C_3$.  For the reference triple $\Delta=\{0,1,\infty\}$ the $3$-cycle $(0\,1\,\infty)$ is realized by $t\mapsto1/(1-t)$, an element of determinant $1$ lying in $\PSL(2,q)$ (\textup{Lemma \ref{lem:sqclass}}), so $|S_{\Delta}|=3$.  Stabilizers of triples in the same $S$-orbit are conjugate, hence the orbit of $\Delta$ has size $|S|/|S_{\Delta}|=q(q^2-1)/6=\binom{q+1}{3}$, the total number of $3$-subsets of $\Pcal$.  Since this equals $\bigl|\binom{\Pcal}{3}\bigr|$, the orbit is all of $\binom{\Pcal}{3}$:  $\PSL(2,q)$ is $3$-homogeneous, and $S_{\Ocal}\cong C_3$ with $|S_{\Ocal}|=3$ for every $\Ocal$.  The two determinant classes of \textup{Lemma \ref{lem:sqclass}} are fused into this single orbit.  The square-class label is not well defined in this case:  the $\PGL(2,q)$-stabilizer of a triple contains a transposition of nonsquare determinant, so representatives of the two determinant classes are $\PSL(2,q)$-conjugate;  instead, the orbit-size calculation above directly proves that $\PSL(2,q)$ is $3$-homogeneous.
\item $q\equiv1\pmod4$:  transpositions exist;  the two square classes are the two orbits (transitivity on each class is classical, cf.\ \cite{huppert1967,dixon1996});  on each, image $S_3$.  The square class attached to an element $g\in\PGL(2,q)$ with $g(\Ocal_0)=\Ocal$ is well defined:  if $h(\Ocal_0)=\Ocal$ as well, then $h^{-1}g$ stabilizes $\Ocal_0$ setwise, and its image in $S_3$ is realized inside $\PSL(2,q)$ (transpositions and $3$-cycles both lie in $\PSL(2,q)$ here), while the pointwise kernel is trivial;  hence $h^{-1}g\in\PSL(2,q)$ and $g,h$ have the same determinant square class.
\end{enumerate}
\end{proof}

\begin{remark}[Counting check for the orbit structure]
\label{rem:psl2-count}
For $q\equiv3\pmod4$, $|S_{\Ocal}|=3$ \textup{(Theorem \ref{thm:psl2}(2)}\textup{)};  since $|S|=q(q^2-1)/2$, the orbit size is $q(q^2-1)/6=\binom{q+1}{3}$, the total number of $3$-subsets:  the orbit is all of $\binom{\Pcal}{3}$, i.e.\ $S$ is $3$-homogeneous.  For $q\equiv1\pmod4$, each of the two orbits has stabilizer of order $6$, hence size $q(q^2-1)/12=\frac12\binom{q+1}{3}$, and the two orbits partition $\binom{\Pcal}{3}$.  These counts confirm the orbit structure of \textup{Theorem \ref{thm:psl2}} directly from $|S|$ and $|S_{\Ocal}|$.
\end{remark}

\begin{remark}
For $q$ odd the orbit structure is governed by the square class described in \textup{Section \ref{sec:psl2}} as follows:  for $q\equiv3\pmod4$ the two classes are \emph{fused} by $\PSL(2,q)$ \textup{(}the group is $3$-homogeneous\textup{)}, and for $q\equiv1\pmod4$ \emph{both} classes admit the transposition $t\mapsto1-t$ \textup{(}because $-1$ is a square\textup{)}, so both have stabilizer $S_3$.  The statement is verified computationally in \textup{Chapter \ref{sec:verify}} for $q=5,7,8,9,11$:  $|S_{\Ocal}|=6$ \textup{(}two orbits\textup{)}, $3$ \textup{(}one orbit\textup{)}, $6$ \textup{(}one orbit\textup{)}, $6$ \textup{(}two orbits\textup{)}, $3$ \textup{(}one orbit\textup{)}, respectively.
\end{remark}

\begin{corollary}[Designs from $\PSL(2,q)$]
\label{cor:psl2-designs}
Each listed system is simple, its blocks being three-point subsets (not a multiset family).  The triple systems with point set $\mathrm{PG}(1,q)$ and automorphism group containing $\PSL(2,q)$ are:
\begin{enumerate}
\item $q$ even:  one system, $2$-\mbox{$(q+1,3,q-1)$} \textup{(}all $3$-subsets;  $\lambda=q-1$ by \eqref{eq:lambda}\textup{)};
\item $q\equiv3\pmod4$:  one system, again $2$-\mbox{$(q+1,3,q-1)$} \textup{(}all $3$-subsets\textup{)};
\item $q\equiv1\pmod4$:  two orbit triple systems \textup{(}the two square classes of \textup{Theorem \ref{thm:psl2}\textup{(}3\textup{)}}\textup{)}, each with $\lambda=(q-1)/2$;  their union is the complete system on $\mathrm{PG}(1,q)$, with $\lambda=(q-1)/2+(q-1)/2=q-1$ \textup{(}\textup{Proposition \ref{prop:union-orbits}}\textup{)}.
\end{enumerate}
\textup{(}For $q=4$:  $\lambda=q-1=3$:  the complete $2$-$(5,3,3)$ design on $5$ points, with $\PSL(2,4)\cong A_5$.\textup{)}
\end{corollary}
\begin{proof}
Immediate from \textup{Theorem \ref{thm:psl2}} and \eqref{eq:lambda}.  For $q$ even, $|S|=q(q^2-1)$ and $|S_{\Ocal}|=6$ give $\lambda=q-1$;  for $q\equiv3\pmod4$, $|S|=q(q^2-1)/2$ and $|S_{\Ocal}|=3$ give $\lambda=q-1$ as well;  for $q\equiv1\pmod4$ each of the two $S_3$-orbits has $|S_{\Ocal}|=6$ and $\lambda=(q^2-1)/(2(q+1))=(q-1)/2$, an integer.
\end{proof}

\subsection{Overgroups: \texorpdfstring{$\PGL(2,q)$}{PGL(2,q)} and \texorpdfstring{$\PpGL(2,q)$}{P$\Gamma$L(2,q)}}

Now let $G$ be any group with $\PSL(2,q)\le G\le\PpGL(2,q)=\PGL(2,q)\rtimes C_f$.

\begin{theorem}[$\PGL(2,q)$, $\PpGL(2,q)$]
\label{thm:pgl}
Let $\PSL(2,q)\le G\le\PpGL(2,q)$, $q=p^f$, and let $\Ocal$ be a $3$-subset of $\mathrm{PG}(1,q)$; here and below $G_{\Ocal}$ and $G_{(\Ocal)}$ denote the setwise and the pointwise stabilizer of $\Ocal$, respectively.
\begin{enumerate}
\item If $G=\PGL(2,q)$:  $G_{\Ocal}\cong S_3$, $G_{(\Ocal)}=1$, for all $\Ocal$;  single orbit on $\binom{\Pcal}{3}$.
\item If $G\supsetneq\PGL(2,q)$, let $F\le C_f$ be the \textup{(}cyclic\textup{)} group of field automorphisms contained in $G$ (i.e.\ the image of $\mathrm{Gal}(\Fq/\F_p)$ in $\mathrm{Aut}(\Fq)$, or a subgroup of that image, acting naturally on $\mathrm{PG}(1,q)$).  Then $G$ is transitive on $\binom{\Pcal}{3}$ and for \emph{every} $3$-subset $\Ocal$,
\[
G_{\Ocal}=G_{(\Ocal)}\rtimes S_3,\qquad |G_{\Ocal}|=6\,|F|,\qquad |G_{(\Ocal)}|=|F|.
\]
The kernel $G_{(\Ocal)}$ consists of the elements $g_\sigma=h_\sigma\circ\sigma$ \textup{(}$\sigma\in F$\textup{)}, where $h_\sigma\in\PGL(2,q)$ is the unique projective transformation with $h_\sigma(\sigma(x))=x$ for $x\in\Ocal$;  in particular the elements of the kernel need not be field automorphisms and exist for \emph{all} $\Ocal$ \textup{(}they are the ``twisted'' elements fixing the three points\textup{)}.
\end{enumerate}
\end{theorem}
\begin{proof}
(1) $\PGL(2,q)$ is sharply $3$-transitive:  $G_{\Ocal}$ maps onto $S_3$ and fixes no point of $\Ocal$ pointwise (a projective transformation fixing three points is the identity).
(2) $\PGL(2,q)\le G$ already realizes all of $S_3$ on $\Ocal$ and is transitive on $\binom{\Pcal}{3}$, so $G$ is too.  For the kernel:  an element of $G_{\Ocal}$ is $g=\sigma\circ h$ with $\sigma\in F$, $h\in\PGL(2,q)$;  it fixes $\Ocal$ pointwise iff $h(\sigma(x))=x$ for all $x\in\Ocal$.  Since $\PGL(2,q)$ is sharply $3$-transitive, for each $\sigma\in F$ there is a unique $h_\sigma$ with $h_\sigma(\sigma(x))=x$ for the three points $x\in\Ocal$, and $h_\sigma\circ\sigma\in G_{(\Ocal)}$ (the product lies in $G$ because $h_\sigma\in\PGL(2,q)\le G$ and $\sigma\in G$).  Distinct $\sigma$ give distinct kernel elements, and every kernel element arises this way, so $|G_{(\Ocal)}|=|F|$.  The image is all of $S_3$ from (1);  the sequence splits since $S_3$ is realized inside $\PGL(2,q)\le G$.
\end{proof}

\begin{remark}
\label{rem:twisted}
By \textup{Theorem \ref{thm:pgl}\textup{(}2\textup{)}} the kernel element for a triple $\Ocal$ is the \emph{twisted} projective transformation $h_\sigma\circ\sigma$, fixing the three points;  only for triples $\Ocal$ lying in the fixed subline \textup{(}e.g.\ $\mathrm{PG}(1,p)$ inside $\mathrm{PG}(1,p^f)$\textup{)} is the kernel element the field automorphism itself.  The computation in \textup{Chapter \ref{sec:verify}} confirms it:  $|\PpGL(2,9)_{\Ocal}|=12=6\cdot2$ for \emph{every} $\Ocal$ \textup{(}not $6$ off the subline\textup{)}, and $|\PpGL(2,25)_{\Ocal}|=12=6\cdot2$ likewise.
\end{remark}

\begin{lemma}[Determinant of the twisted correction]
\label{lem:twisted-det}
Let $q=p^f$, let $\sigma$ be the field automorphism $t\mapsto t^{p^k}$ of $\mathrm{PG}(1,q)$, and let $\Ocal=g\{0,1,\infty\}$ with $g\in\PGL(2,q)$.  The unique $h_\sigma\in\PGL(2,q)$ with $h_\sigma(\sigma(x))=x$ for all $x\in\Ocal$ is $h_\sigma=g\circ\sigma(g)^{-1}$, and
\[
\det(h_\sigma)=\delta^{1-p^k},\qquad \delta=\det(g),
\]
which is a square in $\Fq^{\times}$ for every $\delta$:  for $p$ odd, $p^k-1$ is even, so $\delta^{1-p^k}=(\delta^{(1-p^k)/2})^2$;  for $p=2$ every element of $\Fq^{\times}$ is a square.  Hence $h_\sigma\in\PSL(2,q)$ for every triple $\Ocal$ and every field automorphism $\sigma$.
\end{lemma}
\begin{proof}
For $x=g(y)$ with $y\in\{0,1,\infty\}$ one has $\sigma(x)=(\sigma g\sigma^{-1})(y)$, and $(g\circ\sigma(g)^{-1})(\sigma g\sigma^{-1})(y)=g(y)=x$;  as $\PGL(2,q)$ is sharply $3$-transitive, the element with this property on the three points is unique.  The determinant is $\det(g)\det(\sigma(g))^{-1}=\delta\,\sigma(\delta)^{-1}=\delta^{1-p^k}$, and the square claim follows.
\end{proof}

\begin{theorem}[Field overgroups not containing $\PGL(2,q)$]
\label{thm:psl2-field}
Let $q=p^f$ and $G=\PSL(2,q)\cdot F$ with $1
\ne F\le C_f$ a cyclic group of field automorphisms acting on $\mathrm{PG}(1,q)$ ($F$ the image of $\mathrm{Gal}(\Fq/\F_p)$, or a subgroup of that image); let $\Ocal$ be a $3$-subset, with $G_{\Ocal}$ and $G_{(\Ocal)}$ its setwise and pointwise stabilizers.  Then $|G_{(\Ocal)}|=|F|$ for every $\Ocal$ \textup{(}the twisted kernel elements as in \textup{Theorem \ref{thm:pgl}\textup{(}2\textup{)}}, now with $h_\sigma\in\PSL(2,q)$:  this holds because for $q\equiv3\pmod4$ the map $h_\sigma$ has square determinant automatically, and for $q\equiv1\pmod4$ or $q$ even it is the square-type choice\textup{)}, and the image in $S_3$ is as in \textup{Theorem \ref{thm:psl2}}:
\begin{enumerate}
\item $q$ even:  single orbit, image $S_3$, $|G_{\Ocal}|=6|F|$;
\item $q\equiv3\pmod4$:  single orbit, image $C_3$, $|G_{\Ocal}|=3|F|$;
\item $q\equiv1\pmod4$:  two orbits, image $S_3$, $|G_{\Ocal}|=6|F|$ on each.
\end{enumerate}
\end{theorem}
\begin{proof}
Field automorphisms normalize $\PSL(2,q)$ and preserve the square class (a square $s$ is mapped to $s^{p^k}$, again a square), so the orbit structure of \textup{Theorem \ref{thm:psl2}} is unchanged.  The kernel:  as in \textup{Theorem \ref{thm:pgl}(2)}, for each $\sigma\in F$ there is a unique $h_\sigma\in\PGL(2,q)$ with $h_\sigma(\sigma(x))=x$ for $x\in\Ocal$;  by \textup{Lemma \ref{lem:twisted-det}}, $\det(h_\sigma)$ is a square for every $\Ocal$ and every $\sigma$, so $h_\sigma\in\PSL(2,q)$ and $h_\sigma\circ\sigma\in G_{(\Ocal)}$.  Hence $|G_{(\Ocal)}|=|F|$ and $|G_{\Ocal}|=|F|\cdot|\overline{G}_{\Ocal}|$ with the image as stated.

(Computed in \textup{Chapter \ref{sec:verify}}:  $|\PSL(2,9)\cdot2_{\Ocal}|=12$ (two orbits), $|\PSL(2,25)\cdot2_{\Ocal}|=12$ (two orbits), $|\PSL(2,27)\cdot3_{\Ocal}|=9$ (one orbit).)
\end{proof}

The statements of this section are verified computationally in \textup{Chapter \ref{sec:verify}} for the projective line:  $\PSL(2,q)$ with $q=5,7,8,9,11$, $\PGL(2,q)$ with $q=5,7,8,9$, $\PpGL(2,q)$ with $q=9,25$, and the field overgroups $\PSL(2,9)\cdot2$, $\PSL(2,25)\cdot2$, $\PSL(2,27)\cdot3$.  In each case the computed orbit and stabilizer data agree with \textup{Theorem \ref{thm:psl2}}, \textup{Theorem \ref{thm:pgl}} and \textup{Theorem \ref{thm:psl2-field}}.

% ============================================================
\section[Projective-space and symplectic groups]{The projective-space groups \texorpdfstring{$\PSL(n,q)$, $n\ge3$}{PSL(n,q), n>=3}, and the symplectic groups \texorpdfstring{$\Sp(2n,2)$}{Sp(2n,2)}}
\label{sec:psln}
% ============================================================

\subsection{The groups \texorpdfstring{$\PSL(n,q)$}{PSL(n,q)} on projective space}

Let $n\ge3$, $S=\PSL(n,q)$, and let $\Pcal=\mathrm{PG}(n-1,q)$ be the set of $1$-spaces of $V=\Fq^n$;  $v=(q^n-1)/(q-1)$.  The action of $S$ on $\Pcal$ is $2$-transitive. Fix a point $P=\langle v_1\rangle$; the point stabilizer $S_P$ is a parabolic subgroup with Levi decomposition $U\rtimes L$, where $L$ acts on the quotient $V/P$ through $\GL(n-1,q)$ and $U$ is the unipotent radical. We verify transitivity of $S_P$ on $\Pcal\setminus\{P\}$ in two steps. First, every point $Q
\ne P$ determines the line $PQ$, equivalently a $1$-space of $V/P$; the Levi factor contains $\GL(n-1,q)$, which is transitive on the $1$-spaces of $V/P$, so $S_P$ carries any line through $P$ onto any other. Second, for a fixed line $L$ through $P$, the subgroup of $S_P$ stabilizing $L$ induces on $L\setminus\{P\}$ the translation subgroup of $\PGL(L)$ of order $q$:  the unipotent radical $U$ contains the transvections with axis $P$ and centre any point of $L\setminus\{P\}$, and these transvections act on the line as the regular translation group, transitive on the $q$ points of $L\setminus\{P\}$.  Thus $S_P$ is transitive on $\Pcal\setminus\{P\}$. Note that the line $PQ$ is \emph{not} an $S_P$-invariant: elements of $S_P$ send $PQ$ to $PQ'$; the line is a function of the pair rather than an invariant constraining the two-point action (the collinear/non-collinear distinction arises only at the third order, Theorem \ref{thm:psln}(1)). This family, $n\ge3$, is part of the classical list of \cite{curtis1976}; every instance in the tested degree range is verified in \textup{Chapter \ref{sec:verify}}.

\begin{theorem}[Three-point stabilizers in $\PSL(n,q)$ on projective space]
\label{thm:psln}
Let $n\ge3$, $S=\PSL(n,q)$ acting on $\Pcal=\mathrm{PG}(n-1,q)$, and $\Ocal$ a $3$-subset of $\Pcal$.
\begin{enumerate}
\item $S$ has exactly two orbits on $\binom{\Pcal}{3}$:  the \emph{collinear} triples \textup{(}contained in a projective line\textup{)} and the \emph{non-collinear} triples.  Every overgroup $S\le G\le\PpGL(n,q)$ has the same two orbits.
\item \emph{Collinear triples.}  The image in $S_3$ is all of $S_3$ \textup{(}the line stabilizer induces $\PGL(2,q)$, sharply $3$-transitive on its $q+1$ points\textup{)}, and
\[
|S_{\Ocal}|=6\,|S_{(\Ocal)}|,\qquad
|S_{(\Ocal)}|=\frac{(q-1)\,q^{2(n-2)}\,|\SL(n-2,q)|}{\gcd(n,q-1)},
\]
\textup{(}with $|\SL(1,q)|=1$;  for $n=3$ this reads $q^2(q-1)/(3,q-1)$\textup{)}.  The pointwise stabilizer consists of the elements acting as a scalar on the line \textup{(}extended with determinant $1$\textup{)}, the unipotent radical with axis the line, and $\SL(n-2,q)$ on a complement.
\item \emph{Non-collinear triples.}  The image in $S_3$ is again all of $S_3$ \textup{(}three non-collinear points span a plane;  a permutation $\sigma\in S_3$ of the three points is realized by the permutation matrix composed with a diagonal scalar $\mathrm{diag}(\lambda_1,\lambda_2,\lambda_3)$ with $\lambda_1\lambda_2\lambda_3=\varepsilon(\sigma)^{-1}$, which fixes each of the three points projectively and makes the total determinant $1$;  this works for every $n\ge3$\textup{)}, and
\[
|S_{\Ocal}|=6\,|S_{(\Ocal)}|,\qquad
|S_{(\Ocal)}|=
\begin{cases}
(q-1)^2/(3,q-1), & n=3,\\[2pt]
\dfrac{(q-1)^3\,q^{3(n-3)}\,|\SL(n-3,q)|}{\gcd(n,q-1)}, & n\ge4.
\end{cases}
\]
\textup{(}The pointwise stabilizer fixes the subspace spanned by the three points:  a diagonal part $\mathrm{diag}(\lambda_1,\lambda_2,\lambda_3)$ with $\lambda_1\lambda_2\lambda_3\det B=1$, a triangular part of order $q^{3(n-3)}$, and $B\in\GL(n-3,q)$ on a complement.  The diagonal parameters $\lambda_1,\lambda_2$ range freely \textup{(}$(q-1)^2$ choices\textup{)} and $\lambda_3$ is then determined by the determinant condition;  since $|\GL(m,q)|=(q-1)|\SL(m,q)|$, the count $(q-1)^2q^{3(n-3)}|\GL(n-3,q)|$ equals the displayed $(q-1)^3q^{3(n-3)}|\SL(n-3,q)|$.\textup{)}
\end{enumerate}
\end{theorem}
\begin{proof}
(1) Collinearity is preserved by $\PpGL(n,q)$;  transitivity on collinear triples follows from transitivity on lines (a line is determined by two points) and $3$-transitivity of $\PGL(2,q)$ on a line;  transitivity on non-collinear triples follows because three non-collinear points span a plane and $\PGL(n,q)$ is transitive on planes and on triangles inside a plane.
(2) The line stabilizer maps onto $\PGL(2,q)$ on the line (Levi part), hence onto $S_3$ on $\Ocal$.  For the kernel:  an element fixing the line pointwise has the block form $\left(\begin{smallmatrix}\lambda I_2&\ast\\0&B\end{smallmatrix}\right)$ with $\lambda^2\det B=1$;  counting the choices for $\lambda$, for the unipotent part ($q^{2(n-2)}$ entries) and for $B$ (with $\det B=\lambda^{-2}$, i.e.\ $|\SL(n-2,q)|$ choices for each $\lambda$) gives $(q-1)\,q^{2(n-2)}\,|\SL(n-2,q)|$ in $\SL(n,q)$, and passage to $\PSL(n,q)$ divides by $|\Z(\SL(n,q))|=\gcd(n,q-1)$ \textup{(the stabilizers just counted contain the scalar center $\Z(\SL(n,q))$, which acts trivially on projective points and hence on $\Ocal$)}\textup{)}.
(3) Fixing the three points forces $g$ to fix the spanned $3$-space, with block form $\left(\begin{smallmatrix}D&X\\0&B\end{smallmatrix}\right)$, $D=\mathrm{diag}(\lambda_1,\lambda_2,\lambda_3)$, $X\in M_{3\times(n-3)}(\F_q)$, $B\in\GL(n-3,q)$, and determinant condition $\lambda_1\lambda_2\lambda_3\det B=1$.  The entries of $X$ give $q^{3(n-3)}$ choices;  $\lambda_1,\lambda_2$ are free ($(q-1)^2$ choices) and $\lambda_3$ is determined by the condition;  $B$ ranges over all of $\GL(n-3,q)$, since any determinant is absorbed into $\lambda_3$.  This yields $(q-1)^2q^{3(n-3)}|\GL(n-3,q)|=(q-1)^3q^{3(n-3)}|\SL(n-3,q)|$ elements of $\SL(n,q)$, and passage to $\PSL(n,q)$ divides by $|\Z(\SL(n,q))|=\gcd(n,q-1)$ \textup{(again, the center lies in every counted stabilizer and acts trivially projectively)}\textup{)}, proving (3).  (For $n=3$ there is no $X$ or $B$;  the three non-collinear points form a basis and the count is $(q-1)^2/(3,q-1)$.)  The stabilizer formulas are verified for the five smallest parameter sets in \textup{Proposition \ref{prop:psl-stab-computed}} \textup{(Chapter \ref{sec:verify})}.
\end{proof}

\begin{proposition}[Small-parameter verification for $\PSL(n,q)$ on projective space]
\label{prop:psl-stab-computed}
For the five smallest parameter sets $(n,q)=(3,2),(3,3),(4,2),(4,3),(5,2)$, the stabilizer formulas of \textup{Theorem \ref{thm:psln}} are verified directly in GAP \textup{(supplementary script \texttt{verify\_psl\_stab.g};  Chapter \ref{sec:verify})}:  the computed collinear setwise stabilizer orders $24,108,576,11664,64512$ and non-collinear ones $6,24,48,648,2304$ agree with the formulas \textup{(as do the pointwise orders $4,18,96,1944,10752$ resp.\ $1,4,8,108,384$)}\textup{)}, in each case with induced image $S_3$ and with $\lambda=q-1$ on the collinear orbit and $\lambda=q^2$ (for $n=3$) resp.\ $v-1-q$ (for $n\ge4$) on the non-collinear one.  In particular $\PSL(4,3)$ on $40$ points has kernel $|S_{(\Ocal)}|=108=(q-1)^3q^3/\gcd(4,2)$ on the non-collinear orbit, matching \textup{Table \ref{tab:verify}}.  In addition the script checks the orbit--stabilizer identity $|\Ocal|=|G|/|S_{\Ocal}|$ and that the two orbit sizes sum to $\binom{v}{3}$ \textup{(for $v=40$ the sum equals $\binom{40}{3}=9880$)};  the stabilizer orders and the orbit counts are thus cross-consistent, and not only the $\lambda$-values.  The statement is a finite result for the five listed parameter sets;  no family claim is involved \textup{(Remark \ref{rem:open})}.
\end{proposition}

\begin{corollary}[Designs from $\PSL(n,q)$ on projective space]
\label{cor:psln-designs}
With $v=(q^n-1)/(q-1)$:
\begin{enumerate}
\item $n=3$:  the collinear orbit gives $2$-\mbox{$(q^2+q+1,3,q-1)$} and the non-collinear orbit gives $2$-\mbox{$(q^2+q+1,3,q^2)$};
\item $n\ge4$:  the collinear orbit gives $2$-\mbox{$(v,3,q-1)$} --- for $q=2$ this is the Steiner triple system $\mathrm{STS}(2^n-1)$ on the lines of $\mathrm{PG}(n-1,2)$ --- and the non-collinear orbit gives $2$-\mbox{$(v,3,v-1-q)$}.
\end{enumerate}
By \textup{Proposition \ref{prop:union-orbits}} the union of the two orbits is the complete system $2$-$(v,3,v-2)$ on projective space:  for $n=3$, $\lambda=q-1+q^2=q^2+q-1=v-2$;  for $n\ge4$, $\lambda=q-1+v-1-q=v-2$.
\textup{(}The values follow from \eqref{eq:lambda} and \textup{Theorem \ref{thm:psln}};  e.g.\ $\PSL(3,3)$ on $13$ points:  $\lambda=2,9$;  $\PSL(4,2)=\Alt(8)$ on $15$:  $\lambda=1,12$;  $\PSL(3,5)$ on $31$:  $\lambda=4,25$;  $\PSL(5,2)$ on $31$:  $\lambda=1,28$;  $\PSL(4,3)$ on $40$:  $\lambda=2,36$;  $\PSL(6,2)$ on $63$:  $\lambda=1,60$.\textup{)}
\end{corollary}

\subsection{The groups \texorpdfstring{$\Sp(2n,2)$}{Sp(2n,2)} on the two halves of the quadric}

Let $n\ge3$ and $S=\Sp(2n,2)=\Omega(2n+1,2)$, of order $|S|=2^{n^2}\prod_{i=1}^n(2^{2i}-1)$.  The space $\mathrm{PG}(2n,2)$ carries a nonsingular quadratic form of Witt index $n$;  the nonsingular points fall into two orbits of $\Omega(2n+1,2)$, of sizes
\[
v_\pm=2^{2n-1}\pm2^{n-1},
\]
and $S$ acts $2$-transitively on each of them (the point stabilizer is the orthogonal group $O^{\mp}(2n,2)$ of the perpendicular space, which is transitive on the remaining points of the half).  This family is part of the classical list of \cite{curtis1976}.  For $n=3$ ($S=\Sp(6,2)$, $|S|=1451520$) the census gives, for the two primitive representations $\mathrm{PrimitiveGroup}(28,6)$ and $\mathrm{PrimitiveGroup}(36,10)$ (\textup{Chapter \ref{sec:verify}}):

\begin{table}[htbp]
\centering
\begin{tabular}{ccl}
\toprule
$v$ & $|S_{\Ocal}|$: kernel (image) & $\lambda$ \\
\midrule
28 & $1152$: $192$ ($S_3$),  $720$: $120$ ($S_3$) & $10$, $16$ \\
36 & $384$: $64$ ($S_3$),  $432$: $72$ ($S_3$) & $18$, $16$ \\
\bottomrule
\end{tabular}
\caption{$\Sp(6,2)$ on the two halves of the nonsingular quadric:  three-point stabilizers and design parameters.}\label{tab:sp62}
\end{table}

\begin{corollary}[Designs from $\Sp(2n,2)$]
\label{cor:sp2n-designs}
For $n=3$ the two actions yield the orbit triple systems
\[
2-\mbox{$(28,3,10)$},\quad 2-\mbox{$(28,3,16)$},\quad 2-\mbox{$(36,3,16)$},\quad 2-\mbox{$(36,3,18)$},
\]
and their unions give the complete systems $2$-$(28,3,26)$ and $2$-$(36,3,34)$ \textup{(}\textup{Proposition \ref{prop:union-orbits}}\textup{)}.
For every $n\ge3$ \textup{Theorem \ref{thm:sp2n-triples}} determines the two orbit systems on each half:  on the half of degree $v_s=2^{2n-1}+s2^{n-1}$ ($s=\pm1$) the parameters are $\lambda_0=2^{2n-2}+s2^{n-1}-2$ and $\lambda_1=2^{2n-2}$, and their union is the complete system $2$-\mbox{$(v_s,3,v_s-2)$} \textup{(\textup{Proposition \ref{prop:union-orbits}}\textup{)}}.  For $n\ge4$ the degrees $2^{2n-1}\pm2^{n-1}$ \textup{(}$=120,136,496,528,\dots$\textup{)} lie outside the tested degree range of this monograph, but the classification is supplied by \textup{Theorem \ref{thm:sp2n-triples}}.
\end{corollary}

\begin{remark}
\label{rem:sp2n}
For every $n\ge3$, $\Sp(2n,2)$ is $2$-transitive on each half of the nonsingular quadric \textup{(}the point stabilizer $\Orth^s(2n,2)$, $s=\pm1$, is transitive on the remaining points\textup{)}.  On each half the three-point orbits are \emph{exactly two}, and the complete orbit classification with the stabilizer structure and the design parameters is given by \textup{Theorem \ref{thm:sp2n-triples}};  for $n=3$ it reduces to \textup{Table \ref{tab:sp62}}.  The following structural formula for higher parabolic subgroups is available as a complement:  if $U$ is a totally isotropic $k$-space and $P_k=\Stab_{\Sp(2n,2)}(U)$, then
\[
P_k=Q_k\rtimes\bigl(\GL(k,2)\times\Sp(2n-2k,2)\bigr),
\qquad
|Q_k|=2^{\,2k(n-k)+\binom{k+1}{2}},
\]
so $|P_k|=2^{\,2k(n-k)+\binom{k+1}{2}}\,|\GL(k,2)|\,|\Sp(2n-2k,2)|$.  This formula concerns higher parabolic subgroups;  it is retained for completeness and is not needed for \textup{Theorem \ref{thm:sp2n-triples}}, which determines the three-subset stabilizers directly.
\end{remark}
\begin{theorem}[Three-subset orbits of $\Sp(2n,2)$]
\label{thm:sp2n-triples}
Let $n\ge3$, let $V$ be a $2n$-dimensional symplectic space over $\F_2$, and let
\[
G=\Sp(V)=\Sp(2n,2).
\]
For $\varepsilon\in\F_2$, let $\Omega_\varepsilon$ be the set of quadratic forms on $V$ polarizing to the symplectic form and having Arf invariant $\varepsilon$.  Put
\[
s=(-1)^\varepsilon,
\qquad
v_s=|\Omega_\varepsilon|
   =2^{2n-1}+s2^{n-1}.
\]
Then $G$ is $2$-transitive on $\Omega_\varepsilon$ and has exactly two orbits
\[
\Tcal_0,\Tcal_1
\]
on $\binom{\Omega_\varepsilon}{3}$.

For $X=\{Q_1,Q_2,Q_3\}$, the orbit containing $X$ is determined by
\[
\delta(X)
=
\Arf(Q_1+Q_2+Q_3)+\varepsilon
\in\F_2;
\]
namely $X\in\Tcal_{\delta(X)}$.  The corresponding orbit triple systems have parameters
\[
\lambda_0=2^{2n-2}+s2^{n-1}-2,
\qquad
\lambda_1=2^{2n-2}.
\]
In particular,
\[
\lambda_0+\lambda_1=v_s-2.
\]

For every $X\in\Tcal_\delta$, the induced action of $G_X$ on $X$ is $S_3$.  More precisely,
\[
G_X\cong
\begin{cases}
U\rtimes
 \bigl(\GL(2,2)\times \Orth^s(2n-4,2)\bigr),\quad |U|=2^{\,4n-7},
 &\delta=0,\\[2mm]
S_3\times \Orth^s(2n-2,2),
 &\delta=1,
\end{cases}
\]
and
\[
G_{(X)}\cong
\begin{cases}
U\rtimes \Orth^s(2n-4,2),
 &\delta=0,\\[2mm]
\Orth^s(2n-2,2),
 &\delta=1.
\end{cases}
\]
Thus
\[
|\Tcal_\delta|
=
\frac{v_s(v_s-1)\lambda_\delta}{6},
\qquad
|G_X|
=
\frac{6|G|}{v_s(v_s-1)\lambda_\delta}.
\]
\end{theorem}

\begin{proof}
The $2$-transitivity of $G$ on $\Omega_\varepsilon$ is standard:  $\Sp(2n,2)$ is transitive on the quadratic forms polarizing to the symplectic form and having a fixed Arf invariant (the forms of the prescribed Arf invariant form a single affine orbit of size $2^{2n-1}\pm2^{n-1}$), and the stabilizer $\Orth^s(2n,2)$ is transitive on the remaining forms --- equivalently, fixing $Q\in\Omega_\varepsilon$, $\Orth(Q)$ is transitive on the nonzero $Q$-singular vectors, which parametrize the other forms of the same Arf type by $Q_a\longmapsto a$;  this is what allows the reduction below to the classification of the ordered pairs $(a,b)$.
Fix $Q\in\Omega_\varepsilon$.  Every quadratic form polarizing to the fixed symplectic form is uniquely of the form
\[
Q_a(x)=Q(x)+B(x,a)
\]
for $a\in V$, and
\[
\Arf(Q_a)=\Arf(Q)+Q(a).
\]
Consequently, after fixing the first point, an ordered triple of distinct elements of $\Omega_\varepsilon$ has the form
\[
(Q,Q_a,Q_b),
\]
where $a,b$ are distinct nonzero $Q$-singular vectors.

Since
\[
Q+Q_a+Q_b=Q_{a+b},
\]
we have
\[
\Arf(Q+Q_a+Q_b)+\varepsilon
=Q(a+b)=B(a,b).
\]
By Witt's lemma, $\Orth(Q)$ is transitive on the ordered pairs of distinct nonzero singular vectors having a prescribed value of $B(a,b)$.  Hence the two values $B(a,b)=0,1$ give exactly two orbits, and the displayed Arf invariant is a complete invariant.

Fix a nonzero singular vector $a$.  The form induced on $a^\perp/\langle a\rangle$ has dimension $2n-2$ and type $s$.  It follows that
\[
|\{b:Q(b)=0,\ B(a,b)=0\}|
=
2^{2n-2}+s2^{n-1}.
\]
After deleting $b=0,a$, this gives
\[
\lambda_0=2^{2n-2}+s2^{n-1}-2.
\]
The total number of singular vectors is $v_s$, so
\[
\lambda_1
=
v_s-(2^{2n-2}+s2^{n-1})
=
2^{2n-2}.
\]

Put $R=Q_1+Q_2+Q_3$.  The three differences $Q_i+R$ correspond to the three nonzero vectors of a $2$-space $W$.  If $\delta=0$, then $W$ is totally singular for $R$, and its stabilizer is
\[
U\rtimes
 \bigl(\GL(2,2)\times \Orth^s(2n-4,2)\bigr),
\qquad |U|=2^{\,4n-7},
\]
where $U$ is the unipotent radical of the stabilizer of the totally singular $2$-space $W$;  $U$ is \emph{not} elementary abelian:  fixing a Witt decomposition $V=W\oplus E\oplus W'$ with $\dim E=2n-4$, the transformations of $U$ moving the two directions of $W$ are parametrized by $a,b\in E$, and their commutator is governed by $B_E(a,b)$;  since $E$ is nondegenerate and $n\ge3$, one may choose $B_E(a,b)=1$, giving a nontrivial commutator.  The factor $\GL(2,2)\cong S_3$ permutes the three nonzero vectors of $W$.  If $\delta=1$, then $W$ is a nondegenerate minus-type plane and
\[
V=W\perp W^\perp,
\qquad
\Orth(W)\cong \Orth^-(2,2)\cong S_3,
\]
while $W^\perp$ has type $s$.  This gives
\[
G_X\cong S_3\times \Orth^s(2n-2,2).
\]
Taking the kernels of the actions on the three points gives the stated pointwise stabilizers.  The orbit-size and design-parameter formulas now follow from orbit--stabilizer and double counting.
\end{proof}

\begin{remark}
The two descriptions of the $\Sp(2n,2)$-actions are the same family:  fixing $Q\in\Omega_\varepsilon$, the bijection $Q_a\longmapsto a$ identifies $\Omega_\varepsilon$ with the set of $Q$-singular vectors of $V$, so $\Omega_0$ and $\Omega_1$ have $v_+=2^{2n-1}+2^{n-1}$ and $v_-=2^{2n-1}-2^{n-1}$ elements, matching the two halves of the nonsingular quadric in \textup{(4.2)};  the point stabilizer of the action on $\Omega_\varepsilon$ is the orthogonal group $\Orth^s(2n,2)$, $s=(-1)^\varepsilon$.  For $n=3$ the parameters of \textup{Theorem \ref{thm:sp2n-triples}} reduce to $(\lambda_0,\lambda_1)=(10,16)$ for $v=28$ and $(18,16)$ for $v=36$, and the setwise stabilizers $\Sp(6,2)_X$ computed in \textup{Table \ref{tab:sp62}}, namely $1152=2^5\,|\GL(2,2)|\,|\Orth^-(2,2)|=32\cdot6\cdot6$ and $720=6\,|\Orth^-(4,2)|=6\cdot120$ for $v=28$, and $384=2^5\,|\GL(2,2)|\,|\Orth^+(2,2)|=32\cdot6\cdot2$ and $432=6\,|\Orth^+(4,2)|=6\cdot72$ for $v=36$, agree with the two cases $\delta=0,1$ of the theorem.
\end{remark}

% ============================================================
\section{The unitary groups \texorpdfstring{$\PSU(3,q)$}{PSU(3,q)} and their almost simple overgroups}
\label{sec:psu3}
% ============================================================

Let $q=p^f$, $d=\gcd(3,q+1)$, $S=\PSU(3,q)$ acting on the Hermitian unital
\[
\mathcal{U}=\{[x:y:z]\in\mathrm{PG}(2,q^2): x^{q+1}+y^{q+1}+z^{q+1}=0\},
\]
with $|\mathcal{U}|=q^3+1$.  We record the standard facts about the geometry and the stabilizers \textup{(used as classical input, cf.\ \cite{mitchell1911})}, and then the orbit structure of $S$ and of selected overgroups on $\binom{\mathcal{U}}{3}$, in the parameter ranges explicitly stated below;  theoretical, computed and unresolved cases are separated in the theorem statements \textup{(Theorem \ref{thm:main-psu})}:  the collinear triples are treated via the classical unital geometry \textup{(Lemma \ref{lem:unital}(4),(5))} together with the counting argument of \textup{Theorem \ref{thm:collinear}}, and the non-collinear triples are classified by the two-point-stabilizer reduction of \textup{Theorem \ref{thm:main-psu}} (the orbit counts are confirmed independently by the computations of \textup{Table \ref{tab:psu-general}}).

\begin{lemma}[Geometry of the unital]
\label{lem:unital}
Let $\mathcal{U}$ be the Hermitian unital in $\mathrm{PG}(2,q^2)$, $q\ge3$.
\begin{enumerate}
\item Every line of $\mathrm{PG}(2,q^2)$ meets $\mathcal{U}$ in $1$ point \textup{(}tangent\textup{)} or $q+1$ points \textup{(}secant block\textup{)};  every pair of points of $\mathcal{U}$ lies on a unique block;  there are $q^2(q^2-q+1)$ blocks, each of size $q+1$, and each point lies on $q^2$ blocks.  In particular, no $3$-subset with two points on a block and the third off it exists:  a $3$-subset is either \emph{collinear} \textup{(}contained in a block\textup{)} or \emph{non-collinear}.
\item There is no pair of distinct isotropic points $u
\ne v\in\mathcal{U}$ with $u\perp v$ (orthogonality taken with respect to the standard Hermitian form $x^{q+1}+y^{q+1}+z^{q+1}$ on the underlying $\F_{q^2}^3$):  such a pair would span a totally singular $2$-space $W$ (both vectors are isotropic and mutually orthogonal, so $W\subseteq W^{\perp}$, and $\dim W=2$ forces equality), contradicting the fact that the Hermitian form has Witt index $1$ \textup{(its maximal totally singular subspaces are the $1$-spaces of the unital}\textup{)}.  Consequently no ``self-polar'' or ``pairwise orthogonal'' triples exist.
\item $|S|=q^3(q^3+1)(q^2-1)/d$;  $S$ is $2$-transitive on $\mathcal{U}$;  the point stabilizer $S_{\alpha}$ has order $q^3(q^2-1)/d$ and is transitive on $\mathcal{U}\setminus\{\alpha\}$ \textup{(}an imprimitive quotient of $\SU(2,q)$-type on the $q^3$ remaining points with blocks the $q^2$ blocks through $\alpha$\textup{)}.
\item Two-point stabilizer:  for distinct $\alpha,\beta\in\mathcal{U}$,
\[
|S_{\alpha,\beta}|=\frac{q^2-1}{d},\qquad
S_{\alpha,\beta}\cong \Z_{(q^2-1)/d}
\]
\textup{(}a cyclic group;  on the block through $\alpha,\beta$ it acts transitively on the remaining $q-1$ points, and off the block it has the $qd$ regular orbits of length $(q^2-1)/d$\textup{)}.  Three-point pointwise stabilizer:
\[
S_{(\Ocal)}\cong
\begin{cases}
\Z_{(q+1)/d}, & \Ocal\ \text{collinear},\\
1, & \Ocal\ \text{non-collinear},
\end{cases}
\]
\textup{(}for non-collinear triples the pointwise stabilizer is trivial for every $q$:  it is contained in $S_{\alpha,\beta}$, which acts regularly off the block by (4)\textup{)}.
\item Block stabilizer:  if $B$ is a block, $S_B$ induces on $B$ the group $\PGL(2,q)$ \textup{(sharply $3$-transitive on the $q+1$ points of $B$}\textup{)}:  the block stabilizer in $\PGU(3,q)$ maps onto $\PGL(2,q)$ with kernel the cyclic pointwise stabilizer $\Z_{q+1}$ of $B$ (a classical result of Mitchell \cite{mitchell1911}, used here as input and not reproved), and the image of $S_B=\PGU(3,q)_B\cap S$ is then all of $\PGL(2,q)$.  Indeed $|\PGU(3,q)_B|=(q+1)|\PGL(2,q)|$;  the index $[\PGU(3,q)_B:S_B]$ equals $d=(3,q+1)$;  for $d=1$ there is nothing to prove, and for $d=3$ the diagonal element $\delta=\mathrm{diag}(1,1,\zeta)$, $\zeta^{q+1}=1$, $\zeta
\ne1$, fixes pointwise the block $B_0=\{z=0\}\cap\mathcal{U}$ and maps onto a generator of $\PGU(3,q)/S\cong C_3$, while $\PGU(3,q)$ is transitive on blocks \textup{($S$ is $2$-transitive on $\mathcal{U}$, and blocks are the pairs of unital points}\textup{)}, so the index is $d$ for every block;  and $|S_B\cap\Z_{q+1}|=(q+1)/d$.  Hence $|S_B|=(q+1)|\PGL(2,q)|/d$, the pointwise stabilizer $S_{(B)}$ of the block in $S$ has order $|S_B\cap\Z_{q+1}|=(q+1)/d$, and the image of $S_B$ in $\PGL(2,q)$ has order $|S_B|/|S_{(B)}|=|\PGL(2,q)|$, so it is the whole group.  In particular every $3$-cycle of $S_3$ on a collinear triple is realizable in $S$ and lifts to an element of $S_{\Ocal}$.
\end{enumerate}
\end{lemma}
\begin{proof}
(1),(2) standard facts on Hermitian varieties (e.g.\ \cite{mitchell1911});  the Witt index is $1$, giving (2).
(3) $|\PSU(3,q)|=q^3(q^3+1)(q^2-1)/d$ and $2$-transitivity is classical \cite{huppert1967,mitchell1911}.
(4) The two-point stabilizer is a cyclic group (the classical Singer-type subgroup of the unitary group;  see \cite{mitchell1911});  the pointwise stabilizer statements follow by restricting to a block (on which the relevant group is a subgroup of $\PGL(2,q)$ fixing $3$ points, hence trivial) or from the two-point case.
(5) That the block stabilizer in $\PGU(3,q)$ maps onto $\PGL(2,q)$ with kernel $\Z_{q+1}$ is a classical theorem of Mitchell \cite{mitchell1911} on the ternary linear groups \textup{(the automorphism group of the Hermitian curve in $\mathrm{PG}(2,q^2)$)}\textup{)};  the displayed order count then shows that the image of $S_B$ is all of $\PGL(2,q)$.
\end{proof}

\subsection{The stabilizer of a collinear triple}

\begin{theorem}[Collinear triples]
\label{thm:collinear}
Let $\Ocal=\{x,y,z\}$ be a collinear $3$-subset of $\mathcal{U}$ \textup{(}contained in a block $B$\textup{)}, and $S=\PSU(3,q)$, $d=\gcd(3,q+1)$.  Then
\[
S_{\Ocal}\cong \Z_{(q+1)/d}\rtimes S_3,
\qquad |S_{\Ocal}|=\frac{6(q+1)}{d},
\qquad S_{(\Ocal)}=\Z_{(q+1)/d},
\]
and the image of $S_{\Ocal}$ on $\Ocal$ is all of $S_3$.
\end{theorem}
\begin{proof}
Using the classical input of \textup{Lemma \ref{lem:unital}(4),(5)} \textup{(cf.\ \cite{mitchell1911})}:  the pointwise stabilizer of $\Ocal$ is cyclic of order $(q+1)/d$;  the block stabilizer induces $\PGL(2,q)$, sharply $3$-transitive on $B$, so the setwise stabilizer of $\Ocal$ maps onto $S_3$;  a preimage of each permutation is obtained inside $S_B$.  The kernel of $S_{\Ocal}\to S_3$ is $S_{(\Ocal)}$, and the sequence splits because $S_3$ is realized inside $S_B$ (a $3$-cycle and a transposition of the three points, cf.\ (5)).  Hence $S_{\Ocal}=\Z_{(q+1)/d}\rtimes S_3$, of order $6(q+1)/d$.
\end{proof}

\begin{remark}
The formula $6(q+1)/d$ is confirmed for every $q\le49$ in \textup{Table \ref{tab:magma}}.  \textup{(}An alternative value $6(q+1)^2/d$ and a ``self-polar'' orbit are inconsistent with the computations and are omitted.\textup{)}
\end{remark}

\subsection{Non-collinear triples:  the classification}

\begin{theorem}[Unitary two-point suborbits and the three-point orbit number]
\label{thm:main-psu}
Let $q\ge3$, $d=\gcd(3,q+1)$, and let $S=\PSU(3,q)$ act on the Hermitian unital $\mathcal{U}$ with $|\mathcal{U}|=q^3+1$.  For a pair $\alpha
\ne\beta\in\mathcal{U}$ let $B$ be the block through $\alpha,\beta$, $T=S_{\alpha,\beta}$ the pointwise stabilizer of the pair, and $\tau\in S$ an element interchanging $\alpha$ and $\beta$;  by the classical normalizer structure \textup{(\cite{mitchell1911,huppert1967})} $\tau$ can always be chosen of order $2$ \textup{(verified for every $q\le27$ in \textup{Table \ref{tab:psu-general}})}.  Then
\begin{enumerate}
\item $S$ is $2$-transitive on $\mathcal{U}$;  $T\cong\Z_{(q^2-1)/d}$ is cyclic and acts regularly on $\mathcal{U}\setminus B$, so the $T$-orbits off $B$ are the $qd$ regular orbits of length $(q^2-1)/d$.  The setwise stabilizer is $G_{\{\alpha,\beta\}}=\langle T,\tau\rangle$ with $G_{\{\alpha,\beta\}}/T\cong C_2$, and $\tau$ normalizes $T$ by the automorphism
\[
\tau\,t\,\tau^{-1}=t^{-q}\qquad(t\in T).
\]
\item \emph{Fixed-orbit number.}  Let $k$ be the number of regular $T$-orbits off $B$ fixed by $\tau$ as a set.  Then
\[
k=\begin{cases}
0, & q\ \text{even},\\
d, & q\ \text{odd}.
\end{cases}
\]
Consequently $G_{\{\alpha,\beta\}}$ has, off $B$, $k$ orbits of length $(q^2-1)/d$ and $(qd-k)/2$ orbits of length $2(q^2-1)/d$.
\item \emph{Orbit number.}  The stabilizer image on a non-collinear triple is one of $S_3,C_3,1$ \textup{(trivial pointwise stabilizer, Lemma \ref{lem:psu-pointwise})};  by \textup{Lemma \ref{lem:suborbit}} each $S_3$- or $C_3$-suborbit contributes one triple orbit, while the trivial suborbits fuse three-to-one \textup{(the three $2$-subsets of the triple form one orbit under $S_3$ and $C_3$, three singleton orbits under the trivial image)}\textup{)}.  With the $C_3$- and trivial-suborbit multiplicities of \textup{Proposition \ref{prop:psu-computed}} \textup{(theoretical in the $d=1$ rows, proved for every $q$ with $3\mid q+1$ in the odd-$d=3$ rows \textup{(Proposition \ref{prop:psu-odd})} and for every $q=2^{2m+1}$ in the even-$d=3$ rows \textup{(Proposition \ref{prop:psu-even})};  the pattern is proved for all $q$ in the $d=1$ rows and verified for $q\le27$ in \textup{Table \ref{tab:psu-general}}}\textup{)}:
\[
\begin{aligned}
\sigma_3(S,\mathcal{U})
&=1+\#\textup{(S}_3\textup{-orbits)}+\#\textup{(C}_3\textup{-orbits)}\\
&\qquad+\tfrac13\#\textup{(trivial suborbits)}\\
&=\begin{cases}
\dfrac{q+3}{2}, & q\ \text{odd},\ d=1,\\[4pt]
\dfrac{5q+11}{6}, & q\ \text{odd},\ d=3\ \textup{(Proposition \ref{prop:psu-odd})},\\[4pt]
1+\dfrac q2, & q\ \text{even},\ d=1,\\[4pt]
\dfrac{5q+2}{6}, & q\ \text{even},\ d=3\ \textup{(}q=2^{2m+1}>2,\ \textup{Proposition \ref{prop:psu-even}}\textup{)},
\end{cases}
\end{aligned}
\]
\textup{in particular $\sigma_3(\PSU(3,8),\mathcal{U})=7$ and $\sigma_3(\PSU(3,32),\mathcal{U})=27$;  for $q$ even with $d=3$, equivalently $q=2^{2m+1}>2$, the $C_3$- and trivial-suborbit multiplicities are proved for every $q$ in \textup{Proposition \ref{prop:psu-even}}:  $(q-2)/2$ and $(q+1)/3$, respectively.}
\item \emph{Stabilizer structure.}  The block orbit consists of the collinear triples, with setwise stabilizer $\Z_{(q+1)/d}\rtimes S_3$ of order $6(q+1)/d$ and $\lambda=q-1$ \textup{(\textup{Theorem \ref{thm:collinear}}\textup{)}}.  Every non-collinear triple has trivial pointwise stabilizer \textup{(Lemma \ref{lem:psu-pointwise})}.  The remaining stabilizer refinement, namely the setwise stabilizer on the $k$ fixed
regular orbits having order $6$ and image $S_3$, is proved for $d=1$ and, for
the odd-$d=3$ case, by \textup{Lemma \ref{lem:psu-odd-fixed}};  the $C_3$-/trivial-
image distribution on the merged orbits is proved theoretically for $d=1$ and for
$d=3$ in \textup{Propositions \ref{prop:psu-odd} and \ref{prop:psu-even}}
\emph{(}the even-$d=3$ values in \textup{Proposition \ref{prop:psu-even}}, the
odd-$d=3$ values in \textup{Proposition \ref{prop:psu-odd}}\emph{)}, while
\textup{Proposition \ref{prop:psu-computed}} records the small-$q$ computational
refinement.  The $\lambda$-values then follow:  $\lambda=(q^2-1)/d$ on the fixed
orbits, and $\lambda=2(q^2-1)/d$ or $6(q^2-1)/d$ on the merged $C_3$- or
trivial-orbits.
\end{enumerate}
For $q=3^a$ \textup{(}$a\ge1$, $d=1$\textup{)} the theorem reduces to the previously stated classification \textup{(one fixed orbit with stabilizer $6$, $(q-1)/2$ merged orbits with stabilizer $3$, $\sigma_3=(q+3)/2$;  moreover in this range the field automorphism $\sigma\colon[x:y:z]\mapsto[x^q:y^q:z^q]$ fixes every $S$-orbit as a set, so no fusion occurs under $A=\PpGU(3,q)$ and $\sigma_3(A)=\sigma_3(S)$, as recorded in \textup{Theorem \ref{thm:pgu}(2),(3)}}).
\end{theorem}
\begin{proof}
(1) $2$-transitivity and the cyclic two-point stabilizer are classical \textup{(Lemma \ref{lem:unital}(3),(4));  cf.\ \cite{mitchell1911,huppert1967}}.  The regularity off the block follows exactly as in the $q=3^a$ case:  an element of $T$ fixing $x
\notin B$ fixes $\{\alpha,\beta,x\}$ pointwise, which contradicts the classical Singer-type action of $T$ \textup{(cf.\ \cite{mitchell1911})};  hence the $qd$ regular $T$-orbits off $B$, each of length $|T|=(q^2-1)/d$.  The extension $G_{\{\alpha,\beta\}}=\langle T,\tau\rangle$ with $G_{\{\alpha,\beta\}}/T\cong C_2$ is the classical normalizer of the (toroidal) two-point stabilizer \textup{(\cite{mitchell1911,huppert1967})}:  $T$ is the Singer cycle, and its normalizer in $\PSU(3,q)$ is $\langle T,\tau\rangle$ with $\tau$ an involutory unitary reflection interchanging $\alpha,\beta$ and acting on $T$ by the anti-toroidal automorphism
\[
\tau\,t\,\tau^{-1}=t^{-q}\qquad(t\in T),
\]
the standard structure of the torus normalizer (cf.\ \cite{mitchell1911,huppert1967});  the identity is verified computationally for every $q\le27$ in \textup{Table \ref{tab:psu-general}}, and the existence of an order-$2$ representative interchanging $\alpha,\beta$ is confirmed there and in \textup{Table \ref{tab:magma}}.
\smallskip

(2) We first determine the fixed points of $\tau$, taken of order $2$ (which exists for every $q$ by (1)).  Since $\tau$ is an isometry of order $2$ of a $3$-dimensional nondegenerate Hermitian space $V$ over $\F_{q^2}$ with Witt index $1$:
\emph{if $q$ is odd}, $\tau$ has eigenvalues $\{1,1,-1\}$ with an eigenvalue-$1$ plane $L_\tau$ and an eigenvalue-$(-1)$ line $p_\tau$, and $V=L_\tau\perp p_\tau$ is an orthogonal direct sum.  Since $\tau$ interchanges $\alpha$ and $\beta$, neither $\alpha$ nor $\beta$ lies on $L_\tau\cup p_\tau$, and $L_\tau
\ne B$;  moreover $\alpha+\beta\in B$ is a fixed vector of $\tau$ (in projective notation $\tau(\alpha+\beta)=\alpha+\beta$), and it is anisotropic:  $h(\alpha+\beta,\alpha+\beta)=2h(\alpha,\beta)$, and $h(\alpha,\beta)
\ne 0$ because the polar line of the isotropic point $\alpha$ in the hyperbolic plane $B$ is the tangent $\langle\alpha\rangle$ (cf.\ \textup{Lemma \ref{lem:unital}(2)}), so $\beta
\notin\alpha^{\perp}$.  Hence $\alpha+\beta\in B\cap(L_\tau\cup p_\tau)$, and the point $w_0=L_\tau\cap B$ is anisotropic:  if $w_0$ were isotropic then $w_0\in B\cap\mathcal{U}$ would be a $\tau$-fixed unital point, and the remaining fixed point of the involution $\tau|_B$ on the projective line $B$ would lie on $p_\tau$, forcing $\tau(\alpha+\beta)=-(\alpha+\beta)$ (as $p_\tau$ is the eigenvalue-$(-1)$ line), contradicting $\tau(\alpha+\beta)=\alpha+\beta$;  so $w_0
\notin\mathcal{U}$.  Also $L_\tau$ cannot be anisotropic:  the Witt index is additive over the orthogonal direct sum, $1=\mathrm{WI}(L_\tau)+\mathrm{WI}(p_\tau)$, and $p_\tau$ has $\mathrm{WI}=0$, forcing $\mathrm{WI}(L_\tau)=1$.  Hence $L_\tau$ is a hyperbolic plane and $L_\tau\cap\mathcal{U}$ consists of $q+1$ points, all off $B$ (as $w_0
\notin\mathcal{U}$);  $p_\tau$ is anisotropic (again by the Witt-index sum), so $p_\tau\cap\mathcal{U}=\varnothing$.  Consequently $\tau$ fixes exactly $q+1$ points off $B$ and none on $B$.
\emph{If $q$ is even}, $\tau$ is an isometric transvection (all eigenvalues $1$):  it fixes a plane $L_\tau$ pointwise, $L_\tau$ is degenerate with radical $w_0\in\mathcal{U}$, and $w_0\in B$ (the restriction $\tau|_B$ fixes a point of the line $B$, which lies in $L_\tau\cap B=\{w_0\}$).  Hence $\tau$ fixes the single unital point $w_0\in B$ and no point off $B$.
\smallskip

Now let $O=T\cdot x$ be a regular $T$-orbit off $B$.  Since $\tau T\tau^{-1}=T$, $\tau(O)$ is again a regular $T$-orbit, and $\tau(O)=O$ holds exactly when $O$ meets the fixed-point set $L_\tau\cap\mathcal{U}$:  if $y\in O\cap(L_\tau\cap\mathcal{U})$, then $\tau(y)=y\in\tau(O)\cap O$, so $\tau(O)=O$;  conversely, if $\tau(O)=O$ and $q$ is odd, write $x\in O$, $\tau(x)=t_0x$, $t_0\in T$;  the map $tx\mapsto\tau(tx)=\tau t\tau^{-1}\tau(x)=t^{-q}t_0x$ has a fixed point iff $t^{-q}t_0=t$ for some $t$, i.e.\ iff $t_0\in T^{q+1}$, where $T^{q+1}=\{u^{q+1}:u\in T\}$ is a subgroup of index $\gcd(q+1,|T|)$ in the cyclic group $T$;  the inclusion $t_0\in T^{q+1}$ is equivalent to $O$ containing a $\tau$-fixed point \textup{(the points $y\in O$ with $\tau(y)=y$ are exactly those of the form $t x$ with $t^{q+1}=t_0$, of which there are $\gcd(q+1,|T|)$ provided one exists;  one always exists by the first part, applied to the hyperbolic line $L_\tau$)}.  Hence
\[
k=\#\{T\text{-orbits meeting }L_\tau\cap\mathcal{U}\}
=\frac{|L_\tau\cap\mathcal{U}|}{\gcd(|T|,q+1)}.
\]
Since $\gcd(|T|,q+1)=\gcd((q^2-1)/d,q+1)=(q+1)/d$ \textup{($d\mid q+1$ and $\gcd(q-1,d)=1$)} and $|L_\tau\cap\mathcal{U}|=q+1$ for $q$ odd, $k=d$;  for $q$ even there are no fixed points off $B$, and any $\tau$-fixed $T$-orbit $O$ has odd length $|T|=(q^2-1)/d$ on which $\tau$ acts as an involution, hence would contain a $\tau$-fixed point off $B$, contradiction;  so $k=0$.  The values $k=1$ for $q=3,9$ (where the classical normalizer still contains order-$2$ representatives) are contained in the uniform statement and verified in \textup{Table \ref{tab:psu-general}}.
\smallskip

The formula of \textup{Theorem \ref{thm:main-psu}(3)} follows from \textup{Lemma \ref{lem:suborbit}} and (2):  off the block, $G_{\{\alpha,\beta\}}$ has the $k$ fixed orbits of length $|T|$ (image $S_3$, one triple orbit each) and the $(qd-k)/2$ merged pairs of length $2|T|$, each merged pair carrying either image $C_3$ (one triple orbit) or trivial image (three suborbits fusing into one triple orbit);  the block orbit contributes the collinear triples \textup{(Theorem \ref{thm:collinear}}).  The closed forms of (3) are then obtained from the multiplicity data of \textup{Proposition \ref{prop:psu-computed}}:  for $d=1$ the counts are proved for every $q$ (one fixed $S_3$-orbit together with $(q-1)/2$ resp.\ $q/2$ $C_3$-merged orbits and no trivial suborbit), and for $d=3$ the counts $(q-5)/2$ and $(q+1)/3$ of $C_3$-orbits and $q+1$ of trivial suborbits hold for $q$ odd \textup{(computed values for $q\le49$;  verified for $q\le27$ in \textup{Table \ref{tab:psu-general}}}\textup{)}, while for $q$ even (equivalently $q=2^{2m+1}$ when $d=3$) the $C_3$- and trivial-suborbit multiplicities are proved for every $q$ in \textup{Proposition \ref{prop:psu-even}} \textup{(}$(q-2)/2$ and $(q+1)/3$ respectively\textup{)}.
\smallskip

(3) By \textup{Lemma \ref{lem:psu-pointwise}} the pointwise stabilizer of a non-collinear triple is trivial, so $S_{\Ocal}$ maps injectively into $S_3$ with stabilizer image $C_3$, $S_3$ or trivial.  On a fixed regular orbit $O$ \textup{(q odd;  for $q$ even there are none, by (2))} choose, by (2), a point $x\in O\cap(L_\tau\cap\mathcal U)$ fixed by $\tau$;  then $\tau(x)=x$, so $\tau$ permutes $\{\alpha,\beta,x\}$ as the transposition $(\alpha\,\beta)$ and lies in $S_{\{\alpha,\beta,x\}}$, whence the stabilizer image contains a transposition and $|S_{\Ocal}|\in\{2,6\}$.  By contrast, for $q$ even \textup{(equivalently $q=2^{2m+1}$ when $d=3$)} there are no fixed orbits by (2), and the even-$d=3$ closed forms of \textup{Proposition \ref{prop:psu-even}} account for every non-collinear orbit with stabilizer image $C_3$ or trivial, so no non-collinear stabilizer image contains a transposition;  a transposition-bearing stabilizer image $S_3$ among the non-collinear triples occurs only for $q$ odd, on the fixed orbits.  The value $6$ \textup{(image $S_3$)} on every fixed orbit is established in
\textup{Lemma \ref{lem:psu-odd-fixed}} for $q$ odd with $3\mid q+1$, the
occurrence of an order-$3$ element cyclically permuting $\{\alpha,\beta,x\}$ for
$x\in L_\tau\cap\mathcal U$, forced by the symmetry $\tau(x)=x$;  this is part of
the computed census for $q=3,9,27$ \textup{(Table \ref{tab:psu-pgu};  Chapter
\ref{sec:verify})};  for $q$ even there are no fixed orbits by (2).  With
$|S_{\Ocal}|=6$ on the fixed orbits, \eqref{eq:lambda} gives
$\lambda=6(q^2-1)/(d\cdot6)=(q^2-1)/d$.  On a merged orbit $\{\alpha,\beta,x\}$ the image is $C_3$ precisely when some order-$3$ element permutes the three points cyclically;  the numbers of such orbits in each congruence class of $q$ are the integers listed in \textup{Proposition \ref{prop:psu-computed}}, obtained from \textup{Table \ref{tab:magma}} for $q\le49$ (and, for the $C_3$-counts, confirmed by the class analysis of the order-$3$ elements of $\PSU(3,q)$:  the $C_3$-orbits correspond to the $3$-cycles of the order-$3$ elements with one fixed point, cf.\ the remark below).  The $\lambda$-values follow from \eqref{eq:lambda}.

\end{proof}

\begin{proposition}[Diagonal-overgroup orbit number of $\PGU(3,q)$]\label{prop:psu-diag}
Let $q\ge3$, $d=\gcd(3,q+1)$, and let $\PGU(3,q)$ act on the Hermitian unital $\mathcal{U}$.  For every $q$,
\[
\sigma_3(\PGU(3,q),\mathcal{U})
=\begin{cases}
\dfrac{q+3}{2}, & q\ \text{odd},\\[4pt]
\dfrac{q+2}{2}, & q\ \text{even}.
\end{cases}
\]
When $d=3$ this is the fusion count of the $S$-orbits on triples under the diagonal overgroup $\PGU(3,q)/S\cong C_3$, obtained from the fusing of the $S$-orbits sharing a $\delta$-orbit, where $\delta\in\PGU(3,q)\setminus S$ is the diagonal element of order $3$ \textup{(cf.\ \textup{Theorem \ref{thm:pgu}(1)}}\textup{)};  no fusion occurs for $d=1$, and then $(q+3)/2$ resp.\ $(q+2)/2$ coincides with the $d=1$ closed forms $\sigma_3(\PSU(3,q))=(q+3)/2$ resp.\ $1+q/2$ of \textup{Theorem \ref{thm:main-psu}(3)} \textup{(the diagonal overgroup equals $S$ for $d=1$)}\textup{)}.  In particular $q=2$ gives $\sigma_3(\PGU(3,2))=2$ \textup{(Table \ref{tab:psu-general})}.
\end{proposition}
\begin{proof}
By \textup{Lemma \ref{lem:suborbit}} the suborbit map in $\PGU(3,q)$ is a bijection on those triple orbits whose induced image contains a $3$-cycle \textup{(sufficient condition of \textup{Lemma \ref{lem:suborbit}}}\textup{)}.  The collinear orbits carry the full $S_3$ from $\PGL(2,q)$.  For the non-collinear orbits the image is $C_3$ or $S_3$:  the diagonal element $\delta$ of \textup{Theorem \ref{thm:pgu}(1b)} supplies a $3$-cycle on every non-collinear orbit that meets a $\delta$-cycle, and the order-$3$ elements already in $S$ act likewise on the remaining non-collinear orbits.  That every non-collinear $\PGU(3,q)$-orbit is covered in this way, equivalently that the induced image of every non-collinear triple orbit contains a $3$-cycle, is established for every $q$ by the independent coordinate proof below \textup{(the order-$3$ diagonal element of \textup{Theorem \ref{thm:pgu}(1b)} cyclically permutes the three points of a non-collinear triple, so \textup{Lemma \ref{lem:suborbit}} applies)}\textup{)}.  By \textup{Lemma \ref{lem:suborbit}}, $\sigma_3(\PGU(3,q),\mathcal{U})=1+\#\{$orbits of $G^{\mathrm{PGU}}_{\{\alpha,\beta\}}$ off $B\}$, where $G^{\mathrm{PGU}}_{\{\alpha,\beta\}}$ is the setwise stabilizer of the pair in $\PGU(3,q)$, of order $2(q^2-1)$ \textup{(the cyclic torus $T^{\mathrm{PGU}}\cong\Z_{q^2-1}$ together with $\tau$)}.  The cyclic part acts regularly off $B$ (of length $q^2-1$), giving $q$ regular $T^{\mathrm{PGU}}$-orbits;  $\tau$ fixes as sets exactly the orbits meeting the hyperbolic line $L_\tau$ of \textup{Theorem \ref{thm:main-psu}(2)}, and their number is
\[
\#\{T^{\mathrm{PGU}}\text{-orbits meeting }L_\tau\cap\mathcal{U}\}
=\frac{|L_\tau\cap\mathcal{U}|}{\gcd(|T^{\mathrm{PGU}}|,q+1)}
=\frac{q+1}{q+1}=1
\qquad(q\ \text{odd}),
\]
while for $q$ even $\tau$ has no fixed point off $B$ by \textup{Theorem \ref{thm:main-psu}(2)}, so no regular orbit is fixed.  Hence $G^{\mathrm{PGU}}_{\{\alpha,\beta\}}$ has, off $B$, $1+(q-1)/2$ orbits for $q$ odd and $q/2$ orbits for $q$ even, giving $\sigma_3(\PGU(3,q))=(q+3)/2$ and $(q+2)/2$ respectively.  Equivalently, for $q$ odd the $q$ regular $T^{\mathrm{PGU}}$-orbits pair up under $\tau$ with the single fixed orbit, and the $d=3$ diagonal element $\delta$ realizes this fusion inside $\PGU(3,q)$ (cf.\ (1b) of \textup{Theorem \ref{thm:pgu}}).  The case $q=2$ gives $\sigma_3(\PGU(3,2))=2$, in agreement with the formula $(q+2)/2$ \textup{(Table \ref{tab:psu-general})}.

\smallskip
\emph{Independent coordinate proof of the orbit number (every $q$).}  Let $K=\F_{q^2}$, $\bar x=x^q$, and let $h(x_0,x_1,x_2)=x_0\bar x_2+x_1\bar x_1+x_2\bar x_0$ be the standard Hermitian form, so that $\mathcal{U}=\{\langle(x_0,x_1,x_2)\rangle:h(x_0,x_1,x_2)=0\}$.  Put $P=\langle(1,0,0)\rangle$ and $Q=\langle(0,0,1)\rangle$.  As $\PGU(3,q)$ is $2$-transitive on $\mathcal{U}$, every three-subset is $\{P,Q,R\}$ for a third point $R$;  off the block $PQ$ write $R=R(x,y)=\langle(x,y,1)\rangle$ with $y
\ne 0$, the isotropy condition being $x+\bar x+y\bar y=0$.  Set $u=y\bar y\in\F_q^\times$ and $z=x/u$;  then $z\in\mathcal{Z}:=\{z\in K:z+\bar z=-1\}$, and the trace map $\operatorname{Tr}_{K/\F_q}$ has kernel of size $q$, so $|\mathcal{Z}|=q$.  An element of the ordered two-point stabilizer $G_{(P,Q)}$ is diagonal $D=\operatorname{diag}(a,b,c)$ with $a\bar c=b\bar b=c\bar a$, so with $t=a/c$, $s=b/c$ it acts as $R(x,y)\mapsto R(tx,sy)$, where $t=s\bar s\in\F_q^\times$;  hence $u\mapsto tu$ and $x\mapsto tx$, so $z$ is fixed, and conversely any two points with the same parameter $z$ are joined by $\operatorname{diag}(t,s,1)$ with $s=y'/y$.  Thus $G_{(P,Q)}$ acts transitively on the points of parameter $z$, and its non-collinear orbits are parameterised by $z\in\mathcal{Z}$.  The matrix
\[
w=\begin{pmatrix}0&0&1\\0&1&0\\1&0&0\end{pmatrix}
\]
interchanges $P,Q$, preserves $h$, has $\det w=-1$, and sends $R(x,y)$ to $R(x^{-1},y/x)$, whence $z\mapsto\bar z=-1-z$.  So $G_{\{P,Q\}}=G_{(P,Q)}\rtimes\langle w\rangle$ has, off $B$, the orbit space $\mathcal{Z}/\langle\iota\rangle$ with $\iota\colon z\mapsto-1-z$.  The fixed points of $\iota$ solve $2z=-1$:  for $q$ odd there is the unique fixed point $z=-1/2$ and $|\mathcal{Z}/\langle\iota\rangle|=(q+1)/2$ by Burnside;  for $q$ even, $2=0$ gives no fixed point and $|\mathcal{Z}/\langle\iota\rangle|=q/2$.  The points of the block $PQ$ other than $P,Q$ are $\langle(x,0,1)\rangle$ with $x\in T:=\{x\in K:x+\bar x=0\}$, a one-dimensional $\F_q$-space on which the diagonal elements $t\in\F_q^\times$ act regularly, so all collinear triples form a single $\PGU(3,q)$-orbit.  Therefore $\sigma_3(\PGU(3,q),\mathcal{U})=1+(q+1)/2$ for $q$ odd and $1+q/2$ for $q$ even, as stated.
\end{proof}

\begin{remark}[The diagonal-overgroup stabilizer distribution]\label{rem:psu-diag-dist}
For $q$ odd with $3\mid q+1$ ($d=3$), the fusion of the $S$-orbits under the diagonal overgroup $G=\PGU(3,q)$ is recorded in \textup{Proposition \ref{prop:psu-diag}}:  $\sigma_3(G)=(q+3)/2$.  The setwise stabilizer distribution behind this count is as follows:  on the blocks there is a single orbit, of setwise stabilizer order $6(q+1)$ (the full cyclic torus $T^{\mathrm{PGU}}\cong\Z_{q+1}$ together with $S_3$);  off the blocks there is $1$ orbit of order $6$ (image $S_3$, the three $S$-fixed orbits of \textup{Proposition \ref{prop:psu-odd}} fused by the order-$3$ diagonal element $\delta$) together with $(q-1)/2$ orbits of order $3$ (image $C_3$).  Indeed $\delta$ cyclically permutes the three points of every non-collinear $S$-orbit on which it acts, so the $C_3$-orbits and the trivial orbits of \textup{Proposition \ref{prop:psu-odd}} merge into the $(q-1)/2=(q-5)/6+(q+1)/3$ order-$3$ orbits of $G$, and no orbit of trivial image survives in the overgroup.  This distribution is verified in GAP for $q=5,11,17,23$:  collinear order $6(q+1)$;  non-collinear order $6$ (once) and order $3$ ($(q-1)/2$ times);  $\sigma_3=(q+3)/2$.  The count $\sigma_3(G)=(q+3)/2$ follows for general $q$ from the suborbit argument of \textup{Proposition \ref{prop:psu-diag}};  the stabilizer-order distribution is verified computationally as stated.
\end{remark}

\begin{lemma}[Trivial pointwise stabilizers of non-collinear triples]\label{lem:psu-pointwise}
Let $S=\PSU(3,q)$, $q\ge3$, act on the Hermitian unital $\mathcal{U}$, and let $T$ be a triple of points not contained in a block.  Then the pointwise stabilizer $S_{(T)}$ is trivial.
\end{lemma}
\begin{proof}
Fix two points $\alpha,\beta$ of $T$ and put $B$ for the block through them.  Any element of $S_{(T)}$ fixes $\alpha$ and $\beta$ pointwise, hence lies in the pointwise two-point stabilizer $T=S_{\alpha,\beta}\cong\Z_{(q^2-1)/d}$;  it also fixes the third point $x\in T\setminus\{\alpha,\beta\}$, which lies off $B$.  By \textup{Theorem \ref{thm:main-psu}(1)} $T$ acts regularly on $\mathcal{U}\setminus B$, so no non-identity element of $T$ fixes $x$.  Hence $S_{(T)}=1$.
\end{proof}

\smallskip
\emph{Coordinate convention for the stabilizer analysis.}  In the proofs below we
work in the normalized Hermitian basis $(e_1,e_2,e_3)$ of $K^3$, characterized by
$h(e_1,e_1)=h(e_2,e_2)=h(e_3,e_1)=h(e_3,e_2)=0$, $h(e_1,e_2)=1$ and $h(e_3,e_3)=1$;
equivalently the Hermitian form is
$Q(x_0,x_1,x_2)=x_0\bar x_1+x_1\bar x_0+x_2\bar x_2$.  This non-degenerate form
has Witt index $1$ and is similar (by a unique isometry of the unital) to the
standard form $h(x_0,x_1,x_2)=x_0\bar x_2+x_1\bar x_1+x_2\bar x_0$ used in the
independent coordinate proof of \textup{Theorem \ref{thm:main-psu}} above, so both
realize the same Hermitian unital $\mathcal{U}$;  the specific coordinate system
in use is always the one indicated.  Under this normalized basis the points
$\alpha=e_1$ and $\beta=e_2$ are isotropic with $h(\alpha,\beta)=1$, $e_3$ is
anisotropic and orthogonal to $\mathrm{Span}\{\alpha,\beta\}$, and a third point
in general position is $\gamma=(1,u,z)$ with
$\mathrm{Tr}_{K/F}(u)+N_{K/F}(z)=0$;  in these coordinates
$h(e_1,\gamma)=\bar u$ and $h(e_2,\gamma)=1$, so the Gram matrix of the triple
$(\alpha,\beta,\gamma)$ is
\[
\Gamma=\begin{pmatrix}0&1&\bar u\\1&0&1\\u&1&0\end{pmatrix}.
\]
The isometry from the standard form
$h(X_0,X_1,X_2)=X_0\bar X_2+X_1\bar X_1+X_2\bar X_0$ of
\textup{Theorem \ref{thm:main-psu}} to $Q$ reads
$(X_0,X_1,X_2)\mapsto(X_0,X_2,X_1)$;  under it the normalized point $\gamma=(1,u,z)$
corresponds to the standard-form point $(1,z,u)$, whose $X_2=1$ normalization is
$(u^{-1},zu^{-1},1)$.  In particular the $C_3$-criterion of
\textup{Lemma \ref{lem:psu-c3}}, $\bar u/u\in(K^1)^3$, reads in the standard-form
coordinates as $\bar x/x\in(K^1)^3$ for the first coordinate $x=u^{-1}$ of that
normalization, the two quotients being mutual inverses and hence of the same class
modulo $(K^1)^3$.

\begin{lemma}[The $C_3$-criterion on the unital]
\label{lem:psu-c3}
Let $q\ge3$ with $d=\gcd(3,q+1)=3$ (equivalently $q=2^{2m+1}>2$, or $q$ odd with
$3\mid q+1$), and write $K=\F_{q^2}$, $F=\F_q$,
$K^1=\{a\in K^{\times}:a^{q+1}=1\}$.  Let $S=\PSU(3,q)$ act on $\mathcal{U}$, fix
$\alpha,\beta\in\mathcal{U}$ with block $B$ through them, and write the third point as
$\gamma=(1,u,z)$ in the coordinates $\alpha=e_1$, $\beta=e_2$, so that
$\gamma\in\mathcal{U}\setminus B$ iff $u
\notin F^{\times}$ and $z
\ne 0$ satisfy
$\mathrm{Tr}_{K/F}(u)+N_{K/F}(z)=0$.  Then the image of $S_{\{\alpha,\beta,\gamma\}}$
in $S_3$ contains a $3$-cycle if and only if
\[
\frac{\bar u}{u}\in (K^1)^3.
\]
\end{lemma}
\begin{proof}
Let $v_1=\alpha$, $v_2=\beta$, $v_3=\gamma$ and let
$\Gamma=\bigl(h(v_i,v_j)\bigr)_{i,j}$ be the Gram matrix of the triple, namely
\[
\Gamma=\begin{pmatrix}0&1&\bar u\\1&0&1\\u&1&0\end{pmatrix}.
\]
An element $g\in\PGU(3,q)$ inducing the cycle $\sigma=(123)$ on $\{v_1,v_2,v_3\}$
is a similarity of $\Gamma$:  $g(v_i)=\lambda_i\,v_{\sigma(i)}$ with
$\lambda_i\in K^{\times}$ and
$\lambda_i\lambda_j^q\,\Gamma_{\sigma(i)\sigma(j)}=
u\,\Gamma_{ij}$ for some
$
u\in F^{\times}$.  This system reduces to
\[
\lambda_1\lambda_2^q=
u,\quad \lambda_1\lambda_3^q=
u\bar u,\quad
\lambda_2\lambda_3^q\,u=
u,\quad \lambda_3\lambda_1^q=
u u,
\]
and admits a solution in $K^{\times}$ for every $u
\ne 0$ (e.g.\
$\lambda_2=\lambda_1/N(u)$, $\lambda_3=
u u/\lambda_1^q$ with $\lambda_1$ chosen so
that $N(\lambda_1)=
u N(u)$).  Hence a $3$-cycle is realized in $\PGU(3,q)$ for
every $u
\ne 0$.  Its class in $\PGU(3,q)/\PSU(3,q)\cong C_3$ is computed as follows.  Let $g$ be the matrix of the cycle in the basis $(v_1,v_2,v_3)$:  $g(v_1)=\lambda_1v_2$, $g(v_2)=\lambda_2v_3$, $g(v_3)=\lambda_3v_1$, so $\det g=\lambda_1\lambda_2\lambda_3$.  From the four equations $\lambda_2=\lambda_1/N(u)$, $\lambda_3=
u u/\lambda_1^q$ and $N(\lambda_1)=
u N(u)$, whence
\[
\det g=\frac{\lambda_1^3\,u}{N(u)^2},
\]
and in particular $\det g^{q+1}=
u^3$, as for every unitary similitude of multiplier $
u$.  Since $\PGU(3,q)=\GU(3,q)/\F_q^{\times}$ and $\PSU(3,q)=\SU(3,q)/\mu_3$, an element represented by a similitude $A$ with multiplier $
u_A$ lies in $\PSU(3,q)$ iff some $c\in\F_q^{\times}$ satisfies $N(c)
u_A=1$ and $c^3\det A=1$;  as $N(c)=c^2$ for $c\in\F_q^{\times}$, this is equivalent to $\det A^2=
u_A^3$, the condition $\det A\in\F_q^{\times}$ being automatic ($\det A^{q+1}=
u_A^3$ and $\det A^2=
u_A^3$ force $\det A^{q-1}=1$).  For our cycle,
\[
\det g^2=
u^3 \iff \frac{\lambda_1^6u^2}{N(u)^4}=\frac{N(\lambda_1)^3}{N(u)^3}
\iff u^{q-1}=(\lambda_1^{q-1})^{-3},
\]
using $N(\lambda_1)=
u N(u)$;  as $\lambda_1^{q-1}\in K^1$ and $(K^1)^3$ is closed
under inversion, this is equivalent to $u^{q-1}=\bar u/u\in(K^1)^3$.  Hence a determinant-$1$ representative exists in $\PSU(3,q)$ (equivalently $S_{\{\alpha,\beta,\gamma\}}$ contains a $3$-cycle) precisely when $\bar u/u\in(K^1)^3$, and the criterion is independent of the choice of $\lambda_1$.  The criterion, together with its
consequences in \textup{Proposition \ref{prop:psu-even}}, is verified for $q=8,32$
in \textup{Chapter \ref{sec:verify}};  the proof uses only the Gram determinant
class modulo $(K^1)^3$ and the identification $\PGU(3,q)/\PSU(3,q)\cong C_3$, and
hence holds verbatim for odd $q$ with $3\mid q+1$ \textup{(supplemented by the
verification for $q=5,11,17,23$ in \textup{Chapter \ref{sec:verify}})}.
\end{proof}

\begin{proposition}[The even-$d=3$ unitary case:  closed forms]
\label{prop:psu-even}
Let $q=2^{2m+1}>2$, so that $d=\gcd(3,q+1)=3$, and let $S=\PSU(3,q)$ act on the
Hermitian unital $\mathcal{U}$.  Then
\begin{enumerate}
\item the setwise pair stabilizer $G_{\{\alpha,\beta\}}$ has exactly $3q/2$
suborbits off the block $B$ (of length $2(q^2-1)/3$), and none is fixed by $\tau$ as
a set;
\item the non-collinear three-point orbits of $S$ comprise $(q-2)/2$ orbits of
$C_3$-type (setwise stabilizer of order $3$, $\lambda=2(q^2-1)/3$) and $(q+1)/3$
orbits of trivial stabilizer ($\lambda=2(q^2-1)$);
\item $\sigma_3(S,\mathcal{U})=1+\tfrac{q-2}{2}+\tfrac{q+1}{3}=\tfrac{5q+2}{6}$.
\end{enumerate}
The values are verified for $q=8,32$ in \textup{Chapter \ref{sec:verify}}.
\end{proposition}
\begin{proof}
(1) By \textup{Theorem \ref{thm:main-psu}(2)}, $k=0$ for $q$ even, so off $B$ the
group $G_{\{\alpha,\beta\}}$ has $(qd-k)/2=3q/2$ merged suborbits of length
$2|T|=2(q^2-1)/3$, none fixed by $\tau$ as a set.
(2) For $\gamma=(1,u,z)$ with $u
\notin F^{\times}$, $z
\ne 0$, the criterion of
\textup{Lemma \ref{lem:psu-c3}} says that $\{\alpha,\beta,\gamma\}$ is of $C_3$-type
iff $\bar u/u\in(K^1)^3$.  The map $uF^{\times}\mapsto\bar u/u$ identifies
$K^{\times}/F^{\times}\cong K^1$;  $|K^1|=q+1$ and $(K^1)^3$ has index $3$ inside
$K^1$ (as $\gcd(3,q+1)=3$).  The collinear coset ($u\in F^{\times}$) corresponds to
$\bar u/u=1\in(K^1)^3$, so the number of non-collinear $C_3$-type cosets is
$|(K^1)^3|-1=(q+1)/3-1=(q-2)/3$.  Each coset carries $q-1$ values of $u$, and for
each such $u$ the $q+1$ values $z
\ne 0$ with $\mathrm{Tr}_{K/F}(u)+N_{K/F}(z)=0$;
hence the $C_3$-type non-collinear points number
$\frac{q-2}{3}(q-1)(q+1)$.  Since $T=S_{\alpha,\beta}$ acts regularly on
$\mathcal{U}\setminus B$ \textup{(Theorem \ref{thm:main-psu}(1))}, these split into
\[
\frac{\frac{q-2}{3}(q-1)(q+1)}{(q^2-1)/3}=q-2
\]
regular $T$-suborbits, which by (1) pair under $\tau$ into $(q-2)/2$ merged
$C_3$-suborbits;  each yields one $C_3$-orbit of $S$ \textup{(Lemma
\ref{lem:suborbit})}.  Thus $m_{C_3}=(q-2)/2$, and by \eqref{eq:lambda}
$\lambda=6(q^2-1)/(d\cdot3)=2(q^2-1)/3$.
(3) By (1) and (2) the trivial merged suborbits number $3q/2-(q-2)/2=q+1$.  A
$C_3$-orbit contributes one merged suborbit, while a trivial orbit (trivial setwise
stabilizer) contributes exactly three:  its three $2$-subsets are singletons under the
trivial image, and $\hat T=G_{\{\alpha,\beta\}}$ of order $2(q^2-1)/3$ acts freely on
the $\lambda=2(q^2-1)$ blocks through $\{\alpha,\beta\}$ of that orbit (the block
stabilizer is $S_{\Ocal}\cap\hat T=1$), splitting them into
$\lambda/|\hat T|=3$ suborbits.  Hence $m_1=(q+1)/3$, $\lambda=2(q^2-1)$, and
\[
\sigma_3(S,\mathcal{U})=1+\frac{q-2}{2}+\frac{q+1}{3}=\frac{5q+2}{6}.
\]
\end{proof}

\begin{lemma}[Fixed orbits carry the full $S_3$ image]
\label{lem:psu-odd-fixed}
Let $q$ be odd with $d=\gcd(3,q+1)=3$, and let $S=\PSU(3,q)$ act on the unital
$\mathcal{U}$.  Let $\alpha,\beta\in\mathcal{U}$ be a pair on a block $B$, let
$\tau\in S_{\{\alpha,\beta\}}$ be the order-$2$ involutory unitary reflection fixing
the hyperbolic plane $L_\tau$ pointwise, and let $x\in L_\tau\cap\mathcal{U}$.  Then
the image of $S_{\{\alpha,\beta,x\}}$ in $S_3$ is the full $S_3$;  in particular
$|S_{\{\alpha,\beta,x\}}|=6$.
\end{lemma}
\begin{proof}
Take coordinates $\alpha=e_1$, $\beta=e_2$ with $h(\alpha,\beta)=1$ and $e_3$
anisotropic and orthogonal to $\mathrm{Span}\{\alpha,\beta\}$, normalized so that
$\tau(\alpha)=\beta$ and $\tau(\beta)=\alpha$.  By \textup{Theorem
\ref{thm:main-psu}(2)}, $L_\tau\cap B$ is the anisotropic point $w_0
\notin\mathcal{U}$,
so $x
\notin B$ and $x$ is not collinear with $\alpha,\beta$.  The first coordinate of
$x$ is nonzero:  otherwise $x\in\mathrm{Span}\{e_2,e_3\}$ with $x_1=0$, and
$h(x,x)=x_3x_3^q=|x_3|^2=0$ forces $x_3=0$, so $x\in\mathrm{Span}\{e_2\}\subseteq B$, a
contradiction.  Normalize $x_1=1$ and write $x=(1,u,z)$.
Since $\tau$ fixes $L_\tau$ pointwise, $\tau(x)=x$;  and $\tau$ being an isometry
fixing $\alpha,\beta$ up to interchange gives $\tau(e_1)=e_2$, $\tau(e_2)=e_1$, while
$\tau(e_3)\perp\mathrm{Span}\{\tau(e_1),\tau(e_2)\}=\mathrm{Span}\{e_1,e_2\}$, hence
$\tau(e_3)=s\,e_3$ with $s^2=1$.  Comparing components of
$\tau(x)=\tau(x_1e_1+x_2e_2+x_3e_3)=x_1e_2+x_2e_1+x_3se_3=x_1e_1+x_2e_2+x_3e_3$
gives $x_2=x_1$, whence $u=x_2/x_1=1$.  By \textup{Lemma \ref{lem:psu-c3}} the image
contains a $3$-cycle, since $\bar u/u=1\in(K^1)^3$.  Together with the transposition
$(\alpha\,\beta)$ induced by $\tau$, the image is the full $S_3$.  By \textup{Lemma
\ref{lem:psu-pointwise}} the pointwise stabilizer is trivial, so
$|S_{\{\alpha,\beta,x\}}|=6$.
\end{proof}

\begin{proposition}[The odd-$d=3$ unitary case:  closed forms]
\label{prop:psu-odd}
Let $q$ be odd with $d=\gcd(3,q+1)=3$, and let $S=\PSU(3,q)$ act on the Hermitian
unital $\mathcal{U}$.  Then
\begin{enumerate}
\item the setwise pair stabilizer $G_{\{\alpha,\beta\}}$ has exactly $3$ fixed
suborbits off the block $B$ (of length $(q^2-1)/3$, image $S_3$) and, off $B$,
$(q-5)/2$ merged $C_3$-suborbits and $q+1$ trivial suborbits;
\item the non-collinear three-point orbits of $S$ comprise $3$ orbits of $S_3$-type
(setwise stabilizer of order $6$, $\lambda=(q^2-1)/3$), $(q-5)/2$ orbits of
$C_3$-type (order $3$, $\lambda=2(q^2-1)/3$), and $(q+1)/3$ orbits of trivial
stabilizer ($\lambda=2(q^2-1)$);
\item $\sigma_3(S,\mathcal{U})=1+3+\tfrac{q-5}{2}+\tfrac{q+1}{3}
=\tfrac{5q+11}{6}$.
\end{enumerate}
The values are verified for $q=5,11,17,23$ in \textup{Chapter \ref{sec:verify}}.
\end{proposition}
\begin{proof}
(1) By \textup{Theorem \ref{thm:main-psu}(2)} $k=d=3$, so there are $3$ fixed regular
$T$-orbits off $B$;  by \textup{Lemma \ref{lem:psu-odd-fixed}} each carries the full
$S_3$ image, giving $m_{S_3}=3$ and, by \eqref{eq:lambda},
$\lambda=6(q^2-1)/(d\cdot6)=(q^2-1)/3$.
(2) By \textup{Lemma \ref{lem:psu-c3}}, a non-collinear $\gamma=(1,u,z)$ is of
$C_3$-type iff $\bar u/u\in(K^1)^3$.  As in \textup{Proposition \ref{prop:psu-even}(2)},
the number of non-collinear $C_3$-type cosets is $|(K^1)^3|-1=(q-2)/3$, each carrying
$q-1$ values of $u$ and $q+1$ values $z
\ne 0$;  hence the $C_3$-type non-collinear
points number $\frac{q-2}{3}(q-1)(q+1)$, which split under the regular action of
$T=S_{\alpha,\beta}$ \textup{(Theorem \ref{thm:main-psu}(1))} into
$\frac{\frac{q-2}{3}(q-1)(q+1)}{(q^2-1)/3}=q-2$ regular $T$-suborbits.  The $3$
fixed orbits satisfy the criterion (their image $S_3$ contains a $3$-cycle), so they
lie among these $q-2$ suborbits;  the remaining $q-5$ contain no $\tau$-fixed point
and pair under $\tau$ into $(q-5)/2$ merged $C_3$-suborbits, each giving one
$C_3$-orbit of $S$ \textup{(Lemma \ref{lem:suborbit})}.  Thus $m_{C_3}=(q-5)/2$.
(3) By \textup{Theorem \ref{thm:main-psu}(2)}, $G_{\{\alpha,\beta\}}$ has, off $B$,
$k+(qd-k)/2=3+(3q-3)/2=(3q+3)/2$ suborbits.  A $C_3$-orbit contributes one merged
suborbit, an $S_3$-orbit one, and a trivial orbit three (its three $2$-subsets are
singletons under the trivial image);  hence
$3+m_{C_3}+3m_1=(3q+3)/2$, giving $3m_1=(3q+3)/2-3-(q-5)/2=q+1$, so
$m_1=(q+1)/3$.  Consequently
\[
\sigma_3(S,\mathcal{U})=1+3+\frac{q-5}{2}+\frac{q+1}{3}
=1+3+\frac{3q-15+2q+2}{6}=\frac{5q+11}{6}.
\]
\end{proof}

\begin{proposition}[Computed stabilizer refinement for $q\le49$]
\label{prop:psu-computed}
Let $S=\PSU(3,q)$, $d=\gcd(3,q+1)$, and let the non-collinear triple orbits be classified by the order of the setwise stabilizer $S_{\Ocal}$ \textup{(}$6$, $3$ or $1$, with trivial pointwise stabilizer in every case, by \textup{Lemma \ref{lem:psu-pointwise}}\textup{)}.  The multiplicities are, for $q\le49$:
\begin{center}
\resizebox{0.95\textwidth}{!}{$\begin{array}{c|c|c|c}
 & q\ \text{odd},\ d=1 & q\ \text{odd},\ d=3 & q\ \text{even} \\ \hline
\#\ \text{orbits with stabilizer }6\ (S_3\text{-image}) & 1 & 3\ \text{(proved)} & 0 \\
\#\ \text{orbits with stabilizer }3\ (C_3\text{-image}) & (q-1)/2 & (q-5)/2\ \text{(proved)} & q/2\ \text{(d=1)},\quad (q-2)/2\ \text{(d=3,\ proved)} \\
\#\ \text{orbits with stabilizer }1 & 0 & (q+1)/3\ \text{(proved)} & 0\ \text{(d=1)},\quad (q+1)/3\ \text{(d=3,\ proved)}
\end{array}$}
\end{center}
where in the last column the $d=1$ case means all $q/2$ merged orbits have image $C_3$ \textup{(}verified for $q=4,16$\textup{)};  for $d=3$ the even-$q$ multiplicities are proved for every $q=2^{2m+1}$ in \textup{Proposition \ref{prop:psu-even}} ($(q-2)/2$ and $(q+1)/3$), while the odd-$d=3$ multiplicities are proved for every $q$ with $3\mid q+1$ in \textup{Proposition \ref{prop:psu-odd}} ($3$, $(q-5)/2$ and $(q+1)/3$).  The present proposition records the small-$q$ computational refinement (for $q=8$ there are $3$ $C_3$-orbits and $3$ trivial orbits, and the trivial suborbits $3q/2-3=9$ fuse three-to-one;  for $q=32$ the values are $15$ and $11$, and the stabilizer orders and $\lambda$-values are listed in \textup{Table \ref{tab:magma}}\textup{)};  the exceptional value $q=2$ is listed separately in \textup{Table \ref{tab:magma}}.  These values are finite computations \textup{(Magma and GAP, \textup{Chapter \ref{sec:verify}}}\textup{)} covering $q\le49$;  together with \textup{Theorem \ref{thm:main-psu}(2)} and \textup{Propositions \ref{prop:psu-odd} and \ref{prop:psu-even}} they give the closed forms of \textup{Theorem \ref{thm:main-psu}(3)} in every row:  the $d=1$ rows for every $q$, the odd-$d=3$ row by \textup{Proposition \ref{prop:psu-odd}}, and the even-$d=3$ row by \textup{Proposition \ref{prop:psu-even}}.
\end{proposition}

\begin{remark}[The fixed line of $\tau$;  the order-$3$ elements]
\label{rem:psu-general-computation}
The geometric content of the proof of (2) is the following.  For $q$ odd the involution $\tau$ fixes the hyperbolic plane $L_\tau$ pointwise and the anisotropic line $p_\tau$ pointwise;  $L_\tau\cap B=\{w_0\}$ is an anisotropic point (proved via $\alpha+\beta$, which is $\tau$-fixed and anisotropic, not lying in $p_\tau$), so the $q+1$ isotropic points of $L_\tau$ are exactly the fixed points of $\tau$ off the block, and $\tau$ fixes none of the block points.  For $q$ even the transvection $\tau$ fixes the tangent plane through $w_0\in B\cap\mathcal{U}$ pointwise, so it has no fixed point off $B$.  This is confirmed by the computed values $\operatorname{fix}(\tau;\mathcal{U}\setminus B)=q+1$ ($q$ odd) and $0$ ($q$ even) of \textup{Table \ref{tab:psu-general}}.
\smallskip

The $C_3$-stabilizer orbits are the non-collinear three-cycles of the order-$3$ elements of $\PSU(3,q)$:  an order-$3$ element $g$ fixes a set $F$ of unital points with $|F|\equiv q^3+1\pmod3$, where $|F|=1$ precisely when $3\mid q$ (then the remaining $q^3$ points split into $q^3/3$ three-cycles), while $|F|\in\{0,3\}$ when $3\mid q+1$;  and a non-collinear triple whose setwise stabilizer contains a $C_3$ is a three-cycle of \emph{exactly one} such element (two order-$3$ elements sharing a non-collinear three-cycle differ by an element fixing that triple pointwise, hence agree by \textup{Theorem \ref{thm:main-psu}(4)});  the collinear three-cycles are the blocks, already accounted for in the block orbit.  The distribution of the $C_3$-orbits among the merged orbits is read off from \textup{Table \ref{tab:magma}} for $q\le49$ and recorded in \textup{Proposition \ref{prop:psu-computed}};  the pattern stated there is the observed one and is proved in the $d=1$ rows and in the $d=3$ rows \emph{(}the odd-$d=3$ rows in \textup{Proposition \ref{prop:psu-odd}}, the even-$d=3$ rows in \textup{Proposition \ref{prop:psu-even}}\emph{)}.  On the fixed orbits the required $3$-cycle, an order-$3$ element cyclically permuting $\{\alpha,\beta,x\}$ with $x\in L_\tau\cap\mathcal U$, is proved for every $q$ odd with $3\mid q+1$ in \textup{Lemma \ref{lem:psu-odd-fixed}}:  the symmetry $\tau(x)=x$ forces the second coordinate $u$ of the normalized point $x=(1,u,z)$ to be $1$, whence $\bar u/u=1\in(K^1)^3$ and the criterion of \textup{Lemma \ref{lem:psu-c3}} applies;  the order-$6$ values of the fixed-orbit column of \textup{Table \ref{tab:magma}} therefore no longer rest on computational input but follow from \textup{Lemma \ref{lem:psu-odd-fixed}}.
\end{remark}

\begin{remark}[The non-collinear types]
\label{rem:psu-types}
For the theoretically settled cases ($d=1$) \textup{Theorem \ref{thm:main-psu}} gives, in addition to the collinear orbit $(\lambda=q-1)$, the non-collinear orbits with stabilizer $S_3$ (the $k$ fixed orbits, $k=1$) and the $(q-1)/2$ merged orbits with stabilizer $C_3$, together with the closed forms $\sigma_3(\PSU(3,q))$ and $\sigma_3(\PGU(3,q))$.  For $q=3^a$ this reduces to one non-collinear orbit with stabilizer $S_3$ and $\lambda=q^2-1$ plus $\frac{q-1}{2}$ orbits with stabilizer $C_3$ and $\lambda=2(q^2-1)$;  for $q$ even with $d=1$ the $(q/2)$ merged orbits all have image $C_3$.  For the remaining cases the stabilizer types and multiplicities are proved in \textup{Propositions \ref{prop:psu-odd} and \ref{prop:psu-even}} \textup{(the odd-$d=3$ case in \textup{Proposition \ref{prop:psu-odd}}, the even-$d=3$ case in \textup{Proposition \ref{prop:psu-even}})}), with the small-$q$ computational refinement recorded in \textup{Proposition \ref{prop:psu-computed}} and \textup{Table \ref{tab:magma}}.
\end{remark}

\begin{remark}
For $q=3$ \textup{(}$a=1$\textup{)} the unital has $28$ points and $S=\PSU(3,3)$ of order $6048$;  \textup{Theorem \ref{thm:main-psu}} gives the three systems $2$-$(28,3,2)$, $2$-$(28,3,16)$ and $2$-$(28,3,8)$ \textup{(}collinear, and non-collinear with stabilizers $3,6$\textup{)}, in agreement with \textup{Table \ref{tab:magma}}.  For $q=2$ \textup{(}excluded from the theorem but included in the tables\textup{)} the stabilizers are $1$ \textup{(}non-collinear\textup{)} and $6$ \textup{(}collinear\textup{)}:  $|\mathcal{U}|=9$, $|S|=72$.
\end{remark}

\begin{corollary}[Design parameters for the unitary family]
\label{cor:psu-designs}
The distinct triple systems with point set the unital and automorphism group containing $\PSU(3,q)$ correspond to the orbits in \textup{Table \ref{tab:magma}}.  Since $|S|=q^3(q^3+1)(q^2-1)/d$ and $v=q^3+1$,
\[
\lambda = \frac{6|S|}{|S_{\Ocal}|\,v(v-1)}
=\frac{6\,q^3(q^3+1)(q^2-1)/d}{|S_{\Ocal}|\,(q^3+1)q^3}
=\frac{6(q^2-1)}{d\,|S_{\Ocal}|},
\]
and $\lambda$ is an integer for every orbit that occurs \textup{(}\textup{Table \ref{tab:magma}}, \textup{Table \ref{tab:verify}}\textup{)}.  In particular, for a collinear triple, $|S_{\Ocal}|=6(q+1)/d$ \textup{(}\textup{Theorem \ref{thm:collinear}}\textup{)} gives
\[
\lambda=\frac{6(q^2-1)}{d\cdot 6(q+1)/d}=q-1,
\]
an integer for every $q$:  the collinear orbit yields the triple system $2$-\mbox{$(q^3+1,3,q-1)$}.  By \textup{Proposition \ref{prop:union-orbits}}, every nonempty union of these orbit designs is a $\PSU(3,q)$-invariant triple system whose parameter is the corresponding subset sum of the orbit parameters.  \textup{Theorem \ref{thm:main-psu}(4)} and \textup{Proposition \ref{prop:psu-computed}} determine the complete non-collinear $\lambda$-list:  the $k$ fixed orbits ($|S_{\Ocal}|=6$) give $\lambda=(q^2-1)/d$, the $C_3$-orbits give $\lambda=2(q^2-1)/d$, and the orbits with trivial stabilizer give $\lambda=6(q^2-1)/d$, with the multiplicities
{\small
\[
\begin{array}{c|c|c|c}
& q\ \text{odd},\ d=1 & q\ \text{odd},\ d=3 & q\ \text{even} \\ \hline
\lambda=(q^2-1)/d & 1 & 3\ \text{(proved)} & 0 \\
\lambda=2(q^2-1)/d & \tfrac{q-1}{2} & \tfrac{q-5}{2}\ \text{(proved)} & \tfrac{q}{2}\ \text{(d=1)},\ \tfrac{q-2}{2}\ \text{(d=3,\ proved)} \\
\lambda=6(q^2-1)/d & 0 & \tfrac{q+1}{3}\ \text{(proved)} & 0\ \text{(d=1)},\ \tfrac{q+1}{3}\ \text{(d=3,\ proved)}
\end{array}
\]
}
with the odd-$d=3$ entries proved for every $q$ with $3\mid q+1$ in \textup{Proposition \ref{prop:psu-odd}} and the even-$d=3$ entries proved for every $q=2^{2m+1}$ in \textup{Proposition \ref{prop:psu-even}}.
Thus, for instance:  $q=3^a$ ($d=1$, odd):  one system $2$-\mbox{$(q^3+1,3,q^2-1)$} and $\frac{q-1}{2}$ copies of $2$-\mbox{$(q^3+1,3,2(q^2-1))$} (e.g.\ $q=3$:  $2$-$(28,3,2)$, $2$-$(28,3,16)$, $2$-$(28,3,8)$;  $q=9$:  $2$-$(730,3,8)$, $2$-$(730,3,80)$ and four copies of $2$-$(730,3,160)$);  $q=5$ ($d=3$):  three systems $2$-$(126,3,8)$ and two systems $2$-$(126,3,48)$, no $C_3$-orbits \textup{($\tfrac{q-5}{2}=0$;  the three fixed $S_3$-orbits give $\lambda=8$ and the $(q+1)/3=2$ trivial orbits give $\lambda=48$, see \textup{Table \ref{tab:magma}})};  $q=4$ ($d=1$):  $\lambda=3$ (collinear) and $30$ (non-collinear);  $q=8$ ($d=3$):  $\lambda=7$ (collinear), $42$ (three $C_3$-orbits) and $126$ (three trivial orbits);  $q=2$ ($d=3$):  $\lambda=1$ (collinear, the affine plane $\mathrm{STS}(9)$) and $6$ (the unique non-collinear orbit).  All values are listed in \textup{Section \ref{sec:params}}.
\end{corollary}

\subsection{Almost simple overgroups of \texorpdfstring{$\PSU(3,q)$}{PSU(3,q)}}

Let $S=\PSU(3,q)\le G\le\PpGU(3,q)=\PGU(3,q)\rtimes C_{2f}$.  Recall $\PGU(3,q)/S=C_d$ with $d=(3,q+1)$, and $\Out(S)=C_d\times C_{2f}$.

\begin{theorem}[Overgroups of $\PSU(3,q)$]
\label{thm:pgu}
Let $S\le G\le\PpGU(3,q)$, $d=(3,q+1)$, and $\Ocal\in\binom{\mathcal{U}}{3}$.
\begin{enumerate}
\item \textbf{Diagonal overgroup.}  If $G=\PGU(3,q)$ \textup{(}$d>1$\textup{)}:
  \begin{enumerate}
  \item collinear $\Ocal$:  $G_{\Ocal}\cong\Z_{q+1}\rtimes S_3$, $|G_{\Ocal}|=6(q+1)$, $G_{(\Ocal)}\cong\Z_{q+1}$ \textup{(}the kernel is the preimage of $\Z_{(q+1)/d}$ under the quotient by the diagonal element of order $3$\textup{)};
  \item non-collinear $\Ocal$:  the image in $S_3$ is as in \textup{Theorem \ref{thm:main-psu}} for $S$;  additionally, if $3\mid q+1$ the diagonal element $\delta\in\PGU(3,q)\setminus S$ of order $3$ acts on the unital with $3$-cycles $\{[x:y:z],[x:y:\zeta z],[x:y:\zeta^2 z]\}$ for suitable coordinates and $\zeta^{q+1}=1$ primitive \textup{(explicit coordinate proof, see below}\textup{)};  if $\Ocal$ is such a $3$-cycle then $G_{\Ocal}$ contains $\delta$ and the image contains $C_3$.  The fusion picture is now settled for every $q$:  $\sigma_3(\PGU(3,q),\mathcal{U})=(q+3)/2$ for $q$ odd and $(q+2)/2$ for $q$ even \textup{(Proposition \ref{prop:psu-diag}}\textup{)} \textup{(for $d=1$ the overgroup equals $S$, and these values coincide with the $d=1$ closed forms of \textup{Theorem \ref{thm:main-psu}(3)})}, so $\PGU(3,q)$ fuses exactly $(3q-1)/2$ of the $S$-orbits for $q$ odd and $q-1$ of them for $q$ even \textup{(including $q=2$, where $\sigma_3(\PGU(3,2))=2$)}.
  \end{enumerate}
\item \textbf{Field overgroup.}  If $G=S\rtimes C_{f'}$ \textup{(}$f'\mid f$, field automorphisms only\textup{)}:  field automorphisms fix pointwise the sub-unital over the subfield $\F_{p^{f/f'}}$;  if $\Ocal$ lies in that sub-unital then $G_{(\Ocal)}$ contains the corresponding field automorphisms \textup{(}a cyclic group\textup{)};  the image in $S_3$ is unchanged from \textup{Theorem \ref{thm:main-psu}} \textup{(}field automorphisms preserve collinearity and the orbit types, since they conjugate $S$ to itself and normalize the unital\textup{)}.  For $q=3^a$ the full field part is cyclic of order $2f$, generated by the Frobenius $\sigma_p\colon[x:y:z]\mapsto[x^p:y^p:z^p]$;  its unique order-$2$ element is the conjugation $\sigma\colon[x:y:z]\mapsto[x^q:y^q:z^q]$, for which the GAP computation of \textup{Table \ref{tab:psu-pgu}} in the range $q=3,9,27$ verifies cycle type $1^{q+1}2^{q(q^2-1)/2}$ on the $q^3+1$ points and fixes every suborbit as a set, hence every $S$-orbit on triples as a set, so the conjugation alone causes no fusion.  By \textup{Proposition \ref{prop:psu-field-nofusion}} the conjugation is proved not to fuse for every $q\ge3$;  the full field part $C_{2f}$ (generated by $\sigma_p$) may fuse for $f>1$ \textup{(Remark \ref{rem:psu-field})}.
\item \textbf{Full overgroup.}  $G=\PpGU(3,q)$:  by the general reduction of \textup{Proposition \ref{prop:outer}}, $S_{\Ocal}\le G_{\Ocal}$ and $G_{\Ocal}/S_{\Ocal}$ embeds into $\PpGU(3,q)/S$.  For $q=3^a$ with $d=1$, $\PGU(3,q)=S$ and $\PpGU(3,q)/S$ is the field part $C_{2f}$, generated by the Frobenius $\sigma_p$ of order $2f$;  its unique order-$2$ element is the conjugation $\sigma_q=\sigma_p^f$, which alone causes no fusion for every $q\ge3$ (Proposition~\ref{prop:psu-field-nofusion}), while the full field part $C_{2f}$ fuses for $f>1$ (Remark~\ref{rem:psu-field}).  In particular
\[
\sigma_3(\PpGU(3,q),\mathcal{U})=\sigma_3(\PSU(3,q),\mathcal{U})=\frac{q+3}{2}
\]
holds for $q=3$ ($f=1$, field part $C_2$), and the computed values $\sigma_3(\PpGU(3,q))=4,7$ obtain for $q=9,27$.  The diagonal part is proved in \textup{Proposition \ref{prop:psu-diag}} and the field part $C_{f'}$ fixes pointwise the sub-unital over the subfield and preserves collinearity and the orbit types \textup{((2))}, but its fusion of $S$-orbits on triples outside the sub-unital is not claimed \textup{(Remark \ref{rem:psu-field})}.
\end{enumerate}
\end{theorem}
\begin{proof}
(1a) As in \textup{Theorem \ref{thm:collinear}}, with the pointwise stabilizer now of order $q+1$:  the block stabilizer has shape $\Z_{q+1}\rtimes\PGL(2,q)$-type (a quotient of the full unitary point stabilizer), and its kernel on a collinear triple is the cyclic group $\Z_{q+1}$ (a direct coordinate computation with the standard Hermitian form;  cf.\ \cite{mitchell1911});  the diagonal element $\delta$ of order $3$ acts with $3$-cycles on the unital and does \emph{not} fix the triple pointwise, so the kernel is the cyclic part, not the diagonal.
(1b) In the standard coordinates of \textup{Lemma \ref{lem:unital}}, the diagonal element is $\delta\colon[x:y:z]\mapsto[x:y:\zeta z]$ with $\zeta^{q+1}=1$, $\zeta
\ne1$, of order $3$ when $3\mid q+1$.  A unital point satisfies $x^{q+1}+y^{q+1}+z^{q+1}=0$;  since $(\zeta z)^{q+1}=z^{q+1}$, the three points $[x:y:z]$, $[x:y:\zeta z]$, $[x:y:\zeta^2 z]$ all lie on the unital, and they are distinct precisely when $z
\ne 0$.  Hence $\delta$ permutes the unital points with $3$-cycles on the set $z
\ne 0$ and fixes pointwise the sub-unital $z=0$;  in particular a triple of the stated form is a $3$-cycle for $\delta$, so $G_{\Ocal}$ contains $\delta$ and the image contains $C_3$.  A transposition of the three points is not realized by $\delta$.
(2) In the standard coordinates a field automorphism is $\sigma\colon[x:y:z]\mapsto[x^{p^k}:y^{p^k}:z^{p^k}]$;  it fixes precisely the points of the sub-unital $\mathcal{U}'$ over the subfield $\F_{p^{f/f'}}$ \textup{(those with coordinates in $\F_{p^{f/f'}}$}\textup{)} pointwise, it preserves the Hermitian equation \textup{(the form is defined over $\F_p$}\textup{)}, collinearity is preserved, and orbit types are preserved because $\sigma$ normalizes $S$.
\end{proof}

\begin{proposition}[The field conjugation causes no fusion for general $q$]
\label{prop:psu-field-nofusion}
Let $q\ge3$, let $S=\PSU(3,q)\le G=\PGU(3,q)$ act on the Hermitian unital $\mathcal{U}$, and let $\sigma\colon[x:y:z]\mapsto[x^q:y^q:z^q]$ be the field conjugation of order $2$ \textup{(the field automorphism fixing $\F_q$ pointwise}\textup{)}.  Then $\sigma$ fixes every $G$-orbit on three-subsets of $\mathcal{U}$ as a set, and consequently
\[
\sigma_3(\langle G,\sigma\rangle,\mathcal{U})=\sigma_3(G,\mathcal{U}),
\]
so the field conjugation causes no additional fusion beyond the diagonal fusion of $G$, for every $q\ge3$.  \textup{(The full field-automorphism group is analysed separately in \textup{Remark \ref{rem:psu-field}}:  it fuses certain $G$-orbits, e.g.\ for $q=8$ and $q=25$}\textup{)}.
\end{proposition}
\begin{proof}
The automorphism $\sigma$ fixes pointwise the sub-unital over $\F_q$ \textup{(the points with coordinates in $\F_q$}\textup{)}, preserves the Hermitian equation and collinearity, and normalizes $G$, so it sends $G$-orbits to $G$-orbits and preserves their setwise stabilizer orders and induced images in $S_3$.  The collinear orbit is unique \textup{(the blocks}\textup{)} and is preserved;  it remains to show that $\sigma$ preserves every non-collinear orbit.

Fix the base pair $\{\alpha,\beta\}=[e_1],[e_2]$ with block $B$, and write $G_{\{\alpha,\beta\}}=\langle T,\tau\rangle$, where $T\cong\Z_{(q^2-1)/d}$ acts regularly off $B$ and $\tau$ is an order-$2$ element interchanging $\alpha,\beta$ \textup{(Theorem \ref{thm:main-psu}(1);  for $d=3$ the full torus $T^{\mathrm{PGU}}\cong\Z_{q^2-1}$ of \textup{Proposition \ref{prop:psu-diag}}}\textup{)}.  By \textup{Lemma \ref{lem:suborbit}} every non-collinear $G$-orbit on triples is represented by $\{\alpha,\beta,\gamma\}$ with $\gamma\in\mathcal{U}\setminus B$ \textup{(the suborbit map is surjective, and bijective when the induced image contains a $3$-cycle}\textup{)}.  Write $\gamma=(1,u,z)$, so that $\gamma\in\mathcal{U}\setminus B$ iff $u
\notin\F_q^{\times}$ and $z
\ne 0$, and set $a=\overline{u}/u\in K^1$, where $K=\F_{q^2}$ and $K^1=\{w\in K^{\times}:w^{q+1}=1\}$.  Since $\tau\in G_{\{\alpha,\beta\}}$ interchanges $\alpha$ and $\beta$, it acts on the parameters $a$ of the off-$B$ points by reversal $a\mapsto a^{-1}$ \textup{(the pairing $t\mapsto t^{-q}$ of \textup{Proposition \ref{prop:psu-diag}} is exactly this reversal on $K^1$;  cf.\ \textup{Remark \ref{rem:psu-diag-dist}}\textup{)}};  hence every non-collinear $G$-orbit, being represented by $\{\alpha,\beta,\gamma\}$ and by $\{\alpha,\beta,\tau(\gamma)\}$, contains points with parameters $a$ and $a^{-1}$;  the orbit is reversal-closed.  Now $\sigma$ sends $\gamma=(1,u,z)$ to $(1,u^q,z^q)$, hence maps the parameter $a$ to $\overline{u^q}/u^q=(\overline{u}/u)^q=a^q=a^{-1}$ \textup{(as $a\in K^1$}\textup{)}.  Thus $\sigma$ maps the reversal-closed parameter set of a non-collinear orbit to itself;  since the non-collinear $G$-orbits are parameterised by their reversal classes \textup{(Proposition \ref{prop:psu-diag} for $d=3$;  Theorem \ref{thm:main-psu}(3) for $d=1$}\textup{)}, $\sigma$ fixes every non-collinear orbit as a set.  Every $G$-orbit on triples is therefore fixed by $\sigma$ as a set;  since $\sigma$ normalizes $G$, every $\langle G,\sigma\rangle$-orbit on triples is a union of $G$-orbits, and fixing each $G$-orbit forces the two orbit sets to coincide, so $\sigma_3(\langle G,\sigma\rangle,\mathcal{U})=\sigma_3(G,\mathcal{U})$.
\end{proof}

\begin{remark}[The field part for general $q$]\label{rem:psu-field}
For $q=3,9,27$ the field conjugation $\sigma\colon[x:y:z]\mapsto[x^q:y^q:z^q]$ of order $2$ fixes every $S$-orbit on three-subsets as a set \textup{(verified in GAP, Table \ref{tab:psu-pgu}:  cycle type $1^{q+1}2^{q(q^2-1)/2}$ on the $q^3+1$ points, with every $G_{\{\alpha,\beta\}}$-suborbit fixed as a set}\textup{)}.  By \textup{Proposition \ref{prop:psu-field-nofusion}} this no-fusion conclusion for the conjugation holds for every $q\ge3$:  $\sigma$ fixes every $\PGU(3,q)$-orbit on three-subsets as a set, so $\sigma_3(\langle\PGU(3,q),\sigma\rangle)=\sigma_3(\PGU(3,q))$;  it reproduces the odd-$d=3$ statement of \textup{Proposition \ref{prop:psu-field-nofusion}} and the verified values $q=3,9,27$, and was checked directly in GAP for $q=5,11$ \textup{($\sigma$ added as a permutation on the $q^3+1$ points yields no new three-subset orbits}\textup{)}.  The full field part $\PpGU(3,q)/\PGU(3,q)\cong C_{2f}$ is generated by the Frobenius automorphism $\sigma_p\colon[x:y:z]\mapsto[x^p:y^p:z^p]$ ($p$ the characteristic), of which the conjugation $\sigma_q$ is the $f$-th power, $\sigma_q=\sigma_p^f$ (its unique order-$2$ element).  This generator structure explains the dichotomy:  $\sigma_q$ maps the parameter $a$ to $a^q=a^{-1}$, hence fixes every reversal-closed class and never fuses (Proposition~\ref{prop:psu-field-nofusion}, for every $q\ge3$), whereas $\sigma_p$ maps $a$ to $a^p$ and does fuse certain $\PGU(3,q)$-orbits whenever $f>1$.  In the coordinate parameterization of \textup{Proposition \ref{prop:psu-diag}}, the non-collinear $\PGU(3,q)$-orbits are the classes $Z/\langle\iota\rangle$ with $Z=\{z\in K:z+\bar z=-1\}$ and $\iota\colon z\mapsto-1-z$, and $\sigma_p$ acts on these classes by $z\mapsto z^p$;  hence $\sigma_p$ fixes a $\PGU(3,q)$-orbit precisely when $z^p=z$ or $z^p=-1-z$ for its representative $z$.  The number of such classes depends on the arithmetic of $q=p^f$ and is not given by a single closed form, so the full field part remains a computed value for the four cases $q=8,9,25,27$ rather than a family-wide statement.  Concretely, it is checked in GAP that $\sigma_3(\PpGU(3,q))=3$ for $q=8$ (even $d=3$;  $\sigma_3(\PGU(3,q))=(q+2)/2=5$), $=4$ for $q=9$ (odd $d=1$;  $\sigma_3(\PGU(3,q))=(q+3)/2=6$), $=8$ for $q=25$ (odd $d=1$, $f=2$;  $\sigma_3(\PSU(3,q))=(q+3)/2=14$) and $=7$ for $q=27$ (odd $d=1$;  $\sigma_3(\PGU(3,q))=15$), whereas the conjugation $\sigma_q$ alone leaves $\sigma_3$ unchanged in each case.  For $q=3$ ($f=1$) the field part is generated by $\sigma_p=\sigma_q$ alone and causes no fusion.  The orbit classification of the socle $\PSU(3,q)$ and of the diagonal overgroup $\PGU(3,q)$ is unaffected \textup{(Theorem \ref{thm:main-psu};  Theorem \ref{thm:pgu}(1),(2))}.
\end{remark}

\begin{remark}[The two parametrizations of the non-collinear orbits]\label{rem:psu-z-a}
The non-collinear orbits carry two coordinate parametrizations, related as follows.  In the standard coordinates of \textup{Proposition \ref{prop:psu-diag}} the third point is $R(x,y)=\langle(x,y,1)\rangle$ \textup{($y
\ne 0$, $x+\bar x+y\bar y=0$)}, with $u=y\bar y\in\F_q^{\times}$ and $z=x/u\in\mathcal Z=\{z\in K:z+\bar z=-1\}$.  Under the isometry $(X_0,X_1,X_2)\mapsto(X_0,X_2,X_1)$ of the coordinate convention and normalization of the first coordinate to $1$, the same point reads $\gamma=(1,u',z')$ with $u'=x^{-1}$, $z'=yx^{-1}$ \textup{($\mathrm{Tr}_{K/F}(u')+N_{K/F}(z')=(x+\bar x+y\bar y)/(x\bar x)=0$, so $\gamma\in\mathcal{U}\setminus B$}\textup{)}, and the $\PSU(3,q)$-splitting parameter of \textup{Lemma \ref{lem:psu-c3}} is
\[
a=\frac{\bar u'}{u'}=u'^{q-1}=(uz)^{-(q-1)}=z^{-(q-1)}\in K^1,
\]
the factor $u^{q-1}=1$ vanishing because $u\in\F_q^{\times}$.  Hence:
\begin{itemize}
\item the diagonal pairing $\iota\colon z\mapsto-1-z$ \textup{($=\bar z$ for $z\in\mathcal Z$)} of \textup{Proposition \ref{prop:psu-diag}} acts on $a$ as $a\mapsto\bar a=a^{-1}$ \textup{(}since $a\in K^1$\textup{)}:  the non-collinear $\PGU(3,q)$-orbits are exactly the reversal classes $\{a,a^{-1}\}$, as used in \textup{Proposition \ref{prop:psu-field-nofusion}};
\item the Frobenius $\sigma_p$ acts as $z\mapsto z^p$ and $a\mapsto a^p$, and the conjugation $\sigma_q$ as $z\mapsto z^q$ and $a\mapsto a^q=a^{-1}$ \textup{(Remark \ref{rem:psu-field})};
\item the $\PSU(3,q)$-split of a non-collinear orbit is governed by $a\in(K^1)^3$ \textup{(Lemma \ref{lem:psu-c3})}, and the two coordinate systems realize the same split:  $a\in(K^1)^3 \iff z^{-(q-1)}\in(K^1)^3$.
\end{itemize}
\end{remark}

\begin{remark}
The structure $\Out(\PSU(3,q))=C_{(3,q+1)}\times C_{2f}$ used above is standard \textup{(}see e.g.\ the classical theory of unitary groups, \cite{huppert1967,dixon1996}\textup{)}.
\end{remark}

\subsection{Computational confirmation \textup{(Magma and GAP)} for $\PSU(3,q)$}

The following table was computed in Magma \cite{magma} and re-checked in GAP by the same stabilizer computation;  it lists, for each $q$, the set of orders $|S_{\Ocal}|$ occurring on the orbits of $\PSU(3,q)$ on $\binom{\mathcal{U}}{3}$, where $S=\PSU(3,q)$.  (The degrees $q^3+1$ exceed $200$ for $q\ge6$, so many rows lie outside the exhaustive enumeration of \textup{Chapter \ref{sec:verify}};  for those rows the stabilizer orders were computed directly by the Magma and GAP computations reported here, not by the exhaustive run.)

\begin{landscape}
\begin{table}[p]
\centering
\small
\begin{tabular}{cccll}
\toprule
$q$ & $d$ & collinear $|S_{\Ocal}|$ & non-collinear $|S_{\Ocal}|$ & non-collinear $\lambda$ \\
\midrule
2 & 3 & 6 & $\{1\}$ & $\{6\}$ \\
3 & 1 & 24 & $\{3,6\}$ & $\{8,16\}$ \\
4 & 1 & 30 & $\{3\}$ & $\{30\}$ \\
5 & 3 & 12 & $\{1,6\}$ & $\{8,48\}$ \\
7 & 1 & 48 & $\{3,6\}$ & $\{48,96\}$ \\
8 & 3 & 18 & $\{1,3\}$ & $\{42,126\}$ \\
9 & 1 & 60 & $\{3,6\}$ & $\{80,160\}$ \\
11 & 3 & 24 & $\{1,3,6\}$ & $\{40,80,240\}$ \\
13 & 1 & 84 & $\{3,6\}$ & $\{168,336\}$ \\
16 & 1 & 102 & $\{3\}$ & $\{510\}$ \\
17 & 3 & 36 & $\{1,3,6\}$ & $\{96,192,576\}$ \\
19 & 1 & 120 & $\{3,6\}$ & $\{360,720\}$ \\
23 & 3 & 48 & $\{1,3,6\}$ & $\{176,352,1056\}$ \\
25 & 1 & 156 & $\{3,6\}$ & $\{624,1248\}$ \\
27 & 1 & 168 & $\{3,6\}$ & $\{728,1456\}$ \\
29 & 3 & 60 & $\{1,3,6\}$ & $\{280,560,1680\}$ \\
31 & 1 & 192 & $\{3,6\}$ & $\{960,1920\}$ \\
32 & 3 & 66 & $\{1,3\}$ & $\{2046,682\}$ \\
37 & 1 & 228 & $\{3,6\}$ & $\{1368,2736\}$ \\
41 & 3 & 84 & $\{1,3,6\}$ & $\{560,1120,3360\}$ \\
43 & 1 & 264 & $\{3,6\}$ & $\{1848,3696\}$ \\
47 & 3 & 96 & $\{1,3,6\}$ & $\{736,1472,4416\}$ \\
49 & 1 & 300 & $\{3,6\}$ & $\{2400,4800\}$ \\
\bottomrule
\end{tabular}
\caption{Unitary groups:  the stabilizer orders $|S_{\Ocal}|$ on three-subsets of the Hermitian unital for $q\le49$.  The collinear stabilizer is $6(q+1)/d$ (third column);  the last column gives $\lambda=6(q^2-1)/(d\,|S_{\Ocal}|)$ on the non-collinear orbits (\textup{Corollary \ref{cor:psu-designs}}).}\label{tab:magma}
\end{table}
\end{landscape}

The value $6(q+1)/d$ is exactly the collinear stabilizer order in every row ($q=8,d=3$:  $18$;  $q=11,d=3$:  $24$;  $q=16,d=1$:  $102$;  etc.), confirming \textup{Theorem \ref{thm:collinear}}.  The last two columns record the non-collinear stabilizer orders and the resulting non-collinear $\lambda$-values, in agreement with \textup{Theorem \ref{thm:main-psu}} and \textup{Corollary \ref{cor:psu-designs}}:  the occurring stabilizer orders are $\{3,6\}$ for odd $q$ with $d=1$ (one $S_3$-orbit and $(q-1)/2$ $C_3$-orbits), $\{1,3,6\}$ for odd $q$ with $d=3$ (three $S_3$-orbits, $(q-5)/2$ $C_3$-orbits and $(q+1)/3$ trivial orbits), $\{3\}$ for even $q$ with $d=1$ (all $q/2$ merged orbits have image $C_3$), and $\{1,3\}$ for even $q$ with $d=3$, $q\ge8$ \textup{(proved for every $q=2^{2m+1}$ in \textup{Proposition \ref{prop:psu-even}}:  $(q-2)/2$ $C_3$-orbits and $(q+1)/3$ trivial orbits;  for $q=8$:  three $C_3$-orbits and three trivial orbits;  for $q=32$:  $15$ and $11$)};  $q=2$ is the exceptional row.

The suborbit structure underlying \textup{Theorem \ref{thm:main-psu}} (including the fixed-orbit number $k$ and the diagonal fusion under $\PGU(3,q)$) was computed in GAP \cite{gap} for all $q\le27$ and for the larger prime powers listed in \textup{Table \ref{tab:psu-general}}.  The construction is explicit and independent of the orbit--stabilizer identities:  the group $\SU(3,q)$ is built from its defining representation over $\F_{q^2}$, the unital $\mathcal{U}$ is the orbit of a projective fixed point under the projective action, and the base pair $\{\alpha,\beta\}$ is a pair of points of the block through the two isotropic base vectors;  for each $q$ the program computes the two-point setwise stabilizer $G_{\{\alpha,\beta\}}=\langle T,\tau\rangle$ ($|T|=(q^2-1)/d$, $\tau$ an order-$2$ element interchanging $\alpha,\beta$), its suborbits on $\mathcal{U}\setminus B$, the fixed points of $\tau$ off $B$, the fixed-orbit number $k$, the induced image on a representative triple of each suborbit \textup{(distinguishing the $S_3$-, $C_3$- and trivial suborbits of \textup{Proposition \ref{prop:psu-computed}}}\textup{)}, and (for $d=3$) the orbits of $\langle T,\tau,\delta\rangle$ where $\delta$ is the order-$3$ diagonal element;  the numbers $\sigma_3(\PSU(3,q))$ and $\sigma_3(\PGU(3,q))$ are read off from \textup{Lemma \ref{lem:suborbit}} as $1+f+c+t/3$ \textup{(where $f,c,t$ count the $S_3$-, $C_3$- and trivial suborbits, the latter fusing three-to-one)}\textup{)} resp.\ $1+\#\{\textup{suborbits of }\langle T,\tau,\delta\rangle\}$ \textup{(all non-collinear images contain a $3$-cycle, so the bijection of \textup{Lemma \ref{lem:suborbit}} applies)}\textup{)}.  The results are displayed in \textup{Table \ref{tab:psu-general}}:  in every row $k$ equals $0$ for $q$ even and $d$ for $q$ odd, $\sum_t\operatorname{fix}(\tau t)=q^2-1$ for $q$ odd and $0$ for $q$ even, $\operatorname{fix}(\tau;\mathcal{U}\setminus B)=q+1$ for $q$ odd and $0$ for $q$ even \textup{(with the order-$2$ representative chosen)}, and the values $\sigma_3(\PSU(3,q))$, $\sigma_3(\PGU(3,q))$ agree with the closed forms of \textup{Theorem \ref{thm:main-psu}(3)} and \textup{Proposition \ref{prop:psu-diag}}.  For $q=3,9,27$ this reproduces the earlier computation (block orbit, $k=1$, $\sigma_3=(q+3)/2$);  in that range the conjugation $\sigma\colon[x:y:z]\mapsto[x^q:y^q:z^q]$ of order $2$ fixes every $S$-orbit as a set (cycle type $1^{q+1}2^{q(q^2-1)/2}$), so the conjugation alone causes no fusion.  For $q=3$ ($f=1$) the field part is trivial, whereas for $q=9,27$ the full field part, generated by the Frobenius $\sigma_p\colon[x:y:z]\mapsto[x^p:y^p:z^p]$, does fuse:  GAP gives $\sigma_3(\PpGU(3,q))=4$ for $q=9$ and $=7$ for $q=27$ (Remark~\ref{rem:psu-field}).

\begin{table}[htbp]
\centering
\footnotesize
\renewcommand{\arraystretch}{1.1}
\setlength{\tabcolsep}{2pt}
\caption{The Hermitian unital for $q=3,9,27$ \textup{(GAP)}:  $v=q^3+1$;  $T=S_{\alpha,\beta}$ cyclic of order $q^2-1$;  the fourth column lists the orbits of $G_{\{\alpha,\beta\}}=\langle T,\tau\rangle$ on $\mathcal{U}\setminus\{\alpha,\beta\}$;  $\sigma\colon[x:y:z]\mapsto[x^q:y^q:z^q]$ is the field conjugation of order $2$;  $\sigma_3(\PSU(3,q))=\sigma_3(\PpGU(3,q))$ by the no-fusion column \textup{(Theorem \ref{thm:main-psu};  the conjugation alone;  the full field part $C_{2f}$ does fuse for $q=9,27$, see Remark~\ref{rem:psu-field})}.}\label{tab:psu-pgu}
\begin{tabular}{ccccclc}
\toprule
$q$ & $v$ & $|S_{\{\alpha,\beta\}}|$ & $G_{\{\alpha,\beta\}}$-suborbits & cycle type of $\sigma$ & \shortstack{$\sigma$ fixes\\ each suborbit?} & \shortstack{$\sigma_3(\PSU)=$\\ $\sigma_3(\PpGU)$} \\
\midrule
$3$ & $28$ & $8$ & $2;\ 16,\ 8$ & $1^4 2^{12}$ & yes & $3$ \\
$9$ & $730$ & $80$ & $8;\ 160^{\times4},\ 80$ & $1^{10} 2^{360}$ & yes & $6$ \\
$27$ & $19684$ & $728$ & $26;\ 1456^{\times13},\ 728$ & $1^{28} 2^{9828}$ & yes & $15$ \\
\bottomrule
\end{tabular}
\end{table}

\begin{table}[htbp]
\centering
\footnotesize
\renewcommand{\arraystretch}{1.1}
\setlength{\tabcolsep}{3pt}
\caption{The Hermitian unital for every $q\le27$ \textup{(GAP)}:  $v=q^3+1$;  $d=\gcd(3,q+1)$;  $T=S_{\alpha,\beta}$ cyclic of order $|T|=(q^2-1)/d$;  $k$ the number of $T$-orbits off the block $B$ fixed by $\tau$ as a set;  $\operatorname{fix}(\tau)\setminus B$ the number of points of $\mathcal{U}\setminus B$ fixed by $\tau$ (an order-$2$ representative is taken in every row, which exists by the classical normalizer structure);  the column $\sum_{t\in T}\operatorname{fix}(\tau t)$ is the off-block Burnside count of the $\tau$-coset, which by the proof of \textup{Proposition \ref{prop:psu-diag}} equals $q^2-1$ for $q$ odd and $0$ for $q$ even \textup{(directly verified for $q=5,7,11,13$)};  $\sigma_3(\PSU)$ and $\sigma_3(\PGU)$ are the orbit numbers of \textup{Theorem \ref{thm:main-psu}(3)} and \textup{Proposition \ref{prop:psu-diag}} \textup{(Lemma \ref{lem:suborbit}}\textup{)}.  The row $q=2$ is the exceptional case $|T|=(q^2-1)/d=1$:  the closed forms of \textup{Theorem \ref{thm:main-psu}} require $q\ge3$ and do not apply there;  the displayed values $\sigma_3(\PSU(3,2))=2$, $\sigma_3(\PGU(3,2))=2$ are the direct GAP computations for the two degree-$9$ representations \textup{(the former agrees with the $\PSU(3,2)$ row of \textup{Table \ref{tab:verify}}\textup{)};  the latter is the diagonal extension, computed in GAP as $\PGU(3,2)$ on the same $9$ points}\textup{)}.}\label{tab:psu-general}
\begin{tabular}{cccccccc}
\toprule
$q$ & $d$ & $v$ & $|T|$ & $k$ & $\operatorname{fix}(\tau)\setminus B$ & $\sum_{t\in T}\operatorname{fix}(\tau t)$ & $\sigma_3(\PSU),\ \sigma_3(\PGU)$ \\
\midrule
$2$  & $3$ & $9$     & $1$    & $0$ & $0$ & $0$    & $2,\ 2$ \\
$3$  & $1$ & $28$    & $8$    & $1$ & $4$ & $8$    & $3,\ 3$ \\
$4$  & $1$ & $65$    & $15$   & $0$ & $0$ & $0$    & $3,\ 3$ \\
$5$  & $3$ & $126$   & $8$    & $3$ & $6$ & $24$   & $6,\ 4$ \\
$7$  & $1$ & $344$   & $48$   & $1$ & $8$ & $48$   & $5,\ 5$ \\
$8$  & $3$ & $513$   & $21$   & $0$ & $0$ & $0$    & $7,\ 5$ \\
$9$  & $1$ & $730$   & $80$   & $1$ & $10$ & $80$  & $6,\ 6$ \\
$11$ & $3$ & $1332$  & $40$   & $3$ & $12$ & $120$ & $11,\ 7$ \\
$13$ & $1$ & $2198$  & $168$  & $1$ & $14$ & $168$ & $8,\ 8$ \\
$16$ & $1$ & $4097$  & $255$  & $0$ & $0$ & $0$    & $9,\ 9$ \\
$17$ & $3$ & $4914$  & $96$   & $3$ & $18$ & $288$ & $16,\ 10$ \\
$19$ & $1$ & $6860$  & $360$  & $1$ & $20$ & $360$ & $11,\ 11$ \\
$23$ & $3$ & $12168$ & $176$  & $3$ & $24$ & $528$ & $21,\ 13$ \\
$25$ & $1$ & $15626$ & $624$  & $1$ & $26$ & $624$ & $14,\ 14$ \\
$27$ & $1$ & $19684$ & $728$  & $1$ & $28$ & $728$ & $15,\ 15$ \\
\bottomrule
\end{tabular}
\end{table}

For $q=27$ the group orders are $|S|=282056445216$ and $|\PpGU(3,27)|=1692338671296=6|S|$ \textup{(}confirming that the field part is cyclic of order $2f=6$, generated by the Frobenius $\sigma_p\colon[x:y:z]\mapsto[x^3:y^3:z^3]$;  the conjugation $\sigma_q=\sigma_p^3$ spans its unique order-$2$ subgroup, of order $2$, and the same holds for $q=3,9$ with $|S|=6048,42573600$ and field-part orders $2,4$\textup{)}\textup{)}.  Together with \textup{Table \ref{tab:magma}}, these computations verify \textup{Theorem \ref{thm:main-psu}} and \textup{Theorem \ref{thm:pgu}(2),(3)} in the stated range:  $\sigma_3(\PSU(3,q))=(q+3)/2$ for $q=3,9,27$, and the conjugation alone causes no fusion, while the full field part fuses for $q=9,27$ ($\sigma_3(\PpGU(3,q))=4,7$, see Remark~\ref{rem:psu-field});  for the other $q$ of \textup{Table \ref{tab:psu-general}} the values of $k$ and of $\sigma_3(\PSU(3,q))$, $\sigma_3(\PGU(3,q))$ agree with the closed forms of \textup{Theorem \ref{thm:main-psu}(3)} and \textup{Proposition \ref{prop:psu-diag}}.

% ============================================================
\section{The Suzuki groups \texorpdfstring{$\Sz(q)$}{Sz(q)}}
\label{sec:sz}
% ============================================================

Let $q=2^{2m+1}>2$ and $S=\Sz(q)$, acting $2$-transitively on the $q^2+1$ points of the Suzuki ovoid $\mathcal{O}$.  The group $S$ is a Zassenhaus group of degree $q^2+1$:  the two-point stabilizer is cyclic of order $q-1$ (the split torus) and acts semiregularly on the remaining $q^2-1$ points;  $S$ is $2$-transitive but not $3$-transitive;  moreover $3
\nmid|\Sz(q)|=|q^2(q^2+1)(q-1)|$, since $q^2+1=2^{2(2m+1)}+1\equiv1+1\equiv2\pmod3$ and $q-1\equiv1\pmod3$, $q^2\equiv1\pmod3$.

\begin{theorem}[Three-point stabilizers for the Suzuki group]
\label{thm:sz}
Let $S=\Sz(q)$, $q=2^{2m+1}>2$, acting on the Suzuki ovoid with $v=q^2+1$, and $\Ocal$ a $3$-subset.
\begin{enumerate}
\item The pointwise stabilizer of any $3$-subset is trivial:  $S_{(\Ocal)}=1$.
\item If the three points lie on a common Suzuki circle \textup{(}an orbit of a sub-Suzuki subgroup $\Sz(q_0)$, $q_0=2^{2m'+1}$, on the ovoid, equivalently a sub-ovoid of the Suzuki ovoid\textup{)}, then $S_{\Ocal}\cong C_2$;  if the three points do not lie on a Suzuki circle, then $S_{\Ocal}=1$ \textup{(classical geometric input, \cite{suzuki1962}}\textup{)}.
\item In particular $S_{\Ocal}$ is never $S_3$ or $C_3$:  since $3
\nmid|S|$, no element of order $3$ exists, and the image in $S_3$ is contained in a $C_2$.
\item The circle triples form a single $S$-orbit \textup{(classical geometric input, \cite{suzuki1962}}\textup{)}.  The triples not on a circle form $\frac{q-2}{6}$ further orbits \textup{(Theorem \ref{thm:sz-orbits})}, each with trivial setwise stabilizer \textup{(by (2)}\textup{)}.  The two-orbit structure for $q=8$ is verified independently in \textup{Table \ref{tab:verify}} \textup{(Proposition \ref{prop:sz8-computed})}.
\end{enumerate}
\end{theorem}
\begin{proof}
The following classical facts are due to Suzuki \cite{suzuki1962} and are used as input.  Key facts:  the point stabilizer $S_{\alpha}$ is a Frobenius group of order $q^2(q-1)$ acting transitively on the remaining $q^2$ points;  the two-point stabilizer is cyclic of order $q-1$ (the split torus $S_{u,v}\cong C_{q-1}$) and acts semiregularly on the remaining $q^2-1$ points (Zassenhaus);  hence fixing three points forces the identity.  The dichotomy ``an involution swapping $x,y$ and fixing $z$ exists precisely for the circle triples'' is part of the classical Suzuki ovoid geometry (cf.\ \cite{suzuki1962});  it is used here as input and is verified computationally for $q=8$ \textup{(Table \ref{tab:verify})}.  The circles are the blocks of the classical $2$-$(q^2+1,q+1,1)$ design carried by the Suzuki ovoid (cf.\ \cite{suzuki1962}):  any two points lie on a unique circle.  The circle class is a single $S$-orbit:  given circles $B,B'$ and $x,y\in B$, $x',y'\in B'$, the $2$-transitivity of $S$ supplies $g$ with $g\{x,y\}=\{x',y'\}$, and $g(B)$ is the unique circle through $g(x),g(y)$, hence $g(B)=B'$.  To see that the non-circle stabilizer is trivial, let $X=\{x,y,z\}$ be a three-subset not contained in a circle.  Since $3
\nmid|S|$, the image of $S_X$ in $\Sym(X)\cong S_3$ contains no element of order $3$;  by (1) the kernel of $S_X\to\Sym(X)$ is trivial, so the image is trivial or a $C_2$ generated by a transposition.  If the image were nontrivial, some $g\in S_X$ would interchange two points of $X$ and fix the third;  then $g^2$ fixes $X$ pointwise, so $g^2=1$ by (1), and $g$ is an involution of the type excluded by the dichotomy for non-circle triples.  Hence $S_X=1$ for every non-circle triple.  The orbit--design parameters follow in \textup{Corollary \ref{cor:sz-designs}}.
\end{proof}

\begin{theorem}[The Suzuki orbit number]\label{thm:sz-orbits}
Let $S=\Sz(q)$, $q=2^{2m+1}>2$, act on the $q^2+1$ points of the Suzuki ovoid.  The number of $S$-orbits on three-subsets is
\[
\boxed{\;\sigma_3(S)=\frac{q+4}{6}\;}
\]
with $\sigma_3(\Sz(q))$ integral since $q=2^{2m+1}\equiv2\pmod6$.  Consequently the circle class is the single orbit of \textup{Theorem \ref{thm:sz}(4)} and the generic class consists of $\frac{q+4}{6}-1=\frac{q-2}{6}$ orbits, each with trivial setwise stabilizer \textup{(Theorem \ref{thm:sz}(2)}\textup{)}.
\end{theorem}
\begin{proof}
We apply Burnside's lemma;  the cycle types of $S$ on the ovoid are classical (Suzuki \cite{suzuki1962};  the same input underlies Bradley's closed form, \cite{bradley2014}, Theorem 3.3.2).  For $g\in S$ let $\eta_3(g)=\binom{f}{3}+fc_2+c_3$ be the number of three-subsets fixed by $g$, where $f=\operatorname{fix}(g)$ and $c_i$ is the number of $i$-cycles of $g$ on the ovoid.  The class input of \textup{Theorem \ref{thm:sz}} gives:
(i) $g=1$:  $\eta_3(1)=\binom{q^2+1}{3}$;
(ii) an involution (a nonidentity element of a Sylow $2$-subgroup $Q$ of order $q^2$, lying in the point stabilizer of its unique fixed point) fixes one point and has cycle type $1^12^{q^2/2}$:  the only fixed three-subsets are $\{p,x,y\}$ with $p$ the fixed point and $\{x,y\}$ a $2$-cycle, so $\eta_3=\binom{1}{3}+1\cdot\frac{q^2}{2}=\frac{q^2}{2}$;
(iii) an element of order $4$ fixes one point and has no $2$- or $3$-cycle:  $\eta_3=0$;
(iv) an element of $H_0$-type (order $m>1$, $m\mid q-1$) fixes two points and has no $2$- or $3$-cycle (since $m$ is odd and $3
\nmid q-1$):  $\eta_3=\binom{2}{3}=0$;
(v) an element of $H_{\pm}$-type (order $m$, $m\mid q\pm r+1$, $r^2=2q$) fixes no point:  $\eta_3=0$.
The class sizes enter through the class weights $1/|S|$, $1/q^2$, $1/q$, $1/(2(q-1))$, $1/(4(q\pm r+1))$ of \cite{bradley2014}, Theorem 3.3.2;  only the identity and the involutions contribute:
\[
\sigma_3(S)
=\frac{1}{q^2(q-1)(q^2+1)}\binom{q^2+1}{3}
+\frac{1}{q^2}\cdot\frac{q^2}{2}
=\frac{q+1}{6}+\frac{1}{2}
=\frac{q+4}{6}.
\]
For $q=8$ this gives $\sigma_3(\Sz(8))=2$, in agreement with \textup{Proposition \ref{prop:sz8-computed}} and \textup{Table \ref{tab:verify}}.
\end{proof}

\begin{corollary}[Designs from $\Sz(q)$]
\label{cor:sz-designs}
Since $|S|=q^2(q^2+1)(q-1)$ and $v=q^2+1$, the formula \eqref{eq:lambda} gives
\[
\lambda=\frac{6|S|}{|S_{\Ocal}|\,v(v-1)}
=\frac{6\,q^2(q^2+1)(q-1)}{|S_{\Ocal}|\,(q^2+1)q^2}
=\frac{6(q-1)}{|S_{\Ocal}|}.
\]
For a circle triple $|S_{\Ocal}|=2$ \textup{(Theorem \ref{thm:sz}(2))}, giving $\lambda=3(q-1)$;  for a non-circle triple $|S_{\Ocal}|=1$ \textup{(Theorem \ref{thm:sz}(2))}, giving $\lambda=6(q-1)$.  Thus the complete list of $\Sz(q)$-orbit parameters is
\[
3(q-1),\quad\underbrace{6(q-1),\ldots,6(q-1)}_{(q-2)/6\ \text{times}},
\]
i.e.\ the circle system $2$-$(q^2+1,3,3(q-1))$ and $(q-2)/6$ copies of $2$-$(q^2+1,3,6(q-1))$ \textup{(Theorem \ref{thm:sz}(4)}\textup{)}.  For $q=8$ this gives the two systems $2$-$(65,3,21)$ and $2$-$(65,3,42)$, whose union has $\lambda=63$ \textup{(Proposition \ref{prop:union-orbits});  Proposition \ref{prop:sz8-computed};  Table \ref{tab:verify})}.
\end{corollary}

\begin{remark}[Field automorphisms]\label{rem:sz-field}
$\Aut(\Sz(q))=\Sz(q)\cdot(2m+1)$ \textup{(}field automorphisms of odd order $2m+1$\textup{)}.  Field automorphisms preserve the circle orbit \textup{(Theorem \ref{thm:sz-field-reduction}(1)}\textup{)};  the quotient action on the generic class is the Frobenius action $[\lambda]\mapsto[\lambda^{2^c}]$ on the parameter space $(\mathbb F_q\setminus\{0,1\})/H$, and the number $\sigma_3(\Sz(q)\cdot f')$ is given by the closed form of \textup{Theorem \ref{thm:suzuki-field-final}} for every field quotient $C_{f'}\le C_{2m+1}$ \textup{(Lemma \ref{lem:suzuki-bluher-count})}.  The pointwise kernel inside a fused triple stabilizer remains an orbit-by-orbit invariant of the extension:  for $G=\Sz(q)\cdot f'$, whether the twisted elements $h_\sigma\circ\sigma$ fix a given triple pointwise or permute it depends on the individual orbit.  The case $\Sz(8)\cdot3$ is computed in \textup{Chapter \ref{sec:verify}} \textup{(\textup{Proposition \ref{prop:sz8-computed}};  \textup{Table \ref{tab:verify}}\textup{)}}:  $|G_{\Ocal}|=3$ \textup{(}kernel $1$, image $C_3$, $\lambda=42$\textup{)} on the generic orbit and $|G_{\Ocal}|=6$ \textup{(}kernel $3$, image $C_2$, $\lambda=21$\textup{)} on the circle orbit.  \textup{(}This is a genuine difference from the projective-line case of \textup{Theorem \ref{thm:pgl}\textup{(}2\textup{)}}, where a twisted kernel element exists for every triple.\textup{)}
\end{remark}

\begin{theorem}[Suzuki field extensions: reduction]\label{thm:sz-field-reduction}
Let $S=\Sz(q)$, $q=2^{2m+1}$, and let $S\le A\le\Aut(S)=S\cdot(2m+1)$ with $F=A/S\le C_{2m+1}$.  Let $\mathcal G_q$ denote the set of the $\frac{q-2}{6}$ non-circle $S$-orbits on three-subsets \textup{(Theorem \ref{thm:sz}(4)}\textup{)}.
\begin{enumerate}
\item The circle orbit is fixed by $F$.
\item The action of $A$ on the non-circle three-subset orbits is exactly the induced action of $F$ on $\mathcal G_q$.  If the $F$-orbits on $\mathcal G_q$ are $\mathcal C_1,\ldots,\mathcal C_t$, then
\[
\sigma_3(A,\Ocal)=1+t.
\]
\item If $\mathcal C_i$ has size $e_i$, then the corresponding $A$-orbit on triples is the union of $e_i$ generic $S$-orbits and has design parameter
\[
\boxed{\;\lambda_i=6e_i(q-1)\;}
\]
\textup{(}orbit size $e_i|\Sz(q)|$, since every generic $S$-orbit has stabilizer $1$\textup{)};  the circle orbit has parameter $3(q-1)$.
\end{enumerate}
\end{theorem}
\begin{proof}
(1) The circle property is defined geometrically (a Suzuki sub-ovoid) and is preserved by every automorphism of the ovoid, so $F$ fixes the circle orbit.  (2) Since $S\trianglelefteq A$, the quotient $A/S$ acts on the set of $S$-orbits on three-subsets, and by the normal-subgroup fusion lemma \textup{(Lemma \ref{lem:norm-orbit-fusion};  Proposition \ref{prop:ree-ext-reduction})} the $A$-orbits are precisely the unions of the $S$-orbits belonging to one $F$-orbit;  the circle orbit contributes one $A$-orbit.  (3) A generic $S$-orbit has size $|S|$ (trivial stabilizer, Theorem \ref{thm:sz}(2)), so an $F$-orbit of size $e_i$ gives an $A$-orbit of size $e_i|S|$;  its parameter is
\[
\lambda_i=\frac{6e_i|S|}{v(v-1)}=\frac{6e_i\,q^2(q^2+1)(q-1)}{(q^2+1)q^2}=6e_i(q-1),
\]
and the circle orbit has $\lambda=6(q-1)/2=3(q-1)$.
\end{proof}

The reduction of \textup{Theorem \ref{thm:sz-field-reduction}} is made explicit by the following complete finite-field analysis.  The parametrization $\mathcal G_q\cong(\mathbb F_q\setminus\{0,1\})/H$ below is the classical one carried by the affine part of the Suzuki ovoid (cf.\ \cite{suzuki1962}), and the counting lemma is a statement about the six M\"obius transformations fixing the reference set $\{0,1,\infty\}$.

\begin{lemma}[The complete finite-field count]\label{lem:suzuki-bluher-count}
Let $q=2^a$ with $a$ odd, let $\sigma_b\colon x\mapsto x^{2^b}$ be a field automorphism of $\mathbb F_q$, put
\[
d=(a,b),\qquad q_0=2^d,\qquad n=a/d,
\]
and let
\[
H=\left\{\,x,\ x+1,\ x^{-1},\ (x+1)^{-1},\ x(x+1)^{-1},\ (x+1)x^{-1}\,\right\}\cong S_3
\]
be the six M\"obius transformations of $\mathbb P^1(\mathbb F_q)$ permuting $\{0,1,\infty\}$.  The group $H$ acts freely on $\Lambda_q=\mathbb F_q\setminus\{0,1\}$.  For $h\in H$ put
\[
N_b(h)=\left|\left\{x\in\Lambda_q:\sigma_b(x)=h(x)\right\}\right|.
\]
Then
\[
N_b(1)=q_0-2,
\]
\[
N_b(h)=0\qquad\text{for every transposition }h\in H,
\]
and
\[
N_b(h)=\begin{cases}q_0+1,&3\mid n,\\0,&3
\nmid n\end{cases}\qquad\text{for every $3$-cycle }h\in H.
\]
Consequently, since $H$ acts freely on $\Lambda_q$ and the six equations $\sigma_b(x)=h(x)$ have disjoint solution sets,
\[
\left|\Fix_{\Lambda_q/H}(\sigma_b)\right|=\frac16\sum_{h\in H}N_b(h)=
\begin{cases}\dfrac{2^d-2}{6},&3
\nmid a/d,\\[3mm]2^{d-1},&3\mid a/d,\end{cases}
\]
where in the first case integrality uses $d$ odd, so $2^d\equiv2\pmod6$.
\end{lemma}
\begin{proof}
The free action of $H$ on $\Lambda_q$:  the transposition $x\mapsto x+1$ has no fixed point, $x\mapsto x/(x+1)$ fixes only $0$, the involution $x\mapsto x^{-1}$ fixes only $1$, and the two $3$-cycles fix exactly the roots of $x^2+x+1=0$, which lie in $\mathbb F_4\setminus\mathbb F_2$;  since $a$ is odd, $\mathbb F_4
\not\subseteq\mathbb F_{2^a}=\mathbb F_q$, so none lies in $\Lambda_q$.  Counting $N_b(h)$:  the identity equation $x^{2^b}=x$ has as solutions exactly the fixed field $K=\mathbb F_{q_0}$, hence $N_b(1)=q_0-2$.  For a transposition, write $T=h\circ\sigma_b$;  since $h$ is defined over $\mathbb F_2$ it commutes with $\sigma_b$, and $h^2=1$ gives $T^n=h^n\,\sigma_b^n=h$, because $n$ is odd.  If $x^{2^b}=x+1$ then $T(x)=x$ and $x=T^n(x)=x+1$, a contradiction;  if $x^{2^b}=x^{-1}$ then $T^n(x)=x^{-1}=x$ forces $x=1$ (or the projective degenerations $0,\infty$), outside $\Lambda_q$;  the third transposition is conjugate to the first two in $H$.  For a $3$-cycle $c(x)=(x+1)^{-1}$ the matrix $C=\left(\begin{smallmatrix}0&1\\1&1\end{smallmatrix}\right)$ satisfies $C^3=I$ in characteristic $2$, so $T=c\circ\sigma_b$ satisfies $T^n=c^n\,\sigma_b^n=c^n$.  If $3
\nmid n$, then $c^n$ is a nontrivial $3$-cycle;  a fixed point of $c$ or $c^2$ satisfies $x=(x+1)^{-1}$, i.e.\ $x^2+x+1=0$, whose roots lie in $\mathbb F_4\setminus\mathbb F_2
\not\subset\mathbb F_q$, so $N_b(c)=N_b(c^2)=0$.  If $3\mid n$, then $T^n=c^n=1$;  by Galois descent for semilinear maps, $T$ is conjugate in $\mathrm{P}\Gamma\mathrm{L}(2,q)$ to the pure Frobenius $\sigma_b$, hence $|\Fix_{\mathbb P^1(\mathbb F_q)}(T)|=|\Fix_{\mathbb P^1(\mathbb F_q)}(\sigma_b)|=|\mathbb P^1(\mathbb F_{q_0})|=q_0+1$.  The points $0,1,\infty$ are not among them ($x^{2^b}(x+1)=1$ fails at $0$ and $1$, and $c(\infty)=0
\neq\infty$), so all $q_0+1$ fixed points lie in $\Lambda_q$, giving $N_b(c)=q_0+1$;  the other $3$-cycle $c^2$ is conjugate to $c$ in $H$.  The Burnside average over the free action of $H$ gives the displayed fixed-orbit count.
\end{proof}

\begin{theorem}[Three-subset orbits of Suzuki field extensions]\label{thm:suzuki-field-final}
Let $q=2^a$ with $a=2m+1$ odd, let $S=\Sz(q)$, and let
\[
A=S\rtimes\langle\sigma_c\rangle\le\Aut(S)=\Sz(q)\cdot(2m+1),
\]
where $\sigma_c\colon x\mapsto x^{2^c}$ is the field automorphism.  Put $d_0=(a,c)$ and $e=a/d_0$, so $|A/S|=e$.
\begin{enumerate}
\item The set of generic $S$-orbits is parametrized by $\mathcal G_q\cong(\mathbb F_q\setminus\{0,1\})/H$ \textup{(Lemma \ref{lem:suzuki-bluher-count})}, and $\sigma_c$ acts on this parameter space by
\[
[\lambda]\longmapsto[\lambda^{2^c}].
\]
\item The number of $A$-orbits on three-subsets is
\[
\boxed{\;\sigma_3(A)=1+\frac1e\sum_{r\mid e}\varphi\!\left(\frac er\right)
\begin{cases}\dfrac{2^{d_0r}-2}{6},&3
\nmid e/r,\\[3mm]2^{d_0r-1},&3\mid e/r,\end{cases}\;}
\]
where the first term is the invariant circle orbit.
\item If an $A$-orbit fuses $r_i$ generic $S$-orbits, then its design parameter is $\lambda_i=6r_i(q-1)$ and its setwise stabilizer has order $|A_X|=e/r_i$;  the circle orbit has parameter $3(q-1)$ and setwise stabilizer of order $2e$.
\end{enumerate}
\end{theorem}
\begin{proof}
(1) This is the classical affine parametrization of the generic orbits together with Lemma \ref{lem:suzuki-bluher-count};  the field automorphism acts by raising the parameter to the corresponding $2$-power.  (2) For $0\le j<e$ we have $\sigma_c^j=\sigma_{cj}$ with $(a,cj)=d_0(e,j)$;  writing $r=(e,j)$, the element $\sigma_c^j$ fixes, by Lemma \ref{lem:suzuki-bluher-count},
\[
\varepsilon_j=\begin{cases}\dfrac{2^{d_0r}-2}{6},&3
\nmid e/r,\\[2mm]2^{d_0r-1},&3\mid e/r\end{cases}
\]
generic $S$-orbits.  By Burnside's lemma on the parameter space $\mathcal G_q$ (the circle orbit contributes the first term $1$;  Theorem \ref{thm:sz-field-reduction}(2)),
\[
t_F=\frac1e\sum_{j=0}^{e-1}\varepsilon_j=\frac1e\sum_{r\mid e}\varphi\!\left(\frac er\right)
\begin{cases}\dfrac{2^{d_0r}-2}{6},&3
\nmid e/r,\\[2mm]2^{d_0r-1},&3\mid e/r,\end{cases}
\]
after merging the divisors $r$ of $e$ (there are $\varphi(e/r)$ elements of order $e/r$ in $\langle\sigma_c\rangle$).  (3) A fused $A$-orbit is the union of $r_i$ generic $S$-orbits of size $|S|$ (Theorem \ref{thm:sz}(2)), so its orbit size is $r_i|S|$ and, by (2),
\[
\lambda_i=\frac{6\,r_i|S|}{v(v-1)}=6r_i(q-1),\qquad |A_X|=\frac{|A|}{r_i|S|}=\frac{e}{r_i},
\]
and the circle orbit has $\lambda=3(q-1)$ and $|A_X|=|A|/(|S|/2)=2e$.
\end{proof}

\begin{remark}[Numerical instances]\label{rem:sz-field-examples}
Theorem \ref{thm:suzuki-field-final} gives $\sigma_3(\Sz(8){:}3)=2$:  the identity fixes $\frac{8-2}{6}=1$ generic orbit and the two non-trivial field automorphisms have $d=1$, $n=3$, hence $\varepsilon=2^{1-1}=1$ each, so $t_F=(1+1+1)/3=1$, in agreement with \textup{Proposition \ref{prop:sz8-computed}}.  For $\Sz(32){:}5$ the four non-trivial field automorphisms have $d=1$ and $3
\nmid n=5$, hence $\varepsilon=0$, and $t_F=5/5=1$, so $\sigma_3(\Sz(32){:}5)=2$.  For $\Sz(512){:}9$ ($a=9$):  the identity fixes $\frac{512-2}{6}=85$ generic orbits, the six elements of order $9$ have $d=1$ and $3\mid n=9$, fixing $2^{1-1}=1$ each, and the two elements of order $3$ have $d=3$ and $3\mid n=3$, fixing $2^{3-1}=4$ each, so $t_F=(85+6\cdot1+2\cdot4)/9=11$ and $\sigma_3(\Sz(512){:}9)=12$.  The values for $q=8$ and $q=32$ are confirmed by the machine computations of \textup{Chapter \ref{sec:verify}} \textup{(Table \ref{tab:verify})}:  $\sigma_3(\Sz(32))=6$ and $\sigma_3(\Sz(32){:}5)=2$;  the degree of $\Sz(512){:}9$ is $262145$, beyond direct machine enumeration, and the displayed value is the arithmetic output of the closed form.
\end{remark}

\begin{proposition}[Computed Suzuki cases]
\label{prop:sz8-computed}
For the concrete degree-$65$ permutation representations
\[
\Sz(8)\cong\mathrm{PrimitiveGroup}(65,6),
\qquad
\Sz(8){:}3\cong\mathrm{PrimitiveGroup}(65,7),
\]
the actions on three-subsets each have two orbits, with parameters
\[
\lambda=42
\qquad\text{and}\qquad
\lambda=21.
\]
The corresponding simple designs are $2$-$(65,3,42)$ and $2$-$(65,3,21)$.
The orbit sizes are $29120$ and $14560$, summing to $\binom{65}{3}=43680$.  For $\Sz(8)$ the induced images on the triples have orders $1$ and $2$ with trivial pointwise kernels;  for $\Sz(8){:}3$ they have orders $3$ and $2$ with pointwise kernels of orders $1$ and $3$, respectively.
\end{proposition}
\begin{proof}
This is the output of the GAP census (\textup{Table \ref{tab:verify}}):  in both cases the orbit sizes sum to $\binom{65}{3}$, the orbit--stabilizer identity $|X_i^G|\,|G_{X_i}|=|G|$ holds for every orbit, and $\lambda_i=6|G|/\bigl(|G_{X_i}|\,v(v-1)\bigr)$ \eqref{eq:lambda}.  It is a statement about the two specified degree-$65$ permutation representations;  the general orbit numbers for arbitrary field extensions $\Sz(q){:}f$ are given by the closed form of \textup{Theorem \ref{thm:suzuki-field-final}} \textup{(verified additionally for $\Sz(32){:}5$ in \textup{Chapter \ref{sec:verify}})}.
\end{proof}

% ============================================================
\section{The Ree groups \texorpdfstring{$\Ree(q)$}{Ree(q)}}
\label{sec:ree}
% ============================================================

Let $q=3^{2m+1}>3$ and $S=\Ree(q)={}^2G_2(q)$, acting $2$-transitively on the $q^3+1$ points of the Ree unital $\mathcal{R}$ (see \cite{ree1961,luneburg1966} for the construction and geometry of the group and the unital;  the $2$-transitive action of degree $q^3+1$;  $|S|=q^3(q^3+1)(q-1)$, and now $3\mid|S|$).  The group is a Zassenhaus group of degree $q^3+1$ with two-point stabilizers of order $q-1$;  the Sylow $3$-subgroups of order $q^3$ occur in the kernel of the point stabilizer, rather than as two-point stabilizers.
\begin{example}[The smallest Ree instance $q=3$]\label{ex:ree-q3}
The smallest value $q=3^{2m+1}$ with $m=0$ gives ${}^2G_2(3)\cong\PpGL(2,8)$ acting on the $28$ points of the Ree unital \textup{($|{}^2G_2(3)|=1512=3^3\cdot28\cdot2$)}.  The three-point orbit data, recorded in \textup{Table \ref{tab:verify}}, consist of six orbits with replication parameters
\[
\lambda=6,\ 4,\ 6,\ 2,\ 6,\ 2,
\]
i.e.\ three orbits with setwise stabilizer of order $2$ \textup{($\lambda=6=3(q-1)$)}, one orbit of order $3$ \textup{($\lambda=4=2(q-1)$)}, and the collinear and non-collinear orbits of order $6$ \textup{($\lambda=2=q-1$)};  there is no orbit with trivial stabilizer ($N_1=0$).  These values instantiate \textup{Theorem \ref{thm:ree-structure}(2)} at $q=3$, namely $(N_1,N_2,N_3,N_6)=(0,3,1,2)$, and serve as the small-case check on which the analysis for $q>3$ does not rely \textup{(Remark \ref{rem:ree-remark2})}.
\end{example}

The following statement combines the classical construction of \cite{ree1961,luneburg1966} with the present analysis and with the GAP computation for the smallest case $^2G_2(3)\cong\PpGL(2,8)$ (degree $28$, \textup{Table \ref{tab:verify}}).  This smallest case is the nonsimple derived-parameter instance and is \emph{not} used as evidence for the simple socle family $q>3$.

\begin{theorem}[Ree reduction]
\label{thm:ree}
Let $S=\Ree(q)={}^2G_2(q)$, $q=3^{2m+1}>3$, act in its standard $2$-transitive action of degree $q^3+1$, and let $\Ocal$ be a $3$-subset of the Ree unital \textup{(}$v=q^3+1$\textup{)}.
\begin{enumerate}
\item For every three-point orbit with representative $\Ocal$, $S_{\Ocal}\le S_3$ after quotienting by the pointwise stabilizer;  if the pointwise stabilizer is trivial, then
\[
|S_{\Ocal}|\in\{1,2,3,6\}.
\]
\item For each computed value of $q$, the actual orbit data are those listed in \textup{Table \ref{tab:verify}}.
\end{enumerate}
\end{theorem}
\begin{proof}
The action of $S_{\Ocal}$ on $\Ocal$ gives a homomorphism $S_{\Ocal}\longrightarrow S_3$ whose kernel is the pointwise stabilizer.  Therefore the image is one of the subgroups $1$, $C_2$, $C_3$, $S_3$.  If the pointwise stabilizer is trivial, the stabilizer itself is isomorphic to one of these four groups.  The occurrence and number of the corresponding orbits require an independent Ree-specific orbit calculation and are not inferred merely from this embedding;  the computed values are listed in \textup{Table \ref{tab:verify}}.
\end{proof}

\begin{remark}\label{rem:ree-remark2}
For the smallest simple case $q=27$ the GAP suborbit computation of \textup{Chapter \ref{sec:verify}} determines the collinear triple stabilizer explicitly:  the three points of a block of $\mathcal R$ form one $S$-orbit with setwise stabilizer of order $6$, pointwise stabilizer of order $2$ and induced image $C_3$ on the three points, hence $\lambda=q-1=26$.  For arbitrary $q>3$ the same conclusion follows from the classical block stabilizer $C_2\times\PSL(2,q)$ of the Ree unital \textup{\cite{luneburg1966}} together with \textup{Theorem \ref{thm:ree-structure}(3)} below, which only needs the order of the two-point stabilizers and the occurrence of one collinear orbit.  The general formula $\lambda=6|S|/\bigl(|S_{\Ocal}|\,v(v-1)\bigr)$ \textup{(Corollary \ref{cor:ree-designs})} then gives $\lambda=q-1$ for the collinear orbit and the values $6(q-1)$, $3(q-1)$, $2(q-1)$, $q-1$ for the four families of \textup{Theorem \ref{thm:ree-structure}}.
\end{remark}

The orbit \emph{number} of the Ree socle on three-subsets is, however, now determined for every $q$ by the class data compiled in \textup{Lemma \ref{lem:ree-classes}} below.  The input is classical \textup{(Ward's class table, as systematised in \cite{bradley2014}, Lemma 3.5.2 and Table 3.5)};  all entries of \textup{Lemma \ref{lem:ree-classes}} are independently verified for $q=27$ by the GAP computation reported in \textup{Chapter \ref{sec:verify}}.

\begin{lemma}[Ree class input]\label{lem:ree-classes}
Let $S=\Ree(q)={}^2G_2(q)$, $q=3^{2m+1}>3$, act on $\Omega$ with $|\Omega|=q^3+1$, and write $f(g)=\operatorname{fix}_\Omega(g)$ for the number of fixed points of $g\in S$.  The non-identity conjugacy class types of $S$, together with their fixed-point numbers, centralizer orders and numbers of classes, are as follows \textup{\cite{bradley2014}, Lemma 3.5.2 and Table 3.5}:
\begin{center}
\footnotesize
\renewcommand{\arraystretch}{1.1}
\setlength{\tabcolsep}{3pt}
\begin{tabular}{@{}llcc@{}}
\toprule
class type & $f(g)$ & $|C_S(g)|$ & number of classes \\
\midrule
$J$ (involution) & $q+1$ & $q(q+1)(q-1)$ & $1$ \\
$X$ (order $3$, in the centre of a Sylow $3$-subgroup) & $1$ & $q^3$ & $1$ \\
$T$, $T^{-1}$ (order $3$) & $1$ & $2q^2$ & $1$ each \\
$Y$, $YT$, $YT^{-1}$ (order $9$) & $1$ & $3q$ & $1$ each \\
$JT$, $JT^{-1}$ (order $6$) & $1$ & $2q$ & $1$ each \\
$R^a
\neq1$, order $m\mid(q-1)/2$ & $2$ & $q-1$ & $(q-3)/2$ \\
$S^a
\neq1$, order $m\mid(q+1)/4$ & $0$ & $q+1$ & $(q-3)/2$ \\
$JR^a
\neq R^a$ & $2$ & $q-1$ & $(q-3)/2$ \\
$JS^a
\neq S^a$ & $0$ & $q+1$ & $(q-3)/2$ \\
$V$ (order $m\mid q+1-\sqrt{3q}$) & $0$ & $q+1-\sqrt{3q}$ & $(q-\sqrt{3q})/6$ \\
$W$ (order $m\mid q+1+\sqrt{3q}$) & $0$ & $q+1+\sqrt{3q}$ & $(q+\sqrt{3q})/6$ \\
\bottomrule
\end{tabular}
\end{center}
Together with the identity class ($f= q^3+1$), these give the $q+8$ conjugacy classes of $S$ \textup{\cite{bradley2014}, Lemma 3.5.1}.  The cycle types needed below are
\begin{gather*}
\pi_J=1^{q+1}2^{(q^3-q)/2},\qquad
\pi_X=\pi_T=\pi_{T^{-1}}=1^13^{q^3/3},\\
\pi_{JT}=\pi_{JT^{-1}}=1^13^{q/3}6^{(q^3-q)/6},
\end{gather*}
and $\pi_Y=\pi_{YT}=\pi_{YT^{-1}}=1^19^{q^3/9}$ \textup{\cite{bradley2014}, Proposition 3.5.3, with the entries for the $JT$-classes corrected to the GAP-verified form $1^13^{q/3}6^{(q^3-q)/6}$ of \textup{Chapter \ref{sec:verify}}}.  Every element of the remaining class types fixes at most two points and has no cycle of length $2$ or $3$.
\end{lemma}

\begin{lemma}[Three-subsets fixed by a permutation]\label{lem:eta3}
Let $g$ be a permutation of a finite set $\Omega$, let $f=\operatorname{fix}_\Omega(g)$, and let $c_2(g)$, $c_3(g)$ be the numbers of $2$- and $3$-cycles of $g$.  The number of $3$-subsets of $\Omega$ fixed setwise by $g$ is
\[
\eta_3(g)=\binom{f}{3}+c_3(g)+f\,c_2(g).
\]
\end{lemma}
\begin{proof}
A $3$-subset $X$ satisfies $X^g=X$ if and only if $X$ is a union of $g$-cycles of length at most $3$:  either three fixed points, or one $3$-cycle, or one $2$-cycle together with one fixed point.  These three alternatives are disjoint and give the displayed count.
\end{proof}

\begin{lemma}[Pointwise stabilizers of non-collinear triples]\label{lem:ree-pointwise}
Let $S={}^2G_2(q)$, $q=3^{2m+1}>3$, act on the Ree unital, and let $T=\{u,v,x\}$ be a triple of points not contained in a block.  Then the pointwise stabilizer $S_{(T)}$ is trivial.
\end{lemma}
\begin{proof}
By \textup{Lemma \ref{lem:ree-classes}}, every non-identity element of $S$ outside the involution class $J$ fixes at most two points in the degree-$(q^3+1)$ action \textup{(the order-$3$ classes $X,T,T^{-1}$ and the order-$6$ classes $JT,JT^{-1}$ fix one point, as do the order-$9$ classes $Y,YT,YT^{-1}$;  the classes $R^a$ and $JR^a$ fix two;  the classes $S^a,JS^a,V,W$ fix none)};  an element of the involution class $J$ fixes exactly the $q+1$ points of one block.  Hence an element fixing $u,v,x$ pointwise must belong to $J$, and the block it fixes pointwise contains the triple $T$, making $T$ collinear, contrary to the hypothesis.  Thus $S_{(T)}=1$.
\end{proof}

\begin{lemma}[Dihedral two-point stabilizer and its suborbits]\label{lem:ree-dihedral}
Let $S={}^2G_2(q)$, $q=3^{2m+1}>3$, act $2$-transitively on the $q^3+1$ points of the Ree unital $\mathcal R$, and fix two points $u
\ne v$.  The setwise two-point stabilizer has order $2(q-1)$ and is a dihedral group $D_{2(q-1)}$, acting on the remaining $q^3-1$ points with orbit lengths
\[
q-1,\quad (q+1)\times(q-1),\quad \frac{q^2-1}{2}\times 2(q-1),
\]
namely one orbit of the $(q-1)$ points of the block $L_{uv}$ through $u,v$, the $q+1$ invariant orbits carried by the $q+1$ points of the block $L'$ fixed pointwise by the transposition $\tau\in S$ exchanging $u$ and $v$ \textup{(the $J$-involution whose fixed block is $L'$;  the $q-1$ transpositions of $S_{\{u,v\}}$ fix pairwise distinct blocks, and $L'$ is the one through no point of $\{u,v\}$)}, and the $\frac{q^2-1}{2}$ paired orbits of the remaining points.
\end{lemma}
\begin{proof}
The pointwise two-point stabilizer $S_{u,v}$ is cyclic of order $q-1$ \textup{(Zassenhaus group, \cite[pp.~256--259]{luneburg1966})}.  The transposition $\tau\in S$ exchanging $u$ and $v$ exists because the number of involutions exchanging a given pair is $(q^3-q)/2$ \textup{(from the $J$-class data of \textup{Lemma \ref{lem:ree-classes}}:  $|S|/|C_S(J)|$ involutions, each swapping $(q^3-q)/2$ pairs)}.  The involution $\tau$ is of type $J$ and fixes pointwise the $q+1$ points of a block $L'
\ne L$;  in particular $u,v
\notin L'$.

By the standard Zassenhaus normalizer structure of the Ree action \textup{\cite[pp.~256--259]{luneburg1966}}, the setwise stabilizer of two distinct points $u,v$ is
\[
S_{\{u,v\}}\cong C_{q-1}\rtimes C_2,
\]
where the cyclic factor is the pointwise two-point stabilizer $S_{u,v}$ and the involution $\tau$ acts on $C_{q-1}$ by inversion;  equivalently $\tau c\tau=c^{-1}$ for every $c\in S_{u,v}$.  Hence, with $c$ a generator of $S_{u,v}$,
\[
S_{\{u,v\}}=\langle c,\tau\rangle\cong D_{2(q-1)}.
\]
That is, by the cited normalizer structure the setwise two-point stabilizer is $S_{\{u,v\}}=\langle S_{u,v},\tau\rangle\cong D_{2(q-1)}$;  in particular its order is $|S_{\{u,v\}}|=2(q-1)$.
(We record that the weaker observation that $\tau$ does \emph{not} centralize $S_{u,v}$ is not sufficient for this conclusion:  the automorphism of $C_{q-1}$ induced by $\tau$ could in general be $c\mapsto c^r$ for any $r$ with $r^2\equiv1\pmod{q-1}$, and for composite $q-1$ these need not be $\pm1$.  The dihedral conclusion is therefore taken from the cited normalizer structure rather than from an ad hoc argument.)

\emph{Suborbits.}  On the block $L\setminus\{u,v\}$ ($q-1$ points) the stabilizer of $x\in L\setminus\{u,v\}$ inside $S_{u,v}$ has order $2$ (the unique involution of $S_{u,v}$ fixes the block pointwise), so $S_{u,v}$ has one orbit of length $(q-1)/2$ there;  $\tau$ permutes this orbit with the orbit of $\tau(x)$, giving one $S_{\{u,v\}}$-orbit of length $q-1$.  On $\Omega\setminus L$ the action of $S_{u,v}$ is free \textup{(the pointwise stabilizer of a non-collinear triple is trivial by \textup{Lemma \ref{lem:ree-pointwise}}, so no non-identity element of $S_{u,v}$ fixes a third point off $L$)};  hence $q(q+1)$ free orbits of length $q-1$.  The transposition $\tau$ fixes pointwise exactly the $q+1$ points of $L'$, one on each of $q+1$ free $S_{u,v}$-orbits, and has no fixed point on any other free orbit:  on a free orbit $O$, $\tau$ reverses the $c$-action, so a fixed point $x$ would satisfy $c^kx=\tau$-fixed for every $k$, forcing $O$ to contain the $q+1$ fixed points of $\tau$, i.e.\ $O$ is an invariant orbit through a point of $L'$.  Thus $\tau$ fixes $q+1$ free orbits and pairs the remaining $q(q+1)-(q+1)=q^2-1$ free orbits, giving $(q^2-1)/2$ paired $S_{\{u,v\}}$-orbits of length $2(q-1)$.
\end{proof}

\begin{theorem}[Three-subset orbit number of $\Ree(q)$]\label{thm:ree-triples}
Let $S=\Ree(q)={}^2G_2(q)$, $q=3^{2m+1}$, act $2$-transitively on the $q^3+1$ points of the Ree unital.  Then the number of $S$-orbits on three-subsets is
\[
\sigma_3(S)=\frac{(q+3)^2}{6}.
\]
In particular $\sigma_3(\Ree(3))=6$, $\sigma_3(\Ree(27))=150$, and $\sigma_3(\Ree(243))=10086$.  The value $6$ agrees with the direct enumeration of \textup{Table \ref{tab:verify}} for ${}^2G_2(3)\cong\PpGL(2,8)$;  the value $150$ is independently confirmed by the class computation of \textup{Chapter \ref{sec:verify}}.
\end{theorem}
\begin{proof}
By Burnside's lemma $\sigma_3(S)=\frac1{|S|}\sum_{g\in S}\eta_3(g)$, and by \textup{Lemma \ref{lem:eta3}} it suffices to know, for each class, the fixed-point number and the numbers of $2$- and $3$-cycles.  Using \textup{Lemma \ref{lem:ree-classes}} and summing over the conjugacy classes, the identity class and the $J$-, $(X,T,T^{-1})$-, $(JT,JT^{-1})$-families give
\begin{align*}
\sigma_3(S)
&=\frac{1}{q^3(q^3+1)(q-1)}\binom{q^3+1}{3}
 +\frac{1}{q^3-q}\Bigl(\binom{q+1}{3}+(q+1)\frac{q^3-q}{2}\Bigr)\\
&\quad+\frac{q+1}{q^3}\cdot\frac{q^3}{3}
 +\frac{1}{q}\cdot\frac{q}{3},
\end{align*}
and every other class contributes $0$:  the $Y$-family has $\eta_3(\pi_Y)=0$ since $\pi_Y$ has no $2$- or $3$-cycle, and the remaining types fix at most two points and have no $2$- or $3$-cycle.  Here the coefficient $\frac1{q^3-q}$ is $1/|C_S(J)|$, the coefficient $\frac{q+1}{q^3}$ is $\frac1{q^3}+\frac1{2q^2}+\frac1{2q^2}$ for the three order-$3$ classes, and the coefficient $\frac1q$ is $2\cdot\frac1{2q}$ for the two order-$6$ classes.  Simplifying,
\[
\frac{q^2+q+1}{6}+\frac{3q+4}{6}+\frac{q+1}{3}+\frac13
=\frac{q^2+6q+9}{6}=\frac{(q+3)^2}{6},
\]
where the first term is $\binom{q^3+1}{3}/\bigl(q^3(q^3+1)(q-1)\bigr)$ and the second is $\bigl[\binom{q+1}{3}+(q+1)(q^3-q)/2\bigr]/(q^3-q)$.  Since $q=3^{2m+1}\equiv3\pmod6$, we have $q+3=6k$ with $k=(3^{2m}+1)/2\in\mathbb N$, so $(q+3)^2/6=6k^2$ is an integer.  For $q=3$ the class input of \textup{Lemma \ref{lem:ree-classes}} requires $q>3$ \textup{\cite{bradley2014}};  the value $6$ for ${}^2G_2(3)\cong\PpGL(2,8)$ is taken from the direct enumeration of \textup{Table \ref{tab:verify}}, and it agrees with the formula.  The value for $q=27$ ($150$) and the class count of \textup{Chapter \ref{sec:verify}} give an independent check of the $q>3$ case.
\end{proof}

\begin{lemma}[Counting of transpositions and $3$-cycles]\label{lem:ree-count}
Let $S={}^2G_2(q)$, $q=3^{2m+1}>3$, act on the Ree unital, fix two points $u,v$, and let $L'$ be the block fixed pointwise by the transposition $\tau\in S_{\{u,v\}}$.  We take as standing assumptions the dihedral structure $S_{\{u,v\}}\cong D_{2(q-1)}$ and the suborbit decomposition of \textup{Lemma \ref{lem:ree-dihedral}}, and the trivial pointwise stabilizer of every non-collinear triple of \textup{Lemma \ref{lem:ree-pointwise}}.  Then, among the $S_{\{u,v\}}$-orbits on the remaining points:
\begin{enumerate}
\item exactly $q$ of the paired orbits contain a point $x$ with $(13)$- or $(23)$-type transposition in $S_{\{u,v,x\}}$;
\item exactly $(q-1)/2$ of the paired orbits contain a point $x$ with a $3$-cycle in $S_{\{u,v,x\}}$;
\item exactly one of the $q+1$ invariant orbits (equivalently one $x\in L'$) has a $3$-cycle in $S_{\{u,v,x\}}$.
\end{enumerate}
\end{lemma}
\begin{proof}
We first record a uniformity fact.  Let $\mathcal R_3=\{r\in S:\mathrm{ord}(r)=3\}$;  by \textup{Lemma \ref{lem:ree-classes}} $|\mathcal R_3|=(q^3+1)(q^2-1)$, every $r\in\mathcal R_3$ fixes exactly one point and acts on the remaining $q^3$ points as a product of $q^3/3$ three-cycles.  Counting pairs $(r,(u,v))$ with $r(u)=v$:  each $r$ contributes $3\cdot q^3/3=q^3$ ordered pairs, so the total is $q^3(q^3+1)(q^2-1)$;  since $S$ is $2$-transitive, the distribution over the $(q^3+1)q^3$ ordered pairs $(u,v)$ is uniform, and therefore
\[
\bigl|\{r\in S:\mathrm{ord}(r)=3,\ r(u)=v\}\bigr|=q^2-1.
\]
Fix $(u,v)$ and put $L=L_{uv}$;  write $\mathcal R=\{r:\mathrm{ord}(r)=3,\ r(u)=v\}$ and $x=r(v)$.  Then $x\ne u,v$, and $r$ cyclically permutes $\{u,v,x\}$, so $x$ is of $3$-cycle type and the map $r\mapsto x$ is injective \textup{(the order-$3$ elements of $S_{\{u,v,x\}}$ are $r$ and $r^{-1}$, of which exactly one sends $u$ to $v$)}.  Hence the orbit counts of the lemma are recovered below from the cardinalities of the parts of $\mathcal R$ with $x$ of the respective kinds, divided by the length of one orbit of the corresponding kind.

\emph{Collinear part.}  For $x\in L\setminus\{u,v\}$ the block stabilizer acts as $\PSL(2,q)$ on the projective line $L$, which is $3$-transitive;  the two order-$3$ elements of $S_{\{u,v,x\}}$ cyclically permute the three points, exactly one of them sending $u$ to $v$.  Hence exactly $|L\setminus\{u,v\}|=q-1$ elements of $\mathcal R$ have $x\in L\setminus\{u,v\}$.

\emph{Two independent counts.}  \textup{(i) \emph{$3$-cycles.}}  Write $n_3^i$ and $n_3^p$ for the numbers of invariant resp.\ paired orbits of $3$-cycle type.  The elements of $\mathcal R$ with third point $x$ of $3$-cycle type number $n_3^i(q-1)+2n_3^p(q-1)$ \textup{(an invariant orbit has length $q-1$, a paired orbit $2(q-1)$, and every point of such an orbit is of $3$-cycle type, the property being $S_{\{u,v\}}$-invariant)}\textup{)}.  Together with the collinear part,
\begin{equation}
q^2-1=(q-1)+n_3^i(q-1)+2n_3^p(q-1),\qquad\text{i.e.}\quad 2n_3^p+n_3^i=q.
\label{eq:ree-star1}
\end{equation}
\textup{(ii) \emph{Transpositions.}}  The same uniformity with the transposition class $J$ \textup{(Lemma \ref{lem:ree-classes}:  $|J|=q^2(q^3+1)/(q+1)$, each transposition fixing $q+1$ points)} gives
\[
\bigl|\{\tau'\in J:\tau'(v)=v\}\bigr|
=\frac{|J|(q+1)}{q^3+1}=q^2,
\]
and a transposition fixing both $u$ and $v$ pointwise fixes the block $L$ pointwise:  the pointwise two-point stabilizer $S_{u,v}$ is cyclic of order $q-1$ \textup{(Zassenhaus;  cf.\ \cite{luneburg1966}, as in Lemma \ref{lem:ree-dihedral})}, and its unique involution $c^{(q-1)/2}$ lies in the unique involution class $J$ of $S$ \textup{(Lemma \ref{lem:ree-classes}:  every involution of $S$ is of class $J$)};  the members of $J$ fix pointwise the $q+1$ points of a block \textup{(Lemma \ref{lem:ree-dihedral})}, and since $c^{(q-1)/2}$ fixes $u$ and $v$, that block is the block $L$ through $u,v$.  Hence
\[
\bigl|\{\tau'\in J:\tau'(u)=u,\ \tau'(v)=v\}\bigr|=1.
\]
The remaining $q^2-1$ transpositions fix $v$ and move $u$;  each realizes a $(13)$-type permutation of $\{u,v,\tau'(u)\}$, and symmetrically for $(23)$.  The map $\tau'\mapsto\tau'(u)$ is injective on them \textup{(if two such transpositions agreed on $u$, their product would fix $\{u,v,\tau'(u)\}$ pointwise, forcing the identity by \textup{Lemma \ref{lem:ree-pointwise}}, the trivial pointwise stabilizer of a non-collinear triple)}, so the $(13)$-type points number $q^2-1=(q-1)(q+1)$;  such a point is necessarily non-collinear \textup{(a collinear triple lies on a block fixed pointwise by the block involution, so its pointwise stabilizer is nontrivial, and its induced image is $C_3$, whence no transposition fixes one of its points and swaps the other two; see the block-stabilizer computation in the proof of \textup{Lemma \ref{lem:ree-dihedral}})}\textup{)}.  A paired orbit whose stabilizer contains a transposition has length $2(q-1)$ and contains exactly $q-1$ points of $(13)$-type and $q-1$ of $(23)$-type \textup{(the involution $\tau$ conjugates $(13)$ to $(23)$)};  the $n_3^i$ invariant orbits of $3$-cycle type contain $q-1$ points of $(13)$-type each \textup{(their stabilizer has image $S_3$)}.  Hence, $t$ denoting the number of paired transposition orbits,
\begin{equation}
\bigl(t+n_3^i\bigr)(q-1)=(q-1)(q+1),\qquad\text{i.e.}\quad t+n_3^i=q+1.
\label{eq:ree-star2}
\end{equation}

\emph{The orbit number.}  To pin down $n_3^i$ we combine the two counts above with the fusion rule for unordered pairs, recorded here:  an $S$-orbit of a non-collinear triple meets the set of triples containing $\{u,v\}$ in $3$, $2$, $1$ or $1$ of the $S_{\{u,v\}}$-orbits according as its setwise stabilizer has order $1,2,3$ or $6$ \textup{(the three unordered pairs of a triple lie in one $S_{\{u,v\}}$-orbit exactly when the induced image contains a $3$-cycle, and in two exactly when the image is $\langle(12)\rangle$, since $S$ is transitive on the relevant pairs)}.  Of the $q+1$ invariant orbits, $n_3^i$ carry a $3$-cycle \textup{(image $S_3$)} and the remaining $q+1-n_3^i$ have image $\langle(12)\rangle$;  the latter fuse two-to-one with the $t=q+1-n_3^i$ paired transposition orbits \textup{(by \eqref{eq:ree-star2})} into $N_2=q+1-n_3^i$ $S$-orbits, while the $n_3^p$ paired orbits of $3$-cycle type give $N_3=n_3^p$ \textup{(no fusion)}.  With $N_1=q(q-3)/6$ free orbits \textup{(three-to-one fusion)} and $N_6=2$, the total $\sigma_3(S)=(q+3)^2/6$ of \textup{Theorem \ref{thm:ree-triples}} gives
\[
\frac{q(q-3)}6+\bigl(q+1-n_3^i\bigr)+\frac{q-n_3^i}{2}+2=\frac{(q+3)^2}{6},
\]
where $n_3^p=(q-n_3^i)/2$ is substituted from \eqref{eq:ree-star1};  the left-hand side equals $\frac{q^2}{6}+q+3-\frac32n_3^i$ and the right-hand side $\frac{q^2}{6}+q+\frac32$, whence $n_3^i=1$.  This proves (3) \textup{(an invariant orbit of $3$-cycle type contains exactly one point of $L'$, by the suborbit structure of \textup{Lemma \ref{lem:ree-dihedral}};  equivalently exactly one $x\in L'$ carries a $3$-cycle)}\textup{.}  Substitution back into \eqref{eq:ree-star1} and \eqref{eq:ree-star2} gives $n_3^p=(q-1)/2$ and $t=q$, proving (2) and (1).

All three counts are confirmed for $q=27$ by the computation of \textup{Chapter \ref{sec:verify}}:  the suborbit distribution $393=1+28+364$ of \textup{Table \ref{tab:verify-ree}} splits as $[1,324],[2,54],[3,14],[6,1]$, i.e.\ $324/3=108$ free, $54/2=27$ transposition (paired and invariant fused two-to-one), $14-1=13$ paired $3$-cycle and $1$ invariant $3$-cycle orbits, in agreement with the three formulas.
\end{proof}

\begin{theorem}[Three-subset orbit structure of $\Ree(q)$]
\label{thm:ree-structure}
Let $S={}^2G_2(q)$, $q=3^{2m+1}>3$, act $2$-transitively on the $q^3+1$ points of the Ree unital $\mathcal R$, and let $N_j$ denote the number of $S$-orbits on three-subsets whose setwise stabilizer has order $j$.
\begin{enumerate}
\item The setwise two-point stabilizer of $S$ is a dihedral group $D_{2(q-1)}$ of order $2(q-1)$, with the suborbit decomposition of \textup{Lemma \ref{lem:ree-dihedral}}:  one orbit of the $q-1$ points of the block $L_{uv}$ through $u,v$, the $q+1$ invariant orbits of length $q-1$ carried by the $q+1$ points of the block $L'$ fixed pointwise by the transposition $\tau$ exchanging $u$ and $v$, and the $\frac{q^2-1}{2}$ paired orbits of the remaining points.
\item The occurring setwise stabilizer orders are $1,2,3,6$, with multiplicities
\[
N_1=\frac{q(q-3)}6,\qquad N_2=q,\qquad N_3=\frac{q-1}{2},\qquad N_6=2.
\]
\item The orbit of $N_6=2$ is split into one \emph{collinear} orbit (the $b$ triples lying on blocks;  pointwise stabilizer $C_2$, induced image $C_3$, hence setwise stabilizer $6$) and one \emph{non-collinear} orbit (trivial pointwise stabilizer, induced image $S_3$, setwise stabilizer $6$).  Every non-collinear triple has trivial pointwise stabilizer.
\end{enumerate}
In particular $N_1+N_2+N_3+N_6=(q+3)^2/6$, and the orbit-size identity $N_1+N_2/2+N_3/3+N_6/6=(q^2+q+1)/6$ holds.  For $q=27$ the four values are $(108,27,13,2)$, verified by the direct suborbit computation of \textup{Chapter \ref{sec:verify}};  for $q=3$ ($^2G_2(3)\cong\PpGL(2,8)$) the same formulas give $(0,3,1,2)$, matching the six orbits with $\lambda=6,4,6,2,6,2$ of \textup{Table \ref{tab:verify}}.
\end{theorem}
\begin{proof}
Since $S$ is $2$-transitive, fix two points $u,v$ and put $L=L_{uv}$.  By \textup{Lemma \ref{lem:ree-dihedral}} the setwise two-point stabilizer is $S_{\{u,v\}}=\langle c,\tau\rangle\cong D_{2(q-1)}$, with $S_{u,v}=\langle c\rangle$ cyclic of order $q-1$ and $\tau$ the transposition exchanging $u$ and $v$, and its action on the remaining $q^3-1$ points has exactly the suborbits of item (1).  This proves (1).

\emph{Stabilizer orders.}  For a representative $T=\{u,v,x\}$ with $x$ in the block-orbit, $S_T$ has order $6$:  the pointwise part is $C_2$ (the block involution) and the induced image is $C_3$, the latter following from the classical block stabilizer $C_2\times\PSL(2,q)$ acting on the block as $\PSL(2,q)$ on the projective line \textup{\cite{luneburg1966}} (for $q=27$ verified directly in \textup{Chapter \ref{sec:verify}}).  For non-collinear $x$ the pointwise stabilizer is trivial by \textup{Lemma \ref{lem:ree-pointwise}}, so $|S_T|$ equals the order of the induced image, i.e.\ $1$, $2$ or $3$ according as $S_T$ contains no transposition, a transposition but no $3$-cycle, or a $3$-cycle;  the image is $S_3$ exactly when both occur, giving order $6$.

On the $q+1$ invariant orbits (the points of $L'$), the transposition $\tau$ is present, so $|S_T|\ge2$.  The counting Lemma \ref{lem:ree-count} shows that exactly one of these $q+1$ orbits carries a $3$-cycle (equivalently, a third-point $x\in L'$ for which some order-$3$ element of $S$ cyclically permutes $\{u,v,x\}$), and that among the $(q^2-1)/2$ paired orbits exactly $q$ carry a transposition and $(q-1)/2$ carry a $3$-cycle, the remaining $q(q-3)/2$ being free.  Since a $3$-cycle together with $\tau$ generates $S_3$, the distribution (2) follows at the level of $S_{\{u,v\}}$-orbits:  one orbit of order $6$ (collinear), $q+1$ of order $2$ (the invariant orbits without a $3$-cycle), one of order $6$ (the invariant orbit with a $3$-cycle, image $S_3$), $q$ of order $2$ and $(q-1)/2$ of order $3$ among the paired orbits, and $q(q-3)/2$ of order $1$.

\emph{Fusion.}  It remains to pass from $S_{\{u,v\}}$-orbits to $S$-orbits.  Let $O$ be the $S$-orbit of a non-collinear $T=\{u,v,x\}$, and consider the elements of $O$ containing $u,v$ as a subset.  For $A\in\binom{T}{2}$ let $O_A$ be the set of $gT$ with $g^{-1}\{u,v\}=A$, i.e.\ $g(A)=\{u,v\}$.  The three sets $O_{\{u,v\}}$, $O_{\{u,x\}}$, $O_{\{v,x\}}$ are $S_{\{u,v\}}$-invariant, mutually disjoint, and by the transitivity of $S$ on the relevant pairs they form three $S_{\{u,v\}}$-orbits when $|S_T|=1$ (free case):  indeed $S_{\{u,v\}}$ acts transitively on the representatives $g$ with $g(A)=\{u,v\}$, and the map $g\mapsto gT$ is then $S_{\{u,v\}}$-equivariant.  When $|S_T|=2$ with image $\langle(12)\rangle$, the $3$-cycle-free transposition $g$ realising $A=\{u,x\}$ or $A=\{v,x\}$ is conjugate into the image, so exactly two of the three sets survive as distinct $S_{\{u,v\}}$-orbits;  when the image is $C_3$ (respectively contains a $3$-cycle), the three $2$-subsets of $T$ lie in one orbit, giving a single $S_{\{u,v\}}$-orbit.  Therefore the $S$-orbit of a non-collinear triple contains $3$, $2$, $1$, $1$ of the $S_{\{u,v\}}$-orbits according as $|S_T|=1,2,3,6$.  Combining with the suborbit counts:  the $q(q-3)/2$ free orbits fuse three-to-one into $N_1=q(q-3)/6$;  the $q$ transposition orbits of order $2$ among the paired orbits fuse two-to-one with $q$ of the $q+1$ invariant orbits of order $2$ into $N_2=q$ (one invariant orbit is left, the one with the $3$-cycle);  the $(q-1)/2$ orbits of order $3$ do not fuse, giving $N_3=(q-1)/2$;  the collinear orbit and the order-$6$ invariant orbit give $N_6=2$.  This proves (2) and (3).
\end{proof}

\begin{corollary}[Designs from $\Ree(q)$]\label{cor:ree-designs}
Since $|S|=q^3(q^3+1)(q-1)$ and $v=q^3+1$, the general formula is
\[
\lambda=\frac{6(q-1)}{|S_{\Ocal}|}.
\]
By \textup{Theorem \ref{thm:ree-triples}} the number of orbit systems is $(q+3)^2/6$, and by \textup{Theorem \ref{thm:ree-structure}} the occurring stabilizer orders and their multiplicities are known for every $q>3$:  $q(q-3)/6$ orbits with $|S_{\Ocal}|=1$ give $2$-$(q^3+1,3,6(q-1))$, $q$ orbits with $|S_{\Ocal}|=2$ give $2$-$(q^3+1,3,3(q-1))$, $(q-1)/2$ orbits with $|S_{\Ocal}|=3$ give $2$-$(q^3+1,3,2(q-1))$, and the two orbits with $|S_{\Ocal}|=6$ (one collinear, one non-collinear) give $2$-$(q^3+1,3,q-1)$.
For the computed case $q=3$ \textup{(${}^2G_2(3)\cong\PpGL(2,8)$ on $28$ points)}, \textup{Table \ref{tab:verify}} records six orbits with $\lambda=6,4,6,2,6,2$, and the formulas $N_1=0,N_2=3,N_3=1,N_6=2$ of \textup{Theorem \ref{thm:ree-structure}} evaluated at $q=3$ give exactly this:  the three orbits with $\lambda=6=3(q-1)$ have stabilizer order $2$, the single orbit with $\lambda=4=2(q-1)$ has stabilizer order $3$, and the two orbits with $\lambda=2=q-1$ have stabilizer order $6$;  no orbit of stabilizer order $1$ (which would have $\lambda=12$) occurs, matching $N_1=0$.  For $q=27$ the distribution is $(108,27,13,2)$ with $\lambda=156,78,52,26$ respectively \textup{(verified, Chapter \ref{sec:verify})};  in particular the $150$ orbit systems consist of $108$ copies of $2$-$(19684,3,156)$, $27$ copies of $2$-$(19684,3,78)$, $13$ copies of $2$-$(19684,3,52)$ and two copies of $2$-$(19684,3,26)$.
\end{corollary}

\begin{lemma}[Fixed orbits for the cubic field extension]\label{lem:ree-ext-cycles}
Let $S={}^2G_2(q_0^3)$ and let $\sigma$ be the field automorphism of order $3$ fixing the sub-unital $\mathcal U_0\cong{}^2G_2(q_0)$ on $q_0^3+1$ points. Put
\[
R_0=|{}^2G_2(q_0)|=q_0^3(q_0^3+1)(q_0-1).
\]
Then $\sigma$ fixes exactly
\[
a=\frac{q_0^2+2q_0+3}{2}
\]
$S$-orbits on three-subsets, namely $(q_0+3)^2/6$ subfield orbits and $q_0^2/3$ free cycle-type orbits (with $S$-setwise stabilizer $1$).
\end{lemma}
\begin{proof}
The automorphism $\sigma$ fixes the $q_0^3+1$ rational points and permutes all other points in $3$-cycles, and its centralizer is $C_S(\sigma)={}^2G_2(q_0)$, of order $R_0$.

On a block (a projective line over $\F_{q_0^3}$) the field automorphism $\sigma$ fixes exactly the $q_0+1$ points of $\mathbb P^1(\F_{q_0})$. Hence a $\sigma$-invariant block contains precisely $q_0+1$ rational points and is one of the sub-unital blocks, whose number is
\[
b_0=q_0^2(q_0^2-q_0+1).
\]
Each such block carries $(q_0^3-q_0)/3$ non-rational $\sigma$-cycles, so the number of collinear $\sigma$-cycles is
\[
b_0\frac{q_0^3-q_0}{3}
 =\frac{q_0^3(q_0^3+1)(q_0-1)}{3}
 =\frac{R_0}{3}.
\]
The total number of non-rational $\sigma$-cycles is $(q_0^9-q_0^3)/3$, and therefore the number of non-collinear $\sigma$-cycles is
\[
N=\frac{q_0^9-q_0^3}{3}-\frac{R_0}{3}
 =\frac{q_0^4(q_0-1)(q_0+1)(q_0^3+1)}{3}.
\]

We next determine how many of these fixed three-subsets lie on subfield orbits. Let $\mathcal O$ be a subfield orbit with $S$-setwise stabilizer $H$ of order $m$, which $\sigma$ centralizes (so $H\leq C_S(\sigma)$), and write $G_T=H\rtimes\langle\sigma\rangle$ for the stabilizer of a representative $T$. A point $sH$ of $\mathcal O$ is fixed by $\sigma$ precisely when $s^{-1}\sigma s\in G_T$. Double-counting gives
\[
|\operatorname{Fix}_{\mathcal O}(\sigma)|=\frac{n_T R_0}{m},
\]
where $n_T$ is the number of elements of $G_T$ conjugate to $\sigma$ in $S$: each such element is $s^{-1}\sigma s$ for exactly $R_0$ elements $s$, and each fixed point accounts for $m$ representatives. An element $h\sigma$ with $h\in H$ has order $3$ exactly when $h$ has order $1$ or $3$, and is then conjugate to $\sigma$ by the Lang--Steinberg theorem (the Lang map of the defining field automorphism is surjective on the connected Ree group); when $h$ is an involution, $h\sigma$ has order $6$ and its fixed points are $\operatorname{Fix}(h)\cap\operatorname{Fix}(\sigma)$, namely the $q_0+1$ points fixed by $h$ on the sub-unital. Thus
\[
n_T=1+\#\{\text{elements of order }3\text{ in }H\}.
\]

By \textup{Theorem \ref{thm:ree-structure}} applied to the sub-unital group ${}^2G_2(q_0)$, the subfield orbits comprise $q_0(q_0-3)/6$ with $H=1$, $q_0$ with $H\cong C_2$, $(q_0-1)/2$ with $H\cong C_3$, and two with $H\cong S_3$ (one collinear, one non-collinear). Subtracting the all-fixed rational triples ($R_0/m$ per orbit) from $|\operatorname{Fix}_{\mathcal O}(\sigma)|$, the $\sigma$-cycles on subfield orbits total
\[
\frac{q_0-1}{2}\cdot\frac{2R_0}{3}
   +2\cdot\frac{2R_0}{6}
 =\frac{R_0(q_0+1)}{3},
\]
of which $R_0/3$ are collinear and $R_0q_0/3$ are non-collinear. Hence the number of non-collinear $\sigma$-cycles on free orbits is
\[
N-\frac{R_0q_0}{3}
 =\frac{q_0^5(q_0-1)(q_0^3+1)}{3}.
\]
On a free orbit ($H=1$) the centralizer $C_S(\sigma)$ acts freely, and the fixed triples are one $C_S(\sigma)$-orbit of length $R_0$; hence the number of free cycle-type orbits is
\[
\frac{1}{R_0}\cdot\frac{q_0^5(q_0-1)(q_0^3+1)}{3}
 =\frac{q_0^2}{3}.
\]
Together with the $(q_0+3)^2/6$ subfield orbits this gives
\[
a=\frac{(q_0+3)^2}{6}+\frac{q_0^2}{3}
 =\frac{q_0^2+2q_0+3}{2},
\]
as claimed.
\end{proof}

\begin{theorem}[Field-automorphism extensions of $\Ree(q)$ of cubic degree]\label{thm:ree-ext}
Let $S={}^2G_2(q)$, $q=3^{2m+1}>3$, act on the Ree unital, and let $G=\Ree(q)\cdot 3$ with $\langle\sigma\rangle=G/S$, $\sigma$ of order $3$;  put $q_0=3^{(2m+1)/3}$.  The field automorphism $\sigma$ fixes exactly the $q_0^3+1$ rational points of the subfield $\F_{q_0}$ pointwise and acts on the remaining points in cycles of length $3$.  A $\sigma$-invariant three-subset $S$-orbit contains either a three-subset of rational points \textup{(``subfield type'')} or a subset $\{x,\sigma x,\sigma^2x\}$ carried by one $3$-cycle \textup{(``cycle type'')}, and the $G$-orbits on three-subsets are precisely the $\sigma$-closures of the $S$-orbits \textup{(Lemma \ref{lem:norm-orbit-fusion})}.  Hence
\[
\sigma_3(G)=a+\frac{\sigma_3(S)-a}{3},\qquad
\sigma_3(S)=\frac{(q+3)^2}{6},
\]
where the number $a$ of $\sigma$-fixed $S$-orbits is
\[
a=\frac{q_0^2+2q_0+3}{2}
\qquad\textup{(Lemma \ref{lem:ree-ext-cycles})}.
\]
For $q=27$, $f'=3$ \textup{($q_0=3$)}:  $a=9$, consisting of $6$ subfield-type orbits \textup{(the $6$ orbits of ${}^2G_2(3)\cong\PpGL(2,8)$ on $28$ points, lifted)} and $3$ cycle-type orbits (all with trivial socle stabilizer), so
\[
\sigma_3(\Ree(27){:}3)=9+\frac{150-9}{3}=56,
\]
and the $56$ orbits of $\Ree(27){:}3$ have setwise stabilizer orders distributed as $1^{35},2^8,3^7,6^3,9^1,18^2$ and pointwise stabilizer orders as $1^{50},3^5,6^1$ \textup{(all computed independently in Chapter \ref{sec:verify})}.
\end{theorem}
\begin{proof}
The description of the fixed points of $\sigma$ is classical (the cycle type $1^{28}3^{6552}$ on the $19684$ points of the Atlas representation of $\Ree(27){:}3$ is verified in Chapter \ref{sec:verify}). A three-subset fixed setwise by $\sigma$ is a union of $\sigma$-cycles of length at most $3$, i.e.\ three fixed points or one $3$-cycle; hence a $\sigma$-fixed $S$-orbit contains such a triple, and conversely. By Lemma \ref{lem:norm-orbit-fusion} the $G$-orbits are the orbits of $\langle\sigma\rangle$ on the $S$-orbits: $a$ fixed orbits and $(\sigma_3(S)-a)/3$ orbits of size $3$, giving the displayed formula. The fixed-orbit number $a=(q_0^2+2q_0+3)/2$ is the content of Lemma \ref{lem:ree-ext-cycles}. For $q=27$, $q_0=3$, the computation of Chapter \ref{sec:verify} constructs $\Ree(27){:}3$ from the Atlas generators, reads off $\sigma$ (order $3$, $28$ fixed points), and computes the $A_{\{u,v\}}$-orbits ($137$ suborbits) and their fusion in $A$: $56$ orbits with the stated stabilizer distributions and the split $a=9=6+3$, in agreement with Lemma \ref{lem:ree-ext-cycles}.
\end{proof}

\begin{proposition}[Galois descent and the fixed-orbit count for the Ree socle]
\label{prop:ree-descent}
Let $S={}^2G_2(q)$, $q=3^{2m+1}>3$, act $2$-transitively on the Ree unital $\Pcal$, let $\sigma$ be a field automorphism of order $f'\mid 2m+1$, and let $G=S\rtimes\langle\sigma\rangle\le\Aut(S)$, so that $G/S\cong C_{f'}$.  Put $q_0=3^{(2m+1)/f'}$.
\begin{enumerate}
\item[(i)] Every $\sigma$-invariant $S$-orbit on $\binom{\Pcal}{3}$ contains a $\sigma$-fixed triple.  Equivalently, the $\sigma$-fixed $S$-orbits are precisely those meeting $\Fix_{\Pcal}(\sigma)=\Pcal(q_0)$, together with, when $3\mid f'$, the cycle-type orbits of the cubic subquotient (Remark \ref{rem:ree-descent-status}).
\item[(ii)] The number $A(f')$ of $\sigma$-fixed $S$-orbits is
\[
A(f')=
\begin{cases}
\dfrac{(q_0+3)^2}{6}=\sigma_3(\Ree(q_0)),&3
\nmid f',\\[6pt]
\dfrac{q_0^2+2q_0+3}{2},&3\mid f'.
\end{cases}
\]
\end{enumerate}
The assertion (i) is unconditional for every $f'\mid 2m+1$ by the proof below \textup{(Lemma \ref{lem:ree-ext-cycles} covers $f'=3$;  the two arguments there and here agree)}.
\end{proposition}
\begin{proof}
For $f'=1$ there is nothing to prove.  For $f'=3$ the assertions are \textup{Lemma \ref{lem:ree-ext-cycles}}.  Assume $f'\ge5$.  We use three elementary facts.  Write $\sigma=F|_S$ for a Steinberg endomorphism $F$ of the connected algebraic group $\mathbf G$ with $S=\mathbf G^{F^{f'}}$ \textup{(the field quotient is cyclic, so we may choose $F$ to generate the same cyclic group as $\sigma$;  for the twisted Ree root datum this is allowed by the standard form of $F$)}.

\emph{(1)  A semilinear element of the correct order is conjugate to $\sigma$.}  Let $x\in S\sigma$ with $x^{f'}=1$, and write $x=sF$ with $s\in S$.  By Lang's theorem \cite{lang1956,steinberg1968} there is $u\in\mathbf G$ with $s=u^{-1}F(u)$;  then $x=u^{-1}Fu$ and
\[
x^{f'}=(u^{-1}Fu)^{f'}=u^{-1}F^{f'}(u)\,F^{f'},
\]
so $x^{f'}=1$ forces $F^{f'}(u)=u$, i.e.\ $u\in\mathbf G^{F^{f'}}=S$.  Hence $x=u^{-1}Fu=u^{-1}\sigma u$ with $u\in S$.

\emph{(2)  Every $\sigma$-invariant $S$-orbit contains a $\sigma$-fixed triple.}  The collinear orbit is unique and contains a triple of the fixed subfield $\F_{q_0}$ (a subfield line carries $q_0+1\ge4$ points), so it is $\sigma$-invariant and contains a $\sigma$-fixed triple.  Let $O$ be a non-collinear $\sigma$-invariant $S$-orbit and choose $T\in O$.  Since $O=\sigma(O)=T^S$, there is $s\in S$ with $\sigma(T)=s(T)$;  put $\delta=s^{-1}\sigma\in S\sigma$.  Then $\delta(T)=T$, and the image of $\delta$ in $G/S\cong C_{f'}$ is $\sigma S
\ne S$, so $\delta
\ne1$.  Let $m\in\{1,2,3\}$ be the order of the permutation induced by $\delta$ on $T$.  If $m\mid f'$, then $\delta^{f'}\in S$ fixes $T$ pointwise, so $\delta^{f'}\in S_{(T)}=1$ (Theorem \ref{thm:ree-structure}:  the pointwise three-point stabilizer of the Ree socle is trivial), and $\delta$ is an $f'$-th power-free element of order dividing $f'$ with $\delta^{f'}=1$;  by (1) $\delta$ is $S$-conjugate to $\sigma$, whence $\sigma(u(T))=u(T)$ for some $u\in S$:  the orbit contains a $\sigma$-fixed triple.  If $\gcd(m,f')=1$ (which covers $m=2$ since $f'$ is odd, and $m=3$ exactly when $3
\nmid f'$), choose $k\equiv1\pmod{f'}$, $k\equiv0\pmod m$;  then $y=\delta^k\in S\sigma$ fixes $T$ pointwise, so $y^{f'}\in S_{(T)}=1$ and, by (1), $y$ is $S$-conjugate to $\sigma$, giving again a $\sigma$-fixed triple in $O$.  The only remaining case is $m=3$ with $3\mid f'$, which is the case $m\mid f'$ treated first:  $\delta$ is $S$-conjugate to $\sigma$, and $u(T)$ is a $\sigma$-fixed triple of cycle type $\{x,\sigma x,\sigma^2x\}$.  Thus (i) holds in all cases.

\emph{(3)  Rigidity of subfield triples.}  Let $\tau$ be a field automorphism of odd order $d$ and let $T,U$ be triples with all points fixed by $\tau$;  if $s(T)=U$ for $s\in S$, then $s\in C_S(\tau)$.  Indeed, put $c=s^{-1}\tau(s)$;  since $\tau(U)=U$ and $s(T)=U$, the element $c$ fixes $T$ pointwise, so $c\in S_{(T)}$.  If $T$ is non-collinear, $S_{(T)}=1$ and $c=1$.  If $T$ is collinear, $S_{(T)}\cong C_2$ is normalized by $\tau$, which acts trivially on it (odd order), and $1=c\,\tau(c)\cdots\tau^{d-1}(c)=c^{d}$, so $c=1$ again.  Hence $\tau(s)=s$.

\emph{(4)  The fixed-orbit count.}  If $3
\nmid f'$, the non-rational $\sigma$-cycles have length $f'\ge5$, so a $\sigma$-fixed triple consists of fixed points and lies in $\Pcal(q_0)$.  By (3) with $\tau=\sigma$, subfield triples are $S(q_0)$-equivalent exactly when they are $S$-equivalent, so by (2) the $\sigma$-fixed $S$-orbits are precisely the subfield orbits, and $A(f')=\sigma_3(\Ree(q_0))=(q_0+3)^2/6$.  If $3\mid f'$, then $\sigma^3$ fixes every $\sigma$-fixed triple pointwise \textup{(the induced permutation on the three points has order $1$ or $3$)}, so every $\sigma$-fixed triple lies in $\Pcal(q_0^3)$;  by (3) with $\tau=\sigma^3$ \textup{(order $f'/3$, odd)} the $S(q)$-equivalence on such triples coincides with the $S(q_0^3)$-equivalence, and by (2) the $\sigma$-fixed $S$-orbits are in bijection with the $\sigma$-fixed $S(q_0^3)$-orbits.  The latter are counted by the cubic case $f'=3$ of \textup{Lemma \ref{lem:ree-ext-cycles}} with base field $\F_{q_0^3}$:  $A(q_0^3,3)=(q_0^2+2q_0+3)/2$.  This proves (ii).
\end{proof}

\begin{remark}[Status of the Ree descent]
\label{rem:ree-descent-status}
The Galois-descent assertion (i) and the closed form (ii) of \textup{Proposition \ref{prop:ree-descent}} are proved unconditionally for every $f'\mid 2m+1$, without any hypothesis.  The cycle-type orbits mentioned in (i) are exactly the triples $\{x,\sigma x,\sigma^2x\}$ with $x\in\Pcal(q_0^3)$ not in $\Pcal(q_0)$ \textup{(they can occur only when $3\mid f'$;  for $f'=3$ they are the $3$ orbits giving $a=9=6+3$ and $\sigma_3(\Ree(27){:}3)=56$, verified in Chapter \ref{sec:verify})}.  Their number is absorbed in the closed form (ii):  for $3\mid f'$ every $\sigma$-fixed triple lies in $\Pcal(q_0^3)$, the $S(q)$-equivalence on $\Pcal(q_0^3)$ coincides with the $S(q_0^3)$-equivalence \textup{(rigidity step (3) of the proof of Proposition \ref{prop:ree-descent}, applied to $\tau=\sigma^3$)}, and (ii) counts the $\sigma$-fixed $S(q_0^3)$-orbits by the cubic closed form $A(q_0^3,3)=(q_0^2+2q_0+3)/2$.  In particular the previously open instances $f'=9,27$ are now closed:  $\sigma_3(\Ree(3^9)\cdot9)=7\,176\,734$ (degree $q^3+1=3^{27}+1$;  $\Out(\Ree(3^9))\cong C_9$, so $f'=27$ does not occur for $q=3^9$;  the $f'=27$ case occurs only for $3^{27}\mid q$, i.e.\ $f'=27\mid 2m+1$).
\end{remark}

\begin{proposition}[Extension formula for every field quotient]\label{prop:ree-ext-noncubic}
Let $S={}^2G_2(q)$, $q=3^{2m+1}>3$, act on the Ree unital, let $f'\mid 2m+1$ with $f'\ge5$, and let $G=\Ree(q)\cdot f'$ with $\langle\sigma\rangle=G/S$, $\sigma$ of order $f'$.  By \textup{Proposition \ref{prop:ree-descent}} the $\sigma$-fixed $S$-orbits are counted by $A(f')$ of (ii), and by the orbit-fusion principle the $G$-orbits on three-subsets are the orbits of $\langle\sigma\rangle$ on the $S$-orbits.  Burnside's lemma applied to this finite quotient action gives the closed form, for every $f'\mid 2m+1$,
\[
\sigma_3(G)=\frac1{f'}\sum_{d\mid f'}\varphi(d)\,A\!\left(3^{(2m+1)/d},d\right),
\qquad
A(Q,d)=
\begin{cases}
\dfrac{(Q+3)^2}{6},&3
\nmid d,\\[6pt]
\dfrac{Q^2+2Q+3}{2},&3\mid d,
\end{cases}
\]
where $d$ is the order of $\sigma^j$ in the sum over $j=0,\ldots,f'-1$ ($\varphi(d)$ elements of order $d$;  a field automorphism of order $d$ fixes exactly the subfield $\F_{3^{(2m+1)/d}}$ and, by (ii), fixes $A(3^{(2m+1)/d},d)$ of the $S$-orbits).  For example
\[
\sigma_3(\Ree(27){:}3)=\tfrac13\bigl(\sigma_3(S)+2A(3,3)\bigr)=\tfrac13(150+2\cdot9)=56,
\]
\[
\sigma_3(\Ree(243)\cdot5)=\tfrac15\bigl(\sigma_3(S)+4A(3,5)\bigr)=\tfrac15(10086+4\cdot6)=2022,
\]
\[
\begin{aligned}
\sigma_3(\Ree(3^9)\cdot9)&=\tfrac19\bigl(\sigma_3(S)+2A(27,3)+6A(3,9)\bigr)\\
&=\tfrac19(64\,589\,766+2\cdot393+6\cdot9)=7\,176\,734
\end{aligned}
\]
(degrees $19684,14348908,3^{27}+1$ respectively).  The first two values agree with the computations of \textup{Chapter \ref{sec:verify}}.
\end{proposition}

\begin{remark}[Field extensions of non-cubic degree]\label{rem:ree-field}
For a field automorphism $\sigma$ of order $f'\mid 2m+1$ acting on the Ree unital, the Galois-descent assertion (i) and the fixed-orbit count (ii) of \textup{Proposition \ref{prop:ree-descent}} are proved unconditionally for every $f'$:  \textup{Lemma \ref{lem:ree-ext-cycles}} covers $f'=3$ ($a=(q_0^2+2q_0+3)/2$, and $\sigma_3(\Ree(27){:}3)=56$ is verified in \textup{Chapter \ref{sec:verify}, Table \ref{tab:verify-ree}}), and the proof of \textup{Proposition \ref{prop:ree-descent}} covers every $f'\ge5$ by the Lang-conjugation step (1), the fixed-triple step (2) and the rigidity step (3) there, without any intermediate-subfield hypothesis.  Hence \textup{Proposition \ref{prop:ree-ext-noncubic}} is unconditional for every $f'\ge5$, e.g.\ $\sigma_3(\Ree(243)\cdot5)=2022$ and $\sigma_3(\Ree(3^9)\cdot9)=7\,176\,734$;  the previously open instances $f'=9,27$ are closed by the cubic reduction of \textup{Remark \ref{rem:ree-descent-status}} ($\Out(\Ree(3^9))\cong C_9$, so $f'=27$ requires $27\mid 2m+1$).
\end{remark}

By the general orbit-fusion lemma \textup{(Lemma \ref{lem:norm-orbit-fusion}, Chapter 2)} the extension problem separates into two independent tasks, and no new orbit can occur beyond fusion of the socle orbits.

\begin{proposition}[Ree extension:  orbit reduction]
\label{prop:ree-ext-reduction}
Let $S={}^2G_2(q)\trianglelefteq G\le\Aut(S)$, $q=3^{2m+1}>3$, with $S$ acting on the point set $\Pcal$ and $\Omega=\binom{\Pcal}{3}$.  If the $S$-orbits on $\Omega$ are $\Delta_1,\ldots,\Delta_r$, then the $G$-orbits on $\Omega$ are exactly the sets
\[
\Delta_{\mathcal C}=\bigcup_{\Delta_i\in\mathcal C}\Delta_i
\]
as $\mathcal C$ runs over the orbits of $G/S$ on $\{\Delta_1,\ldots,\Delta_r\}$;  in particular $|\Delta_{\mathcal C}|=\sum_{\Delta_i\in\mathcal C}|\Delta_i|$.
\end{proposition}
\begin{proof}
This is \textup{Lemma \ref{lem:norm-orbit-fusion}} applied to the action of $G$ on $\binom{\Pcal}{3}$.
\end{proof}

The stabilizer orders for the extension can be recovered from the quotient action as follows.  For a three-subset $X$ with $\Delta=X^S$,
\[
|G_\Delta|=|S|\,\bigl|(G/S)_\Delta\bigr|,
\tag{R.1}
\]
where $G_\Delta=\{g\in G:\Delta^g=\Delta\}$;  for a fused orbit $\mathcal C=\Delta^{G/S}$,
\[
|G_X|=\frac{|G|}{|\Delta_{\mathcal C}|}=\frac{|G|}{\displaystyle\sum_{\Gamma\in\mathcal C}|\Gamma|},
\tag{R.2}
\]
and the design parameter is
\[
\lambda_{\mathcal C}=\frac{6|\Delta_{\mathcal C}|}{v(v-1)}=\frac{6|G|}{|G_X|\,v(v-1)}.
\tag{R.3}
\]
For every group $G\le\Sym(\Pcal)$ the canonical exact sequence
\[
1\longrightarrow G_{(X)}\longrightarrow G_X\longrightarrow\overline G_X\longrightarrow1
\tag{R.4}
\]
holds with $\overline G_X\le\Sym(X)\cong S_3$, giving
\[
|G_X|=|G_{(X)}|\,|\overline G_X|.
\tag{R.5}
\]
No splitting of (R.4) is assumed:  $G_X$ must not be replaced by $G_{(X)}\rtimes\overline G_X$ without constructing a complement.  Moreover, for an extension $S\le G\le\Aut(S)$, $S_{(X)}\le G_{(X)}$, and $G_{(X)}/S_{(X)}$ is the subgroup of $G/S$ represented by automorphisms fixing the three points of $X$ individually;  the equality $G_{(X)}=S_{(X)}$ requires a separate proof and cannot be inferred merely from the fact that $G/S$ is a field-automorphism group.  Likewise the induced image may enlarge:  $\overline S_X=S_X/S_{(X)}\le G_X/G_{(X)}=\overline G_X$, and the inclusion need not be an equality.

\begin{theorem}[Ree extension:  reduction to socle orbits]
\label{thm:ree-ext-reduction}
For ${}^2G_2(q)\le G\le\Aut({}^2G_2(q))$, $q=3^{2m+1}>3$, the classification of the $G$-orbits on three-subsets, together with all triple stabilizer orders and design parameters, is equivalent to:
\begin{enumerate}
\item the classification of the three-subset orbits of the Ree socle ${}^2G_2(q)$;  and
\item the induced action of $G/{}^2G_2(q)$ on the finite set of socle orbits.
\end{enumerate}
In particular, the action of a Ree almost simple overgroup on three-subsets reduces to the action of its outer quotient on the three-subset orbits of the Ree socle.  This reduction is the basis of \textup{Theorem \ref{thm:ree-ext}}:  once the number $a$ of $\sigma$-fixed $S$-orbits is known, every extension quantity follows from (R.1)--(R.5).  The fixed-orbit number is unconditional for every $f'\mid 2m+1$:  $a=(q_0^2+2q_0+3)/2$ for $3\mid f'$ \textup{(Lemma \ref{lem:ree-ext-cycles} for $f'=3$;  cubic reduction of Proposition \ref{prop:ree-descent}(ii) and Remark \ref{rem:ree-descent-status} for $f'\ge5$)}, and $a=(q_0+3)^2/6=\sigma_3(\Ree(q_0))$ for $3
\nmid f'$ \textup{(Proposition \ref{prop:ree-descent}(ii)}.  Thus for $G=\Ree(27){:}3$ one has $a=9$ and $\sigma_3(G)=56$ with the full stabilizer distribution $1^{35},2^8,3^7,6^3,9^1,18^2$;  the Ree extension formula is closed for every field quotient, including the previously open $f'=9,27$ \textup{(see Remark \ref{rem:ree-descent-status} for $\sigma_3(\Ree(3^9)\cdot9)=7\,176\,734$)}.
\end{theorem}
\begin{proof}
The orbit statement is \textup{Proposition \ref{prop:ree-ext-reduction}};  the stabilizer and design-parameter formulas are (R.1)--(R.5) together with the orbit--stabilizer theorem and \eqref{eq:lambda}.
\end{proof}

% ============================================================
\section{The alternating groups}
\label{sec:alt}
% ============================================================

Let $\Pcal=\{1,\dots,n\}$, $n\ge5$, and $S=\Alt(n)$ (or $G=\Sym(n)$).

\begin{theorem}
\label{thm:alt}
For $n\ge5$ and $\Ocal$ a $3$-subset of $\{1,\dots,n\}$:
\begin{align*}
S_{\Ocal} &= \Alt(n-3)\rtimes S_3, & |S_{\Ocal}| &= 3\,(n-3)!,\\
\Sym(n)_{\Ocal} &= \Sym(n-3)\rtimes S_3, & |\Sym(n)_{\Ocal}| &= 6\,(n-3)!,
\end{align*}
with pointwise stabilizers $\Alt(n-3)$ and $\Sym(n-3)$ respectively, and image $S_3$ in both cases.  For $n=4$:  $S=\Alt(4)$ on $4$ points, $S_{\Ocal}\cong C_3$ \textup{(}$|S_{\Ocal}|=3$\textup{)};  $\Sym(4)_{\Ocal}=S_3$.
\end{theorem}
\begin{proof}
For $\Ocal=\{x,y,z\}$, the pointwise stabilizer in $\Alt(n)$ is $\Alt(n-3)$ (order $(n-3)!/2$), and in $\Sym(n)$ is $\Sym(n-3)$ (order $(n-3)!$).  The image on $\Ocal$:  for every $\sigma\in S_3$ there is an even permutation of $\{1,\dots,n\}$ inducing $\sigma$ on $\Ocal$ and fixing the other points (for $n\ge5$:  use $\sigma$ on $\Ocal$ together with a transposition of two of the remaining $n-3$ points to adjust the sign);  hence the image is $S_3$ and $S_{\Ocal}=\Alt(n-3)\rtimes S_3$ with $|S_{\Ocal}|=6\cdot(n-3)!/2=3(n-3)!$.  The $n=4$ case is immediate ($\Alt(4)$ acts as the tetrahedral group on $4$ points;  the stabilizer of a $3$-subset is the stabilizer of its complement, of order $3$).
\end{proof}

\begin{corollary}
The triple system on $\{1,\dots,n\}$ with blocks all $3$-subsets is the complete $2$-$(n,3,n-2)$ design with full automorphism group $\Sym(n)$;  for $\Alt(n)$ the same block set is invariant, with the same parameters $\lambda=n-2$.
\end{corollary}

% ============================================================
\section{The sporadic groups}
\label{sec:spor}
% ============================================================

\subsection{The Mathieu groups}\label{sec:mathieu}

For the Mathieu groups in their natural actions the analysis is forced by transitivity:

\begin{theorem}
\label{thm:mathieu}
Let $G$ be one of $M_{11},M_{12},M_{22},M_{23},M_{24}$ \textup{(}degrees $11,12,22,23,24$\textup{)}, or the extensions $M_{12}.2$, $M_{22}.2$.  Then each of the actions used here is at least $3$-transitive, all $3$-subsets form one orbit, and for $\Ocal\in\binom{\Pcal}{3}$:
the exact sequence
\[
1\longrightarrow G_{(\Ocal)}\longrightarrow G_{\Ocal}\longrightarrow S_3\longrightarrow1
\]
holds \textup{(3-transitivity realizes every $\sigma\in S_3$ as an induced permutation;  a semidirect-product splitting is not asserted)} with
\[
|G_{\Ocal}|=\frac{|G|}{\binom{v}{3}},\qquad
|G_{(\Ocal)}|=\frac{|G|}{v(v-1)(v-2)}.
\]
Concretely \textup{(}$|M_{11}|=7920,\ |M_{12}|=95040,\ |M_{22}|=443520,\ |M_{23}|=10200960,\ |M_{24}|=244823040$\textup{)}:
\begin{center}
\begin{tabular}{ccccc}
\toprule
$G$ & $v$ & $|G_{\Ocal}|$ & $|G_{(\Ocal)}|$ & $\overline{G}_{\Ocal}$ \\
\midrule
$M_{11}$ \textup{(}deg 11\textup{)} & 11 & 48 & 8 & $S_3$ \\
$M_{11}$ \textup{(}deg 12\textup{)} & 12 & 36 & 6 & $S_3$ \\
$M_{12}$ & 12 & 432 & 72 & $S_3$ \\
$M_{12}.2$ & 12 & 864 & 144 & $S_3$ \\
$M_{22}$ & 22 & 288 & 48 & $S_3$ \\
$M_{22}.2$ & 22 & 576 & 96 & $S_3$ \\
$M_{23}$ & 23 & 5760 & 960 & $S_3$ \\
$M_{24}$ & 24 & 120960 & 20160 & $S_3$ \\
\bottomrule
\end{tabular}
\end{center}
\end{theorem}
\begin{proof}
The transitivity degrees of the Mathieu actions are classical \textup{(cf.\ \cite{curtis1976,cameron1999})}:  the actions used here are at least $3$-transitive:  $M_{11}$ (deg 11) is $4$-transitive, $M_{11}$ (deg 12) is $3$-transitive, $M_{12}$ is $5$-transitive, $M_{22}$ is $3$-transitive, $M_{23}$ is $4$-transitive, $M_{24}$ is $5$-transitive, and the extensions $M_{12}.2$, $M_{22}.2$ act $3$-transitively on the same $12$ resp.\ $22$ points \textup{(an overgroup of a $3$-transitive group is $3$-transitive in the same action;  the degrees and actions are those of the natural Mathieu representations, checked against the GAP rows of \textup{Table \ref{tab:verify}}}\textup{)}.  Hence each is transitive on ordered triples and therefore on $3$-subsets, and realizes every $\sigma\in S_3$ on the fixed set $\Ocal$ (by $3$-transitivity applied to the ordered triple $(\sigma(x),\sigma(y),\sigma(z))$).  The orders follow from the exact sequence \eqref{eq:ses};  the displayed sequence $1\to G_{(\Ocal)}\to G_{\Ocal}\to S_3\to1$ is exact, and a semidirect-product splitting is not asserted.
\end{proof}

\begin{remark}
The kernel of $G_{\Ocal}$ on a $3$-subset is nontrivial for the Mathieu groups:  e.g.\ $|M_{11}|/(11\cdot10\cdot9)=8$ for degree $11$.  The table above is verified by direct computation in \textup{Chapter \ref{sec:verify}}.
\end{remark}

\subsection{The \texorpdfstring{$2$}{2}-transitive sporadic groups of lower transitivity}

The sporadic simple groups with a $2$-transitive (but not $3$-transitive) action are $HS$ (degree $176$) and $Co_3$ (degree $276$).  (The groups $\Suz$, $\mathrm{Fi}_{22}$ and $J_1$ are not $2$-transitive in any action:  $|\Suz|/(1782\cdot1781)$, $|\mathrm{Fi}_{22}|/(14080\cdot14079)$ and $|J_1|/(266\cdot265)$ are not integers, cf.\ \textup{Theorem \ref{thm:2trans}(2)}(h).)  For each, the point stabilizer is a maximal subgroup as follows (Atlas \cite{conway1985,wilson1996}):
\begin{table}[htbp]
\centering
\begin{tabular}{llll}
\toprule
$G$ & $v$ & $G_{\alpha}$ & $|G_{\alpha}|$ \\
\midrule
$HS$ & 176 & $U_3(5):2$ & 252000 \\
$Co_3$ & 276 & $McL:2$ & 1796256000 \\
\bottomrule
\end{tabular}
\caption{Point stabilizers of $HS$ and $Co_3$ in their $2$-transitive actions.}\label{tab:sporadic-stab}
\end{table}

(Here $|HS|=44352000$;  the degree-$176$ point stabilizer is $U_3(5):2$ of order $252000$ \cite{wilson1996}.  $|Co_3|=495766656000$ and the degree-$276$ point stabilizer is $McL:2$ of order $2|McL|=1796256000$.)

\begin{theorem}[Framework for $HS$ and $Co_3$]
\label{thm:spor-fw}
Let $G\in\{HS,Co_3\}$ in its $2$-transitive action, and $\Ocal$ a $3$-subset.  Then
\begin{enumerate}
\item $G$ is $2$-transitive:  $|G_{\alpha,\beta}|=|G|/(v(v-1))$ equals $1440$ \textup{(}for $HS$\textup{)} and $6531840$ \textup{(}for $Co_3$\textup{)} \textup{(here $G_{\alpha,\beta}$ is the ordered-pair stabilizer}\textup{)}.
\item Fix an unordered pair $P=\{\alpha,\beta\}$ with setwise stabilizer $G_P=G_{\{\alpha,\beta\}}$.  By \textup{Proposition \ref{prop:pair-orbits}} the map $\gamma^{G_P}\mapsto\{\alpha,\beta,\gamma\}^{G}$ is a well-defined surjection from the $G_P$-orbits on $\Pcal\setminus P$ onto the $G$-orbits on $\binom{\Pcal}{3}$;  for $HS$ and $Co_3$ the suborbit computation of \textup{Chapter \ref{sec:verify}} shows that it is a bijection \textup{(}3 resp.\ 2 orbits\textup{)}.  For $\Ocal=\{\alpha,\beta,\gamma\}$ with $\gamma\in\Pcal\setminus P$, writing $G_{P,\gamma}=G_P\cap G_{\gamma}$ and $C_2(P)=\{\mathrm{id},(\alpha\,\beta)\}\le\Sym(\Ocal)$,
\[
|G_{\Ocal}|=\frac{|G_P|}{|\gamma^{G_P}|}\cdot\frac{|\overline{G}_{\Ocal}|}{|\overline{G}_{\Ocal}\cap C_2(P)|},
\qquad
|G_P|=|G_{\alpha,\beta}|\cdot[G_P:G_{\alpha,\beta}],
\]
where the index $[G_P:G_{\alpha,\beta}]$ is $1$ or $2$, and equals $2$ precisely when some element of $G$ swaps $\alpha,\beta$ \textup{(}the case for $HS$ and $Co_3$\textup{)}.
\end{enumerate}
\end{theorem}

\begin{remark}[The sporadic two-point stabilizers]
\label{rem:spor-correction}
The two-point stabilizers of $HS$ and $Co_3$ have orders $1440$ and $6531840$, respectively \textup{(}Atlas data, \cite{conway1985,wilson1996}\textup{)};  in particular $G_{\Ocal}$ is not the trivial group.  The analysis uses the framework of \textup{Theorem \ref{thm:spor-fw}}:  it requires the suborbit structure of $G_{\{\alpha,\beta\}}$, which is a finite computation carried out in GAP and recorded in \textup{Table \ref{tab:verify}}.  For $HS$ the two-point stabilizer has order $1440$, and it acts on the remaining $174$ points with suborbits of sizes $12,72,90$;  the three-point setwise stabilizers are $G_{\Ocal}$ of orders $720$ \textup{(}kernel $120$\textup{)}, $120$ \textup{(}kernel $20$\textup{)} and $96$ \textup{(}kernel $16$\textup{)}, all with image $S_3$, giving $\lambda=12,72,90$;  with $|G_{\{\alpha,\beta\}}|=2880$ and $|\gamma^{G_P}|=12,72,90$ these orders match the formula of \textup{Theorem \ref{thm:spor-fw}(2)}:  $|G_{\Ocal}|=|G_P|/|\gamma^{G_P}|\cdot|\overline{G}_{\Ocal}|/|\overline{G}_{\Ocal}\cap C_2(P)|=240\cdot6/2=720$, $40\cdot6/2=120$, $32\cdot6/2=96$.  For $Co_3$ the corresponding data are determined in \textup{Proposition \ref{prop:co3}} below \textup{(finite computation, \textup{Chapter \ref{sec:verify}}\textup{)}}:  the two-point stabilizer has order $6531840$ \textup{(ordered)} and acts on the remaining $274$ points with suborbits of sizes $162,112$, and the three-point setwise stabilizers have orders $241920,349920$ with pointwise stabilizers $40320,58320$, both with image $S_3$, giving $\lambda=162,112$.  The values above are obtained from the explicit suborbit computations recorded here and in \textup{Table \ref{tab:verify}}:  for $HS$ the occurring setwise stabilizer orders $720,120,96$ (all with image $S_3$) give $\lambda=12,72,90$, and for $Co_3$ the orders $241920,349920$ (image $S_3$) give $\lambda=162,112$.  The orbit numbers $r=3$ \textup{(}for $HS$\textup{)} and $r=2$ \textup{(}for $Co_3$\textup{)} are confirmed independently of the enumeration by the class-function count of \textup{Remark \ref{rem:char}} \textup{(}\textup{Chapter \ref{sec:verify}}\textup{)}.
\end{remark}

\begin{remark}
For $M_{12}.2$ and $M_{22}.2$ the results are in \textup{Theorem \ref{thm:mathieu}}.  The overgroup $HS.2$ acts $2$-transitively on the same $176$ points, and its three-point stabilizer data are determined in \textup{Proposition \ref{prop:hs2}} below.
\end{remark}

\begin{proposition}[The overgroup $HS.2$ on 176 points]
\label{prop:hs2}
Let $G=HS.2=\Aut(HS)$ act $2$-transitively on the $176$ points of the Higman--Sims graph.  The orbits of $G$ on $\binom{\Pcal}{3}$ coincide with the three orbits of $HS$;  the three-point setwise stabilizers have orders
\[
|G_{\Ocal}|=1440,\ 240,\ 192
\]
with $\lambda=12,72,90$ respectively, the pointwise stabilizers have orders $240,40,32$, and the induced image in $S_3$ is $S_3$ in each case.
\end{proposition}
\begin{proof}
The block counts of the three $HS$-orbits are $|HS|/|HS_{\Ocal}|=44352000/720=61600$, $44352000/120=369600$ and $44352000/96=462000$ \textup{(Remark \ref{rem:spor-correction})}, with $\lambda=12,72,90$ pairwise distinct.  Every $G$-orbit on triples is a union of $HS$-orbits, and $\lambda$ is a design parameter of the underlying block set, so orbits with distinct $\lambda$ cannot be fused by $G$;  hence $G$ has exactly the same three orbits on triples.  The orbit--stabilizer theorem then gives $|G_{\Ocal}|=|G|/\mathrm{blocks}=88704000/61600=1440$, $88704000/369600=240$, $88704000/462000=192$.  The induced image contains the $HS$-image $S_3$ and is contained in $S_3$, hence equals $S_3$, and the pointwise stabilizer orders are $1440/6=240$, $240/6=40$, $192/6=32$.  The arithmetic input is verified in GAP \textup{(Chapter \ref{sec:verify};  supplementary script \texttt{hs2\_orbits.g})}:  the three $HS$-orbits are computed directly with stabilizers $720,120,96$, each block count divides $|G|$ with quotient $1440,240,192$, and no nontrivial union of block counts divides $|G|$;  for the $\lambda=12$ and $\lambda=72$ orbits the normalizer of the \textup{(}center-free\textup{)} setwise stabilizer inside $\Aut(HS)$ is computed to have orders $1440$ and $240$, twice the $HS$-values, consistent with the orbit--stabilizer arithmetic.  The statement is a finite result for the degree-$176$ action;  no family claim is involved \textup{(Remark \ref{rem:open})}.
\end{proof}
% ============================================================
\begin{proposition}[The sporadic group $Co_3$ on $276$ points]
\label{prop:co3}
Let $G=Co_3$ act $2$-transitively on the $276$ points \textup{(}point stabilizer $McL{:}2$, \cite{conway1985,wilson1996}\textup{)}.  Then $G$ has exactly two orbits on $\binom{\Pcal}{3}$, with
\[
\begin{array}{c|ccccc}
\text{orbit} & \lambda & |G_{\Ocal}| & |G_{(\Ocal)}| & \overline{G}_{\Ocal} & \#\text{blocks}\\
\hline
1 & 162 & 241920 & 40320 & S_3 & 2049300\\
2 & 112 & 349920 & 58320 & S_3 & 1416800
\end{array}
\]
and $\binom{276}{3}=3466100=2049300+1416800$.
\end{proposition}
\begin{proof}
Take the degree-$276$ permutation representation of $Co_3$ from the Atlas \textup{(\texttt{Co3G1-p276aB0}, \cite{conway1985,wilson1996})}.  GAP \textup{(Chapter \ref{sec:verify};  supplementary script \texttt{co3\_orbits.g})} verifies $|G|=495766656000$, transitivity $2$, and the setwise two-point stabilizer order $|G_{\{\alpha,\beta\}}|=13063680=2|G|/(276\cdot275)$;  its orbits on the remaining $274$ points have sizes $162$ and $112$.  The two representatives $\{1,2,3\}$ and $\{1,2,4\}$ give distinct $G$-orbits on triples \textup{(the bijection of \textup{Proposition \ref{prop:pair-orbits}} holds here, as for $HS$\textup{)}}.  For the two orbits the setwise stabilizers have orders $241920$ and $349920$, the pointwise stabilizers $40320$ and $58320$, and the induced image is $S_3$ in both cases;  the orbit--stabilizer theorem gives $2049300=495766656000/241920$ and $1416800=495766656000/349920$ blocks, summing to $\binom{276}{3}$, and $\lambda=6|G|/(|G_{\Ocal}|\,276\cdot275)=162,112$ is integral.  The orbit count $r=2$ is confirmed independently by the class-function count of \textup{Remark \ref{rem:char}}:  $\sum_{g\in G}n(g)=991533312000=2|G|$ \textup{(Chapter \ref{sec:verify})}.  The statement is a finite result for the degree-$276$ action;  no family claim is involved \textup{(Remark \ref{rem:open})}.
\end{proof}
% ============================================================
\section{The almost simple type in general}
\label{sec:as}
% ============================================================

In this section we give the systematic treatment of $2$-transitive groups with simple socle $S$ and $S\lhd G\le\Aut(S)$.

\begin{lemma}[Extension lemma]
\label{lem:ext}
Let $S\lhd G\le\Aut(S)$, both $2$-transitive on $\Pcal$ \textup{(}$|\Pcal|\ge4$\textup{)}, $\Ocal$ a $3$-subset.
\begin{enumerate}
\item $S_{\Ocal}\le G_{\Ocal}$ and $S_{(\Ocal)}\le G_{(\Ocal)}$.
\item The image $\overline{G}_{\Ocal}$ contains $\overline{S}_{\Ocal}$ and satisfies
$\overline{G}_{\Ocal}/\overline{S}_{\Ocal}\cong G_{\Ocal}/(S_{\Ocal}\cdot G_{(\Ocal)})$;  in particular
$|G_{\Ocal}|\le |S_{\Ocal}|\cdot|\Out(S)|$.
\item If $G_{\Ocal}=G_{(\Ocal)}\rtimes\overline{G}_{\Ocal}$ and the outer automorphism group acts on the three-point orbits, then the $G$-orbits on $\binom{\Pcal}{3}$ are unions of $S$-orbits, with at most $|\Out(S)|$ orbits of each $S$-type.
\end{enumerate}
\end{lemma}
\begin{proof}
(1) Clear.  (2) Identifying $G$ with its image in $\Aut(S)$ (faithful, since $S$ is $2$-transitive and simple), we have $G\cap\Inn(S)=S$ and hence $G/S\le\Out(S)$.  The kernel of $G_{\Ocal}\to G/S$ is $G_{\Ocal}\cap S=S_{\Ocal}$, so $|G_{\Ocal}/S_{\Ocal}|\le|\Out(S)|$ and the bound follows.  The isomorphism $\overline{G}_{\Ocal}/\overline{S}_{\Ocal}\cong G_{\Ocal}/(S_{\Ocal}\cdot G_{(\Ocal)})$ follows from \eqref{eq:ses} and the third isomorphism theorem, since $S_{\Ocal}\lhd G_{\Ocal}$ and $G_{(\Ocal)}\lhd G_{\Ocal}$.
(3) follows from the orbit-stabilizer theorem and the fact that $G= S\cdot(\text{outer part})$.
\end{proof}

\begin{proposition}[Outer extensions:  general reduction]
\label{prop:outer}
Let $S\lhd G$ be a $2$-transitive extension acting on $\Pcal$, and let $X\in\binom{\Pcal}{3}$.  Then $S_X\le G_X$ and the quotient $G_X/S_X$ embeds into $G/S$.  Thus outer automorphisms can only enlarge the stabilizer and can fuse $S$-orbits, but the exact kernel, image, and orbit fusion must be computed for the particular extension.
\end{proposition}

\begin{proof}
The inclusion $S_X\le G_X$ is immediate.  The map $G_X\longrightarrow G/S$ has kernel $G_X\cap S=S_X$, so the first isomorphism theorem gives an embedding $G_X/S_X\hookrightarrow G/S$.  If $X^S$ is an $S$-orbit, then the $G$-orbit of $X$ is a union of $S$-orbits because $S\lhd G$.  Whether several $S$-orbits fuse and whether an outer element acts trivially or nontrivially on $X$ is not determined by the abstract quotient alone.
\end{proof}

\begin{theorem}[Summary of the almost simple analysis]
\label{thm:as-summary}
The almost simple overgroups $G\supsetneq S$ treated in this monograph, together with the status of their three-point analysis, are summarized in \textup{Table \ref{tab:as-summary}}.
\end{theorem}

\begin{table}[htbp]
\centering
\caption{Almost simple overgroups and their effect on the three-point stabilizer $G_{\Ocal}$.}\label{tab:as-summary}
\footnotesize
\renewcommand{\arraystretch}{1.0}
\setlength{\tabcolsep}{4pt}
\begin{tabular}{@{}>{\raggedright\arraybackslash}p{2.1cm}>{\raggedright\arraybackslash}p{2.7cm}>{\raggedright\arraybackslash}p{5.8cm}>{\raggedright\arraybackslash}p{1.2cm}@{}}
\toprule
$S$ & $G$ & three-point stabilizer $G_{\Ocal}$ & section \\
\midrule
$\PSL(2,q)$ & $\PGL(2,q)$ & $S_3$ \textup{(}all $\Ocal$\textup{)};  single orbit & \ref{sec:psl2} \\
$\PSL(2,q)$ & $\PpGL(2,q)$ & $S_3$ with kernel $C_{f'}$ \textup{(}twisted elements, all $\Ocal$\textup{)};  single orbit, $|G_{\Ocal}|=6f'$ & \ref{sec:psl2} \\
$\PSL(2,q)$ & $\PSL(2,q)\cdot C_{f'}$ & image $S_3$ \textup{(}$q$ even or $q\equiv1\pmod4$\textup{)} resp.\ $C_3$ \textup{(}$q\equiv3\pmod4$\textup{)};  kernel $C_{f'}$;  $|G_{\Ocal}|=6f'$ \textup{(}resp.\ $3f'$\textup{)} & \ref{sec:psl2} \\
$\PSU(3,q)$ & $\PGU(3,q)$ & collinear:  $\Z_{q+1}\rtimes S_3$ (proved);  non-collinear:  $S_3$, $C_3$, trivial as in \textup{Thm.\ \ref{thm:main-psu}} (proved in the $d{=}1$ rows and in the $d{=}3$ rows, the odd-$d{=}3$ case in Prop.\ \ref{prop:psu-odd} and the even-$d{=}3$ case in Prop.\ \ref{prop:psu-even};  the small-$q$ stabilizer refinement is computed for $q\le49$);  diagonal fusion proved for every $q$:  $\sigma_3(\PGU(3,q))=\frac{q+3}{2}$ ($q$ odd) or $\frac{q+2}{2}$ ($q$ even) when $d=3$ \textup{(Prop.\ \ref{prop:psu-diag})} & \ref{sec:psu3} \\
$\PSU(3,q)$ & $\PSU(3,q).f'$ & field part:  the conjugation $\sigma$ ($x\mapsto x^q$) is proved not to fuse for every $q\ge3$ \textup{(\ref{prop:psu-field-nofusion}:  $\sigma_3(\langle\PGU(3,q),\sigma\rangle)=\sigma_3(\PGU(3,q))$)};  the full field part $C_{2f}$ fuses certain orbits ($\sigma_3(\PpGU(3,q))=3,4,8,7$ for $q=8,9,25,27$;  Remark~\ref{rem:psu-field}) & \ref{sec:psu3} \\
$\Sz(q)$ & $\Sz(q).f'$ & closed for every field quotient $C_{f'}\le C_{2m+1}$:  $\sigma_3(\Sz(q)\cdot f')$ by the closed form of \textup{Thm.\ \ref{thm:suzuki-field-final}}, stabilizers from \textup{Lemma \ref{lem:suzuki-bluher-count}};  $\Sz(8)\cdot3$ and $\Sz(32){:}5$ computed in \textup{Table \ref{tab:verify}} and \textup{Chapter \ref{sec:verify}} & \ref{sec:sz} \\
$\Ree(q)$ & $\Ree(q).f'$ & socle orbit structure closed for all $q>3$ \textup{(Thm.\ \ref{thm:ree-structure})};  extension closed for every field quotient $f'\mid 2m+1$:  $a=(q_0^2+2q_0+3)/2$ for $3\mid f'$ \textup{(Lem.\ \ref{lem:ree-ext-cycles}, Prop.\ \ref{prop:ree-descent}(ii))}, $a=(q_0+3)^2/6$ for $3
\nmid f'$ \textup{(Prop.\ \ref{prop:ree-descent}(ii))};  $\sigma_3(G)=\frac1{f'}\sum_{d\mid f'}\varphi(d)\,A(3^{(2m+1)/d},d)$ \textup{(Prop.\ \ref{prop:ree-ext-noncubic})};  $\sigma_3(\Ree(27){:}3)=56$ with stabilizers $1^{35},2^8,3^7,6^3,9^1,18^2$;  $\sigma_3(\Ree(3^9)\cdot9)=7\,176\,734$ & \ref{sec:ree} \\
$\Alt(n)$ & $\Sym(n)$ & $S_3$ with kernel $\Sym(n-3)$ & \ref{sec:alt} \\
$M_{12}$ & $M_{12}.2$ & $S_3$, $|G_{\Ocal}|=864$ & \ref{sec:spor} \\
$M_{22}$ & $M_{22}.2$ & $S_3$, $|G_{\Ocal}|=576$ & \ref{sec:spor} \\
$HS$ & $HS.2$ & $|G_{\Ocal}|=1440,240,192$;  kernel $240,40,32$;  image $S_3$ \textup{(Prop.\ \ref{prop:hs2})} & \ref{sec:spor} \\
\bottomrule
\end{tabular}
\end{table}

The main mechanism for every family is the same:  the outer automorphisms either (a) fuse two $S$-orbits on $\binom{\Pcal}{3}$ (as $\PGL(2,q)$ fuses the square/non-square classes of $\PSL(2,q)$ for $q\equiv1\pmod4$), or (b) leave the $S$-orbit structure but contribute to the kernel of the action on the three points:  the kernel element attached to an outer automorphism $\sigma$ is the \emph{twisted} element $h_\sigma\circ\sigma$ fixing the three points pointwise (\textup{Theorems \ref{thm:pgl}}(2) and \ref{thm:psl2-field}).  Case (a) is exactly what happens for $\PGL(2,q)\ge\PSL(2,q)$;  for the diagonal part of $\Out$ on the other rows the mechanism is (b), and for the field parts of $\PSU(3,q)$ and $\Ree(q)$ the kernel/image behaviour is analysed in \textup{Remarks \ref{rem:psu-field} and \ref{rem:ree-field}} respectively.

% ============================================================
\section{The affine type}
\label{sec:affine}
% ============================================================

Let $V=\F_p^d$ (so $|V|=p^d$).  The point set is $\Pcal=V$ with $v=p^d$, and $G=V\rtimes G_0$ acts by $(v,g_0): x\mapsto g_0(x)+v$, where $G_0\le\GL(d,p)$ is transitive on $V\setminus\{0\}$.
The affine $2$-transitive classification used throughout is that of Hering \cite{hering1985};  we also refer to the modern surveys \cite{cameron1999,dixon1996} for the consistency of the family list below with the classical literature on sharply $2$-transitive groups and on the classical affine groups.

\subsection{Reduction to the three linear equations}

Since $G$ is $2$-transitive, every $3$-subset is of the form $\{x,\,x+u,\,x+w\}$ with $x\in V$ and $u,w\in V\setminus\{0\}$, $u
\ne w$;  by translating we reduce to
\[
\Ocal=\{0,u,w\}.
\]

\begin{lemma}[The three-point stabilizer in affine groups]
\label{lem:affine}
Let $G=V\rtimes G_0$ as above, $\Ocal=\{0,u,w\}$.  An element $(v,g_0)\in G$ lies in $G_{\Ocal}$ iff $g_0$ and $v$ satisfy one of the following:
\begin{enumerate}
\item[\textup{\textup{(}T0\textup{)}}] \emph{fixed point $0$, swap $u\leftrightarrow w$}:  $v=0$, $g_0\in G_0$, $g_0(u)=w$, $g_0(w)=u$  \textup{(}a transposition fixing $0$;  if additionally $g_0(u)=u$ and $g_0(w)=w$, then $g_0\in G_{(\Ocal)}$\textup{)};
\item[\textup{\textup{(}T1\textup{)}}] \emph{fixed point $u$}:  $v=w$, $g_0(w)=-w$, $g_0(u)=u-w$;
\item[\textup{\textup{(}T2\textup{)}}] \emph{fixed point $w$}:  $v=u$, $g_0(u)=-u$, $g_0(w)=w-u$;
\item[\textup{\textup{(}C\textup{)}}] \emph{$3$-cycle} $(0\mapsto u\mapsto w\mapsto0)$:  $v=u$, $g_0(u)=w-u$, $g_0(w)=-u$;
\item[\textup{\textup{(}I\textup{)}}] identity:  $v=0$, $g_0=\mathrm{id}$ \textup{(}fixes all three points\textup{)}.
\end{enumerate}
In particular:
\begin{enumerate}
\item The pointwise stabilizer is $G_{(\Ocal)}=\{g_0\in G_0: g_0(u)=u,\ g_0(w)=w\}$.
\item A transposition fixing $0$ exists iff $G_0$ contains an element swapping $u$ and $w$;  a transposition fixing $u$ exists iff $G_0$ contains an element with $g_0(u)=u-w,\ g_0(w)=-w$;  a $3$-cycle exists iff $G_0$ contains an element with $g_0(u)=w-u,\ g_0(w)=-u$.
\item If $g_0$ satisfies \textup{(}C\textup{)}, then the affine element $(u,g_0)$ permutes $\{0,u,w\}$ as the $3$-cycle $(0\,u\,w)$;  in particular $g_0^3$ acts as the identity on the plane $\langle u,w\rangle$.
\end{enumerate}
\end{lemma}
\begin{proof}
$(v,g_0)\in G_{\Ocal}$ iff $\{v,\ g_0(u)+v,\ g_0(w)+v\}=\{0,u,w\}$.  The point $v$ must be one of $0,u,w$ (as $g_0$ is invertible, $v\ne g_0(u)+v$ etc.).  The four cases are obtained by matching the three images to the three points in all six ways;  (C) is the case $v=u$ with $g_0(u)+u=w$, $g_0(w)+u=0$, i.e.\ $g_0(u)=w-u$, $g_0(w)=-u$.  The affine element $(u,g_0)$ sends $0\mapsto u$, $u\mapsto g_0(u)+u=w$ and $w\mapsto g_0(w)+u=0$, so it permutes $\{0,u,w\}$ as the $3$-cycle $(0\,u\,w)$;  on the plane $\pi=\langle u,w\rangle$ the matrix of $g_0$ in the basis $(u,w)$ is $\begin{pmatrix}-1&-1\\1&0\end{pmatrix}$, whose cube is the identity.  (The action of $g_0$ on a complement of $\pi$ is not constrained by (C), so $g_0$ may have order a multiple of $3$;  only the induced image on $\Ocal$ matters.)
\end{proof}

\begin{remark}
\textup{Lemma \ref{lem:affine}} is the precise ``structure and solution'' of the three-point stabilizer announced in the Introduction:  computing $G_{\Ocal}$ for the affine type is exactly solving the linear system \textup{(}T0\textup{)}--\textup{(}C\textup{)} inside $G_0$, together with the pointwise stabilizer \textup{(}1\textup{)}.  The solutions depend only on the $G_0$-orbit structure of pairs $\{u,w\}$ of nonzero vectors.
\end{remark}

\begin{remark}[The six permutations realized on a triple]
\label{rem:affine-six}
Write the affine element as $g(x)=Ax+b$, $A\in G_0$, $b\in V$.  Since $g(0)=b\in X=\{0,u,w\}$, the image of $G_X$ in $\Sym(X)$ is realized by the following six cases \textup{(}each row lists the induced permutation, the translation part $b$, and the images $A(u),A(w)$\textup{)}:
\[
\begin{array}{c|c|c|c}
\text{permutation on }X & b & A(u) & A(w)\\ \hline
\mathrm{id} & 0 & u & w\\
(u\ w) & 0 & w & u\\
(0\ u) & u & -u & w-u\\
(0\ w) & w & u-w & -w\\
(0\ u\ w) & u & w-u & -u\\
(0\ w\ u) & w & -w & u-w
\end{array}
\]
For instance, $(0\ u\ w)$ requires $g(0)=u$, $g(u)=w$, $g(w)=0$, whence $b=u$, $A(u)=w-u$, $A(w)=-u$;  the other rows are obtained in the same way.  The pointwise stabilizer is therefore $G_{(X)}=\{A\in G_0: A(u)=u,\ A(w)=w\}$, embedded by $x\mapsto Ax$ (translation part zero), as in \textup{Lemma \ref{lem:affine}(1)}.  The two $3$-cycles are inverses of each other;  only one of them is an independent criterion, but listing both is convenient for checking computations and signs.
\end{remark}

\subsection{The sharply \texorpdfstring{$2$}{2}-transitive and \texorpdfstring{$\GammaL(1,q)$}{$\Gamma$L(1,q)} families}

\begin{theorem}[$\AGL(1,p^d)$]
\label{thm:agl1}
Let $G=\AGL(1,p^d)=V\rtimes\F_{p^d}^{\times}$ \textup{(}$V=\F_{p^d}$, affine maps $t\mapsto at+b$\textup{)}, sharply $2$-transitive of degree $p^d$, $p^d\ge4$, and $\Ocal=\{0,u,w\}$.
\begin{enumerate}
\item $G_{(\Ocal)}=1$ for all $\Ocal$, and $G_{\Ocal}\cong\overline{G}_{\Ocal}\le S_3$.
\item A transposition \textup{(}fixing $0$ or $u$ or $w$\textup{)} exists iff $p$ is odd and one of the three points is the midpoint of the other two:
\[
\exists\ \text{transposition} \iff p\ \text{odd and } 2m=x+y\ \text{for }\{x,y,m\}=\Ocal.
\]
In that case $\overline{G}_{\Ocal}=C_2$ \textup{(}unless a $3$-cycle also exists, giving $S_3$\textup{)}.
\item A $3$-cycle exists iff either
\[
\begin{gathered}
p
\ne3,\quad 3\mid p^d-1,\quad w=(a+1)u,\\
\text{where }a\in\F_{p^d}^{\times}\ \text{is a primitive cube root }(a^2+a+1=0,\ a
\ne1),
\end{gathered}
\]
or
\[
p=3\ \text{and } \Ocal=\{x,x+u,x+2u\}\ \text{is an arithmetic progression}.
\]
\item Consequently \textup{(}with $q=p^d$\textup{)}:
  \begin{enumerate}
  \item $p=2$:  no transpositions exist \textup{(}$p$ even\textup{)};  a $3$-cycle exists iff $3\mid q-1$, i.e.\ iff $d$ is even \textup{(}since $3\mid2^d-1\iff d$ even\textup{)}.  Hence $G_{\Ocal}=1$ for all $\Ocal$ when $d$ is odd \textup{(}$q=8,32,128,\ldots$;  the smallest admissible example is $q=8$, since $q=2$ is excluded by $q\ge4$\textup{)}, while for $d$ even \textup{(}$q=4,16,\ldots$\textup{)} the triples of cube-root type have $G_{\Ocal}\cong C_3$ and the remaining triples have $G_{\Ocal}=1$;  in particular $\AGL(1,4)\cong A_4$ has $G_{\Ocal}\cong C_3$ for every $\Ocal$ \textup{(}single orbit\textup{)};
  \item $p=3$:  $G_{\Ocal}=S_3$ iff $\Ocal$ is an arithmetic progression;  otherwise $G_{\Ocal}=1$;
  \item $p
\ne2,3$:  $G_{\Ocal}=C_2$ \textup{(}midpoint type\textup{)}, $C_3$ \textup{(}cube-root type\textup{)}, or $1$ \textup{(}generic\textup{)}, according to \textup{(}2\textup{)},\textup{(}3\textup{)};  in particular $G_{\Ocal}$ is never $S_3$ for $p
\ne3$.
  \end{enumerate}
\end{enumerate}
\end{theorem}
\begin{proof}
(1) sharply $2$-transitive:  the stabilizer of two points is trivial, so the pointwise stabilizer of three points is trivial.
(2) Solve (T0):  $g_0(t)=at+b$ with $\{b, au+b, aw+b\}=\{0,u,w\}$ and fixed point structure;  the transposition fixing $0$:  $b=0$, $a u=w$, $a w=u$:  $a^2=1$, $a
\ne1$ (as $u
\ne w$), so $a=-1$ (requires $p$ odd) and $w=-u$:  midpoint configuration $u+w=0$.  The transposition fixing $u$:  $b=w$, $a w + w = 0$:  $a=-1$ (p odd), $a u + w = u$:  $-u+w=u$:  $w=2u$:  so $\Ocal=\{0,u,2u\}$:  $u$ is the midpoint of $0,2u$.  The transposition fixing $w$ is analogous and again reduces to the midpoint condition.
(3) (C):  $g_0(t)=at+b$, $b=u$, $a u+u=w$, $a w+u=0$:  $a u = w-u$, $a w = -u$:  substitute:  $a(w-u)=a w - a u = -u - (w-u) = -w$:  and directly $a^2 u = a(w-u) = -w$, $a^2 w = a(-u) = -a u = -(w-u)=u-w$:  so $a^2$ acts on $(u,w)$ as $u\mapsto -w$, $w\mapsto u-w$:  then $a^3$:  $a^3 u = a(-w) = -a w = u$:  so $a^3=1$ on the plane (as $a^3$ fixes $u$ and $w$):  $a^3=1$ in $\Fq^{\times}$, $a
\ne1$:  $a$ a primitive cube root, $a^2+a+1=0$, which requires $p
\ne3$ and $3\mid q-1$.  Conversely if $a^2+a+1=0$, define $g_0(t)=at+u$;  check:  $g_0(0)=u$, $g_0(u)=a u+u=w$ (as $w=(a+1)u$), $g_0(w)=a w+u=a(a+1)u+u=(a^2+a)u+u=(-1)u+u=0$:  a $3$-cycle.  For $p=3$:  the translation $t\mapsto t+u$ is a $3$-cycle on $\{0,u,2u\}$ (arithmetic progression), and $t\mapsto t+u$ with $a=1$:  $g_0(0)=u$, $g_0(u)=2u$, $g_0(2u)=0$:  a $3$-cycle;  conversely a $3$-cycle with $a=1$:  $u+w$-type:  $g_0(0)=u$, $g_0(u)=u+u=2u=w$, $g_0(w)=w+u$ must be $0$:  $w=-u=2u$:  so $\Ocal=\{0,u,2u\}$.
(4) follows from (2),(3) and the field arithmetic:  $3\mid p^d-1\iff p\equiv1\pmod3$ or ($p
\ne3$ and $d$ even with $p\equiv2\pmod3$);  for $p=2$:  $3\mid2^d-1\iff d$ even.
\end{proof}

\begin{remark}
In characteristic $3$ the polynomial $a^2+a+1$ equals $(a-1)^2$, so the equation $a^2+a+1=0$ does not describe primitive cube roots there;  the $3$-cycle of \textup{Theorem \ref{thm:agl1}(3)} in characteristic $3$ is the translation $t\mapsto t+u$ ($a=1$) on the arithmetic progression $\{x,x+u,x+2u\}$, not a cube-root case.  These two cases should not be conflated.
\end{remark}

\begin{theorem}[$\AaL(1,p^d)$]
\label{thm:agamL1}
Let $G=\AaL(1,p^d)=V\rtimes(\Fq^{\times}\rtimes C_f)$ \textup{(}affine maps $t\mapsto a\,\sigma(t)+b$, $\sigma\in C_f$ a subgroup of the field automorphisms of $\Fq$ \textup{(}$1\le f\le d$,\ $f\mid d$;\  $C_f$ is the subgroup of $\mathrm{Gal}(\F_q/\F_p)\cong C_d$ of order $f$\textup{)};  the product $\Fq^{\times}\rtimes C_f$ is a semidirect product because a field automorphism acts on the multipliers, $\sigma\,m_a\,\sigma^{-1}=m_{\sigma(a)}$\textup{)}, $q=p^d$, $\Ocal=\{0,u,w\}$.
\begin{enumerate}
\item The pointwise stabilizer consists of the \emph{twisted semilinear} maps
\[
x\longmapsto\frac{u}{\sigma(u)}\,\sigma(x),
\qquad \sigma\in C_f,\ \sigma(w/u)=w/u,
\]
one for each field automorphism fixing the ratio $r=w/u$ \textup{(}the multiplicative factor $a=u/\sigma(u)$ is then forced by $a\sigma(u)=u$\textup{)};  in particular
\[
|G_{(\Ocal)}|=\bigl|C_f(w/u)\bigr|=\frac{f}{[\,E(w/u):E\,]},
\qquad E=\mathrm{Fix}(C_f)=\F_{p^{d/f}},
\]
the order of the stabilizer of $r=w/u$ under $C_f$ \textup{(for the full group $C_f=\mathrm{Gal}(\Fq/\F_p)$, i.e.\ $f=d$, this reduces to $|G_{(\Ocal)}|=d/[\F_p(w/u):\F_p]$;  for a proper subgroup $C_f$ the correct base field is the fixed field $E=\mathrm{Fix}(C_f)$)} \textup{(}a field automorphism $t\mapsto t^{p^k}$ fixes $w/u$ iff $k$ is a multiple of $[\F_p(w/u):\F_p]$\textup{)}.
\item A transposition in the induced action on $\Ocal$ exists iff $\sigma(r)\in\{1-r,\ r^{-1},\ r/(r-1)\}$ for some $\sigma\in C_f$, and a $3$-cycle in the induced action exists iff $\sigma(r)\in\{1/(1-r),\ (r-1)/r\}$ for some $\sigma\in C_f$, where $r=w/u$ and $\sigma(r)$ denotes the image of $r$ under the field automorphism $\sigma$;  in each condition $\sigma$ ranges over $C_f$.  The three transposition values correspond respectively to the transpositions fixing $0$, $u$ and $w$;  the two $3$-cycle values are inverses of one another.  Here \textup{``the induced action contains a $3$-cycle''} is the precise statement:  when the pointwise kernel is nontrivial, the group element realizing the $3$-cycle need not itself have order $3$.
\item The full stabilizer order is
\[
|G_{\Ocal}|=\sum_{\sigma\in C_f}\#\{\pi\in S_3:\ \sigma(r)=T_{\pi}(r)\},
\]
where $T_{\pi}$ ranges over the six anharmonic transformations of $r$ associated with the six permutations $\pi\in S_3$.  When computing the stabilizer one must count with multiplicity according to the six permutations and must not deduplicate the six values first:  for the midpoint type, the cube-root type and the characteristic-$3$ parameters, the values $T_{\pi}(r)$ can coincide.
\end{enumerate}
\end{theorem}
\begin{proof}
Direct from \textup{Lemma \ref{lem:affine}} and the structure of $\GammaL(1,q)$:  an affine semilinear map $x\mapsto a\sigma(x)+b$ fixes $\{0,u,w\}$ pointwise iff $b=0$, $a\sigma(u)=u$ and $a\sigma(w)=w$, i.e.\ iff $\sigma(w/u)=w/u$ and $a=u/\sigma(u)$;  by Galois theory the field automorphisms fixing $w/u$ are exactly those fixing $\F_p(w/u)$ pointwise, a cyclic group of order $f/[E(w/u):E]$ with $E=\mathrm{Fix}(C_f)=\F_{p^{d/f}}$ \textup{(reducing to $d/[\F_p(w/u):\F_p]$ for the full group $C_f=\mathrm{Gal}(\F_q/\F_p)$)}.  For part\textup{ (2)}:  an element $(a,\sigma)\in G_0$ induces a permutation $\pi$ on $\Ocal$ iff $\{a\sigma(0),\ a\sigma(u),\ a\sigma(w)\}=\{0,u,w\}$;  the condition $a\sigma(0)=0$ forces $b=0$, and solving for the multiplier $a$ shows that this holds exactly when $\sigma(r)=T_{\pi}(r)$, where $T_{\pi}$ is the anharmonic transformation attached to $\pi\in S_3$.  The six transformations are $r,\ 1-r,\ r^{-1},\ (1-r)^{-1},\ r/(r-1),\ (r-1)/r$;  the three transpositions correspond to $1-r$, $r^{-1}$, $r/(r-1)$ and the two $3$-cycles to $1/(1-r)$, $(r-1)/r$.  For part\textup{ (3)}:  for each $\sigma\in C_f$ there is at most one multiplier $a$ realizing a given $\pi$, so the count over $\sigma\in C_f$ and $\pi\in S_3$ with $\sigma(r)=T_{\pi}(r)$ is exact.
\end{proof}

\begin{remark}[Standard parametrization, orbit fusion and design parameters]
\label{rem:anharmonic}
For $\AGL(1,q)$ the triples are conveniently parametrized as $\Ocal_r=\{0,1,r\}$ with $r\in\Fq\setminus\{0,1\}$:  two triples lie in the same $\AGL(1,q)$-orbit iff their parameters lie in the same \emph{anharmonic sextuple}
\[
r,\quad 1-r,\quad r^{-1},\quad (1-r)^{-1},\quad \frac{r}{r-1},\quad \frac{r-1}{r}.
\]
The transposition case is exactly the double coincidence of the sextuple (the midpoint configuration), the $3$-cycle case is exactly $r^2-r+1=0$ \textup{(equivalent to $r=a+1$, $a^2+a+1=0$ of \textup{Theorem \ref{thm:agl1}}(3)}\textup{)}, and the image $S_3$ occurs only in the characteristic-$3$ arithmetic-progression case.  Under $\AaL(1,q)=V\rtimes(\Fq^{\times}\rtimes C_f)$ the orbit-fusion criterion is the complete one:  two triples $\Ocal_r,\Ocal_s$ lie in the same $\AaL(1,q)$-orbit iff
\[
s\in\bigl\{\sigma(r),\ 1-\sigma(r),\ \sigma(r)^{-1},\ (1-\sigma(r))^{-1},\ \tfrac{\sigma(r)}{\sigma(r)-1},\ \tfrac{\sigma(r)-1}{\sigma(r)}\bigr\}
\]
for some $\sigma\in C_f$;  equivalently, for every $\sigma\in C_f$ and every $\pi\in S_3$, an element inducing $\pi$ on $\Ocal_r$ exists iff $\sigma(r)=h_\pi(r)$, where $h_\pi$ ranges over the six anharmonic transformations.  \textup{(This replaces the informal statement that field automorphisms realize a transposition or a $3$-cycle;  it is the complete orbit-fusion criterion.)}

For the designs:  for a $2$-transitive $G$ and an orbit $\mathcal{B}=\Ocal^G$ on triples with $s=|G_{\Ocal}|$, the orbit design has $b=|G|/s$ blocks and replication parameter $\lambda=6b/(v(v-1))$ \textup{(\eqref{eq:lambda})}.  For $\AGL(1,q)$ this gives $b=q(q-1)/s$ and $\lambda=6/s$ with $s\in\{1,2,3,6\}$;  for $\AaL(1,q)=V\rtimes(\Fq^{\times}\rtimes C_f)$ it gives $b=q(q-1)f/s$ and $\lambda=6f/s$.  These parameters are automatically integral, and the blocks are distinct three-subsets \textup{(the designs are simple)}.  For the flag-transitivity claims in the applications \textup{(Chapter \ref{sec:design})}:  a design obtained from a 2-transitive group's triple orbit is automatically point-2-transitive and block-transitive;  flag-transitivity for $G$ itself requires $G_{\Ocal}$ to be transitive on $\Ocal$, equivalently the induced stabilizer $\overline{G}_{\Ocal}$ is $C_3$ or $S_3$.  One must not assert $G$-flag-transitivity for every orbit design indiscriminately.  Also, when the orbit $\Ocal^G$ is generated by applying all $g\in G$ to a representative $\Ocal$, the resulting multiset must be deduplicated \textup{(the blocks are distinct three-subsets)};  the design parameters above already count distinct blocks.
\end{remark}

\begin{example}
$\AGL(1,9)$:  $p=3$:  $G_{\Ocal}=S_3$ iff $\Ocal$ an arithmetic progression \textup{(}in $\F_9$, i.e.\ $\{x,x+u,x+2u\}$\textup{)};  otherwise $1$.  $\AaL(1,9)$:  the Frobenius $t\mapsto t^3$ fixes $\F_3$ pointwise;  the two orbits on triples have $|G_{\Ocal}|=2$ \textup{(}image $C_2$\textup{)} and $12$ \textup{(}image $S_3$, kernel of order $2$\textup{)}, respectively \textup{(}\textup{Table \ref{tab:verify}}\textup{)}.
\end{example}

\subsection{The classical linear families: \texorpdfstring{$\SL(n,q)$}{SL(n,q)}, \texorpdfstring{$\Sp(2a,q)$}{Sp(2a,q)}, \texorpdfstring{$G_2'(q)$}{G2'(q)}}

Fix $p^d=q^n$ with $V=\F_q^n$ the natural module ($n=2a$ for symplectic, $n=6$ for $G_2'$), and let $H\in\{\SL(n,q),\ \Sp(2a,q),\ G_2'(q)\}$ with $H\lhd G_0\le\Gamma$-type.  We solve \textup{Lemma \ref{lem:affine}} for non-collinear and collinear triples.

\begin{theorem}[Non-collinear triples]
\label{thm:affine-classical}
Let $\Ocal=\{0,u,w\}$ with $u,w$ linearly independent, and $G_0\supseteq H$.
\begin{enumerate}
\item \textbf{Pointwise stabilizer} $G_{(\Ocal)}=G_0{}_{u,w}$, the pointwise stabilizer of $u,w$ in $G_0$ \textup{(when $G_0$ is a proper overgroup of $H$ this may be strictly larger than $H_{u,w}$:  a field or diagonal part may contribute elements fixing $u,w$;  we record $H_{u,w}$ as the contribution of the subgroup $H$ below}\textup{)}:
  \begin{enumerate}
  \item $H=\SL(n,q)$, $n\ge3$:  choosing a basis with $u=e_1$, $w=e_2$, the elements of $H_{u,w}$ have the block form $\left(\begin{smallmatrix}I_2&A\\0&B\end{smallmatrix}\right)$ with $B\in\SL(n-2,q)$ and $A\in M_{2\times(n-2)}(\F_q)$ arbitrary \textup{(}$2(n-2)$ free parameters\textup{)};  hence $H_{u,w}\cong q^{2(n-2)}:\SL(n-2,q)$ of order $q^{2(n-2)}|\SL(n-2,q)|$.  For $n=2$:  $H_{u,w}=1$.
  \item $H=\Sp(2a,q)$:  the pointwise stabilizer depends on whether $\langle u,w\rangle$ is nondegenerate or totally isotropic.  If $\beta(u,w)
\ne 0$ \textup{(}nondegenerate plane\textup{)}, then $V=\langle u,w\rangle\perp\langle u,w\rangle^{\perp}$ and $H_{u,w}\cong\Sp(2a-2,q)$, of order $|\Sp(2a-2,q)|$ \textup{(extend by the identity on the perpendicular complement}\textup{)}.  If $\beta(u,w)=0$ \textup{(}totally isotropic plane\textup{)}, then $a\ge2$ and the pointwise stabilizer of the ordered pair $(u,w)$ has the closed form $|H_{u,w}|=q^{4a-5}|\Sp(2a-4,q)|$ \textup{(convention $\Sp(0,q)=1$;  the ordered totally isotropic pairs number $(q^{2a}-1)(q^{2a-1}-q)$, so the orbit--stabilizer formula gives $|H_{u,w}|=|\Sp(2a,q)|/\bigl((q^{2a}-1)(q^{2a-1}-q)\bigr)=q^{4a-5}|\Sp(2a-4,q)|$}\textup{)};  the induced image is $S_3$ for every $q$ \textup{(Remark \ref{rem:sp-isotropic}}\textup{)}.
  \item $H=G_2'(q)$ \textup{(}$p=2$\textup{)}:  $H_{u,w}$ is the pointwise stabilizer of $u,w$ in the natural module $\F_q^6$.  For $q=2^f\ge4$ a closed form is proved in \textup{Theorem \ref{thm:g2-classical}}:  $|H_{u,w}|=q(q^2-1)$ if $h(u,w)
\ne 0$ \textup{(}the spanned $2$-space is nondegenerate for the alternating form $h$ of $\Sp(6,q)$\textup{)}, and $|H_{u,w}|=q^5$ or $q^3$ if $h(u,w)=0$ \textup{(}the spanned $2$-space is totally isotropic, of long-root resp.\ short-root type in the $G_2$ hexagon geometry, cf.\ \cite{cooperstein1979}\textup{)}.  For $q=2$ ($2^6{:}G_2'(2)\cong2^6{:}\PSU(3,3)$, $v=64$, \textup{Table \ref{tab:verify}}\textup{)} the orbit data are $r=3$ with $\lambda=32,24,6$;  $q=2$ is exceptional and is not covered by the general $q$ statement of \textup{Theorem \ref{thm:g2-classical}}.
  \end{enumerate}
\item \textbf{$3$-cycles always exist}:  for every independent pair $(u,w)$ there is $g_0\in H$ with $g_0(u)=w-u$, $g_0(w)=-u$.  Indeed the matrix of $g_0$ in the basis $(u,w)$ is $M=\begin{pmatrix}-1 & -1\\ 1 & 0\end{pmatrix}$ of determinant $1$;  for $H=\Sp(2a,q)$ it preserves the alternating form on $\langle u,w\rangle$ and extends to $\Sp(2a,q)$ in both plane types \textup{(Remark \ref{rem:sp-isotropic})}, and for $G_2'(q)\le\Sp(6,q)$ the element exists by the same symplectic computation.
\item \textbf{Transpositions}:  a transposition fixing $0$ \textup{(}swap $u\leftrightarrow w$\textup{)} exists iff $H$ contains an element with $g_0(u)=w,\ g_0(w)=u$;  a transposition fixing $u$ exists iff $H$ contains $g_0$ with $g_0(u)=u-w,\ g_0(w)=-w$ \textup{(}matrix $\begin{pmatrix}1&0\\ -1&-1\end{pmatrix}$, det $1$\textup{)};  fixing $w$:  matrix $\begin{pmatrix}-1&-1\\0&1\end{pmatrix}$.  Concretely:
  \begin{enumerate}
  \item $H=\SL(n,q)$, $n\ge3$:  all transposition types exist for all $q$ \textup{(}the $2\times2$ blocks above have determinant $1$, and the complement is adjusted by an element of $\SL(n-2,q)$\textup{)};
  \item $H=\SL(2,q)$:  transpositions exist iff $q$ is even \textup{(}then $-1=1$ and the swap $\begin{pmatrix}0&1\\1&0\end{pmatrix}$ has determinant $-1=1$\textup{)};  for $q$ odd:  no transposition of any kind;
  \item $H=\Sp(2a,q)$, nondegenerate case:  transpositions exist iff $q$ is even \textup{(}the swap on $\langle u,w\rangle$ has determinant $-1$, and it preserves $\beta$ iff $\beta(w,u)=\beta(u,w)$, i.e.\ iff the characteristic is $2$\textup{)};  in the totally isotropic case transpositions exist for every $q$ \textup{(Remark \ref{rem:sp-isotropic}}\textup{)};
  \item $H=G_2'(q)$ \textup{(}$p=2$\textup{)}:  transpositions exist.
  \end{enumerate}
\item \textbf{Summary} \textup{(}generic non-collinear triples\textup{)}:
  {\footnotesize\setlength{\tabcolsep}{4pt}%
  \begin{center}
  \begin{tabular}{llll}
  \toprule
  $H$ & image $\overline{G}_{\Ocal}$ & kernel & $|G_{\Ocal}|$ \\
  \midrule
  $\SL(n,q)$, $n\ge3$ & $S_3$ & $q^{2(n-2)}|\SL(n-2,q)|$ & $6\,q^{2(n-2)}|\SL(n-2,q)|$ \\
  $\SL(2,q)$, $q$ even & $S_3$ & $1$ & $6$ \\
  $\SL(2,q)$, $q$ odd & $C_3$ & $1$ & $3$ \\
  $\Sp(2a,q)$, $q$ even, nondeg. & $S_3$ & $\Sp(2a-2,q)$ & $6|\Sp(2a-2,q)|$ \\
  $\Sp(2a,q)$, $q$ odd, nondeg. & $C_3$ & $\Sp(2a-2,q)$ & $3|\Sp(2a-2,q)|$ \\
  $\Sp(2a,q)$, totally isotr. & $S_3$ & $q^{4a-5}|\Sp(2a-4,q)|$ & $6q^{4a-5}|\Sp(2a-4,q)|$ \\
  $G_2'(q)$, $q$ even & $S_3$ & $q(q^2-1)$;  $q^5$;  $q^3$ & $6q(q^2-1)$;  $6q^5$;  $6q^3$ \\
  \bottomrule
  \end{tabular}
  \end{center}}
  for the symplectic rows:  the nondegenerate rows with $\beta(u,w)
\ne 0$ have kernel $\Sp(2a-2,q)$;  the totally isotropic row has kernel $q^{4a-5}|\Sp(2a-4,q)|$ \textup{(Remark \ref{rem:sp-isotropic})}.  For the $G_2'$ row the three kernel values $q(q^2-1)$, $q^5$, $q^3$ correspond respectively to the hyperbolic class ($h(u,w)
\ne 0$) and to the long-root resp.\ short-root isotropic classes ($h(u,w)=0$) of \textup{Theorem \ref{thm:g2-classical}}.  The two plane types form distinct orbits of $\Sp(2a,q)$.  For the smallest case $\ASp(4,2)=2^4{:}\Sp(4,2)\cong2^4{:}S_6$ on $16$ points the non-collinear triples fall into two orbits with $|G_{\Ocal}|=36$ \textup{(}kernel $6$, nondegenerate\textup{)} and $48$ \textup{(}kernel $8$, totally isotropic\textup{)}, both with image $S_3$, giving $\lambda=8$ and $6$ \textup{(\textup{Table \ref{tab:verify}}\textup{)}}.
\end{enumerate}
\end{theorem}
\begin{proof}
(1) Direct from the structure of point stabilizers:  for $\SL(n,q)$ the pointwise stabilizer of the two vectors $u,w$ is the group of matrices $\begin{pmatrix}I_2&*\\0&*\end{pmatrix}$ with the bottom-right block in $\SL(n-2,q)$ (order $q^{2(n-2)}|\SL(n-2,q)|$);  for $\Sp$ fixing a nondegenerate plane pointwise is exactly $\Sp$ of the perpendicular complement (order $|\Sp(2a-2,q)|$).
(2) Let $g_0$ act on $\pi=\langle u,w\rangle$ by the matrix
\[
M=\begin{pmatrix}-1&-1\\1&0\end{pmatrix},
\]
i.e.\ $g_0(u)=w-u$, $g_0(w)=-u$.  With $\Omega=\begin{pmatrix}0&1\\-1&0\end{pmatrix}$ the standard alternating form on $\pi$, $M^{\mathsf T}\Omega M=\Omega$ (direct multiplication) and $\det M=1$, so $g_0$ is an isometry of $(\pi,\beta|_\pi)$ and lies in $\SL(2,q)$.  For $H=\SL(n,q)$, $n\ge2$, we extend by the identity on a complement (for $n\ge3$;  for $n=2$, $\pi=V$).  For $H=\Sp(2a,q)$:  if $\pi$ is nondegenerate, $V=\pi\perp\pi^{\perp}$ and the identity extension on $\pi^{\perp}$ works;  if $\pi$ is totally isotropic, the hyperbolic-basis extension of \textup{Remark \ref{rem:sp-isotropic}} applies.  In both cases the extended element lies in $\Sp(2a,q)$.  By \textup{Lemma \ref{lem:affine}}(C) the affine element $(u,g_0)$ is a $3$-cycle on $\{0,u,w\}$.  Hence $3$-cycles exist on every non-collinear triple, for $H=\SL(n,q)$ ($n\ge2$), $\Sp(2a,q)$ and $G_2'(q)\le\Sp(6,q)$ alike.
(3) The transposition fixing $0$ requires $g_0(u)=w,\ g_0(w)=u$, i.e.\ the matrix $P=\begin{pmatrix}0&1\\1&0\end{pmatrix}$ on $\pi$;  $\det P=-1$, and $P^{\mathsf T}\Omega P=-\Omega$.  Hence, on a nondegenerate plane, the swap lies in $\SL$ (resp.\ $\Sp$) iff $-1=1$, i.e.\ $q$ even.  The transpositions fixing $u$ or $w$ have matrices $\begin{pmatrix}1&0\\-1&-1\end{pmatrix}$, $\begin{pmatrix}-1&-1\\0&1\end{pmatrix}$, both of determinant $-1$, so the same conclusion holds.  For $H=\SL(n,q)$ with $n\ge3$ the sign can be corrected in the complement (scale one complement basis vector by $-1$), so all transposition types exist for all $q$.  For $H=\Sp(2a,q)$ on a totally isotropic plane the swap extends by $\mathrm{diag}(P,P)$ on a hyperbolic basis \textup{(Remark \ref{rem:sp-isotropic})}, so transpositions exist there for every $q$;  the two plane types therefore behave differently.
(4) The table follows from (1)--(3) and the splitting $G_{\Ocal}=G_{(\Ocal)}\rtimes\overline{G}_{\Ocal}$ (the image is realized inside $H$).
\end{proof}

\begin{remark}
\label{rem:sp-isotropic}
For the symplectic rows the two plane types $\beta(u,w)
\ne 0$ and $\beta(u,w)=0$ form distinct $\Sp(2a,q)$-orbits on independent pairs and must be kept apart.  The extension of a linear map on $\pi=\langle u,w\rangle$ to an element of $\Sp(2a,q)$ is immediate when $\pi$ is nondegenerate \textup{(identity on $\pi^{\perp}$}\textup{)} but not when $\pi$ is totally isotropic;  in the latter case one completes $\{u,w\}$ to a hyperbolic basis $\{u,w,f_1,f_2\}$ with $\beta(u,f_1)=\beta(w,f_2)=1$ and all other pairings zero, and extends a $2\times2$ matrix $M$ on $\pi$ by $Z=(M^{\mathsf T})^{-1}$ on $\langle f_1,f_2\rangle$:  a direct multiplication gives $\left(\begin{smallmatrix}M&0\\0&Z\end{smallmatrix}\right)^{\mathsf T}\left(\begin{smallmatrix}0&I_2\\-I_2&0\end{smallmatrix}\right)\left(\begin{smallmatrix}M&0\\0&Z\end{smallmatrix}\right)=\left(\begin{smallmatrix}0&I_2\\-I_2&0\end{smallmatrix}\right)$, so the extension preserves the alternating form and has determinant $1$.  In particular the $3$-cycle matrix $M=\left(\begin{smallmatrix}-1&-1\\1&0\end{smallmatrix}\right)$ of \textup{Theorem \ref{thm:affine-classical}(2)} extends in both cases \textup{(}$\det M=1$\textup{)}, while the swap $P=\left(\begin{smallmatrix}0&1\\1&0\end{smallmatrix}\right)$ extends by $\mathrm{diag}(P,P)$, which is symplectic for every $q$:  thus, on the totally isotropic orbit, transpositions exist for all $q$ \textup{($a\ge2$}\textup{)}, whereas on the nondegenerate orbit they exist exactly for $q$ even.  For $q=2$ the isotropic values in \textup{Table \ref{tab:main}} are computed directly (\textup{e.g.\ $\ASp(4,2)$:  $\lambda=8$ and $6$}), and the parabolic kernel order is not of the form $q^{2a-1}|\Sp(2a-2,q)|$ in general.
\end{remark}

\begin{theorem}[The $G_2'$ rows:  orbits and pointwise stabilizers]
\label{thm:g2-classical}
Let $q=2^{f}\ge4$ and $H=G_2'(q)=G_2(q)$ act on the natural module $V=\F_q^6$, with $H\le\Sp(6,q)$ preserving the alternating form $h$, and let $G=V\rtimes H$ act on the $v=q^6$ affine points \textup{(the $2$-transitive affine row of \textup{Table \ref{tab:hering}}}\textup{)}.
\begin{enumerate}
\item $H$ is transitive on the nonzero vectors, and the stabilizer of a nonzero vector is the maximal parabolic subgroup $H_u=q^5{:}\SL(2,q)$ of order $|H_u|=q^6(q^2-1)$ \textup{(classical:  cf.\ \cite{kleidmanliebeck};  the same value follows from $|H|=q^6(q^6-1)(q^2-1)$ and $|V\setminus\{0\}|=q^6-1$}\textup{)}.  Consequently $H_u$ fixes every point of the line $\langle u\rangle$ \textup{(by linearity}\textup{)}.
\item The orbits of $H_u$ on $V\setminus\{0,u\}$ are the following \textup{(the transitivity is part of the classical $G_2$ hexagon geometry, cf.\ \cite{cooperstein1979};  the sizes and stabilizer orders are then determined by the orbit--stabilizer identity, and the numerical data are verified for $q=4$ in \textup{Chapter \ref{sec:verify}}}\textup{)}:
  \begin{itemize}
  \item[(a)] the $q-2$ fixed points $\{\lambda u\}$, $\lambda\in\F_q^{\times}\setminus\{1\}$ \textup{(collinear points}\textup{)};
  \item[(b)] for each $c\in\F_q^{\times}$ the set $W_c=\{w\in V:\ h(u,w)=c\}$ \textup{(so $\langle u,w\rangle$ is a nondegenerate, hyperbolic $2$-space}\textup{)}, a single orbit of size $q^5$, with pointwise stabilizer order $|H_{u,w}|=q(q^2-1)$;
  \item[(c)] two orbits contained in $u^{\perp}\setminus\{0,u\}$ \textup{(so $\langle u,w\rangle$ is totally isotropic}\textup{)}:  the long-root class of size $q(q^2-1)$ with $|H_{u,w}|=q^5$, and the short-root class of size $q^3(q^2-1)$ with $|H_{u,w}|=q^3$.
  \end{itemize}
  The size check $q^5-q=q(q^2-1)+q^3(q^2-1)$ accounts for all isotropic non-collinear vectors.
\item \textbf{The three-subset orbits of $G$ on the affine points.}  $G$ has $r=q+1+c(q)$ orbits on three-subsets, with $c(q)=1+\frac{q-4}{6}$ if $3\mid q-1$ and $c(q)=\frac{q-2}{6}$ otherwise \textup{($q=2^f$;  the first case is $f$ even, the second $f$ odd)}:
  \begin{itemize}
  \item[(a)] \textbf{collinear} $\{0,u,\lambda u\}$:  the pointwise stabilizer of every collinear triple is $H_u$ of order $q^6(q^2-1)$;  the induced image in $S_3$ contains no transposition \textup{(a swap of $u$ and $\lambda u$ would be a linear map with $g(u)=\lambda u$, $g(\lambda u)=u$, forcing $\lambda^2=1$, i.e.\ $\lambda=1$ in characteristic $2$;  hence no transposition exists on the line}\textup{)}.  The collinear $G$-orbits are the orbits of $S_3\cong\langle s,i\rangle$ on $\Fq\setminus\{0,1\}$, where $s(\lambda)=\lambda+1$ and $i(\lambda)=\lambda^{-1}$ \textup{(the translation by $u$ moves the base point, and $\{0,u,\lambda u\}=\{0,\lambda u,u\}$}\textup{)}:
    \begin{itemize}
    \item if $3\mid q-1$ \textup{(}$f$ even\textup{)}:  one orbit $\{\omega,\omega^2\}$ with $\omega$ a primitive cube root, induced image $C_3$ \textup{(realized by the scalar $g=\omega^2$ on the line together with the translation by $u$}\textup{)}, $|G_{\Ocal}|=3q^6(q^2-1)$ and $\lambda=2$;  and $\frac{q-4}{6}$ orbits with induced image $1$, $|G_{\Ocal}|=q^6(q^2-1)$ and $\lambda=6$;
    \item if $3
\nmid q-1$ \textup{(}$f$ odd\textup{)}:  $\frac{q-2}{6}$ orbits, each with induced image $1$, $|G_{\Ocal}|=q^6(q^2-1)$ and $\lambda=6$.
    \end{itemize}
    The block counts are consistent:  the collinear blocks number $\binom{q^6}{3}\cdot2/(q^6-2)=q^6(q^6-1)(q-2)/6$, which equals $|G|/\bigl(3q^6(q^2-1)\bigr)+\frac{q-4}{6}\cdot|G|/\bigl(q^6(q^2-1)\bigr)$ for $f$ even and $\frac{q-2}{6}\cdot|G|/\bigl(q^6(q^2-1)\bigr)$ for $f$ odd;  for $q=4$ this gives the single collinear orbit of size $5591040$ \textup{(Chapter \ref{sec:verify}}\textup{)}.
  \item[(b)] \textbf{hyperbolic} $\{0,u,w\}$ with $h(u,w)
\ne 0$:  $q-1$ orbits \textup{(one for each value of $h(u,w)=c$}\textup{)};  pointwise stabilizer of order $q(q^2-1)$, induced image $S_3$, $|G_{\Ocal}|=6q(q^2-1)$;  $\lambda=q^5$;
  \item[(c)] \textbf{isotropic} $\{0,u,w\}$ with $h(u,w)=0$:  two orbits \textup{(long-root and short-root}\textup{)};  pointwise stabilizers of orders $q^5$ and $q^3$, induced image $S_3$ in both cases, $|G_{\Ocal}|=6q^5$ resp.\ $6q^3$;  $\lambda=q(q^2-1)$ resp.\ $q^3(q^2-1)$.
  \end{itemize}
  \textup{(For the non-collinear orbits the affine stabilizer decomposes as $|G_{\Ocal}|=6|H_{u,w}|$:  the elements of $G_{\Ocal}$ are $(v,g)$ with $v\in\{0,u,w\}$, and for each such $v$ exactly two linear parts realize the corresponding permutation of $\{0,u,w\}$:  one sends $w$ to $u+w$ and one sends $u$ to $u+w$;  both occur, since $w$ and $u+w$ lie in the same $H_u$-orbit whenever $h(u,w)
\ne 0$ resp.\ in the isotropic class when $h(u,w)=0$, by (2).  The non-collinear count is $(q-1)+2$;  adding the collinear count $c(q)$ of the item above gives $r=q+1+c(q)$.  The orbit count $r$, the stabilizer orders and the injectivity of the suborbit correspondence are verified for $q=4$ in \textup{Chapter \ref{sec:verify}} (where $c(4)=1$ and $r=6=q+2$)}\textup{)}.
\end{enumerate}
\end{theorem}
\begin{proof}
The order $|H_u|=q^6(q^2-1)$ is the classical parabolic-stabilizer value \textup{\cite{kleidmanliebeck}}:  $H_u=Q{:}L$ with $Q$ the unipotent radical of order $q^5$ and $L\cong\SL(2,q)$ of order $q(q^2-1)$;  together with the transitivity of $H$ on $V\setminus\{0\}$ this gives (1).  The orbit structure of (2) is the classical suborbit structure of the point stabilizer of $G_2(q)$ on the $6$-dimensional module, i.e.\ the split Cayley hexagon geometry \textup{\cite{cooperstein1979}}:  the hyperbolic ($h(u,w)
\ne 0$) and the two isotropic classes (long-root and short-root $2$-spaces) are the $H_u$-orbits listed, and the transitivity on each is part of that geometry;  the fixed points $\{\lambda u\}$ are fixed by (1).  The stabilizer orders in (2) are then computed from the orbit--stabilizer identity:  $q^6(q^2-1)/q^5=q(q^2-1)$;  $q^6(q^2-1)/(q(q^2-1))=q^5$;  $q^6(q^2-1)/(q^3(q^2-1))=q^3$.  For (3), the pointwise stabilizer of a collinear triple is $H_u$ by (1):  an element fixing $0,u,\lambda u$ pointwise fixes the line $\langle u\rangle$ pointwise and lies in $H_u$, and conversely every element of $H_u$ fixes the line pointwise \textup{(by linearity, (1)}\textup{)}.  No transposition exists by the characteristic-$2$ swap obstruction.  The collinear orbits are governed by the base-point moves:  translation by $u$ sends $\{0,u,\lambda u\}$ to $\{u,0,(\lambda+1)u\}=\{0,u,(\lambda+1)u\}$, and the identity $\{0,u,\lambda u\}=\{0,\lambda u,u\}$ identifies $\lambda$ with $\lambda^{-1}$;  hence the $\lambda$-classes are the orbits of $S_3\cong\langle s,i\rangle$ on $\Fq\setminus\{0,1\}$ as stated, whose number is $1+(q-4)/6$ for $f$ even \textup{(the two-point orbit $\{\omega,\omega^2\}$ plus the $(q-4)/6$ regular orbits}\textup{)} and $(q-2)/6$ for $f$ odd.  On the $\{\omega,\omega^2\}$-orbit the $3$-cycle is realized by the affine element $(u,\omega^2\cdot{})$ \textup{(the scalar $\omega^2$ on the line, available in the Levi factor of the point stabilizer, together with the translation by $u$}\textup{)}, so $|G_{\Ocal}|=3q^6(q^2-1)$;  on every other collinear orbit the image is $1$ and $|G_{\Ocal}|=q^6(q^2-1)$.  For the non-collinear orbits the decomposition $|G_{\Ocal}|=6|H_{u,w}|$ is the enumeration of $(v,g)$ described in the theorem \textup{(the six cosets realize $S_3$;  the occurrence of both linear parts for $v=u$ and $v=w$ uses that $w$ and $u+w$ are in the same $H_u$-orbit, which is part of the orbit list of (2)}\textup{)}.  The parameters $\lambda$ follow from \eqref{eq:lambda}:  $\lambda=6|G|/\bigl(|G_{\Ocal}|\,v(v-1)\bigr)$ with $|G|=q^6|H|$, $v=q^6$:  the collinear $\{\omega,\omega^2\}$-orbit ($f$ even):  $6|H|/\bigl(3q^6(q^2-1)(q^6-1)\bigr)=2$;  the collinear image-$1$ orbits:  $6|H|/\bigl(q^6(q^2-1)(q^6-1)\bigr)=6$;  hyperbolic:  $6|H|/\bigl(6q(q^2-1)(q^6-1)\bigr)=q^5$;  isotropic:  $6|H|/\bigl(6q^5(q^6-1)\bigr)=q(q^2-1)$ and $6|H|/\bigl(6q^3(q^6-1)\bigr)=q^3(q^2-1)$, using $|H|=q^6(q^6-1)(q^2-1)$ throughout.  The numerical case $q=4$ (orbit lengths, stabilizer orders, images and $\lambda$ values $2,1024,60,960$, the collinear stabilizer $|G_{\Ocal}|=184320=3q^6(q^2-1)$ with image $C_3$ and block size $5591040$, and the injectivity of the suborbit correspondence) is verified by the GAP computation of \textup{Chapter \ref{sec:verify}} \textup{(supplementary script \texttt{g2\_full2.g})}.
\end{proof}

\begin{theorem}[Collinear triples]
\label{thm:affine-collinear}
Let $\Ocal=\{0,u,\lambda u\}$ with $\lambda\in\Fq^{\times}\setminus\{1\}$ \textup{(}collinear triple\textup{)}, $H=\SL(2,q)\le G_0$ \textup{(}the other $H$ are analogous;  $G_2'(q)$ and $\Sp(2a,q)$ with $q$ even behave like the $\SL(2,q)$-even row, cf.\ the verification table\textup{)}.
\begin{enumerate}
\item Pointwise stabilizer:  $G_{(\Ocal)}=H_u$, the vector stabilizer of $u$ in $H$;  for $H=\SL(2,q)$ this is the transvection subgroup of order $q$ \textup{(}for $G_0$ with full $\GL(2,q)$, $|G_{(\Ocal)}|=q(q-1)$\textup{)}.
\item On the line $\langle u\rangle$:
  \begin{enumerate}
  \item a transposition fixing $0$ \textup{(}swapping $u,\lambda u$\textup{)} exists iff $\lambda=-1$ and $-1$ is realized on the line:  for $H=\SL(2,q)$ the element $-I$ \textup{(}det $1$\textup{)} realizes it, so this holds for $q$ odd;  for $q$ even $\lambda=-1=1$ is excluded;
  \item a $3$-cycle exists iff $\lambda^2-\lambda+1=0$ is solvable in $\Fq$ \textup{(}i.e.\ $\lambda$ a primitive cube root of unity, $q\equiv1\pmod3$, or $\lambda=-1$ in characteristic $3$\textup{)};  the realizing element is $g_0=\mathrm{diag}(\lambda-1,\,-\lambda)$ \textup{(}det $=(\lambda-1)(-\lambda)=1$ since $\lambda(\lambda-1)=-1$\textup{)} together with the translation by $u$.
  \end{enumerate}
\item Consequently, for $H=\SL(2,q)$:
  {\footnotesize\setlength{\tabcolsep}{3pt}%
  \begin{center}
  \begin{tabular}{llll}
  \toprule
  $q$ & $|G_{\Ocal}|$ & image & condition \\
  \midrule
  $q=2$ & 6 & $S_3$ & degree $4$:  all triples are non-collinear \textup{(}see \textup{Table \ref{tab:verify}}\textup{)} \\
  $q=3^f$ \textup{(char $3$)} & $q$ & $1$ & generic $\lambda$ \\
   & $6q$ & $S_3$ & $\lambda=-1=2$ \\
  $q$ odd, $q\ge5$, $\operatorname{char}\F_q\ne3$ & $q$ & $1$ & generic $\lambda$ \\
   & $2q$ & $C_2$ & $\lambda\in\{-1,2,\tfrac12\}$ \\
   & $3q$ & $C_3$ & $\lambda^2-\lambda+1=0$ \textup{(}$q\equiv1\pmod3$\textup{)} \\
  $q$ even, $q\ge4$ & $q$ & $1$ & generic $\lambda$ \\
   & $3q$ & $C_3$ & $\lambda$ a primitive cube root \textup{(}$q\equiv1\pmod3$\textup{)} \\
  \bottomrule
  \end{tabular}
  \end{center}}
  \textup{(In characteristic $3$, $\{-1,2,\tfrac12\}=\{2\}$ and $(-1)^2-(-1)+1=0$, so at $\lambda=-1$ both the transposition and the $3$-cycle are realized and the image is $S_3$ with $|G_{\Ocal}|=6q$, whereas for generic $\lambda$ the image is trivial with $|G_{\Ocal}|=q$;  verified for $q=9$ in \textup{Chapter \ref{sec:verify}}.)}
\end{enumerate}
\end{theorem}
\begin{proof}
Direct solving of \textup{Lemma \ref{lem:affine}} on the line $\langle u\rangle$:  the conditions reduce to the displayed equations.  For the $3$-cycle:  the affine element $(u,g_0)$ with $g_0$ multiplication by $c$ on the line maps $0\mapsto u$, $u\mapsto(c+1)u$, $\lambda u\mapsto(c\lambda+1)u$;  imposing the cycle $(0\,u\,\lambda u)$ gives $c+1=\lambda$ and $c\lambda+1=0$, hence $c=\lambda-1$ and $\lambda^2-\lambda+1=0$.  The transposition fixing $0$:  $(0,g_0)$ with $g_0$ multiplication by $\lambda$ on the line, $\lambda^2=1$, $\lambda
\ne1$:  $\lambda=-1$.  The kernel is $H_u$ (elements fixing $u$ and $\lambda u$ pointwise on the line are the transvections with axis $\langle u\rangle$), of order $q$ in $\SL(2,q)$.  The table follows.
\end{proof}

\subsection{The exceptional affine cases}

The exceptional cases of \textup{Table \ref{tab:hering}} are treated by the same \textup{Lemma \ref{lem:affine}};  the results (computed and verified in \textup{Chapter \ref{sec:verify}}) are collected here.

\begin{theorem}[Exceptional affine cases]
\label{thm:affine-exc}
Let $G=V\rtimes G_0$ with $G_0$ one of the groups in \textup{Table \ref{tab:hering}}, $V=\F_p^d$.  The stabilizer orders $|G_{\Ocal}|$ and their orbit counts are as follows \textup{(}details in \textup{Table \ref{tab:verify}}\textup{)}:
{\small\setlength{\tabcolsep}{4pt}%
\begin{center}
\begin{tabular}{p{2.6cm}p{1.1cm}p{5.6cm}p{1.5cm}}
\toprule
$G_0$ & $v=p^d$ & $|G_{\Ocal}|$ values \textup{(}orbit counts\textup{)} & $\lambda$ values \\
\midrule
$\SL(2,3)\le\GL(2,5)$, $v=25$ & 25 & $1$ \textup{(}2 orbits\textup{)}, $2$ \textup{(}1 orbit\textup{)}, $3$ \textup{(}4 orbits\textup{)} & $6,\ 3,\ 2$ \\
$\SL(2,3)\le\GL(2,p)$ \textup{(nearfield)}, $p=7,23$ & $p^2$ & nearfield:  $G_{(\Ocal)}=1$, images from the explicit nearfield action & --- \\
$\SL(2,5)\le\GL(2,11)$, $v=121$ & 121 & $1$ \textup{(}16 orbits\textup{)}, $2$ \textup{(}1 orbit\textup{)}, $3$ \textup{(}10 orbits\textup{)} & $6,\ 3,\ 2$ \\
$\SL(2,5)\le\GL(2,p)$ \textup{(nearfield)}, $p=19,29,59$ & $p^2$ & nearfield case as above & --- \\
$\Alt(6)\le\GL(4,2)$, $v=16$ & 16 & $18$ \textup{(}1 orbit\textup{)}, $24$ \textup{(}1 orbit\textup{)} & $8,\ 6$ \\
$\Alt(7)\le\GL(4,2)$, $v=16$ & 16 & $72$ \textup{(}1 orbit\textup{)} & $14$ \\
$\Alt(7)\le\GL(4,2)$, $v=15$ \textup{(}nonzero vectors\textup{)} & 15 & $6$ \textup{(}1 orbit\textup{)}, $72$ \textup{(}1 orbit\textup{)} & $12,\ 1$ \\
$2.S_5\le\GL(4,3)$, $v=81$ \textup{($2.S_5$ is the point stabilizer, of order $240$, not $\SL(2,5)$ of order $120$)} & 81 & $1$ \textup{(}3 orbits\textup{)}, $3$ \textup{(}4 orbits\textup{)}, $18$ \textup{(}1 orbit\textup{)} & $18,\ 6,\ 1$ \\
$2_-^{1+4}\le\GL(4,3)$, $v=81$ & 81 & no $2$-transitive realization:  $|2_-^{1+4}|=32$ cannot be transitive on the $80$ nonzero vectors of $\F_3^4$ \textup{(}and no $2$-transitive group of degree $81$ in the GAP census has a linear part containing $2_-^{1+4}$\textup{)} & --- \\
$\PSU(3,3)\le\GL(6,2)$, $v=64$ & 64 & $18$ \textup{(}1 orbit\textup{)}, $24$ \textup{(}1 orbit\textup{)}, $96$ \textup{(}1 orbit\textup{)} & $32,\ 24,\ 6$ \\
$\SL(2,13)\le\GL(6,3)$, $v=729$ & 729 & $18$ \textup{(1 orbit)}, $3$ \textup{(4 orbits)}, $1$ \textup{(39 orbits)} \textup{(Proposition \ref{prop:aff729})} & $1,\ 6,\ 18$ \\
\bottomrule
\end{tabular}
\end{center}}
\end{theorem}

\begin{remark}
For the nearfield cases $G_0=\SL(2,p)$ embedded in $\GL(2,p)$ as the multiplicative group of a non-Desarguesian \textup{(Zassenhaus)} nearfield, the affine group is sharply $2$-transitive of degree $p^2$, so $G_{(\Ocal)}=1$ and $G_{\Ocal}\cong\overline{G}_{\Ocal}$;  the images follow the explicit nearfield action of $\SL(2,p)$ on $\F_p^2$ once the correct nearfield multiplication is used.  The classical reference is \cite{huppert1967};  the computations in \textup{Chapter \ref{sec:verify}} use the explicit $2$-transitive representations of $\mathrm{AGL}(2,5)$ with linear part $\SL(2,3)$ and of $\mathrm{AGL}(2,11)$ with linear part $\SL(2,5)$ \textup{(}Hering's construction\textup{)}, where the three-point orbits split as listed.  The $\Alt(7)\le\GL(4,2)$ row with $v=15$ gives, for the orbit with $|G_{\Ocal}|=72$, the Steiner triple system $\mathrm{STS}(15)$ \textup{(}$\lambda=1$\textup{)}:  the $35$ triples $\{a,b,c\}$ of nonzero vectors with $a+b+c=0$ in $\F_2^4$.  For the remaining row \textup{(}$v=729$\textup{)}:  the representation is realized explicitly as $G=3^6\rtimes\rho(\SL(2,13))$ acting on $V=\F_3^6$, where $\rho$ is obtained from either of the two $6$-dimensional faithful irreducible $\F_3$-modules of $\SL(2,13)$ \textup{(constructible in GAP as $\texttt{IrreducibleModules}(\SL(2,13),\F_3)$;  both choices yield the same orbit data, Proposition \ref{prop:aff729})}.  The image of $\rho$ is transitive on the $728$ nonzero vectors of $\F_3^6$, and the stabilizer of a nonzero vector has order $2184/728=3$.  The complete three-subset orbit structure of this representation is computed in Proposition \ref{prop:aff729}:  $44$ orbits with $\lambda=1,6,18$.  The result is a finite statement for the specified degree-$729$ representations;  no classification for every embedding $\rho$ or for a family is asserted \textup{(Remark \ref{rem:open})}.
\end{remark}
\begin{proposition}[Computed affine case:  $\SL(2,13)\le\GL(6,3)$]
\label{prop:aff729}
Let $\rho\colon\SL(2,13)\to\GL(6,3)$ be obtained from either of the two $6$-dimensional faithful irreducible $\F_3$-modules of $\SL(2,13)$, and let $G=3^6\rtimes\rho(\SL(2,13))$ act on the points of $V=\F_3^6$ \textup{(}$v=729$\textup{)}.  Then $|G|=3^6|\SL(2,13)|=1592136$, $G$ is $2$-transitive on $V$, and $G$ has exactly $44$ orbits on the three-point subsets of $V$:
\begin{itemize}
\item[(a)] one orbit with $\lambda=1$:  orbit size $88452$, $|G_{\Ocal}|=18$, pointwise stabilizer of order $3$ and image $S_3$;  the corresponding simple design is a $2$-$(729,3,1)$ Steiner triple system;
\item[(b)] four orbits with $\lambda=6$:  orbit size $530712$, $|G_{\Ocal}|=3$;  for three of them the pointwise stabilizer is trivial and the image is $C_3$, and for one the pointwise stabilizer has order $3$ and the image is trivial;
\item[(c)] thirty-nine orbits with $\lambda=18$:  orbit size $1592136$, $|G_{\Ocal}|=1$.
\end{itemize}
The orbit sizes sum to $\binom{729}{3}=64304604$;  the Burnside count on three-subsets gives $44$ orbits and the count on ordered triples gives $247$;  and $\lambda=6|G|/\bigl(|G_{\Ocal}|\,v(v-1)\bigr)$ \eqref{eq:lambda} holds on every orbit.
\end{proposition}
\begin{proof}
This is the output of the GAP computation of Chapter \ref{sec:verify} \textup{(supplementary script \texttt{r729\_merge.g})}:  the group is built as the affine semidirect product on the points of $\F_3^6$, the orbit representatives are deduplicated with \texttt{RepresentativeAction}, and for every orbit the orbit--stabilizer identity $|X_i^G|\,|G_{X_i}|=|G|$, the Burnside count, the sum of the orbit sizes and the ordered-triple count $247$ are verified by the program.  The statement is a finite result for the two specified degree-$729$ permutation representations;  it does not assert a classification for every embedding $\rho$ or any family statement \textup{(Remark \ref{rem:open})}.
\end{proof}

% ============================================================
\section{The classification:  proof status and parameters}
\label{sec:main}
% ============================================================

The classification is stated in \textup{Theorem \ref{thm:main}} (\textup{Section 1}).  Here we record the proof status of every family listed there, the $\lambda=1$ (Steiner) systems, and the resulting design parameters.

\begin{remark}[Proof status of the families]
\label{rem:proofstatus}
Every row of \textup{Table \ref{tab:main}} is either proved for all parameters or explicitly marked as computational or not treated in the indicated section;  the proof status of each family is stratified as proved, computed or not treated in \textup{Remark \ref{rem:open}} \textup{(the Ree field extension is unconditional for every field quotient by \textup{Proposition \ref{prop:ree-descent}} and \textup{Proposition \ref{prop:ree-ext-noncubic}}, including the previously open $f'=9,27$ \textup{(Remark \ref{rem:ree-descent-status})}, marked accordingly in the $\Ree(q)$ row of \textup{Table \ref{tab:main}};  the full field part $C_{2f}$ of $\PpGU(3,q)$ is computed for $q=8,9,25,27$, while the conjugation is proved not to fuse for every $q$ by \textup{Proposition \ref{prop:psu-field-nofusion}}, and neither affects a row of \textup{Table \ref{tab:main}})}.  The stabilizer statements for $\PSL(2,q)$ and its overgroups, $\PSL(n,q)$ \textup{(}$n\ge3$\textup{)}, $\Alt(n)$/$\Sym(n)$ and the affine families of \textup{Section \ref{sec:affine}} are proved for all parameters in the corresponding sections \textup{(for $G_2'(q)$ the complete orbit and stabilizer classification of \textup{Theorem \ref{thm:g2-classical}} for every $q=2^f\ge4$, verified for $q=4$ in \textup{Chapter \ref{sec:verify}};  the case $q=2$ is the computational row $2^6{:}\PSU(3,3)$ of \textup{Table \ref{tab:verify}}}\textup{)}.  For $\Sp(2n,2)$ the classification is complete for every $n\ge3$ \textup{(Theorem \ref{thm:sp2n-triples}, \textup{Section \ref{sec:psln}})}:  two orbits on each half with the parameters and stabilizer structure displayed there;  for $n=3$ this agrees with the computed values of \textup{Table \ref{tab:sp62}}.  For $\Sz(q)$ the circle-orbit analysis, the pointwise-triviality results, the stabilizer dichotomy of the circle class and the triviality of the generic stabilizer are proved for all $q$ on the basis of the classical input of \textup{Theorem \ref{thm:sz}};  the orbit number $\sigma_3(\Sz(q))=(q+4)/6$ is proved in \textup{Theorem \ref{thm:sz-orbits}} \textup{(Corollary \ref{cor:sz-designs})}, and the field overgroups are closed by the finite-field count of \textup{Lemma \ref{lem:suzuki-bluher-count}} and the closed form of \textup{Theorem \ref{thm:suzuki-field-final}} for every field quotient $C_{f'}\le C_{2m+1}$, with numerical verification for $q=8,32$ in \textup{Table \ref{tab:verify}}.  For $\Ree(q)$ the number of three-subset orbits is proved for every $q=3^{2m+1}$:  $\sigma_3=(q+3)^2/6$ \textup{(Theorem \ref{thm:ree-triples}, from the classical class data of \cite{bradley2014})}, and the full orbit structure (the stabilizer-order distribution $N_1=\frac{q(q-3)}6$, $N_2=q$, $N_3=\frac{q-1}2$, $N_6=2$ and the collinear/non-collinear split) is determined for every $q>3$ \textup{(Theorem \ref{thm:ree-structure})};  for the field-automorphism overgroups the extension formula $\sigma_3(G)=a+(\sigma_3(S)-a)/f'$ holds unconditionally for $f'=3$ \textup{(Theorem \ref{thm:ree-ext};  $a=(q_0^2+2q_0+3)/2$, Lemma \ref{lem:ree-ext-cycles})}, and for $f'\ge5$ by the descent argument of \textup{Proposition \ref{prop:ree-descent}} \textup{(Proposition \ref{prop:ree-ext-noncubic};  $a=(q_0+3)^2/6$ for $3
\nmid f'$, $a=(q_0^2+2q_0+3)/2$ for $3\mid f'$)};  in particular $\sigma_3(\Ree(27){:}3)=56$ with the full stabilizer distribution \textup{(Chapter \ref{sec:verify};  Theorem \ref{thm:ree-ext})} and $\sigma_3(\Ree(3^9)\cdot9)=7\,176\,734$, so the Ree extension formula is closed for every field quotient;  for $\PSU(3,q)$ the collinear case is established for all $q$ from the classical unital geometry \textup{(Lemma \ref{lem:unital}(4),(5), cf.\ \cite{mitchell1911})} together with the counting argument of \textup{Theorem \ref{thm:collinear}} \textup{(the underlying classical facts are used as input and are not reproved here)};  the non-collinear case is classified by \textup{Theorem \ref{thm:main-psu}}:  the orbit number is $\sigma_3=(q+3)/2$ for $q$ odd with $d=1$, $(5q+11)/6$ for $q$ odd with $d=3$ \textup{(Proposition \ref{prop:psu-odd})}, $1+q/2$ for $q$ even with $d=1$, and $(5q+2)/6$ for $q$ even with $d=3$ \textup{(Proposition \ref{prop:psu-even}, }$q=2^{2m+1}$\textup{)} \textup{(closed forms of (3);  $k=0$ for $q$ even and $k=d$ for $q$ odd, with the fixed orbits carrying the full $S_3$ image by Lemma \ref{lem:psu-odd-fixed} when $q$ is odd with $3\mid q+1$}\textup{)}, with the stabilizer distribution of (4) and the small-$q$ refinement of \textup{Proposition \ref{prop:psu-computed}} for $q\le49$;  the diagonal fusion under $\PGU(3,q)$ is proved for every $q$ ($\sigma_3(\PGU(3,q))=(q+3)/2$ for $q$ odd and $(q+2)/2$ for $q$ even when $d=3$, no fusion when $d=1$;  Proposition \ref{prop:psu-diag});  the field conjugation of $\PpGU(3,q)$ is proved to cause no fusion for every $q\ge3$ \textup{(Proposition \ref{prop:psu-field-nofusion})}, while the full field part $C_{2f}$ (generated by the Frobenius $\sigma_p$) fuses certain $\PGU(3,q)$-orbits ($\sigma_3(\PpGU(3,q))=3$ for $q=8$, $4$ for $q=9$, $8$ for $q=25$ and $7$ for $q=27$;  Remark~\ref{rem:psu-field});  the closed forms are verified by GAP for every $q\le27$ \textup{(Table \ref{tab:psu-general};  supplemented for the odd-$d=3$ row by $q=5,11,17,23$ in Chapter \ref{sec:verify};  the stabilizer census of \textup{Table \ref{tab:magma}} covers $q\le49$)}.  The sporadic rows and the Hering exceptions are computational \textup{(\textup{Chapter \ref{sec:verify}}\textup{)}}, and the $v=729$ row is computational \textup{(Proposition \ref{prop:aff729};  Chapter \ref{sec:verify})}.
\end{remark}

% ============================================================
\section{The \texorpdfstring{$2$}{2}-homogeneous but not \texorpdfstring{$2$}{2}-transitive groups}
\label{sec:2homog}
% ============================================================

The main classification \textup{Theorem \ref{thm:main}} concerns $2$-transitive groups.  For completeness we record the parallel analysis for the $2$-homogeneous groups that are \emph{not} $2$-transitive.  These lie just outside the main class, but the orbit--design correspondence of \textup{Proposition \ref{prop:each-orbit}} applies verbatim (it only requires transitivity on unordered pairs), so their three-point orbit data yield invariant triple systems, graphs, codes and association schemes in exactly the same way.

\begin{lemma}[Classification of the $2$-homogeneous groups]
\label{lem:2homog-class}
Let $G$ be a finite $2$-homogeneous permutation group that is not $2$-transitive.  Then $G$ is affine one-dimensional, $G\le\AaL(1,q)$ on $v=q=p^d$ points with $q$ odd, $G$ is solvable, and the point stabilizer $G_0$ has odd order.  In particular $|G|$ is odd, so $G$ contains no involution, and no almost simple example occurs.
\end{lemma}

\begin{proof}
The $2$-homogeneous groups were classified by Kantor \cite{kantor1985}; the only $2$-homogeneous groups that are not $2$-transitive are the affine one-dimensional groups $G\le\AaL(1,q)$ with $q$ odd.  If $G_0$ contained an element of order $2$, then together with $2$-homogeneity this element would invert the action on the remaining points and make $G$ $2$-transitive (a classical observation of Livingstone--Wagner \cite{livingstone1965}):  an involution fixing one point of a pair acts transitively on the other, so transitivity on unordered pairs upgrades to transitivity on ordered pairs.  Hence $|G_0|$ is odd, and $|G|=v|G_0|$ is odd because $v=q$ is odd.  Since $|G|$ is odd, $G$ is solvable \cite{feit1963}, and no almost simple (non-solvable) group is $2$-homogeneous without being $2$-transitive.
\end{proof}

The principal family, and the one with the cleanest orbit data, is the index-$2$ subgroup of $\AGL(1,q)$ obtained by keeping only the squares as multipliers.

\begin{theorem}[Orbit data of the square-subgroup affine group]
\label{thm:2homog}
Let $q\equiv3\pmod4$ be a prime power \textup{(so $q=p^d$ with $p\equiv3\pmod4$ and $d$ odd)}, let $V=(\F_q,+)$, let $M=(F_q^*)^2$ be the subgroup of squares, and let $G=V\rtimes M$ act on $V$.  Then $G$ is $2$-homogeneous but not $2$-transitive, $|G|=q(q-1)/2$ is odd, and the number of $G$-orbits on the three-point subsets of $V$ is
\[
\sigma_3(G)=
\begin{cases}
\dfrac{q+2}{3}, & q\equiv7\pmod{12},\\[4pt]
\dfrac{q-2}{3}, & q\equiv11\pmod{12},\\[4pt]
\dfrac{q}{3}, & q=3^d\ (d\ \text{odd}).
\end{cases}
\]
Every three-point orbit $X$ is a $2$-\mbox{$(q,3,\lambda_X)$} design with $\lambda_X\in\{1,3\}$, and the distribution is
\[
\begin{array}{c|c|c}
\text{$q$ range} & \#\{X:\lambda_X=1\} & \#\{X:\lambda_X=3\}\\
\hline
q\equiv7\pmod{12} & 2 & (q-4)/3\\
q\equiv11\pmod{12} & 0 & (q-2)/3\\
q=3^d\ (d\ \text{odd}) & 1 & q/3-1
\end{array}
\]
The $\lambda=1$ orbits are Steiner triple systems (their number is $2$ for $q\equiv7\pmod{12}$, e.g.\ $q=7$ gives two distinct Fano planes, $1$ for $q=3^d$, and $0$ for $q\equiv11\pmod{12}$, consistently with the existence of $\mathrm{STS}(q)$ iff $q\equiv1,3\pmod6$), with setwise stabilizer of order $3$ and induced image $C_3$;  the $\lambda=3$ orbits have trivial setwise stabilizer and image $1$.  In particular no three-point orbit has induced image $C_2$ or $S_3$, since $|G|$ is odd.  The values are verified by direct enumeration for $q=7,11,19,23,31,43,47$ and $q=3$.
\end{theorem}

\begin{proof}
\emph{Transitivity.}  The multiplier group $M$ has index $2$ in $F_q^*$ and, since $q\equiv3\pmod4$, $-1$ is not a square, so $-1
\notin M$;  hence $M$ is transitive on the $|F_q^*|/\{\pm1\}=(q-1)/2$ two-subsets $\{x,-x\}$, i.e.\ on unordered pairs of nonzero points, and $G$ is $2$-homogeneous.  $M$ is not transitive on $F_q^*$ ($|M|=(q-1)/2<q-1$), so $G$ is not $2$-transitive.

\emph{Orbit count (Burnside).}  The $q(q-1)/2$ elements of $G$ are the maps $g\colon x\mapsto mx+s$ with $m\in M$, $s\in\F_q$;  let $f(g)$ be the number of three-point subsets fixed by $g$.
\begin{itemize}
\item[(a)] $g=\mathrm{id}$ ($m=1$, $s=0$):  $f=\binom q3$.
\item[(b)] $m=1$, $s
\ne 0$ (pure translation, order $p$):  if $p>3$ no three-point subset is fixed;  if $p=3$, $g$ is a product of $q/3$ three-cycles and fixes exactly $q/3$ three-subsets.
\item[(c)] $m
\ne1$:  $g$ has the unique fixed point $s/(1-m)$ and its cycles on the remaining points are the $|m|$-cycles;  it fixes a three-point subset exactly when $|m|=3$, in which case it fixes $(q-1)/3$ of them.
\end{itemize}
An element of $M$ has order $3$ iff $3\mid|M|=(q-1)/2$, i.e.\ iff $q\equiv1\pmod3$.  For $q\equiv7\pmod{12}$ ($q\equiv1\pmod3$ and $q\equiv3\pmod4$) $M$ contains exactly two elements of order $3$, each giving $q$ affine maps, whence
\[
\sigma_3=\frac{1}{q(q-1)/2}\left[\binom q3+2q\cdot\frac{q-1}{3}\right]=\frac{q+2}{3}.
\]
For $q\equiv11\pmod{12}$ ($q\equiv2\pmod3$) $M$ has no order-$3$ element and $p>3$, so
\[
\sigma_3=\frac{1}{q(q-1)/2}\binom q3=\frac{q-2}{3}.
\]
For $q=3^d$ the translations are three-cycles and $3
\nmid q-1$, so $M$ has no order-$3$ element, giving
\[
\sigma_3=\frac{1}{q(q-1)/2}\left[\binom q3+(q-1)\frac q3\right]=\frac q3.
\]

\emph{The $\lambda$-distribution.}  Each three-point orbit $X$ is a $2$-\mbox{$(q,3,\lambda_X)$} design by $2$-homogeneity;  its number of blocks is both $\lambda_X\cdot q(q-1)/6$ and $|G|/|G_X|$, whence $\lambda_X=3/|G_X|$ and $\lambda_X\in\{1,3\}$.  Let $a$ (resp.\ $b$) count the orbits with $\lambda=1$ (resp.\ $3$).  The sum of the $\lambda_X$ over all orbits equals $q-2$ \textup{(the number of three-subsets of the form $\{0,1,t\}$)}, i.e.\ $a+3b=q-2$;  together with $a+b=\sigma_3$ this yields $b=(q-2-\sigma_3)/2$ and $a=\sigma_3-b$, giving the displayed distributions.  In particular $a=0$ exactly for $q\equiv11\pmod{12}$, matching the non-existence of $\mathrm{STS}(q)$ for $q\equiv5\pmod6$;  for $q=7$ one obtains $a=2$ (two distinct Fano planes) and $b=1$, as checked.  Since $\lambda_X=3/|G_X|$, the $\lambda=1$ orbits have $|G_X|=3$ and induced image $C_3$, and the $\lambda=3$ orbits have $|G_X|=1$ and image $1$.
\end{proof}

\begin{remark}[Combinatorial consequences]
\label{rem:2homog-comb}
Because $|G|$ is odd, $G$ contains no involution;  this is the single most important difference from the $2$-transitive case, and it forces every three-point stabilizer to have induced image in $\{1,C_3\}$.
\begin{itemize}
\item[(a)] \emph{Triple systems.}  Each three-point orbit is already a $2$-\mbox{$(q,3,\lambda)$} design, $\lambda=1$ or $3$;  the $G$-invariant triple systems are the nonempty unions with $\lambda$ the corresponding subset sum \textup{(\textup{Proposition \ref{prop:union-orbits}}}, applied verbatim), and all of them have odd point stabilizers.  The Steiner systems arise exactly in the ranges of \textup{Theorem \ref{thm:2homog}}.
\item[(b)] \emph{Graphs.}  For any $2$-\mbox{$(q,3,\lambda)$} orbit, the incidence graph (points versus blocks) is biregular with valencies $(3,r)$, $r=\lambda(q-1)/2$, and the spectral results of \textup{\S\ref{sec:graphs}} apply verbatim;  by the criterion of \textup{Theorem \ref{thm:stblock}} the block graph is strongly regular whenever $\lambda=1$ and $q\ge9$ (the Steiner case), and the $\lambda\ge2$ rows are regular but generally not strongly regular, as in \textup{\S\ref{sec:graphs}}.
\item[(c)] \emph{Codes.}  The design code (point-row space over $\F_2$) and the constant-weight code are defined exactly as in \textup{\S\ref{sec:coding}};  since $r=\lambda(q-1)/2$, the self-orthogonality criterion of \textup{Theorem \ref{thm:descode}(iii)} is decided by the parities of $r$ and $\lambda$.
\item[(d)] \emph{Association schemes.}  The orbit partition is the class partition of a Schurian coherent configuration on triples \textup{(\textup{Proposition \ref{prop:valencies}})} with class number $\sigma_3$ of \textup{Theorem \ref{thm:2homog}};  because $G$ has no involution the configuration is not symmetric \textup{(generous transitivity on triples fails)}, in contrast with the $\Sym(\Pcal)$ case.
\end{itemize}
\end{remark}

\begin{remark}[The general affine $1$-dimensional case]
\label{rem:2homog-general}
The same analysis applies to any $2$-homogeneous subgroup $G=V\rtimes M$ of $\AaL(1,q)$ with $|G_0|$ odd;  by \textup{Lemma \ref{lem:2homog-class}} these are precisely all $2$-homogeneous non-$2$-transitive groups, and there is no almost simple example.  The square-subgroup case $M=(F_q^*)^2$ is \textup{Theorem \ref{thm:2homog}}.  If $M$ is enlarged by a field automorphism $\sigma\colon x\mapsto x^p$ of odd order $d$, the point stabilizer is $M\rtimes\langle\sigma\rangle$ of order $(q-1)d/2$ and the setwise stabilizer of a pair is $\langle\sigma\rangle$ \textup{(it fixes both points;  $-1
\notin M$ gives no transposition)}.  Then $\sigma$ fixes pointwise the $(p-2)$ points of $\F_p\setminus\{0,1\}$ and cyclically permutes the points of $\F_q\setminus\F_p$ in cycles of lengths dividing $d$, and the three-point orbits are obtained by fusing these pair-suborbits as in \textup{Theorem \ref{thm:2homog}};  the dichotomy $\lambda_X\in\{1,3\}$ persists, since $|G|$ remains odd.
\end{remark}

% ============================================================

\chapter{Combinatorial applications: triple systems, graphs, codes and association schemes}
\label{sec:design}

In this chapter we develop the combinatorial applications of the three-point orbit classification of \textup{Chapter \ref{sec:prelim}}.  The basic object is the following.

The basic combinatorial object is a \emph{triple system} $\mathcal{D}=(\Pcal,\bcal)$, where $\Pcal$ is a finite set of \emph{points} and $\bcal$ is a collection of $3$-element subsets of $\Pcal$ (\emph{blocks}) with the property that every pair of points is contained in a constant number of blocks: the pair $\{x,y\}$ lies in the same number $\lambda$ of blocks independently of $x,y$. Equivalently, $\mathcal{D}$ is a simple $2$-design with block size $k=3$; when $\lambda=1$ it is a Steiner triple system in the classical sense.

Every automorphism group $G\le\Sym(\Pcal)$ of such a design decomposes $\binom{\Pcal}{3}$ into orbits, and the orbit-design correspondence of \textup{Propositions \ref{prop:union-orbits} and \ref{cor:all-invariant}} shows that the $G$-invariant triple systems are exactly the nonempty unions of the three-point orbits.

\begin{corollary}[Block-transitive triple systems]\label{cor:blocktrans}
Let $\Tcal=\Tcal_I$ be a $G$-invariant simple triple system on $\Pcal$, with $\varnothing
\ne I\subseteq\{1,\ldots,r\}$ \textup{(Propositions \ref{prop:union-orbits} and \ref{cor:all-invariant})}.
\begin{enumerate}
\item[(i)] $G$ is block-transitive on $\Tcal$ if and only if $|I|=1$, i.e.\ $\Tcal$ is a single orbit $\Tcal_i$.
\item[(ii)] In particular, every Steiner triple system $2$-$(v,3,1)$ with a $2$-transitive automorphism group is block-transitive; equivalently, every Steiner row $\lambda=1$ of \textup{Table \ref{tab:main-lambda}} is a single orbit.
\item[(iii)] A union $\Tcal_I$ with $|I|\ge2$ (hence $\lambda_I\ge2$) is not block-transitive.
\end{enumerate}
\end{corollary}
\begin{proof}
(i) If $|I|=1$, $\Tcal$ is a $G$-orbit by construction, so $G$ is block-transitive.  If $|I|\ge2$, choose $i
\ne j$ in $I$; blocks of $\Tcal_i$ and $\Tcal_j$ lie in distinct $G$-orbits, so no element of $G$ maps one to the other, and $G$ is not block-transitive.
(ii) Let $\Tcal$ be a $2$-$(v,3,1)$ design and $G$ $2$-transitive.  Given blocks $X=\{a,b,c\}$ and $Y=\{a',b',c'\}$, $2$-transitivity gives $g\in G$ with $a^g=a'$, $b^g=b'$.  The block through $\{a,b\}$ is unique (namely $X$), and the block through $\{a',b'\}$ is unique (namely $Y$); hence $X^g=Y$, and $G$ is block-transitive.  By (i), $\Tcal$ is a single orbit.  (The argument fails for $\lambda\ge2$, when $\lambda$ blocks pass through $\{a,b\}$ and $g$ may send $X$ to any of them.)
(iii) is immediate from (i).
\end{proof}

\begin{corollary}[Flag-transitive triple systems]\label{cor:flagtrans}
Let $\Tcal_i$ be a single orbit (hence block-transitive), with induced image $\overline{G}_{\Ocal_i}\le S_3$ in the exact sequence of \textup{Theorem \ref{thm:main}(2)}.
\begin{enumerate}
\item[(i)] $G$ is flag-transitive on $\Tcal_i$ if and only if $\overline{G}_{\Ocal_i}$ acts transitively on the three points of $\Ocal_i$, if and only if
\[
\overline{G}_{\Ocal_i}\in\{C_3,S_3\};
\]
images $1$ and $C_2$ are not flag-transitive.
\item[(ii)] Reading the image column of \textup{Table \ref{tab:main}}, the classification is:
\begin{itemize}
\item[(a)] \emph{All orbits flag-transitive:}  $\PSL(n,q)$ with $n\ge3$ (both orbits, $S_3$, \textup{Theorem \ref{thm:psln}}); $\Sp(2n,2)$ (both orbits, $S_3$, \textup{Theorem \ref{thm:sp2n-triples}}); $\Alt(n)$/$\Sym(n)$, the Mathieu groups, $HS$ and $Co_3$ (every orbit, $S_3$); and the projective overgroups $\PGL(2,q)/\PpGL(2,q)$ (single orbit, $S_3$).
\item[(b)] \emph{Some orbits flag-transitive:}  $\PSL(2,q)$: all three-point orbits are flag-transitive --- the single orbit for $q$ even ($S_3$) or $q\equiv3\pmod4$ ($C_3$), and the two orbits for $q\equiv1\pmod4$ ($S_3$); $\PSU(3,q)$: the collinear orbit ($S_3$) and, for $q$ odd, the $\lambda=(q^2-1)/d$ orbits ($S_3$) and $\lambda=2(q^2-1)/d$ orbits ($C_3$), and for $q$ even the non-collinear merged orbits ($C_3$), are flag-transitive, while the $\lambda=6(q^2-1)/d$ orbits (image $1$) are not; $\Ree(q)$: the $\lambda=2(q-1)$ orbits ($C_3$) and the two $\lambda=q-1$ orbits (collinear $C_3$, non-collinear $S_3$) are flag-transitive, while the $\lambda=6(q-1)$ (image $1$) and $\lambda=3(q-1)$ (image $C_2$) orbits are not; and on the affine side $\AGL(1,q)$ ($\lambda=1$ with $S_3$, $\lambda=2$ with $C_3$) and $G_2'(q)$ (collinear $\lambda=2$ with $C_3$, hyperbolic and isotropic with $S_3$) are flag-transitive on those orbits, while their image-$1$ and image-$C_2$ collinear orbits are not.
\item[(c)] \emph{No orbit flag-transitive:}  $\Sz(q)$: the circle orbit has image $C_2$ and every generic orbit image $1$.
\end{itemize}
The remaining affine families ($\ASL(n,q)/\AGL(n,q)$ for $n\ge2$, $\ASp$, $\AaL(1,q)$ and the Hering exceptions) are decided by the same criterion from the image column of \textup{Table \ref{tab:main}} and the affine theorems of \textup{Section \ref{sec:affine}}.  Every individual triple-orbit design $\Tcal_i$ is block-transitive, and is flag-transitive exactly in the cases listed.  A nontrivial union $\Tcal_I=\bigcup_{i\in I}\Tcal_i$ with $|I|\ge2$ is \emph{not} block-transitive ($G$ preserves each constituent orbit, so no element of $G$ carries a block of $\Tcal_i$ onto a block of $\Tcal_j$ for distinct $i,j\in I$) and hence is not flag-transitive under $G$.  Consequently the only flag-transitive $G$-invariant triple systems on $\Pcal$ are the flag-transitive single orbits themselves.
\end{enumerate}
\end{corollary}
\begin{proof}
(i) A flag is an incident pair $(\alpha,X)$ with $\alpha\in X$.  Since $G$ is block-transitive on $\Tcal_i$, it is flag-transitive exactly when the block stabilizer $G_{\Ocal_i}$ is transitive on the three points of $\Ocal_i$ (the standard block-transitive/flag-transitive reduction).  Via the exact sequence $1\to G_{(\Ocal_i)}\to G_{\Ocal_i}\to\overline{G}_{\Ocal_i}\to1$, this is transitivity of the image; among the subgroups of $S_3$, precisely $C_3$ and $S_3$ are transitive on three points, whereas $C_2$ (a single transposition) fixes one point and $1$ fixes all.
(ii) is obtained by reading the image attached to each orbit in \textup{Table \ref{tab:main}}, using the image computations in the cited theorems.  (Block's classical gcd criterion, that a block-transitive $2$-design with $\gcd(k,v)=1$ is flag-transitive, is a sufficient condition; here $k=3$, so it applies whenever $3
\nmid v$, but the image criterion above is both necessary and sufficient and also covers the case $3\mid v$.)
\end{proof}

\begin{remark}[Steiner triple systems ($\lambda=1$)]
\label{rem:sts}
A triple system with $\lambda=1$ is a Steiner triple system.  In the framework of \textup{Proposition \ref{prop:union-orbits}}, a nonempty union of three-point orbits is a Steiner triple system if and only if it consists of a single orbit with $\lambda_i=1$ \textup{(}each orbit parameter is a positive integer, so $\lambda_I=\sum_{i\in I}\lambda_i=1$ forces $|I|=1$ and $\lambda_i=1$\textup{)}.  The Steiner triple systems admitting a $2$-transitive automorphism group have been classified by M.~Hall \cite{hall1985}:  such a system is necessarily an affine space $\mathrm{AG}(r,3)$ \textup{(}$v=3^r$\textup{)} or a projective space $\mathrm{PG}(s,2)$ \textup{(}$v=2^{s+1}-1$\textup{)}, apart from the single exception of the $2$-transitive representation of $\Alt(7)$ on the $15$ points of $\mathrm{PG}(3,2)$.  In the affine case the proof is due to J.~I.~Hall \textup{(}see \cite{hall1985}\textup{)}, and independently Kantor and Shult obtained similar results \textup{(}footnote in \cite{hall1985}\textup{)};  the wider context of $2$-transitive linear spaces is treated by Kantor \cite{kantor1985}.  For completeness we record the $\lambda=1$ rows of \textup{Theorem \ref{thm:main}}, all consistent with Hall's classification and confirmed by the exhaustive computation of \textup{Chapter \ref{sec:verify}} \textup{(}which, together with the integrality criterion \eqref{eq:lambda}, leaves no further $\lambda=1$ orbit in the tested degree range $v\le200$\textup{)}.  The projective-space systems:  the collinear triples of $\mathrm{PG}(n-1,2)$ \textup{(}$\PSL(n,2)$, $n=3,4,5,6$\textup{)} give $\mathrm{STS}(7)$, $\mathrm{STS}(15)$, $\mathrm{STS}(31)$, $\mathrm{STS}(63)$.  The affine systems:  the characteristic-$3$ arithmetic-progression orbit of $\AGL(1,q)$ \textup{(}$q\equiv0\pmod3$\textup{)} gives $\mathrm{STS}(q)$ for $q=3,9,27,81$ \textup{(}the affine space $\mathrm{AG}(r,3)$, $q=3^r$\textup{)}, and the collinear orbits of $\ASL(n,q)$/$\AGL(n,q)$ in characteristic $3$ give the affine spaces $\mathrm{AG}(2,3)$, $\mathrm{AG}(3,3)$, $\mathrm{AG}(4,3)$ \textup{(}that is, $\mathrm{STS}(9)$, $\mathrm{STS}(27)$, $\mathrm{STS}(81)$\textup{)}, together with a special non-collinear orbit of $\ASL(2,9)$ which, by Hall's theorem, is abstractly isomorphic to $\mathrm{AG}(4,3)$.  The group $\PSU(3,2)\cong3^2{:}Q_8$ of degree $9$, and its extensions, give the affine plane $\mathrm{STS}(9)$ \textup{(}the unitary case $q=2$ of Hall's analysis\textup{)}, and $\Alt(7)\le\GL(4,2)$ on the $15$ nonzero vectors of $\F_2^4$ \textup{(}$a+b+c=0$\textup{)} gives $\mathrm{STS}(15)$, the same system as the collinear triples of $\PSL(4,2)\cong\Alt(8)$ \textup{(}Hall's exception\textup{)}.  The groups $3^4{:}(2.S_5)$ and $\ASp(4,3)$ of degree $81$ give further $\mathrm{STS}(81)$, which by Hall's theorem are abstractly isomorphic to $\mathrm{AG}(4,3)$.
\end{remark}
\section{Design parameters and integer sequences}
\label{sec:params}
% ============================================================

We record the design parameters $(v,3,\lambda)$ of \textup{Theorem \ref{thm:main}} and the resulting integer sequences.  In every case the parameters of the $G$-invariant triple systems are the nonempty subset sums of the orbit parameters (\textup{Proposition \ref{prop:union-orbits}});  the orbit parameters of the principal families are summarized in \textup{Table \ref{tab:unions}}.

\begin{table}[htbp]
\centering
\footnotesize
\setlength{\tabcolsep}{3pt}
\caption{Orbit parameters and parameters of invariant unions.}\label{tab:unions}
\begin{tabular}{p{1.6cm}p{2.7cm}p{3.7cm}p{3.5cm}}
\toprule
Group family & Orbit types & Orbit parameters & Parameters of nonempty unions \\
\midrule
$\PSL(2,q)$ & projective line & $q-1$;  $(q-1)/2$ twice ($q\equiv1\pmod4$) & subset sums;  complete system $\lambda=q-1$ \\
$\PSL(n,q)$ & collinear;  non-collinear & $q-1$;  $q^2$ ($n=3$), $v-1-q$ ($n\ge4$) & subset sums;  complete system $\lambda=v-2$ \\
$\Sp(2n,2)$ & two orbits on each half & $\lambda_0=2^{2n-2}+s2^{n-1}-2$;  $\lambda_1=2^{2n-2}$ ($s=\pm1$;  $n=3$:  $10,16$ ($v=28$), $18,16$ ($v=36$)) & subset sums;  complete systems $\lambda=v_s-2$ ($n=3$:  $26,34$) \\
$\PSU(3,q)$ & collinear;  non-collinear & $q-1$;  $(q^2-1)/d$ $\times k$, $2(q^2-1)/d$ $\times\#C_3$, $6(q^2-1)/d$ $\times\#1$ with $k=0$ ($q$ even) or $d$ ($q$ odd) \textup{(Thm.\ \ref{thm:main-psu}(3);  Table \ref{tab:magma}, \ref{tab:psu-general})} & subset sums of the known orbit parameters \textup{(Proposition \ref{prop:union-orbits})} \\
$\Sz(q)$ & circle;  generic & $3(q-1)$ \textup{(proved)};  $6(q-1)$ \textup{(proved)} & $3(q-1)$;  $6(q-1)$, $9(q-1)$ \textup{(computed $q{=}8$)} \\
$\Ree(q)$ & $(q+3)^2/6$ orbits & $\lambda=6(q-1)/|S_{\Ocal}|$:  $N_1=\frac{q(q-3)}6$ with $\lambda=6(q-1)$, $N_2=q$ with $\lambda=3(q-1)$, $N_3=\frac{q-1}2$ with $\lambda=2(q-1)$, $N_6=2$ \textup{(collinear, non-collinear)} with $\lambda=q-1$ \textup{(Thm.\ \ref{thm:ree-structure})};  $q=3$:  $6,4,6,2,6,2$;  $q=27$:  $(N_1,N_2,N_3,N_6)=(108,27,13,2)$, $\lambda=156,78,52,26$ \textup{(verified)} & subset sums of the known orbit parameters \textup{(Proposition \ref{prop:union-orbits})} \\
$\Alt(n)$/\allowbreak$\Sym(n)$, \\ Mathieu & single orbit & $n-2$ resp.\ $v-2$ & single system \\
$HS$;  $Co_3$ & $3$;  $2$ orbits & $72,90,12$;  $162,112$ & subset sums;  complete systems $\lambda=174,274$ \\
\bottomrule
\end{tabular}
\end{table}

\subsection{The projective line family}
For $v=q+1$ ($\PSL(2,q)$ and its overgroups):
\begin{enumerate}
\item $q$ even or $q\equiv3\pmod4$:  the single system is $2$-$(q+1,3,q-1)$;  $\lambda=q-1$ ranges over $1,2,3,6,7,10,15,18,22,26,30,31,42,46,\dots$ (the values $q-1$ for $q$ a prime power with $q$ even or $q\equiv3\pmod4$).
\item $q\equiv1\pmod4$:  the two orbit systems are $2$-$(q+1,3,(q-1)/2)$;  $\lambda=(q-1)/2$ ranges over $2,4,6,8,10,12,14,16,\dots$ (all even integers $\ge2$), and their union is the complete system $2$-$(q+1,3,q-1)$.
\end{enumerate}

\subsection{The alternating and symmetric families}
For $\Sym(n)$ and $\Alt(n)$:  $\lambda=n-2$ (\textup{Theorem \ref{thm:alt}}):  the complete system on $n$ points ($3,4,5,\dots$), with $|S_{\Ocal}|=6(n-3)!$ resp.\ $3(n-3)!$ and kernel $\Sym(n-3)$ resp.\ $\Alt(n-3)$.

\subsection{The projective-space and symplectic families}
For $\PSL(n,q)$ on $\mathrm{PG}(n-1,q)$, $v=(q^n-1)/(q-1)$ (\textup{Corollary \ref{cor:psln-designs}}):
\begin{enumerate}
\item $n=3$:  $2$-\mbox{$(q^2+q+1,3,q-1)$} and $2$-\mbox{$(q^2+q+1,3,q^2)$};  instances:  $2$-$(13,3,2)$, $2$-$(13,3,9)$;  $2$-$(21,3,3)$, $2$-$(21,3,16)$;  $2$-$(31,3,4)$, $2$-$(31,3,25)$;  $2$-$(57,3,6)$, $2$-$(57,3,49)$;  $2$-$(73,3,7)$, $2$-$(73,3,64)$;  $2$-$(91,3,8)$, $2$-$(91,3,81)$;
\item $n\ge4$:  $2$-\mbox{$(v,3,q-1)$} (for $q=2$:  $\mathrm{STS}(2^n-1)$:  $n=3,4,5,6$ give $\mathrm{STS}(7),\mathrm{STS}(15),\mathrm{STS}(31),\mathrm{STS}(63)$) and $2$-\mbox{$(v,3,v-1-q)$}:  $2$-$(15,3,12)$, $2$-$(31,3,28)$, $2$-$(40,3,36)$, $2$-$(63,3,60)$, $2$-$(85,3,80)$.
\end{enumerate}
In both cases the union of the two orbits is the complete system $2$-$(v,3,v-2)$ on projective space (\textup{Corollary \ref{cor:psln-designs}}).
For $\Sp(2n,2)$ on the quadric halves (\textup{Corollary \ref{cor:sp2n-designs}}), each half of degree $v_s=2^{2n-1}+s2^{n-1}$ ($s=\pm1$) carries the two systems $2$-\mbox{$(v_s,3,\lambda_0)$} and $2$-\mbox{$(v_s,3,\lambda_1)$} with $\lambda_0=2^{2n-2}+s2^{n-1}-2$, $\lambda_1=2^{2n-2}$ (\textup{Theorem \ref{thm:sp2n-triples}});  for $n=3$:  $2$-$(28,3,10)$, $2$-$(28,3,16)$, $2$-$(36,3,16)$, $2$-$(36,3,18)$, and their unions $2$-$(28,3,26)$, $2$-$(36,3,34)$.

\subsection{The Mathieu families}
All Mathieu groups are $3$-transitive, so each gives the complete system with $\lambda=v-2$:  $M_{11}$ (deg $11$):  $9$;  $M_{11}$ (deg $12$), $M_{12}$, $M_{12}.2$:  $10$;  $M_{22}$, $M_{22}.2$:  $20$;  $M_{23}$:  $21$;  $M_{24}$:  $22$.

\subsection{The affine families}
For $\AGL(1,q)$:  $\lambda=6/|G_{\Ocal}|\in\{6,3,2,1\}$.  The value $\lambda=1$ occurs exactly when the relevant orbit has setwise stabilizer of order $6$ (image $S_3$);  this happens for the characteristic-$3$ arithmetic-progression orbit $\{x,x+u,x+2u\}$, which yields the Steiner triple system $2$-$(q,3,1)$ precisely for $q\equiv0\pmod3$.  The existence of primitive cube roots in $\F_q$ ($q\equiv1\pmod3$, $q
\ne3$) produces $C_3$-orbits with $|G_{\Ocal}|=3$ and $\lambda=2$ (the systems $2$-$(q,3,2)$ for $q\equiv1\pmod6$ or $q=4$), but does not by itself yield a Steiner triple system.  For $\ASL(2,q)$:  the non-collinear orbits give $\lambda=q$ ($q$ even) resp.\ $\lambda=2q$ ($q$ odd) (for $q=4$:  $2$-$(16,3,4)$;  $q=8$:  $2$-$(64,3,8)$;  $q=5$:  $2$-$(25,3,10)$;  $q=3$:  $2$-$(9,3,6)$), and the collinear orbit gives $\lambda=6q/|G_{\Ocal}|\in\{1,2,3,6\}$, obtained from $|G|=q^2|G_0|$, $v=q^2$ and \eqref{eq:lambda} (this equals $6/|\overline{G}_{\Ocal}|$ precisely when the pointwise kernel has order $q$, as it does for the collinear orbit of $\ASL(2,q)$) (for $q=3$:  $\lambda=1$:  the affine plane $\mathrm{STS}(9)$;  $q=4$:  $\lambda=2$;  $q=5$:  $\lambda=3$;  $q=8$:  $\lambda=6$).  For $\ASL(n,q)$, $n\ge3$:  a single non-collinear orbit with $\lambda=q(q^{n-1}-1)$ (e.g.\ $2$-$(27,3,24)$ for $q=3$, $2$-$(64,3,60)$ for $q=4$) and the collinear orbit as above.  For the Hering exceptions:  $2$-$(25,3,6)$, $2$-$(25,3,3)$, $2$-$(25,3,2)$;  $2$-$(121,3,6)$, $2$-$(121,3,3)$, $2$-$(121,3,2)$;  $2$-$(16,3,8)$, $2$-$(16,3,6)$ ($2^4{:}\Alt(6)$);  $2$-$(16,3,14)$ ($2^4{:}\Alt(7)$);  $2$-$(15,3,12)$ and the Steiner triple system $\mathrm{STS}(15)$ ($\Alt(7)$ on the $15$ nonzero vectors of $\F_2^4$).  In every case the parameters of the nonempty unions are the corresponding subset sums of the listed orbit parameters (\textup{Proposition \ref{prop:union-orbits}}).  For $G_2'(q)$ ($v=q^6$, $p=2$), \textup{Theorem \ref{thm:g2-classical}} gives the complete orbit and stabilizer data for every $q=2^f\ge4$:  the $c(q)$ collinear orbits \textup{($c(q)=1+(q-4)/6$ if $3\mid q-1$, else $(q-2)/6$}\textup{)}, consisting of one orbit $\{\omega,\omega^2\}$ with $\lambda=2$ (image $C_3$) when $3\mid q-1$ and the remaining collinear orbits with $\lambda=6$ (pointwise stabilizer $q^6(q^2-1)$, image $1$, throughout), together with the $q-1$ hyperbolic orbits with $\lambda=q^5$ (kernel $q(q^2-1)$, image $S_3$) and the two isotropic orbits with $\lambda=q(q^2-1)$ and $q^3(q^2-1)$ (kernels $q^5$, $q^3$, image $S_3$);  the existence statements of \textup{Theorem \ref{thm:affine-classical}(2),(3)} are proved for all $q$ through the embedding $G_2'(q)\le\Sp(6,q)$, and the numerical data of the general statement are verified for $q=4$ in \textup{Chapter \ref{sec:verify}}.  The case $q=2$ is $2^6{:}G_2'(2)\cong2^6{:}\PSU(3,3)$ with $v=64$, $r=3$ and $\lambda=32,24,6$ (\textup{Table \ref{tab:verify}});  the invariant triple systems are the corresponding unions, with $\lambda$ the subset sums of the orbit parameters of \textup{Theorem \ref{thm:g2-classical}} \textup{(for $q=4$:  subset sums of $\{2,1024,60,960\}$}\textup{)}.

\subsection{The unitary, Suzuki and Ree designs}
For any fixed $2$-transitive group, once its three-point orbit decomposition is known, the general orbit--design proposition gives all invariant triple systems as nonempty unions of those orbits.  For the unitary family the socle decomposition is complete for every $q$ by \textup{Theorem \ref{thm:main-psu}}, together with the diagonal fusion under $\PGU(3,q)$ \textup{(Proposition \ref{prop:psu-diag}, proved for every $q$)};  the closed forms $\sigma_3(\PSU(3,q))$ and $\sigma_3(\PGU(3,q))$ are verified in \textup{Table \ref{tab:psu-general}}.  For the Suzuki family the socle decomposition and the fused decompositions of every field extension $\Sz(q)\cdot f'$ are determined by \textup{Theorem \ref{thm:sz}}, \textup{Theorem \ref{thm:sz-orbits}} and \textup{Theorem \ref{thm:suzuki-field-final}}, together with the explicit finite computations \textup{(Table \ref{tab:verify})}.  For the Ree family the socle decomposition is closed for every $q>3$ by \textup{Theorem \ref{thm:ree-structure}} and the field-automorphism extensions are closed for every $f'\mid 2m+1$ by \textup{Theorem \ref{thm:ree-ext}} together with the fixed-orbit count of \textup{Proposition \ref{prop:ree-descent}(ii)} ($a=(q_0^2+2q_0+3)/2$ for $3\mid f'$, $a=(q_0+3)^2/6$ for $3
\nmid f'$;  Lemma \ref{lem:ree-ext-cycles} covers $f'=3$);  in particular $\Ree(27){:}3$ has $56$ orbits and $\Ree(3^9)\cdot9$ has $7\,176\,734$.
\begin{enumerate}
\item $\PSU(3,q)$ ($v=q^3+1$):  collinear:  $\lambda=q-1$:  $2$-\mbox{$(q^3+1,3,q-1)$};  non-collinear:  by \textup{Theorem \ref{thm:main-psu}(4)} and \textup{Proposition \ref{prop:psu-computed}}, $\lambda=(q^2-1)/d$ on the $k$ fixed orbits, $2(q^2-1)/d$ on the $C_3$-orbits and $6(q^2-1)/d$ on the trivial orbits, with multiplicities $k=0$ ($q$ even) or $d$ ($q$ odd), $\#C_3=(q-1)/2$ ($q$ odd, $d=1$), $(q-5)/2$ ($q$ odd, $d=3$), $q/2$ ($q$ even, $d=1$), and $(q-2)/2$ ($q$ even, $d=3$, proved, \textup{Prop.\ \ref{prop:psu-even}}), $\#1=0$ ($d=1$) or $(q+1)/3$ ($q$ odd, $d=3$) and $(q+1)/3$ ($q$ even, $d=3$, proved, \textup{Prop.\ \ref{prop:psu-even}}):  e.g.\ $q=3$:  $2$-$(28,3,2)$, $2$-$(28,3,16)$, $2$-$(28,3,8)$;  $q=5$:  three copies of $2$-$(126,3,8)$ and two of $2$-$(126,3,48)$;  $q=8$:  $2$-$(513,3,7)$ (collinear), three copies of $2$-$(513,3,42)$ and three of $2$-$(513,3,126)$;  $q=9$:  $2$-$(730,3,8)$ (collinear), $2$-$(730,3,80)$ and four copies of $2$-$(730,3,160)$;  $q=27$:  $2$-$(19684,3,26)$, $2$-$(19684,3,728)$ and thirteen copies of $2$-$(19684,3,1456)$;  $q=4$ ($d=1$):  $\lambda=3$ (collinear) and $30$ (non-collinear);  $q=2$ ($d=3$):  $\lambda=1$ (collinear, $\mathrm{STS}(9)$) and $6$ (the unique non-collinear orbit).
\item $\Sz(q)$ ($v=q^2+1$):  the circle orbit gives $2$-\mbox{$(q^2+1,3,3(q-1))$} for every $q$, and the generic class consists of $\frac{q-2}{6}$ orbits each giving $2$-$(q^2+1,3,6(q-1))$ \textup{(Theorem \ref{thm:sz-orbits};  Corollary \ref{cor:sz-designs})};  for $q=8$ the unique generic orbit is $2$-$(65,3,42)$, whose union with the circle system has $\lambda=63$ \textup{(Proposition \ref{prop:union-orbits})}.
\item $\Ree(q)$ ($v=q^3+1$):  the orbit number is $\sigma_3=(q+3)^2/6$ \textup{(Theorem \ref{thm:ree-triples})} and the orbit structure is closed by \textup{Theorem \ref{thm:ree-structure}}:  $N_1=\frac{q(q-3)}6$ orbits with $\lambda=6(q-1)$, $N_2=q$ with $\lambda=3(q-1)$, $N_3=\frac{q-1}2$ with $\lambda=2(q-1)$, and $N_6=2$ (collinear and non-collinear) with $\lambda=q-1$;  for $q=3$ ($^2G_2(3)\cong\PpGL(2,8)$, $v=28$):  $\lambda=6,4,6,2,6,2$ (six orbit systems, \textup{Table \ref{tab:verify}});  for $q=27$:  $(N_1,N_2,N_3,N_6)=(108,27,13,2)$ with $\lambda=156,78,52,26$, verified in \textup{Chapter \ref{sec:verify}};  the extension $\Ree(27){:}3$ has $56$ orbits with stabilizers $1^{35},2^8,3^7,6^3,9^1,18^2$ \textup{(Theorem \ref{thm:ree-ext})};  and the extension formula is closed for every field quotient $f'\mid 2m+1$ by \textup{Lemma \ref{lem:ree-ext-cycles}} ($f'=3$;  $a=(q_0^2+2q_0+3)/2$) and by the descent argument of \textup{Proposition \ref{prop:ree-descent}} ($f'\ge5$;  $a=(q_0^2+2q_0+3)/2$ for $3\mid f'$, $a=(q_0+3)^2/6$ for $3
\nmid f'$), so in particular $\Ree(3^9)\cdot9$ has $\sigma_3=7\,176\,734$ \textup{(Remark \ref{rem:ree-descent-status};  Proposition \ref{prop:ree-ext-noncubic})}.
\end{enumerate}

\label{sec:other}

\section{Graph applications}
\label{sec:graphs}

The data of an orbit triple system, namely the parameters $(v,3,\lambda)$ and the
suborbit structure of the two-point stabilizers recorded in
\textup{Theorem \ref{thm:main}} and \textup{Section \ref{sec:params}}, determine two graphs of immediate interest: the
incidence graph $\Gamma(\Tcal)$ (points versus blocks) and the block graph on
$\Tcal$.  In this section we translate the orbit data into spectral and
symmetry statements about these graphs.

We use the standard graph terminology of \cite{li2010}: an \emph{arc} of a
graph $\Gamma$ is an ordered adjacent pair; $\Gamma$ is \emph{$G$-arc
transitive} if $G\le\Aut\Gamma$ is transitive on arcs; a graph is
\emph{vertex transitive, edge transitive, biregular} etc.\ in the usual sense.
We recall the two structural facts we use; both are proved in \cite{li2010}.

\begin{lemma}[Orbital graph criteria; \cite{li2010}]\label{lem:orbital}
Let $G\le\Sym(V)$ be transitive.
\begin{enumerate}
\item[(i)] \cite[Proposition 6.3.1]{li2010}: the $G$-orbital graphs (one for each
self-paired orbital, otherwise pairs of complementary directed graphs) are
exactly the $G$-arc-transitive graphs with vertex set $V$; $G$ is primitive iff
every non-trivial orbital graph is connected (Higman's theorem,
\cite[Proposition 6.3.6]{li2010}).
\item[(ii)] \cite[Lemma 7.1.4]{li2010}: a $G$-arc-transitive graph $\Gamma$ is
vertex transitive and $G_v$ is transitive on $\Gamma(v)$; conversely a
connected $G$-vertex-transitive graph with $G_v$ transitive on $\Gamma(v)$ is
$G$-arc transitive.  In particular a graph $\Gamma$ is $G$-arc transitive iff
$G$ is transitive on the set of edges of $\Gamma$ \emph{and} $G_v$ is
transitive on $\Gamma(v)$.
\end{enumerate}
\end{lemma}

\subsection{The incidence graph of an orbit triple system}

Fix a $G$-orbit $\Tcal=\Tcal_i$ on triples, with parameters $(v,3,\lambda)$,
$b=|\Tcal|$, and $r=\lambda(v-1)/2$.

\begin{definition}\label{def:inc}
The \emph{incidence graph} $\Gamma(\Tcal)$ of $\Tcal$ is the bipartite graph
with left part $\Pcal$ (points), right part $\Tcal$ (blocks), and edges the
point--block incidences $x\sim X$ iff $x\in X$.  The \emph{block graph}
$\overline{\Gamma}(\Tcal)$ has vertex set $\Tcal$ and joins distinct blocks
$X,Y$ iff $X\cap Y
\ne\varnothing$.
\end{definition}

Since $G$ preserves $\Tcal$, $G$ acts as automorphisms of both graphs; on
$\Gamma(\Tcal)$ the group $G$ preserves each bipartite part.

\begin{theorem}[Spectrum and structure of the incidence graph]\label{thm:incidence}
Let $\Tcal$ be an orbit triple system with parameters $(v,3,\lambda)$, and let
$N$ be its $v\times b$ incidence matrix, $b=\lambda v(v-1)/6$,
$r=\lambda(v-1)/2$.
\begin{enumerate}
\item[(i)] $\Gamma(\Tcal)$ is a connected bipartite graph with biparts of sizes
$v$ and $b$; every point has degree $r$ and every block has degree $3$,
i.e.\ $\Gamma(\Tcal)$ is biregular with valencies $(3,r)$.  It is connected
with diameter at most $4$: any two points are at distance $2$ (every pair of
points lies in $\lambda\ge1$ blocks), and any two blocks are at distance
most $4$.
\item[(ii)] The adjacency matrix $A=A(\Gamma)$ satisfies
\[
NN^{\top}=(r-\lambda)I+\lambda J,
\]
and $A$ has spectrum
\[
\pm\sqrt{3r}\quad(1\text{ each}),\qquad
\pm\sqrt{r-\lambda}\quad(v-1\text{ each}),\qquad
0\quad(b-v\text{ each}),
\]
where the zero block has multiplicity $b-v$: since $3r>0$ and
$r-\lambda=\lambda(v-3)/2>0$, the matrix $NN^{\top}=(r-\lambda)I+\lambda J$
is positive definite over $\mathbb{R}$, hence $\rank_{\mathbb{R}}N=v$ and
the stated multiplicities hold over $\mathbb{R}$; (the $p$-rank over finite
fields is treated separately in the coding section, \textup{Section
\ref{sec:coding}}).
\item[(iii)] Each block $X=\{a,b,c\}$ spans a star $a-X$, $b-X$, $c-X$; the girth of $\Gamma(\Tcal)$ is $4$ if $\lambda\ge2$ (two blocks share a pair
$\{a,b\}$, giving a $4$-cycle $a\sim X\sim b\sim Y\sim a$) and is exactly $6$ if
$\lambda=1$ and $v\ge7$: two blocks sharing a pair would force $\lambda\ge2$, so
no $4$-cycle exists, while a $6$-cycle always exists.  Indeed, for the unique
block $X=\{a,b,x\}$ containing the pair $\{a,b\}$ and any $d\notin\{a,b,x\}$,
let $Y=\{a,d,y\}$ and $Z=\{b,d,z\}$ be the unique blocks containing $\{a,d\}$
and $\{b,d\}$ respectively;  then $y\notin\{a,b,d,x\}$ (as $y=x$ would put
$\{a,x\}$, and $y=b$ would put $\{a,b\}$, in two blocks) and
$z\notin\{a,b,d,x,y\}$ (as $z=x$, $z=a$ or $z=y$ would likewise put a pair in
two blocks), so $a\sim X\sim b\sim Z\sim d\sim Y\sim a$ is a $6$-cycle.  Hence
every Steiner row of \textup{Table \ref{tab:main}} with $v\ge7$ has girth exactly $6$.
\end{enumerate}
\end{theorem}

\begin{proof}
(i) The point degrees are $r$ by design; the block degree is $3$ by
definition.  Connectedness and the diameter bound: for points $a
\ne b$ there
are $\lambda\ge1$ blocks containing $\{a,b\}$, so $a$ and $b$ are joined via
such a block at distance $2$.  For blocks $X=\{a,x,y\}$ and $Y=\{b,u,w\}$ pick
any block $Z$ containing $\{a,b\}$; then
$X\sim a\sim Z\sim b\sim Y$ is a walk of length $4$.  Hence the graph is
connected, bipartite, and has diameter $\le4$; it is exactly $4$ unless all
blocks pairwise meet in points, i.e.\ unless the design is a projective plane
of order $2$ ($v=7$, the Fano plane), in which case the block--block distance
is $2$ through a common point.
(ii) The $(a,b)$-entry of $NN^{\top}$ is the number of blocks containing
$\{a,b\}$, hence $r$ on the diagonal and $\lambda$ off it, giving
$NN^{\top}=(r-\lambda)I+\lambda J$.  The eigenvalues of the latter are
$3r=(r-\lambda)+\lambda v$ on the all-ones vector and $r-\lambda$ on
$\langle\mathbf1\rangle^{\perp}$.  For
$A=\begin{pmatrix}0&N\\ N^{\top}&0\end{pmatrix}$, a standard block
diagonalization of $A^2=\begin{pmatrix}NN^{\top}&0\\0&N^{\top}N\end{pmatrix}$
shows that the nonzero eigenvalues of $A$ are the square roots of the nonzero
eigenvalues of $NN^{\top}$ together with their negatives; since
$N^{\top}N$ has the same nonzero spectrum as $NN^{\top}$ plus $b-v$ zeros
(if $\rank N=v$), the displayed spectrum follows.
(iii) If $\lambda\ge2$, two distinct blocks contain a common pair
$\{a,b\}$, hence $a\sim X\sim b\sim Y\sim a$ is a $4$-cycle.  If $\lambda=1$
then any two blocks meet in at most one point, so no $4$-cycle exists;  for
$v\ge7$ the Steiner condition $v\equiv1,3\pmod6$ gives at least two blocks
through any point, and a $6$-cycle always exists:  fixing the unique block
$X=\{a,b,x\}$ of the pair $\{a,b\}$ and any $d\notin\{a,b,x\}$, the unique
blocks $Y=\{a,d,y\}$ and $Z=\{b,d,z\}$ satisfy $y,z\notin\{a,b,d,x\}$ and
$y\ne z$ (any coincidence would place a pair in two distinct blocks,
contradicting $\lambda=1$), so $a\sim X\sim b\sim Z\sim d\sim Y\sim a$ is a
$6$-cycle.  Hence the girth is exactly $6$, as stated.
\end{proof}

\begin{remark}\label{rem:integral}
When $3r$ and $r-\lambda$ are both perfect squares, $\Gamma(\Tcal)$ is an
\emph{integral} graph.  This happens, e.g., for the collinear orbit of
$\PSU(3,3)$ on $28$ points ($3r=81$, $r-\lambda=25$) and for the
$\lambda=8$ non-collinear orbit of $\PSU(3,3)$ ($3r=324$, $r-\lambda=100$);
see \textup{Table \ref{tab:as}}.  Integrality is a rare phenomenon and is read off
directly from the orbit data.
\end{remark}

\subsection{Arc-transitivity of the incidence graph}

By \textup{Lemma \ref{lem:orbital}(ii)} (in the bipartite setting) we obtain the
following criterion; it reduces arc-transitivity to the two-point stabilizer,
which is exactly the input available from \textup{Theorem \ref{thm:main}}.

\begin{theorem}[Arc-transitivity criterion]\label{thm:arc}
Let $\Tcal$ be a $G$-orbit triple system, and let
$S_a(\Tcal)=\{X\in\Tcal:\ a\in X\}$ be the star of a point $a$.  The
incidence graph $\Gamma(\Tcal)$ is $G$-arc transitive if and only if the
stabilizer $G_a$ is transitive on $S_a(\Tcal)$.  Equivalently, since
$G$ is $2$-transitive, $\Gamma(\Tcal)$ is $G$-arc transitive if and only if,
for a pair $P=\{a,b\}$, the two-point stabilizer $G_P$ is transitive on the
triples of $\Tcal$ containing $P$ up to the action of $G_a$; more concretely,
the number of $G_a$-orbits on $S_a(\Tcal)$ equals the number of
$G_P$-suborbits on $\{c:\ \{a,b,c\}\in\Tcal\}$ fused by $G_a$.
\end{theorem}

\begin{proof}
The bipartite parts are $\Pcal$ and $\Tcal$, both $G$-invariant, so no element
of $G$ interchanges a point with a block.  Hence arcs $(a,X)$ are
$G$-transitive iff edges $\{a,X\}$ are, and by
\textup{Lemma \ref{lem:orbital}(ii)} this is equivalent to $G_a$ transitive on
$S_a(\Tcal)$.  The reformulation is the standard suborbital fusion: $G_a$
acts on the star, and a star element $\{a,b,c\}$ is determined by the pair
$\{b,c\}$, so the $G_a$-orbits on $S_a(\Tcal)$ correspond to the orbits of
$G_a$ on unordered pairs $\{b,c\}$ with $\{a,b,c\}\in\Tcal$; fixing $b$,
these correspond to the $G_{\{a,b\}}$-suborbits of $c$, counted up to the
fusion induced by the action of $G_a$ on its star.
\end{proof}

\begin{corollary}[Arc-transitive families]\label{cor:arc-families}
The following orbit triple systems have $G$-arc-transitive incidence graphs:
\begin{enumerate}
\item every $3$-transitive family of \textup{Table \ref{tab:main}}, and every
$3$-homogeneous family whose triple stabilizer induces $C_3$ or $S_3$ on the
triple (unique orbit, $\lambda=v-2$);
\item every orbit of $\PSL(2,q)$ on $q+1$ points, $q\equiv1\pmod4$ (both
orbits, $\lambda=(q-1)/2$), and the single orbit for $q$ even or
$q\equiv3\pmod4$;
\item an orbit of $\AGL(1,q)$ on $q$ points is arc transitive if and only if
all blocks of its star at a point lie in a single ratio class (equivalently,
the stabilizer $G_a$ has a single orbit on the star); e.g.\ the $\lambda=2$
orbit for $q=7$ is arc transitive, while the $\lambda=3$ orbit for $q=7$ is
not ($2$ ratio classes, $2$ star orbits).
\end{enumerate}
\end{corollary}

\begin{proof}
(i) For a $3$-transitive family, and for a $3$-homogeneous family whose
triple stabilizer induces $C_3$ or $S_3$, the image $\overline{G}_{\Ocal}$
of the setwise stabilizer is $C_3$ or $S_3$, so the criterion of
\textup{Proposition \ref{prop:arc-image}} applies directly.  This is the
correct condition:  $3$-homogeneity alone is not enough, since $\AGL(1,8)$
is $3$-homogeneous on $8$ points with trivial triple stabilizer, and its
unique orbit triple system is not flag-transitive.
(ii) For $G=\PSL(2,q)$ acting on $\Pcal=\mathbb{F}_q\cup\{\infty\}$ and
$a=0$, identify the remaining points with $\mathbb{F}_q$.  For
$q\equiv1\pmod4$ the two triple orbits are classified by the cross-ratio:
$\{0,1,t\}^{G}$ is determined by the class of
$t\in\mathbb{F}_q\setminus\{0,1\}$ modulo the square subgroup
$(\mathbb{F}_q^{\times})^2$ of index $2$; each orbit has
$\lambda=(q-1)/2$.  The stabilizer $G_0$ is the group of maps
$t\mapsto s^2t+b$ ($s\in\mathbb{F}_q^{\times}$, $b\in\mathbb{F}_q$), which is
transitive on $\mathbb{F}_q$; fixing $b$, the subgroup fixing both points
acts as multiplication by squares on the remaining points and is transitive
on each set $\{c:\ c/b\ \text{lies in a fixed square class}\}$.  Hence $G_0$
has exactly one orbit on the star of a fixed orbit $\Tcal$, so
$\Gamma(\Tcal)$ is $G$-arc transitive.  The argument for the single orbit in
the $3$-homogeneous case is the same, with all of $\mathbb{F}_q^{\times}$
acting.  (iii) A star element $\{a,u,v\}$ is mapped by $G_a=\mathbb{F}_q^{\times}$
(multiplication about $a$) to $\{a,su,sv\}$, so the ratio $v/u$ is a $G_a$-
invariant; the star orbits are therefore in bijection with the ratio classes
occurring among the blocks of the orbit.  The assertion follows.  In each
case the criterion of \textup{Theorem \ref{thm:arc}} applies.
\end{proof}

\begin{proposition}[Arc-transitivity via the stabilizer image]\label{prop:arc-image}
Let $G$ be $2$-transitive on $\Pcal$ and let $\Tcal_i$ be a single orbit
triple system with representative $\Ocal_i$.  For $\Delta\in\Tcal_i$ write
\[
\overline{G}_{\Delta}=G_{\Delta}/G_{(\Delta)}\le\Sym(\Delta)\cong S_3
\]
for the induced image of the setwise stabilizer $G_{\Delta}$ modulo the
pointwise stabilizer $G_{(\Delta)}$;  throughout, an arc of the incidence
graph $\Gamma(\Tcal_i)$ means an incident pair $(\alpha,X)$ with
$\alpha\in X\in\Tcal_i$.  The following conditions are equivalent:
\begin{enumerate}
\item[(i)] $G$ is transitive on the arcs of $\Gamma(\Tcal_i)$;
\item[(ii)] $G$ is flag-transitive on $\Tcal_i$;
\item[(iii)] $\overline{G}_{\Ocal_i}$ is transitive on $\Ocal_i$;
\item[(iv)] $\overline{G}_{\Ocal_i}\cong C_3$ or $S_3$.
\end{enumerate}
In particular $\Gamma(\Tcal_i)$ is $G$-arc transitive if and only if
$\overline{G}_{\Ocal_i}\in\{C_3,S_3\}$.
\end{proposition}
\begin{proof}
(i)$\Leftrightarrow$(ii): an arc $(\alpha,X)$ of $\Gamma(\Tcal_i)$ is by
definition a flag of $\Tcal_i$, and conversely.  (ii)$\Leftrightarrow$(iv):
by \textup{Corollary \ref{cor:flagtrans}(i)}, flag-transitivity is equivalent
to $\overline{G}_{\Ocal_i}\in\{C_3,S_3\}$.  (iii)$\Leftrightarrow$(iv): the
transitive subgroups of $S_3$ on three points are exactly $C_3$ and $S_3$
(the subgroups of $S_3$ are $1$, the three conjugates of $C_2$, $C_3$ and
$S_3$;  among these precisely $C_3$ and $S_3$ act transitively on the three
points), so transitivity of $\overline{G}_{\Ocal_i}$ on $\Ocal_i$ is
equivalent to $\overline{G}_{\Ocal_i}\cong C_3$ or $S_3$.  Finally, by
\textup{Theorem \ref{thm:arc}}, arc-transitivity of $\Gamma(\Tcal_i)$ is
equivalent to transitivity of $G_\alpha$ on the star
$S_\alpha(\Tcal_i)$, which is precisely flag-transitivity (ii), completing
the cycle.
\end{proof}

\begin{proposition}[Arc-transitive orbit designs of the classical rows of \textup{Table \ref{tab:main}}]\label{prop:arc-classification}
Throughout, $\lambda$ denotes the design parameter (the number of blocks
through a pair of points, i.e.\ the \emph{index}) of an orbit triple
system, given by \eqref{eq:lambda} (this is \emph{not} the replication
number $r=\lambda(v-1)/2$, the number of blocks through a point;  note that $\lambda=6/|G_{\Ocal}|$ holds only in the
two-point-regular case, the general formula involving the two-point
stabilizer);  no second meaning is attached to the symbol $\lambda$ in this
monograph.  Reading the image column of \textup{Table \ref{tab:main}} through
\textup{Proposition \ref{prop:arc-image}}:
\begin{enumerate}
\item For $\PSU(3,q)$, $d=\gcd(3,q+1)$, exactly the following orbit triple
systems have $G$-arc-transitive incidence graphs:  the collinear orbit
($\lambda=q-1$, image $S_3$) for every $q$;  the $\lambda=(q^2-1)/d$ orbits
(image $S_3$, $q$ odd, $d$ of them) and the $\lambda=2(q^2-1)/d$ orbits
(image $C_3$) whenever they occur;  the $\lambda=6(q^2-1)/d$ orbits (image
$1$, $(q+1)/3$ of them when $d=3$ --- for $q$ odd by \textup{Proposition
\ref{prop:psu-odd}} and for $q$ even ($q=2^{2m+1}$) by \textup{Proposition
\ref{prop:psu-even}}, e.g.\ three $\lambda=126$ orbits for $q=8$;  none occurs
for $d=1$) are not arc transitive.  In particular for $q=3$ all three orbits are arc transitive
(image column $S_3,C_3,S_3$, one orbit of each type), confirming the
computational census of \textup{Table \ref{tab:magma}}.
\item For $\Sz(q)$, $q=2^{2m+1}>2$ \textup{(the exceptional $\Sz(2)$ is not
simple and is excluded)}, no orbit triple system has arc-transitive incidence
graph:  the circle orbit has image $C_2$ and every generic orbit image $1$.
\item For $\Ree(q)$, $q=3^{2m+1}>3$, the $\lambda=2(q-1)$ orbits (image $C_3$)
and the two $\lambda=q-1$ orbits (collinear image $C_3$, non-collinear image
$S_3$) are arc transitive, while the $\lambda=6(q-1)$ (image $1$) and
$\lambda=3(q-1)$ (image $C_2$) orbits are not;  the small case $q=3$
(${}^2G_2(3)\cong\PpGL(2,8)$ on $28$ points) is covered separately by the
computed census of \textup{Table \ref{tab:verify}} \textup{($\lambda=6,4,6,2,6,2$)}.
\end{enumerate}
\end{proposition}
\begin{proof}
Immediate from \textup{Proposition \ref{prop:arc-image}} and the image column of
\textup{Corollary \ref{cor:flagtrans}(ii)(b),(c)}, with the $\lambda$-values of
\textup{Theorem \ref{thm:main-psu}(4)}, \textup{Corollary \ref{cor:psu-designs}},
\textup{Theorem \ref{thm:sz-orbits}} and \textup{Theorem \ref{thm:ree-structure}}.
\end{proof}

\begin{remark}\label{rem:arc-tt}
In the unitary rows the arc-transitivity data are consistent with, and refine,
the rank-$3$ permutation-character computations of \cite{tongviet2013}.
\end{remark}

\subsection{The block graph}

\begin{theorem}[Block graph: general and Steiner cases]\label{thm:stblock}
Let $\Tcal$ be a $G$-orbit triple system with parameters $(v,3,\lambda)$.
\begin{enumerate}
\item[(i)] The block graph $\overline{\Gamma}(\Tcal)$ is $G$-vertex
transitive and regular of degree
\[
k=3\lambda\frac{v-3}{2}=3r-3\lambda,
\]
because a block $X=\{a,b,c\}$ meets $\lambda(v-3)$ blocks through $a$, and
each of the other blocks through $b$ or $c$ has been counted twice over.
\item[(ii)] If $\lambda=1$ and $v\ge9$, then $\overline{\Gamma}(\Tcal)$ is
strongly regular with parameters
\[
\srg\Bigl(\frac{v(v-1)}6,\ \frac{3(v-3)}2,\ \frac{v+3}2,\ 9\Bigr).
\]
For $v=7$ (the Fano plane, $\GL(3,2)=\PSL(3,2)\cong\PSL(2,7)$ acting on the
$7$ points) the block graph is the complete graph $K_7$, not strongly regular
with the above parameters.
\item[(iii)] If $v\ge9$ and $\lambda\ge2$, the block graph is regular but
generally not strongly regular; its adjacency spectrum is not controlled by
the pair data alone (the block--block intersection numbers depend on the
suborbits of the two-point stabilizers).
\end{enumerate}
\end{theorem}

\begin{proof}
(i) $G$ is transitive on blocks, so vertex transitivity is clear.  For the
degree, fix $X=\{a,b,c\}$ and let $N_1$ be the number of blocks meeting $X$
in exactly one point and $N_2$ the number meeting $X$ in exactly two points.
The latter are the blocks through a pair of $X$ other than $X$ itself: there
are $\lambda-1$ through $\{a,b\}$, $\lambda-1$ through $\{a,c\}$, and
$\lambda-1$ through $\{b,c\}$, all distinct, so $N_2=3(\lambda-1)$.  The
former are the blocks through $a$ containing neither $b$ nor $c$: their
number is $r-1-2(\lambda-1)=\lambda(v-1)/2-2\lambda+1$, and similarly for
$b,c$, with the three families disjoint.  Hence
\[
k=N_1+N_2=3\Bigl(\frac{\lambda(v-1)}2-2\lambda+1\Bigr)+3(\lambda-1)
=\frac{3\lambda(v-3)}2.
\]
The blocks avoiding $X$ altogether are the remaining ones, so
\[
N_0=b-1-k=\frac{\lambda v(v-1)}6-1-\frac{3\lambda(v-3)}2,
\]
and the three layers $N_0,N_1,N_2$ partition the block set.  Note that the
count is uniform in $\lambda$: for $\lambda\ge2$ two blocks may meet in $0$, $1$
or $2$ points, and the computation above accounts for both the $\lambda-1$
blocks through a fixed pair and the blocks through a fixed point avoiding the
other two, without ambiguity.
(ii) Let $\lambda=1$.  Adjacent blocks $X=\{a,b,c\}$, $Y=\{a,d,e\}$
share exactly one point, say $a$.  A common neighbor $Z$ of $X,Y$ is either a
block through $a$ (other than $X,Y$): there are $r-2=(v-5)/2$ of these; or a
block determined by one point of $X\setminus\{a\}$ and one of
$Y\setminus\{a\}$: there are $2\times2=4$ such blocks
$Z=\{b,d,w\}$ with $w
\ne a$, and they are distinct from the blocks through
$a$ and from each other (two of them sharing two points of $X$ would force
them equal, contradicting $\lambda=1$).  Hence
$\lambda_{\overline\Gamma}= (v-5)/2+4=(v+3)/2$.  For nonadjacent blocks
$X,Y$ (disjoint, since $\lambda=1$): any common neighbor $Z$ meets each of
$X,Y$ in exactly one point, so it is determined by a pair of points
$(x\in X,\ y\in Y)$; all $3\cdot3=9$ such blocks exist and are distinct,
giving $\mu_{\overline\Gamma}=9$.  Since $v\ge9$ guarantees the existence of
disjoint pairs of blocks and of intersecting pairs, the claimed strongly
regular parameters follow.
(iii) This is the generic situation; the failure of strong regularity is
visible already for $\lambda=2$: the number of common neighbors of two
blocks meeting in two points depends on the number of further blocks through
that pair, i.e.\ on $\lambda-1$ data, whereas for disjoint blocks it depends
on the suborbits of the two-point stabilizers.
\end{proof}

\begin{remark}[Edge layers of the block graph for $\lambda\ge2$]\label{rem:blocklayers}
For $\lambda\ge2$ the edge set of $\overline{\Gamma}(\Tcal)$ splits into two
$G$-invariant layers: the pairs of blocks meeting in exactly one point and the
pairs meeting in exactly two points (the intersection size is $G$-invariant,
so each layer is a union of $G$-orbits on unordered block pairs).  The
corresponding per-vertex degrees are exactly the numbers $N_1$ and $N_2$ of
the proof above, $N_1=3(r-2\lambda+1)$ and $N_2=3(\lambda-1)$, so the total
degree is $k=N_1+N_2$.  Strong regularity fails precisely when the
common-neighbor count distinguishes the two layers: the common neighbors of
two blocks meeting in two points are controlled by the $\lambda-1$ further
blocks through the common pair, while for blocks meeting in one point they
are controlled by the two-point stabilizer suborbits, and these two data are
independent in general (see \textup{Proposition \ref{prop:blockgraph-srg}}).  For $\lambda=1$ the
second layer is empty and one recovers the strongly regular regime of
\textup{Theorem \ref{thm:stblock}(ii)}; this is the only case in which the block graph
is determined by the pair data alone.
\end{remark}

\begin{proposition}[Strong regularity of the block graph]\label{prop:blockgraph-srg}
Let $\Tcal_i$ be a single orbit triple system with parameters $(v,3,\lambda)$
and block graph $\overline\Gamma_i$, and let $G_X$ be the setwise stabilizer
of a block $X\in\Tcal_i$ in $G$.
\begin{enumerate}
\item \emph{Criterion.}  $\overline\Gamma_i$ is strongly regular if and only
if, for the fixed block $X$, the common-neighbor count
$c(X,Y)=\#\{Z\in\Tcal_i:\ Z\cap X
\ne\varnothing
\ne Z\cap Y\}$ takes a single value on the whole set of blocks meeting $X$ in
exactly $2$ points, a (possibly different) single value on the whole set of
blocks meeting $X$ in exactly $1$ point, and a single value on the whole set
of blocks disjoint from $X$;  equivalently, constancy on the whole adjacency
set and on the whole non-adjacency set.  (Constancy on each individual
$G_X$-orbit alone would be automatic for any $G_X$-invariant function and is
\emph{not} sufficient:  the criterion requires the same value across all
$G_X$-orbits within each of the three intersection layers.  By
block-transitivity the value is then constant over all block pairs, as
required by strong regularity.)  A sufficient condition is that each of the
three layers is a single $G_X$-orbit, together with the required equality of
the common-neighbor values on the two adjacent layers;  when all three
layers are nonempty single $G_X$-orbits the rank of $G$ on $\Tcal_i$ is
$4$, not $3$, and the equality of the common-neighbor values on the
one-point and two-point layers (both adjacent to $X$) must still be
verified.  The criterion is necessary since $G_X$ preserves the
intersection size and the common-neighbor count, and the three layers are
nonempty for $v\ge9$, $\lambda\ge1$ \textup{(for $\lambda=1$ the two-point
layer is empty and the condition on it is vacuous)}.
\item For $\lambda=1$ and $v\ge9$ the criterion holds and yields the strongly
regular parameters of \textup{Theorem \ref{thm:stblock}(ii)};  this settles the
Steiner case family-wide.
\item For $\PSU(3,3)$ \emph{none} of the three block graphs is strongly
regular (computed by exhaustive enumeration of all block pairs):  the
collinear orbit ($v=28,\lambda=2$, block number $252$) has constant counts
$33$ on the one-point layer and $50$ on the two-point layer, but the
zero-point layer takes the two values $15$ and $18$;  the non-collinear
orbits ($\lambda=8$, block number $1008$, and $\lambda=16$, block number
$2016$) split already on the two-point layer, with the two values $212,214$
resp.\ $428,430$.
\item For $\PSU(3,3)$ the three block graphs fail the criterion, as
computed in (3).  For the unitary, Suzuki and Ree rows beyond $\PSU(3,3)$
the criterion is \emph{not} settled family-wide here:  a split of an
intersection layer into several $G_X$-orbits does not by itself force
distinct common-neighbor values, and the full common-neighbor counts have
not been computed for every parameter value.  The strongly regular block
graphs \emph{established} in this monograph are therefore the Steiner cases
of \textup{Theorem \ref{thm:stblock}(ii)} and the $\PSU(3,3)$ failures of
(3);  the remaining unitary, Suzuki and Ree rows are left as a
case-by-case computational question rather than a closed classification.
\end{enumerate}
\end{proposition}
\begin{proof}
(1) The block graph is $G$-vertex transitive.  For a pair of vertices
$X,Y$ the number $c(X,Y)$ of common neighbors is a $G_X$-invariant
function of $Y$, hence constant on $G_X$-orbits;  it is moreover constant on
each intersection layer whenever that layer is a single $G_X$-orbit, which
gives the sufficient direction, while the necessary direction is obtained by
noting that a strongly regular graph has, by definition, a single
$\lambda$-value on adjacent pairs and a single $\mu$-value on non-adjacent
pairs, and each of the three intersection layers consists either of adjacent
or of non-adjacent pairs (a block meeting $X$ in $2$ points is adjacent to
$X$;  meeting it in $1$ point is adjacent;  meeting it in $0$ points is
not).  If $G$ has rank at most $3$ on $\Tcal_i$, then $G_X$ has at most two
orbits off $X$, namely the one- and two-point layers, whence both are single
orbits and the sufficient condition of (1) applies.
(2) is \textup{Theorem \ref{thm:stblock}(ii)} recast in the language of (1).
(3) is the computed statement above;  the counts are obtained by exhaustive
enumeration of the $\binom{b}{2}$ unordered block pairs, which is feasible
for $b\le2016$.
(4) For $\PSU(3,3)$ the statement is the computation of (3).  The
suborbit data of \textup{Chapter \ref{sec:prelim}} supply the $G_X$-orbit
decomposition of the blocks off $X$ through the two-point stabilizers
(\textup{Proposition \ref{prop:pair-orbits}}), but a split of a layer into
several $G_X$-orbits does not by itself imply distinct common-neighbor
values;  a family-wide claim for the unitary, Suzuki and Ree rows would
require the full common-neighbor counts, which are not computed here.
\end{proof}

\begin{example}\label{ex:pg}
[The Fano block graph and $\PG(n-1,2)$]
For the projective geometry $\PG(n-1,2)$, the points are the
$v=2^n-1$ nonzero vectors of $\mathbb{F}_2^n$ and the blocks are the
$b=(2^n-1)(2^{n-1}-1)/3$ triples $\{x,y,x+y\}$.  This is an STS$(v)$
invariant under $\GL(n,2)$, so by \textup{Theorem \ref{thm:stblock}(ii)} the block
graph is $\srg\bigl((2^n-1)(2^{n-1}-1)/3,\ 3(2^{n-1}-2),\ 2^{n-1}+1,\ 9\bigr)$
for $n\ge4$ \textup{(the parameter formula of \textup{Theorem \ref{thm:stblock}(ii)}
requires $v\ge9$, and $v=2^n-1\ge9$ exactly for $n\ge4$)}.  The Fano case
$n=3$ ($v=7$) is singled out:  the block graph is $K_7$, reflecting that in
the Fano plane any two triples meet.  For $n=4$ ($v=15$, the STS$(15)$ on
$\mathbb{F}_2^4$), the block graph is
$\srg(35,18,9,9)$, which is the well-known complement of the
Tutte--Coxeter / Kneser-type graph parameter; in particular it is strongly
regular with the parameters of the $(35,18,9,9)$ graph.
\end{example}

\begin{remark}
The strongly regular block graph of an STS is the classical connection between
Steiner triple systems and strongly regular graphs (see \cite{cameron1999}).
\textup{Theorem \ref{thm:stblock}(ii)} upgrades the classical statement: the
\emph{orbit} triple systems of a $2$-transitive group with $\lambda=1$
necessarily yield these parameters, and the $2$-transitive rows of
\textup{Table \ref{tab:main-lambda}} with $\lambda=1$ are exactly the $\AGL(1,q)$ orbits with
$q\equiv0\pmod3$, i.e.\ the affine Steiner triple systems STS$(q)$,
together with the Hering exceptional rows containing $\lambda=1$:  the $\Alt(7)$-on-$15$ row $(\mathrm{STS}(15))$ and the degree-$81$ row $3^4{:}(2.S_5)$ $(\mathrm{STS}(81))$ \textup{(cf.\ Remark \ref{rem:sts})}.  The geometric
STS's $\PG(n-1,2)$ of \textup{Example \ref{ex:pg}} are orbit systems of the
$2$-transitive group $\PSL(n,2)=\GL(n,2)$ on the points of
$\PG(n-1,2)$ (\textup{Theorem \ref{thm:psln}}), with $\lambda=1$;  they are
included here as an illustration of \textup{Theorem \ref{thm:stblock}(ii)}
alongside the affine and Hering Steiner rows.
\end{remark}

\section{Coding applications}
\label{sec:coding}

The same orbit data also produces, in a canonical way, two codes per orbit
triple system: the binary constant-weight code consisting of the
characteristic vectors of the blocks, and the linear design code spanned by
the rows of the incidence matrix.  In this section we read off their
parameters, namely length, size, minimum distance, rank and self-orthogonality,
from $(v,3,\lambda)$ alone, with no further group-theoretic input.

We use the standard terminology of \cite{ling2004}: an $(n,M,d)$ code over an
alphabet of size $q$ is a set of $M$ words of length $n$ with minimum Hamming
distance $d$; a \emph{constant-weight} code $(n,M,d;w)$ consists of weight-$w$
words.  The classical bounds of Chapter 5 of \cite{ling2004} that we invoke
are the Hamming bound (a perfect code attains it), the Singleton bound
($M\le q^{n-d+1}$, with equality for MDS codes), the Plotkin bound
($d\le n\overline{q}$ for large $M$), and the Johnson bound for
constant-weight codes: writing $A(n,d,w)$ for the maximal size of an
$(n,M,d;w)$ code,
\[
A(n,d,w)\le\Big\lfloor\frac{n}{w}\,A(n-1,d,w-1)\Big\rfloor,
\]
with $A(m,4,2)=\lfloor m/2\rfloor$ and $A(m,4,1)=m$ \cite{ling2004, johnson1962}.

\subsection{Constant-weight codes of orbit triple systems}

Let $\Tcal$ be a $G$-orbit triple system with parameters $(v,3,\lambda)$.

\begin{definition}\label{def:cw}
The \emph{constant-weight code} of $\Tcal$ is
\[
C_{cw}(\Tcal)=\{\,\mathbf1_X:\ X\in\Tcal\,\}\subseteq\mathbb{F}_2^{v},
\]
the set of characteristic vectors of the blocks, viewed as a binary code of
length $v$ and weight $3$.
\end{definition}

The block size $3$ makes the minimum distance computable from the
intersection data alone: for distinct blocks $X,Y$,
\[
d(X,Y)=|\supp\mathbf1_X\triangle\supp\mathbf1_Y|=6-2|X\cap Y|,
\]
and since $|X\cap Y|\le2$, the minimum distance of $C_{cw}(\Tcal)$ is
$d=6-2m$ where $m=\max_{X
\ne Y}|X\cap Y|\in\{0,1,2\}$.  The key observation is
that $m$ is read off from $\lambda$ alone.

\begin{theorem}[Minimum distance of orbit codes]\label{thm:cwdist}
Let $\Tcal$ be a $G$-orbit triple system with parameters $(v,3,\lambda)$.
\begin{enumerate}
\item[(i)] The constant-weight code $C_{cw}(\Tcal)$ has parameters
$(v,\,|\Tcal|,\,d\,;\,3)$ with $d=6-2m$, $m=\max_{X
\ne Y}|X\cap Y|$.
\item[(ii)] If $\lambda\ge2$ then $m=2$ and $d=2$: two distinct blocks
containing the same pair realize distance $2$ (for $v\ge7$), while blocks
meeting in exactly one point realize distance $4$.
\item[(iii)] If $\lambda=1$ then $m=1$ and $d=4$ (for $v\ge7$, the Steiner
triple system condition $v\equiv1,3\pmod6$): blocks meet in at most one
point, and two blocks through a point realize distance $4$.
\item[(iv)] $m=0$ occurs if and only if the blocks of $\Tcal$ are pairwise
disjoint, which for a $2$-transitive orbit forces the degenerate case
$|\Tcal|\le1$ or $v<6$; in particular $d=6$ never occurs for the nontrivial
families of \textup{Table \ref{tab:main}}.
\end{enumerate}
\end{theorem}

\begin{proof}
(i) and the formula for $d$ are immediate.  (ii): $\lambda\ge2$ means that a
point pair $\{a,b\}$ lies in at least two blocks of $\Tcal$, so the two distinct
blocks containing $\{a,b\}$ realize $|X\cap Y|=2$; since $m\le2$ always, the
maximum is attained inside the orbit and $m=2$.  (iii): $\lambda=1$
means no pair of blocks shares two points, so $m\le1$.  If $v\ge7$, the
Steiner condition $r=(v-1)/2\ge3$ gives at least two blocks through any point $a$,
and any two such blocks intersect in $\{a\}$ exactly, so $|X\cap Y|=1$ is
attained inside $\Tcal$; hence $m=1$ and $d=4$.  In both cases the extremal
intersection is realized by explicit blocks of the orbit, not merely inferred
from the replication number.  (iv): $m=0$ forces all
blocks pairwise disjoint; since $G$ is transitive on $\Tcal$ and $2$-transitive
on $\Pcal$, each point lies in $r$ blocks, and pairwise disjointness gives
$3|\Tcal|=vr$ together with $\lambda=0$, contradicting
$\lambda\ge1$ unless $|\Tcal|=1$ (then $v=3$ and $r=1$) or no such orbit
exists; the table families all have $v\ge5$ and $\lambda\ge1$, so $d=6$ is
excluded.
\end{proof}

\begin{corollary}[Classification of the binary orbit codes]
For every nontrivial $2$-transitive orbit triple system,
$d\in\{2,4\}$, with
\[
d=2\iff\lambda\ge2,\qquad d=4\iff\lambda=1.
\]
In particular the orbit codes of the families of \textup{Table \ref{tab:main}} are
all of distance $2$, except exactly the Steiner rows
$\lambda=1$: the $\AGL(1,q)$ orbits with $q\equiv0\pmod3$ and the Hering
exceptions with $\lambda=1$, which have distance $4$.
\end{corollary}

\begin{theorem}[Steiner orbit codes meet the Johnson bound]\label{thm:johnson}
Let $\Tcal$ be a $G$-invariant Steiner triple system on $v$ points
($\lambda=1$), and let $C_{cw}(\Tcal)=\{\mathbf1_X:X\in\Tcal\}\subseteq\mathbb F_2^v$
be the constant-weight orbit code of \textup{Section \ref{sec:coding}}.  By
\textup{Theorem \ref{thm:cwdist}(iii)} it has parameters $(v,|\Tcal|,4;3)$.  The
Johnson bound for binary constant-weight codes states
$A(n,2\delta,w)\le\lfloor\frac nw\,A(n-1,2\delta,w-1)\rfloor$; for
$(n,w,\delta)=(v,3,2)$:
\[
|\Tcal|=A(v,4,3)=\frac{v(v-1)}6,
\]
and equality holds because $|\Tcal|=v(v-1)/6$ by the design count
(\textup{Lemma \ref{lem:lambda}}), which equals the bound.  Hence the code is
optimal among constant-weight codes of the given length, weight and minimum
distance: it attains the Johnson bound.
\end{theorem}

\begin{proof}
By \textup{Theorem \ref{thm:cwdist}(iii)} the code has parameters
$(v,|\Tcal|,4;3)$, and $|\Tcal|=v(v-1)/6$ by the design count
(\textup{Lemma \ref{lem:lambda}}).  The Johnson bound gives
\[
A(v,4,3)\le\Big\lfloor\frac v3\,A(v-1,4,2)\Big\rfloor
=\Big\lfloor\frac v3\Big\lfloor\frac{v-1}2\Big\rfloor\Big\rfloor
=\frac{v(v-1)}6
\]
since $v\equiv1,3\pmod6$:  for $v\equiv3\pmod6$ one has $3\mid v$ and $2\mid(v-1)$, while for $v\equiv1\pmod6$ (e.g.\ $v=7,13$) the factor $(v-1)/2$ is a multiple of $3$, so $\bigl\lfloor\frac v3\lfloor\frac{v-1}2\rfloor\bigr\rfloor=\frac v3\cdot\frac{v-1}2=\frac{v(v-1)}6$ in both cases.  Hence
$A(v,4,3)=|\Tcal|$.
\end{proof}

\begin{example}[Optimal Steiner orbit codes]\label{ex:mj}
The $3$-homogeneous row $v=12$ has $\lambda=10$, so the constant-weight code
has parameters $(12,220,2;3)$.  The Steiner rows of
\textup{Table \ref{tab:main}} produce the optimal codes
$(q, q(q-1)/6, 4; 3)$ for STS$(q)$ under $\AGL(1,q)$, $q\equiv0\pmod3$:
for $q=9$ this is a $(9,12,4;3)$ code, and for $q=27$ a
$(27,117,4;3)$ code.
\end{example}

\subsection{Design codes}

\begin{definition}\label{def:descode}
Let $N$ be the $v\times b$ incidence matrix of $\Tcal$ over a field $F$ (rows indexed by points, columns by blocks).  The \emph{design code} of $\Tcal$ over $F$ is
\[
C_{q}(\Tcal)=\operatorname{RowSpan}_{F}(N)\subseteq F^{b},\qquad q=|F|,
\]
the span of the point rows $\mathbf r_\alpha$ of weight $r$.
(If $F=\mathbb{F}_2$ we write $C_2(\Tcal)$.)
\end{definition}

The group $G$ permutes the block coordinates and stabilizes the point action
compatibly: if $P_g$ is the permutation matrix of $g$ on $\Pcal$ and $Q_g$ the
one on $\Tcal$, then $P_gN=NQ_g$, and $Q_g$ maps $C_q(\Tcal)$ to itself, so
$G\le\PAut(C_q(\Tcal))$.

\begin{theorem}[Rank, all-ones word, self-orthogonality]\label{thm:descode}
Let $\Tcal$ be a $G$-orbit triple system with parameters $(v,3,\lambda)$,
$r=\lambda(v-1)/2$, and let $F=\mathbb{F}_q$ with $\operatorname{char}F=p$.
\begin{enumerate}
\item[(i)] If $p
\nmid3r$ and $p
\nmid(r-\lambda)$, then
$\rank N=v$ and $\dim_F C_q(\Tcal)=v$.
\item[(ii)] The all-ones word $\mathbf1\in C_q(\Tcal)$ whenever $p
\ne3$
(the sum of the $v$ rows equals $3\mathbf1$).
\item[(iii)] For $p=2$: $C_2(\Tcal)$ \textup{(the point-row code of \textup{Definition \ref{def:descode}}}
\textup{)} is self-orthogonal (with respect to the
standard inner product) if and only if $r$ and $\lambda$ are both even: the row
self-products equal $r$ and the distinct-row products equal $\lambda$, so all
inner products vanish in $\mathbb F_2$ exactly when both are even.  This
criterion is independent of (i): a full-rank Gram matrix $NN^{\top}$ does not
by itself imply self-orthogonality, and self-orthogonality is tested directly
on the row inner products, not on the rank.  In
particular:
\begin{itemize}
\item[(a)] the $\PSL(2,q)$ orbits with $q\equiv1\pmod4$
($\lambda=(q-1)/2$ even) are self-orthogonal exactly for $q\equiv1\pmod8$:
indeed $r=(q-1)q/4$ is even iff $(q-1)/4$ is even;
\item[(b)] among the $3$-homogeneous rows, the unique orbit with $\lambda=v-2$ is
self-orthogonal exactly when $r$ and $v-2$ are both even, e.g.\ $M_{22}$ on
$22$ points ($r=210$, $\lambda=20$), while $M_{11},M_{12},M_{23},M_{24}$ and
$\Alt(v),\Sym(v)$ are not in general;
\item[(c)] all non-collinear $\PSU(3,q)$ orbits are self-orthogonal: for $q$ odd,
$\lambda\in\{(q^2-1)/d,2(q^2-1)/d,6(q^2-1)/d\}$ is even (for odd $q$
one has $8\mid q^2-1$, and division by $d\in\{1,3\}$ does not remove the
factor $2$, so $(q^2-1)/d$ is even) and $r=\lambda q^3/2$ is even;
for $q$ even the existing non-collinear types
($\lambda=2(q^2-1)/d$, $6(q^2-1)/d$) are even again and $r$ is even;
\item[(d)] the collinear $\PSU(3,q)$ orbits are self-orthogonal exactly for
$q\equiv1\pmod4$ ($r=(q-1)q^3/2$ is even iff $q\equiv1\pmod4$);
\item[(e)] the trivial $\Sz(q)$ orbits ($\lambda=6(q-1)$, $r=3(q-1)q^2$) are
self-orthogonal, while the circle orbits
($\lambda=3(q-1)$, $q$ even) are not;
\item[(f)] the $\Ree(q)$ orbits with $\lambda\in\{6(q-1),2(q-1)\}$ are
self-orthogonal, those with $\lambda\in\{3(q-1),q-1\}$ are not;
\item[(g)] the second $\Sp(2n,2)$ orbits ($\lambda_1=2^{2n-2}$ even, $r$ even) are
self-orthogonal, the first ones are not.
\end{itemize}
\item[(iv)] The minimum weight of $C_q(\Tcal)$ is at most
\[
\min\{\,r,\ 2(r-\lambda)\,\},
\]
realized by the point rows (weight $r$) and by the \emph{difference} of
two rows belonging to a pair (weight $2(r-\lambda)$;  in characteristic $2$
the difference equals the sum).
\end{enumerate}
\end{theorem}

\begin{proof}
(i) $NN^{\top}=(r-\lambda)I+\lambda J$ has eigenvalues $3r$ and $r-\lambda$;
if both are nonzero in $F$, then $NN^{\top}$ is nonsingular, so
$\rank N\ge\rank NN^{\top}=v\ge\rank N$, and $\rank N=v$.
(ii) The $j$-th entry of the sum of all rows is the number of points in block
$j$, namely $3$; dividing by $3$ (possible for $p
\ne3$) gives $\mathbf1$.
(iii) $C_2(\Tcal)$ is self-orthogonal iff the rows of $N$ are pairwise
orthogonal, iff $NN^{\top}=0$ over $\mathbb{F}_2$, iff (reading the diagonal
and off-diagonal entries of $(r-\lambda)I+\lambda J$) $r\equiv\lambda\equiv0
\pmod2$.  The family-by-family statements follow by substituting the orbit
parameters from \textup{Table \ref{tab:main}}; e.g.\ for
$\PSL(2,q)$, $q\equiv1\pmod4$: $\lambda=(q-1)/2$ is even and
$r=\lambda q/2=(q-1)q/4$ has the parity of $(q-1)/4$, so the code is
self-orthogonal exactly for $q\equiv1\pmod8$; for
non-collinear $\PSU(3,q)$, $q$ odd, $d=1$ or $3$, the parameters
$\lambda\in\{(q^2-1)/d,2(q^2-1)/d,6(q^2-1)/d\}$ are even (as $8\mid
q^2-1$ for odd $q$ and division by $d\in\{1,3\}$ does not remove the factor
$2$) and
$r=\lambda(q^3)/2$ is even (as $q^3$ is odd and $\lambda/2$ even by the same
divisibility), so self-orthogonality holds; the collinear $\PSU(3,q)$ orbit
has $r=(q-1)q^3/2$, which is odd exactly for $q\equiv3\pmod4$, giving the
claimed parity.
(iv) The rows have weight $r$; for two points $\alpha,\beta$ lying in a
common block, the difference $\mathbf r_\alpha-\mathbf r_\beta$ of their
rows has weight $2(r-\lambda)$ (the coordinates of the $\lambda$ common
blocks become $0$);  over characteristic $2$ the difference equals the sum.
Both words lie in $C_q(\Tcal)$, hence the minimum weight is at most the
smaller one.
\end{proof}

\begin{proposition}[Full rank and minimum weight of the unitary design codes]\label{prop:rank-minweight}
Let $\Tcal$ be a single orbit triple system with parameters $(v,3,\lambda)$,
$r=\lambda(v-1)/2$ and $v\ge5$, with incidence matrix $N$ over
$\mathbb{F}_2$.
\begin{enumerate}
\item \emph{Criteria.}  Two independent criteria hold.  First,
$\rank_{\mathbb{F}_2}N=v$ if and only if there is no nonempty subset
$S\subseteq\Pcal$ meeting every block in an even number of points
\textup{(an \emph{even transversal}:  $|S\cap X|$ even for every
$X\in\Tcal$)}\textup{)}.  Second, the minimum weight $d$ of $C_2(\Tcal)$
equals $r$ if and only if no word of weight $<r$ occurs among the nontrivial
combinations $xN$, $0\ne x\in\mathbb{F}_2^v$;  (the row space of $N$ is spanned by the
  point rows, whose weight is the generator weight $r$, while the minimum distance $d$ is the
  minimum weight over all nonzero codewords, i.e.\ over all nontrivial combinations $xN$;
  the two notions are distinguished explicitly)  the inequality $d\le r$ holds
unconditionally (the point rows have weight $r$), and for $v\ge5$ the upper
bound of \textup{Theorem \ref{thm:descode}(iv)} tightens to the single value
$r$ (as $\min\{r,2(r-\lambda)\}=r$, i.e.\
$\lambda(v-1)/2\le\lambda(v-3)$), so the criterion for $d=r$ reduces to the
absence of combinations of weight $<r$;  the finite searches of (2) below
verify this absence only for combinations of at most five rows, so they bound
the minimum weight from below in that range without determining the complete
minimum distance.
\item \emph{Computed values.}  For $\PSU(3,3)$ (all three orbits,
$\lambda=2,8,16$) and $\PSU(3,5)$ (all six orbits, $\lambda=4,8,48$) the
$\mathbb{F}_2$-rank of $N$ is $v$ ($28$ resp.\ $126$):  the design codes of
these orbit triple systems have \emph{full} $\mathbb{F}_2$-rank, and no even
transversal exists.  Moreover for $\PSU(3,3)$ every combination of at most five point rows has
weight at least $r$ ($27,108,216$) --- in particular none has weight $<r$;
the complete minimum distance is not determined by this finite search, as
stated.
(This is a finite search result in that range only;  a determination of the
true minimum distance would require a complete traversal of all $2^v-1$
combinations, a reliable minimum-distance algorithm or a theoretical lower
bound, and is not claimed here.)  The conjectured values $v-2$ or $v-3$
recorded in the literature are not attained by the full-rank codes in the
cases computed.
\item \emph{Consequences.}  In every case of (2) the all-ones word lies in
$C_2(\Tcal)$ \textup{(} $p=2$, $\mathbf1=\sum_{\alpha}\mathbf r_\alpha$,
as $3\equiv1\pmod2$ \textup{Theorem \ref{thm:descode}(ii)}\textup{)}, and in
the self-orthogonal cases
\textup{(Theorem \ref{thm:descode}(iii)(c))} the inclusions
$C_2(\Tcal)\subseteq C_2(\Tcal)^{\perp}$ with $\dim C_2(\Tcal)=v$ give
$\dim C_2(\Tcal)^{\perp}=b-v$, consistent with $b\ge2v$
\textup{($b=\lambda v(v-1)/6\ge2v$ for $v\ge13$, which holds in all cases
above)}\textup{)}.
\end{enumerate}
\end{proposition}
\begin{proof}
(1) $\rank N=v$ means the $v$ point rows are independent over
$\mathbb{F}_2$, i.e.\ $xN=0$ has only the trivial solution;  the $j$-th
coordinate of $xN$ is $|S\cap X_j|$ with $S=\mathrm{supp}(x)$, so a nonzero
solution is exactly an even transversal.  The second statement is the
definition of the minimum weight of the row space.  The equality
$\min\{r,2(r-\lambda)\}=r$ for $v\ge5$ is the elementary inequality
$(v-1)/2\le v-3$.
(2) The ranks are computed with GAP on the standard unital permutation
representations used throughout \textup{Chapter \ref{sec:verify}}:
\textup{AtlasGroup("U3(3)")} of degree $28$ and the primitive group of degree
$126$ and order $126\,000$ \textup{(} $\PSU(3,5)$, found in the GAP
primitive-group library\textup{)};  the incidence matrix of each orbit is
assembled from the orbit file and row-reduced over $\mathbb{F}_2$.  The
implementation is checked against the Fano plane \textup{STS}(7), whose design
code has the classical rank $4$ over $\mathbb{F}_2$.  The minimum-weight
values for $\PSU(3,3)$ are obtained by exhaustive search of the combinations
$xN$ with $x$ of support at most $5$, together with the upper bound $r$ of
(1);  the search establishes the absence of weight-$<r$ words in that range
only, as stated.
(3) follows from \textup{Theorem \ref{thm:descode}(ii),(iii)(c)} and the rank
values of (2).
\end{proof}

\begin{proposition}[All-ones word and the dual code]\label{prop:dual}
Let $\Tcal$ be a $G$-orbit triple system with parameters $(v,3,\lambda)$ and
$r=\lambda(v-1)/2$, and let $F=\mathbb{F}_q$ with $\operatorname{char}F=p$.
\begin{enumerate}
\item[(i)] (\emph{Exact generators of $C_q(\Tcal)$.)}  $C_q(\Tcal)$ is spanned
by the point rows $\mathbf r_\alpha$, $\alpha\in\Pcal$.  If $p
\ne3$, then the
all-ones word has the exact expression
\[
\mathbf1=3^{-1}\sum_{\alpha\in\Pcal}\mathbf r_\alpha\in C_q(\Tcal),
\]
and, provided moreover $p\nmid v$, one has
$C_q(\Tcal)=\langle\,\mathbf1,\ \mathbf r_\alpha-\mathbf r_\beta:
\alpha,\beta\in\Pcal\,\rangle$ (one fixed $\beta$ suffices;  the natural
condition here is $p\nmid v$, not $p\nmid r$).
\item[(ii)] (\emph{Membership of $\mathbf1$ in the dual.)}
$\mathbf1\in C_q(\Tcal)^{\perp}$ if and only if $p\mid r$.  For $p=2$ the
all-ones word therefore lies in $C_2(\Tcal)\cap C_2(\Tcal)^{\perp}$ if and
only if $r$ is even; in particular this holds in every self-orthogonal case
of \textup{Theorem \ref{thm:descode}(iii)}, e.g.\ the $M_{22}$ row ($r=210$) and the
non-collinear $\PSU(3,q)$ rows.
\item[(iii)] (\emph{Exact description of the dual.)}
\[
C_q(\Tcal)^{\perp}=\bigl\{\,u\in F^b:\ \sum_{X\in\Tcal,\ \alpha\in X}u_X=0
\text{ for every }\alpha\in\Pcal\,\bigr\},
\]
i.e.\ $u\in C_q(\Tcal)^{\perp}$ iff the point-block incidence equations
$Nu^{\top}=0$ hold entrywise.  When the rank criterion of
\textup{Theorem \ref{thm:descode}(i)} applies, $\dim C_q(\Tcal)^{\perp}=b-v$.  For
$p=2$, $C_2(\Tcal)^{\perp}$ is spanned by the characteristic vectors of the
\emph{even covers} of $\Pcal$: indeed the entrywise condition says exactly that
the support of $u$ covers every point an even number of times, and conversely
every even cover satisfies the entrywise equations.
\end{enumerate}
\end{proposition}

\begin{proof}
(i) The first assertion is the definition.  For $p
\ne3$, summing the rows
gives $\sum_\alpha\mathbf r_\alpha=3\mathbf1$ coordinatewise (block $X$ has
$3$ points), hence the displayed expression.  For the spanning assertion
assume additionally $p\nmid v$.  If $u=\sum_\alpha c_\alpha\mathbf
r_\alpha$ has zero coordinate sum, then
$0=\sum_Xu_X=3\sum_\alpha c_\alpha$, so (as $p\ne3$) $\sum_\alpha c_\alpha=0$ and
$u=\sum_\alpha c_\alpha(\mathbf r_\alpha-\mathbf r_\beta)\in
\operatorname{span}\{\mathbf r_\alpha-\mathbf r_\beta\}$.  Thus
$C_q(\Tcal)\cap\{\text{zero coordinate sum}\}$ equals the span of the
differences.  It remains to place $\mathbf r_\beta$ itself in
$\langle\mathbf1,\mathbf r_\alpha-\mathbf r_\beta\rangle$.  Summing
the coordinatewise identity $\sum_{\alpha\in\Pcal}\mathbf r_\alpha=3\mathbf1$
in the quotient
$F^b/\operatorname{span}\{\mathbf r_\alpha-\mathbf r_\beta\}$ gives
$v[\mathbf r_\beta]=3[\mathbf1]$, hence, since $p\nmid v$,
$\mathbf r_\beta-(3/v)\mathbf1\in
\operatorname{span}\{\mathbf r_\alpha-\mathbf r_\beta\}$ \textup{(equivalently,
by $3b=vr$, the coordinate sum of $\mathbf r_\beta-(r/b)\mathbf1$ is $0$,
so the same conclusion follows from the first part)}.  Therefore
$\mathbf r_\beta=(3/v)\mathbf1+
(\mathbf r_\beta-(3/v)\mathbf1)\in
\langle\mathbf1,\mathbf r_\alpha-\mathbf r_\beta\rangle$.  The reverse
inclusion is immediate.
(ii) $\mathbf1\cdot\mathbf r_\alpha=\sum_{X
\ni\alpha}1=r$ for every
$\alpha$, so $\mathbf1\in C_q(\Tcal)^{\perp}$ iff $r=0$ in $F$, i.e.\ iff
$p\mid r$.  For $p=2$ one has $p
\ne3$ automatically, so
$\mathbf1\in C_2(\Tcal)$ by (i) and $\mathbf1\in C_2(\Tcal)^{\perp}$ iff
$r$ is even.
(iii) $u\in C_q(\Tcal)^{\perp}$ iff $u\cdot\mathbf r_\alpha=0$ for all
$\alpha$, which is exactly the entrywise condition displayed.  The dimension
claim follows from $\dim C_q(\Tcal)=v$ under the criterion of
\textup{Theorem \ref{thm:descode}(i)}.  For $p=2$ the displayed condition is
equivalent to the support of $u$ being an even cover, so the dual is spanned
by their characteristic vectors.
\end{proof}

\begin{remark}[Block vectors versus point rows;  two conventions]\label{rem:not-selforth}
Two conventions for the linear code of a design coexist in the literature, and the distinction is kept throughout this monograph.  The standard design code of Assmus--Key is the span of the \emph{block vectors} in $\mathbb F_2^v$;  since every block of a triple system has weight $3$, that code is never self-orthogonal over $\mathbb F_2$.  In this monograph the design code $C_q(\Tcal)$ of \textup{Definition \ref{def:descode}} is the row space of the $v\times b$ incidence matrix (the span of the \emph{point rows} $\mathbf r_\alpha$ in $\mathbb F_2^b$), and its self-orthogonality is governed by the parity of $(r,\lambda)$, not by the block size;  in particular the design codes of the non-collinear $\PSU(3,q)$ orbits with $q$ odd are self-orthogonal over $\mathbb F_2$.  The constant-weight code $C_{cw}(\Tcal)$ of \textup{Section \ref{sec:coding}} is the (nonlinear) set of block vectors of weight $3$ in $\mathbb F_2^v$, for which self-orthogonality is not defined.  The two linear objects, the point-row code in $\mathbb F_2^b$ and the block code in $\mathbb F_2^v$, have the same dimension $\rank N$, but their self-orthogonality criteria are opposite in nature (parity of $r,\lambda$ versus the parity of the block size $3$);  the statements of \textup{Theorem \ref{thm:descode}(iii)} refer to the point-row code.
\end{remark}

\subsection{Design codes as modules}

Assume now $F$ is a field with $\operatorname{char}F
\nmid|G|$ and
$\operatorname{char}F
\nmid3r,\ r-\lambda$.  By
\textup{Theorem \ref{thm:descode}(i)} the map
\[
\varphi:\ F^{\Pcal}\longrightarrow C_q(\Tcal),\qquad
\mathbf e_a\longmapsto\text{row }a\text{ of }N,
\]
is a $G$-module isomorphism ($F^{\Pcal}$ carries the permutation action):
$\varphi(x)^{Q_g}=\varphi(x^{P_g^{-1}})$, so $\varphi$ is $G$-equivariant.
For $2$-transitive $G$, the permutation character of $F^{\Pcal}$ is
$1_G+\chi$ with $\chi$ irreducible (classically, $2$-transitivity forces the
existence of a unique invariant linear functional and irreducibility of the
$v-1$ dimensional part over $\mathbb{C}$, hence over fields of characteristic
zero; see \cite{dixon1996}).

\begin{theorem}[Structure of the design code as a $G$-module]\label{thm:module}
Let $\Tcal$ be a $G$-orbit triple system and $N$ its point-by-block incidence
matrix, so that the design code $C_q(\Tcal)$ of \textup{Definition \ref{def:descode}}
is the row space of $N$ in the block-coordinate space $\mathbb F_q^{\mathcal B}$
(with the induced permutation action of $G$ on the block coordinates).  Let $F$
be a field of characteristic $p$ with $p
\nmid|G|$ and $p
\nmid3r,\ r-\lambda$.
Then
$C_q(\Tcal)\cong F^{\Pcal}$ as $G$-modules, where $F^{\Pcal}$ is the
point-permutation module and the isomorphism is the $G$-equivariant
point-row map $\varphi$ of the proof.  Over $\mathbb{C}$,
\[
C_{\mathbb{C}}(\Tcal)\cong 1_G\oplus\pi
\]
with $\pi$ an irreducible $G$-module of dimension $v-1$, the unique
$v-1$-dimensional constituent of the permutation module; in particular the
design code decomposes into exactly two irreducible constituents.  We stress
that the design code $C_q(\Tcal)$ itself lives in the block-coordinate space
$F^{\bcal}$; the decomposition $1_G\oplus\pi$ describes its module structure
\emph{via the point--block isomorphism $\varphi$} of the proof, and in
particular it always carries the natural $G$-invariance under the induced
action on the block coordinates.  The irreducible decomposition is asserted
only under the stated assumptions on the $\ell$-rank of $N$, which guarantee that $\varphi$ is an
isomorphism.
\end{theorem}

\begin{proof}
The map $\varphi\colon C_q(\Tcal)\to F^{\Pcal}$ sending the point row
$\mathbf r_\alpha$ to the basis vector $e_\alpha$ is well defined and injective
because $\rank N=v$ under the stated assumptions (\textup{Theorem \ref{thm:descode}(i)}),
and it is $G$-equivariant: the identity $P_gN=NQ_g$ noted after
\textup{Definition \ref{def:descode}} intertwines the point and block permutation
actions, so $\varphi$ is an isomorphism of $G$-modules.  The decomposition
$1_G\oplus\pi$ is the classical structure of the permutation module of a
$2$-transitive group on $\Pcal$: its permutation character is $1+\chi$ with
$\chi$ irreducible of degree $v-1$ (the unique nontrivial constituent), so over
$\mathbb C$ the module is the direct sum of the trivial module and this
irreducible; see \cite{dixon1996, li2010}.
\end{proof}

\subsection{Summary of graph and code outputs}

The tables below are not an independent computation report: they are the
single output of the orbit classification of \textup{Theorem \ref{thm:main}}, applied
uniformly to the graph and code questions of \textup{Sections \ref{sec:graphs}--\ref{sec:coding}}.
Every entry is either proved from the orbit parameters alone or explicitly
marked as computed on the stated representations.  Throughout: $d$ is the
minimum distance of the constant-weight code $C_{cw}$ (always $2$ except in
the Steiner rows $\lambda=1$); ``self-orth'' refers to $C_2(\Tcal)$ over
$\mathbb{F}_2$; ``rank $=v$'' means the criterion of
\textup{Theorem \ref{thm:descode}(i)} applies with $p=2$ and gives $\dim C_2=v$
(the criterion is a sufficient condition;  no converse is claimed unless
independently computed;  when the criterion fails the parity is
displayed); ``arc'' refers to the
incidence graph $\Gamma(\Tcal)$ (``$\checkmark$'' = proved here,
``c'' = computed, see \textup{Remark \ref{rem:ver}}).

\begin{landscape}
\begingroup
\footnotesize
\setlength{\tabcolsep}{3pt}
\begin{longtable}{@{}p{3.3cm}p{2.3cm}p{2.8cm}p{2.1cm}p{3.4cm}p{0.9cm}p{2.6cm}p{1.7cm}@{}}
\caption{Outputs for the almost simple families (representative
parameters; general formulas in \textup{Section \ref{sec:params}} and
\textup{Theorem \ref{thm:main}}).}\label{tab:as}\\
\toprule
Family $(v)$ & $\lambda$ (orbits) & $b=|\Tcal|$ & $r$ & spectrum
$3r;\ r-\lambda$ & $d$ & self-orth / rank & arc\\
\midrule
\endfirsthead
\multicolumn{8}{c}{\small\itshape continued}\\
\toprule
Family $(v)$ & $\lambda$ (orbits) & $b$ & $r$ & spectrum & $d$ & self-orth / rank & arc\\
\midrule
\endhead
\bottomrule
\endfoot
\endlastfoot
$\PSL(2,q)$, $q\equiv1\pmod4$, $v=q+1$
& $\tfrac{q-1}{2}$ (2)
& $\tfrac{q(q^2-1)}{12}$
& $\tfrac{q(q-1)}4$
& $\tfrac{3q(q-1)}4$
\newline
$\tfrac{(q-1)(q-2)}4$
& 2
& self-orth $\iff q\equiv1\pmod8$;
criterion: rank $=v$ over $\F_2$ when $q\equiv5\pmod8$
& $\checkmark$ \\
$\PSL(2,5)$, $v=6$
& $2$ (2)
& $10$
& $5$
& $15;\ 3$
& 2
& not; rank $=6$
& $\checkmark$ \\
$\PSL(2,13)$, $v=14$
& $6$ (2)
& $182$
& $39$
& $117;\ 33$
& 2
& not; rank $=14$
& $\checkmark$ \\
$\PSL(2,q)$, $q$ even or $q\equiv3\pmod4$, $v=q+1$
& $q-1$ (1)
& $\tfrac{q(q^2-1)}6$
& $\tfrac{q(q-1)}2$
& $\tfrac{3q(q-1)}2$
\newline
$\tfrac{(q-1)(q-2)}2$
& 2
& not (parity);
rank $=v$ over $\F_2$ for $q\equiv3\pmod4$
& $\checkmark$ \\
$\PSL(2,7)$, $v=8$
& $6$ (1)
& $56$
& $21$
& $63;\ 15$
& 2
& not; rank $=8$
& $\checkmark$ \\
$\PGL(2,q)$/$\PpGL(2,q)$, $v=q+1$ (3-hom.)
& $q-1$ (1)
& $\tfrac{q(q^2-1)}6$
& $\tfrac{q(q-1)}2$
& same as above
& 2
& self-orth for $q\equiv1\pmod4$ (e.g.\ $q=5$: $r=10$, $\lambda=4$);
criterion: rank $=v$ over $\F_2$ when $q\equiv3\pmod4$
& $\checkmark$ \\
$\PSU(3,q)$, $v=q^3+1$, collinear
& $q-1$ (1)
& $\tfrac{(q-1)q^3(q^3+1)}6$
& $\tfrac{(q-1)q^3}2$
& $\tfrac{3(q-1)q^3}2$
\newline
$\tfrac{(q-1)(q^3-2)}2$
& 2
& self-orth $\iff q\equiv1\pmod4$;
criterion: rank $=v$ over $\F_2$ when $q\equiv3\pmod4$
& c (see \ref{rem:ver}) \\
$\PSU(3,3)$, $v=28$, collinear
& $2$ (1)
& $252$
& $27$
& $81;\ 25$ (integral)
& 2
& not; rank $=28$
& $\checkmark$ \\
$\PSU(3,q)$, $v=q^3+1$, non-collinear, $d=\gcd(3,q+1)$
& $\tfrac{q^2-1}d$ ($k$), $\tfrac{2(q^2-1)}d$, $\tfrac{6(q^2-1)}d$
& $\lambda v(v-1)/6$ per orbit
& $\lambda q^3/2$
& $3\lambda q^3/2;\ \lambda(q^3-2)/2$
& 2
& self-orth (all); criterion fails ($p=2$)
& c \\
$\PSU(3,3)$, non-coll.
& $8,\ 16$ (2 orbits)
& $1008,\ 2016$
& $108,\ 216$
& $324,\ 648$
\newline
$100,\ 200$ (first integral)
& 2
& self-orth; criterion fails
& $\checkmark$ \\
$\Sz(q)$, $v=q^2+1$, circle / trivial
& $3(q-1)$ (1)
\newline
$6(q-1)$ ($\tfrac{q-2}6$)
& $\tfrac{3(q-1)q^2(q^2+1)}6$
\newline
$\tfrac{6(q-1)q^2(q^2+1)}6$
& $\tfrac{3(q-1)q^2}2$
\newline
$3(q-1)q^2$
& $3$ times $r$; $r-\lambda$
& 2
& circle: not; trivial: self-orth; criterion fails
& c \\
$\Sz(8)$, $v=65$
& $21$ (1); $42$ (1)
& $14560;\ 29120$
& $672;\ 1344$
& $2016;\ 651$ etc.
& 2
& circle not; trivial self-orth
& 2; 3 \\
$\Ree(q)$, $v=q^3+1$
& $6(q-1)$ ($\tfrac{q(q-3)}6$),
\newline
$3(q-1)$ ($q$),
\newline
$2(q-1)$ ($\tfrac{q-1}2$),
\newline
$q-1$ (2)
& $\lambda v(v-1)/6$
& $\lambda q^3/2$
& $3\lambda q^3/2$
\newline
$\lambda(q^3-2)/2$
& 2
& self-orth for $\lambda\in\{6(q-1),2(q-1)\}$;
rank $=v$ over $\F_2$ for $\lambda\in\{q-1,3(q-1)\}$ ($q\equiv3\pmod4$)
& c \\
$\Ree(27)$, $v=19684$
& $156$ (108), $78$ (27), $52$ (13), $26$ (2)
& $\approx1.007{\times}10^{10}$
\newline
$\approx5.04{\times}10^{9}$
\newline
$\approx3.36{\times}10^{9}$
\newline
$\approx1.68{\times}10^{9}$
& $1535274$
\newline
$767637$
\newline
$511758$
\newline
$255879$
& see formulas
& 2
& as above ($\lambda=156,52$ self-orth; $\lambda=78,26$ rank $=v$)
& c \\
$\Sp(2n,2)$, $v=2^{2n-1}\pm2^{n-1}$
& $\lambda_0,\lambda_1$ (1 each)
& $\lambda_i v(v-1)/6$
& $\lambda_i(v-1)/2$
& formulas
& 2
& first: rank $=v$ over $\F_2$; second: self-orth
& c \\
$\Sp(6,2)$, $v=28$
& $10,\ 16$
& $1260;\ 2016$
& $135;\ 216$
& $405;\ 125$; $648;\ 200$
& 2
& rank $=28$; self-orth
& c \\
$\Alt(v)$/$\Sym(v)$, $v\ge5$
& $v-2$ (1)
& $\binom{v}{3}$
& $\tfrac{(v-1)(v-2)}2$
& $\tfrac{3(v-1)(v-2)}2$
\newline
$\tfrac{(v-2)(v-3)}2$
& 2
& parity of $r,v-2$; rank over $\F_2$ depends on $v\bmod4$
& $\checkmark$ \\
$M_{11}$, $v=11$
& $9$ (1)
& $165$
& $45$
& $135;\ 36$ (integral)
& 2
& not; criterion fails ($r-\lambda=36$ even)
& $\checkmark$ \\
$M_{12}$, $v=12$
& $10$ (1)
& $220$
& $55$
& $165;\ 45$
& 2
& not; rank $=12$
& $\checkmark$ \\
$M_{22}$, $v=22$
& $20$ (1)
& $1540$
& $210$
& $630;\ 190$
& 2
& self-orth; criterion fails
& $\checkmark$ \\
$M_{23}$, $v=23$
& $21$ (1)
& $1771$
& $231$
& $693;\ 210$
& 2
& not; criterion fails
& $\checkmark$ \\
$M_{24}$, $v=24$
& $22$ (1)
& $2024$
& $253$
& $759;\ 231$
& 2
& not; rank $=24$
& $\checkmark$ \\
$HS$, $v=176$
& $72,\ 90,\ 12$
& $369600$
\newline
$462000$
\newline
$61600$
& $6300$
\newline
$7875$
\newline
$1050$
& $18900;\ 6228$
\newline
$23625;\ 7785$
\newline
$3150;\ 1038$
& 2
& $72,12$: self-orth; $90$: rank $=176$
& c \\
$Co_3$, $v=276$
& $162,\ 112$
& $2049300$
\newline
$1416800$
& $22275$
\newline
$15400$
& $66825;\ 22113$
\newline
$46200;\ 15288$
& 2
& $162$: rank $=276$; $112$: self-orth
& c \\
\end{longtable}
\endgroup
\end{landscape}

\begin{landscape}
\begingroup
\footnotesize
\setlength{\tabcolsep}{3pt}
\begin{longtable}{@{}p{3.6cm}p{2.6cm}p{3.0cm}p{2.3cm}p{3.6cm}p{1.0cm}p{3.2cm}@{}}
\caption{Outputs for the affine families (formulas; representative
parameters).}\label{tab:aff}\\
\toprule
Family $(v)$ & $\lambda$ (orbits) & $b$ & $r$ & spectrum & $d$ & self-orth / arc\\
\midrule
\endfirsthead
\multicolumn{7}{c}{\small\itshape continued}\\
\toprule
Family $(v)$ & $\lambda$ (orbits) & $b$ & $r$ & spectrum & $d$ & self-orth / arc\\
\midrule
\endhead
\bottomrule
\endfoot
\endlastfoot
$\AGL(1,q)$, $v=q$
& $\lambda=6,3,2,1$ by orbit; e.g.\ $q=7$: $\lambda=3$ (one), $\lambda=2$ (one); $q\equiv0\pmod3$: an orbit with $\lambda=1$
& $\lambda q(q-1)/6$
& $\lambda(q-1)/2$
& $3\lambda(q-1)/2;\ \lambda(q-3)/2$
& 2 (4 for $\lambda=1$)
& parity of $(r,\lambda)$; rank by criterion; arc: c (see \ref{rem:ver})\\
$\AaL(1,q)$, $v=q$
& $\lambda=6f/|G_X|$ by orbit
& as above
& as above
& as above
& 2
& by criterion; arc: criterion\\
$\ASL(n,2)$/
\newline
$\AGL(n,2)$,
\newline
$v=2^n$
& $\lambda=2^n-2$ for $n\ge3$
& $\lambda v(v-1)/6$
& $\lambda(v-1)/2$
& formulas
& 2
& by criterion; arc: criterion\\
$\ASL(2,q)$ ($q$ even), $v=q^2$
& $q$ (1); also collinear row $\lambda\in\{1,2,3,6\}$ in \textup{Theorem \ref{thm:main}}
& $\lambda q^2(q^2-1)/6$
& $\lambda(q^2-1)/2$
& formulas
& 2 (4 if $\lambda=1$)
& by criterion; arc: criterion\\
$\AGL(2,q)$ ($q$ odd), $v=q^2$
& $2q$ (1)
& $2q\cdot q^2(q^2-1)/6$
& $q(q^2-1)$
& $3q(q^2-1);\ q(q^2-3)$
& 2
& by criterion; arc: criterion\\
$\AGL(n,q)$, $n\ge3$
& $q(q^{n-1}-1)$ (1)
& $\lambda q^n(q^n-1)/6$
& $\lambda(q^n-1)/2$
& formulas
& 2
& by criterion; arc: criterion\\
$\ASp(2a,q)$, $v=q^{2a}$
& $q^{2a-1}$ or $2q^{2a-1}$
& $\lambda v(v-1)/6$
& $\lambda(v-1)/2$
& formulas
& 2
& by criterion; arc: criterion\\
$G_2'(q)$, $v=q^6$
& collinear $\lambda=2,6$; hyperbolic $\lambda=q^5$; isotropic
$\lambda=q(q^2-1),q^3(q^2-1)$
\newline
(each of the listed types)
& $\lambda v(v-1)/6$
& $\lambda(v-1)/2$
& formulas
& 2
& by criterion (collinear $\lambda=2$: rank $=v$ over $\F_2$ for $q=4$); arc: criterion\\
Hering exceptions
& $\lambda\in\{6,3,2\},$
\newline
$\{8,6\},14,$
\newline
$\{12,1\},$
\newline
$\{32,24,6\},$
\newline
$\{18,6,1\}$
& $\lambda v(v-1)/6$
& $\lambda(v-1)/2$
& formulas
& 2 (4 if $\lambda=1$)
& by criterion; arc: criterion\\
\end{longtable}
\endgroup
\end{landscape}

\begin{remark}[Status of the ``arc'' entries]\label{rem:ver}
The entries marked ``$\checkmark$'' are proved: the $3$-homogeneous rows by
\textup{Corollary \ref{cor:arc-families}(i)}, and the $\PSL(2,q)$ rows by
\textup{Corollary \ref{cor:arc-families}(ii)}.  The entries marked ``c'' are verified
by direct computation on the stated representations: for $\PSU(3,3)$ each of
the three orbits has a single star orbit (hence is arc transitive, marked
``$\checkmark$'' in \textup{Table \ref{tab:as}}); for $\Sz(8)$ the circle orbit has
$2$ star orbits and the trivial orbit $3$ (hence neither is arc transitive);
for $\AGL(1,7)$ the $\lambda=2$ orbit has $1$ star orbit and the $\lambda=3$
orbit $2$.  The scripts (recorded in \textup{Chapter \ref{sec:verify}} together with the
verification data) reproduce each entry.  We do not claim unproved
arc-transitivity: for the unitary, Suzuki, Ree and sporadic rows the criterion
of \textup{Theorem \ref{thm:arc}} is the statement, and the computational entries are
explicitly bounded by the checked parameters.
\end{remark}

\section{Association schemes on triples}
\label{sec:assoc}

Throughout this section we distinguish strictly between a \emph{Schurian coherent configuration} and an \emph{association scheme}: a coherent configuration is an association scheme exactly when it is \emph{homogeneous}, i.e.\ has a single fiber (equivalently, $G$ is transitive on $X$); for a Schurian coherent configuration every orbital relation is then automatically regular and the intersection numbers close automatically (the orbital algebra is closed under the convolution product), so homogeneity is the only additional condition; the association-scheme claim for a particular family is made only when transitivity on triples is verified.  We record the connection with association schemes, following
\cite{zieschang2005}, the lecture notes \cite{muller2009} and the work of Mesner--Bhattacharya \cite{mesner1990} and
Balmaceda--Briones \cite{balmaceda2021}.  The numerical schemes of the known
$2$-transitive families were computed by Balmaceda and Briones \cite{balmaceda2021};
the contribution recorded here is the closed-form class number of \textup{Proposition \ref{prop:valencies}} below, its uniform dependence on the orbit parameters, and the Steiner-case link to the strongly regular block graphs of \textup{Theorem \ref{thm:stblock}(ii)}, none of which is extracted in \cite{balmaceda2021}.  Let
$X=\binom{\Pcal}{3}$ and let $G$ act on $X$ by $X^g=\{x^g,y^g,z^g\}$.

\begin{definition}[Schurian coherent configuration on triples]\label{def:ast}
The relations $R_i$ are the orbits of $G$ on $X\times X$ (the \emph{orbital
relations}):  $R_i$ is the set of \emph{ordered} pairs $(X,Y)\in X^2$ whose
relative position is fixed by the stabilizer, i.e.\ $(X,Y)$ and $(X',Y')$
lie in the same $R_i$ iff some $g\in G$ sends $(X,Y)$ to $(X',Y')$.  The
diagonal pairs $(X,X)$ form one relation per fiber $X_i^G$
($i=1,\ldots,\sigma_3$), which coincide into a single diagonal relation
$R_0$ when $G$ is transitive on $X$ (i.e.\ $\sigma_3=1$).  The data
$(X,\{R_0,\ldots,R_s\})$ define the \emph{Schurian coherent configuration}
$\mathfrak{X}(G)$ of $G$ on triples;  when every orbital relation is
homogeneous, this is an association scheme (the \emph{Schurian association
scheme}) on triples.
\end{definition}

This is a special case of the general construction of
\cite[Chapter 1]{muller2009}: for any permutation group $H$ on a set $X$, the
orbits of $H$ on $X\times X$ form a \emph{Schurian coherent configuration} (an
association scheme exactly when $H$ is transitive on $X$;  for a Schurian
configuration the intersection numbers close automatically), and its class
number equals the rank of $H$ on $X$.  For
$H=G$ on $\Pcal$ this scheme is the classical orbital scheme of class $1$ (by
$2$-transitivity); the non-trivial structure appears for $H=G$ on
$X=\binom{\Pcal}{3}$.

\begin{proposition}[Class number and valencies]\label{prop:valencies}
Let $X=\binom{\Pcal}{3}$, and let $X_1^G,\dots,X_{\sigma_3}^G$ be the
$G$-orbits on $X$, with representatives $X_i$ and stabilizers
$H_i=G_{X_i}$ ($\sigma_3$ as in \textup{Remark \ref{rem:char}}).
For each $i,j$, write
\[
\rho_{ij}=\#\bigl\{\text{$H_i$-orbits on }X_j^G\bigr\}\qquad(H_i=G_{X_i});
\]
$\rho_{ii}$ counts, in particular, the fixed point $X_i$ itself (the diagonal
pair $(X_i,X_i)$), so $\rho_{ii}\ge1$.  The relations of $\mathfrak{X}(G)$ are
the orbits of $G$ on $X\times X$;  each relation is determined uniquely by the
pair $(i,j)$ of its first- and second-coordinate fibers and by an $H_i$-orbit
on $X_j^G$ (fixing the representative $X_i$), so the class number is
\[
m=\#\{\text{orbital relations of }\mathfrak{X}(G)\}
=\sum_{i=1}^{\sigma_3}\sum_{j=1}^{\sigma_3}\rho_{ij}.
\]
The number of orbital relations is called the \emph{rank} of $G$ on $X$ only
in the transitive case $\sigma_3=1$;  for $\sigma_3>1$ the term \emph{rank}
is not used.
The diagonal contributes one relation per fiber (the trivial $H_i$-orbit
$X_i$), i.e.\ $\sigma_3$ diagonal relations in total.  The valencies of the
relations issued from the fiber $X_i^G$ are the sizes of the $H_i$-suborbits
on the target fibers $X_j^G$, whose sums are $|X_j^G|=|G|/|H_j|$, the orbit
sizes of \textup{Table \ref{tab:main}}.  The scheme is symmetric (every
relation is self-paired) if and only if $G$ is generously transitive on
triples; for $\Sym(\Pcal)$ the relations are the overlap classes
$\{|X\cap Y|=0,1,2,3\}$ and the scheme is symmetric.
\end{proposition}

\begin{proof}
The relations of a Schurian coherent configuration are the orbits of $G$ on $X\times X$, so
their number is the number of orbital relations (the rank of $G$ on $X$
when $\sigma_3=1$).  Fix the representatives $X_i$, one
per fiber.  Each relation $R$ has a unique first-coordinate fiber $X_i^G$
and a unique second-coordinate fiber $X_j^G$; the relations with these two
fibers are in bijection with the $H_i$-orbits on $X_j^G$ via
$R\longmapsto\{Y\in X_j^G:(X_i,Y)\in R\}$ (each pair $(X_i,Y)$ determines one
relation; distinct $H_i$-orbits give distinct $G$-orbits, and every $G$-orbit
arises from its unique first-coordinate fiber), whence
$m=\sum_{i=1}^{\sigma_3}\sum_{j=1}^{\sigma_3}\rho_{ij}$, as stated.  The
diagonal pairs $(X_i,X_i)$ form one orbit per fiber (the trivial $H_i$-orbit),
so the diagonal contributes $\sigma_3$ relations.  The valency of the
relation through $Y\in X_j^G$ is $|Y^{H_i}|$, and the $H_i$-suborbits
partition $X_j^G$, whose size is $|G|/|H_j|$ by the orbit--stabilizer
theorem.  Symmetry means that $(X,Y)$ and $(Y,X)$ lie in the same relation
for every pair, i.e.\ that $G_{X,Y}$ contains an element interchanging $X$
and $Y$; this is the classical notion of generous transitivity on triples
(see \cite[Chapter~1]{muller2009}).  For $\Sym(\Pcal)$ the $G_{X,Y}$-stabilizer
obviously contains such an element (any bijection swapping the two triples),
and the relation of a pair is determined by the overlap $|X\cap Y|$.
\end{proof}

\begin{remark}[Intersection numbers]
The intersection numbers $p_{ij}^k$ of $\mathfrak{X}(G)$ are determined by the
suborbits of the two-point stabilizers: writing $X=\{a,b,c\}$ and letting
$Y=\{a',b',c'\}$ run over the triples with fixed relative position to $X$,
the count depends on the $G_{\{a,b\}}$-suborbit data of
\textup{Proposition \ref{prop:pair-orbits}}.  The general closed formula is due to Mesner and
Bhattacharya \cite{mesner1990}; the specific numerical ASTs of the
$2$-transitive families are tabulated in \cite{balmaceda2021}.  In the
Steiner case $\lambda=1$, the block graph
$\overline{\Gamma}(\Tcal)$ of \textup{Theorem \ref{thm:stblock}(ii)} is a strongly
regular graph.  On the block fiber the graph relations induce a two-class
\emph{fusion} of the coherent configuration induced by $\mathfrak{X}(G)$ on
that fiber:  the adjacency relation ``meeting in exactly one point'' is in
general a union of $G$-orbital relations, and the graph is identified with an
orbital scheme only when that relation is a single orbital;  for the affine
STS$(q)$ under $\AGL(1,q)$, $q\equiv0\pmod3$, this recovers the
classical affine schemes.
\end{remark}

\begin{proposition}[Computational reduction of the intersection numbers]\label{prop:ast-closed}
Let $\Tcal_i$ be a single orbit triple system (block-transitive, with
representative $X_i$ and stabilizer $H_i=G_{X_i}$), and let
$\mathfrak{X}(G)$ be the Schurian coherent configuration of
\textup{Proposition \ref{prop:valencies}}.  The full intersection table
$\{p_{jk}^\ell\}$ of $\mathfrak{X}(G)$ can in principle be computed from the
following data --- a computational scheme rather than a closed
determination:  the two-point data below are essential inputs to the
Mesner--Bhattacharya formula \cite{mesner1990}, and no claim is made that
the data determine the table without it:
\begin{enumerate}
\item the $H_i$-suborbit (valency) arrays of \textup{Proposition \ref{prop:valencies}}, computed by the suborbit method (the stabilizer orders $|H_i|$ alone do not determine the orbit decomposition:  the $H_i$-action on the target fiber must be known, or the double cosets computed);
\item the two-point stabilizer suborbit data of \textup{Proposition \ref{prop:pair-orbits}}, which control the relative position of a pair $(Y,Z)$ of blocks above a fixed block $X_i$ (these are essential:  an intersection number counts blocks $Z$ with fixed position relative to an \emph{ordered} block pair, a two-point datum not recoverable from two-point orbit lengths alone);
\item in the Steiner case $\lambda=1$, the strongly regular block graph of \textup{Theorem \ref{thm:stblock}(ii)} \textup{(its $\lambda,\mu$ parameters are the relevant $p_{jk}^\ell$ above the one-point layer)}.
\end{enumerate}
The valency arrays of (1) alone do not determine the full table:  the two-point
  stabilizer suborbit data of (2) are essential, since $p_{jk}^\ell$ counts blocks whose
  position relative to an ordered block pair is fixed, a two-point datum.  The general
  closed formula realizing this reduction is Mesner--Bhattacharya's
\cite{mesner1990};  the numerical tables of the $2$-transitive families are in
\cite{balmaceda2021}, and the values of \textup{Proposition \ref{prop:astnum}}
($m=32$ for $\AGL(1,7)$, $m=1814$ for $\PSU(3,3)$, with the displayed
double-index arrays) are computed from the explicit actions and checked
against the orbit sizes.
For the Suzuki and Ree rows the ingredients are explicit:  the orbit
sizes are $|G|/|H_i|$ with the stabilizer orders of
\textup{Theorems \ref{thm:sz-orbits} and \ref{thm:ree-structure}}, and the
valency arrays can then be obtained by the suborbit method of (1);  the
full intersection tables of these rows would then follow from the
Mesner--Bhattacharya formula, whose two-point inputs remain to be computed
case by case.
\end{proposition}
\begin{proof}
The relations of $\mathfrak{X}(G)$ are the $G$-orbits on $X_i^G\times X_i^G$,
indexed by the $H_i$-suborbits (1).  An intersection number
$p_{jk}^\ell$ counts, for a fixed ordered pair $(X_i,Y)$ in the $j$-th
relation, the number of $Z$ in the $k$-th relation with $(X_i,Z)$ in the
$\ell$-th relation;  this is a count over triples $Z$ whose relative position
to both $X_i$ and $Y$ is fixed, and the position relative to the pair
$\{X_i,Y\}$ is governed by the two-point stabilizer suborbits (2), exactly
as in \textup{Proposition \ref{prop:pair-orbits}}.  The Steiner case reduces to
the block graph because for $\lambda=1$ the position of $Z$ relative to a
pair of blocks meeting in one point is decided by the $9$ or $(v-5)/2+4$
common neighbors of \textup{Theorem \ref{thm:stblock}(ii)}, giving (3).  The
remaining assertions record where the ingredients occur in this monograph.
\end{proof}

\begin{proposition}[Numerical class numbers (computed)]\label{prop:astnum}
For the affine row $\AGL(1,7)$ on the $35$ triples:  $\sigma_3=2$ and the
two orbits contribute the full double-index array
$\rho_{11}=6$, $\rho_{12}=7$, $\rho_{21}=7$, $\rho_{22}=12$;  the diagonal
suborbit (valency) arrays are $1^2\,3^{4}$ (on the $14$-block orbit,
$|G_X|=3$) and $1^3\,2^{9}$ (on the $21$-block orbit, $|G_X|=2$), and the
cross-fiber arrays are $3^{7}$ resp.\ $2^{7}$, so the class number is
\[
m=\rho_{11}+\rho_{12}+\rho_{21}+\rho_{22}=6+7+7+12=32;
\]
For $\PSU(3,3)$ on the $3276$ triples:  $\sigma_3=3$ and the full
double-index array is
\[
(\rho_{ij})=\begin{pmatrix}13&44&84\\44&193&338\\84&338&676\end{pmatrix},
\qquad
m=\sum_{i,j}\rho_{ij}=13+44+84+44+193+338+84+338+676=1814;
\]
the diagonal suborbit arrays are $1\,3\,8\,24^{10}$ (collinear orbit, $252$
blocks, $|G_X|=24$), $1^3\,3^{45}\,6^{145}$ ($1008$ blocks, $|G_X|=6$) and
$1^6\,3^{670}$ ($2016$ blocks, $|G_X|=3$), and the cross-fiber arrays are
$12^4\,24^{40}$ and $24^{84}$ (from the collinear stabilizer on the $1008$-
and $2016$-block fibers), $3^4\,6^{40}$ and $2^3\,6^{335}$ (from the
$1008$-block stabilizer), and $3^{84}$ and $1^3\,3^{335}$ (from the
$2016$-block stabilizer).  These numbers are computed with the explicit
$3$-dimensional unitary representation of $\PSU(3,3)$ over
$\mathbb{F}_9$ (the $28$ absolutely singular projective points;  the group
order $6048$ is verified), and each array sums to the corresponding orbit
size, which is a cross-check of the computation.
\end{proposition}

\begin{table}[ht]\centering\small
\caption{AST class numbers (the class-number column, from
\textup{Proposition \ref{prop:valencies}}).}\label{tab:ast}
\begin{tabular}{@{}lcll@{}}
\toprule
Family ($v$) & $\sigma_3$ & $m$ & status \\
\midrule
$\Sym(v)$ / $\Alt(v)$, $v\ge7$ & $1$ & $4$ & proved (overlap classes) \\
every family, $v\ge7$ & $\ge1$ & $\ge4$ & proved (overlap lower bound) \\
$\AGL(1,7)$ & $2$ & $32$ & computed (explicit affine action) \\
$\PSU(3,3)$ & $3$ & $1814$ & computed (explicit unitary action) \\
3-transitive rows ($\PGL(2,q)$, $M_{11}$, \dots) & $1$ & $\ge5$ &
criterion; exact value from the\\
& & & three-point stabilizer suborbits \\
\bottomrule
\end{tabular}
\end{table}

\chapter{Computational verification}
\label{sec:verify}
% ============================================================

The computations below have two different purposes.  First, the script checks internal identities such as the orbit--stabilizer relation and the exact-sequence kernel/image orders.  Second, for selected groups, the computed orbit data are compared with the formulas stated in the monograph.  The first kind of check is a consistency test and does not by itself prove the family-wide formulas;  all claims supported only by finite computation are stated with their exact degree and group range;  in particular every \emph{verified} entry depends on the specified software, its version, the input construction and the stated parameter range.  Concretely, the verification distinguishes seven levels:  (i) the GAP census of $2$-transitive primitive groups of degrees $3\le v\le200$ (the total is reported from the run output);  (ii) the dedicated suborbit computations for $HS$ and $Co_3$;  (iii) the degree-$729$ row $3^6{:}\SL(2,13)$, computed for the representation obtained from either $6$-dimensional faithful irreducible $\F_3$-module of $\SL(2,13)$ \textup{(Proposition \ref{prop:aff729})} by explicit generator-based construction \textup{(the explicit generators are recorded in the supplementary script for item (iii), whose output file lists the GAP version, the script name and its checksum)};  and (iv) the unitary suborbit computations for $\PSU(3,q)$ and $\PGU(3,q)$, for every $q\le27$, described in \textup{Section \ref{sec:psu3}} \textup{(Table \ref{tab:psu-general};  the stabilizer census of \textup{Table \ref{tab:magma}} covers $q\le49$)};  for the even-$d=3$ closed forms of \textup{Proposition \ref{prop:psu-even}} the multiplicity distribution and $\sigma_3=(5q+2)/6$ are verified independently for $q=8$ and $q=32$ in the script \texttt{psu\_verify\_even.g} \textup{(}$n_3=(q-2)/2=3,15$, $n_1=(q+1)/3=3,11$, $\sigma_3=7,27$;  $\texttt{psu\_verify\_even.out}$\textup{)}, and the odd-$d=3$ closed forms of \textup{Proposition \ref{prop:psu-odd}} ($3$, $(q-5)/2$, $(q+1)/3$;  $\sigma_3=(5q+11)/6$) are verified independently for $q=5,11,17,23$ in the script \texttt{psu\_odd\_d3\_verify.g} \textup{(}$n_6=3$ throughout, $n_3=(q-5)/2=0,3,6,9$, $\sigma_3=6,11,16,21$;  $\texttt{psu\_odd\_d3\_verify.out}$\textup{)};  and (v) the independent Ree class computation for $^2G_2(27)$ (degree $19684$, $|G|=10073444472$), which does \emph{not} enumerate triples:  the $35$ conjugacy classes are obtained from the Atlas representation of $\Ree(27)$ \textup{(\cite{conway1985,wilson1996})}, the cycle type of a class representative is read off by a direct orbit walk on the $19684$ points, and the class-function count of \textup{Remark \ref{rem:char}} then gives $\sum_g n(g)=150\,|G|$, hence $\sigma_3(\Ree(27))=150$, in agreement with \textup{Theorem \ref{thm:ree-triples}} ($(27+3)^2/6=150$);  and (vi) the Suzuki orbit-merge computations for $\Sz(32)$ and $\Sz(32){:}5$ (degree $1025$), in which the two-point setwise stabilizers are computed, one representative of every stabilizer orbit on the remaining points is selected, and the resulting three-subsets are fused under the full group by pairwise $\texttt{RepresentativeAction}$ tests:  this yields $\sigma_3(\Sz(32))=6$ and $\sigma_3(\Sz(32){:}5)=2$, in agreement with \textup{Theorem \ref{thm:sz-orbits}} ($(32+4)/6=6$) and with \textup{Theorem \ref{thm:suzuki-field-final}} ($t_F=1$);  and (vii) the $G_2(4)$ stabilizer computation for the affine row $2^{12}{:}G_2(4)$ (degree $4096$, $|G|=4096\,|G_2(4)|=1030540492800$), built from the $6$-dimensional $\F_4$-representation of $\texttt{AtlasGenerators("G2(4)",12)}$:  the point stabilizer $P_1$ of order $q^6(q^2-1)=61440$ ($q=4$) has $7$ orbits on $V\setminus\{0,e_1\}$, namely three of length $1024$ with pointwise stabilizer $60=q(q^2-1)$ \textup{(the hyperbolic classes, $h(e_1,w)
\ne 0$}\textup{)}, one of length $60$ with stabilizer $1024=q^5$ and one of length $960$ with stabilizer $64=q^3$ \textup{(the long-root resp.\ short-root isotropic classes, $h(e_1,w)=0$}\textup{)} plus the two fixed collinear points;  the affine group has $6$ three-subset orbits with setwise stabilizers $184320$ (collinear), $360$ (three hyperbolic), $6144$ and $384$ (the two isotropic ones), kernels $61440,60,1024,64$ and images $C_3,S_3,S_3,S_3$, giving $\lambda=2,1024,60,960$ in agreement with \textup{Theorem \ref{thm:g2-classical}}(3) \textup{(here $3\mid q-1$, so the collinear stabilizer $184320=3q^6(q^2-1)$ realizes the $C_3$ image of the $\{\omega,\omega^2\}$-orbit, and the hyperbolic stabilizer is $360=6\cdot60$ with image $S_3$}\textup{)}.  The same class input, applied to the closed-form Burnside sum of \textup{Theorem \ref{thm:ree-triples}}, yields $\sigma_3(\Ree(243))=10086=(243+3)^2/6$ and $\sigma_3(\Ree(3))=6$ \textup{(the latter independently by direct enumeration in Table \ref{tab:verify}}).  All computations were performed with GAP \cite{gap} (Version 4.16), using the primitive groups of the \textsc{primgrp} library:  each group is read as a permutation group, its orbits on $3$-subsets are computed directly, and for one representative of each orbit the program returns the orbit size, $|G_{\Ocal}|$, the kernel order and the \emph{order} of the stabilizer image in $S_3$ \textup{(the table records image orders only;  the induced permutation group on the three points is not identified unless explicitly stated}\textup{)};  the pointwise stabilizer is computed as the intersection of the three point stabilizers $G_{x_1}\cap G_{x_2}\cap G_{x_3}$ \textup{(an explicit check on sample groups confirms that this agrees with the $\texttt{Stabilizer}(G,X,\texttt{OnTuples})$ computation in the GAP version used}\textup{)};  the acting point set is passed explicitly to the counting routines \textup{(e.g.\ $\texttt{count3}(G,[1..v])$}\textup{)};  identifications of the tested groups with the named groups are made through the explicit constructions recorded here \textup{(explicit generator sets and $\texttt{PrimitiveGroup}$ identifiers where listed}\textup{)}, and degree and order alone are never used to identify a group.  Cross-checks are enforced by the program itself:  $|G|$ equals the theoretical group order;  $|G_{\Ocal}|\times$ (orbit size) $=|G|$;  the orbit sizes sum to $\binom{v}{3}$;  kernel $\times$ image $=|G_{\Ocal}|$;  the induced image is a subgroup of $S_3$;  and $\lambda=6|G|/\bigl(|G_{\Ocal}|\,v(v-1)\bigr)$ is an integer.  In addition, for every tested group the program forms all nonempty unions of the computed orbits and records their parameters $\lambda_I$, the subset sums of the orbit parameters (\textup{Proposition \ref{prop:union-orbits}});  distinct unions are reported as distinct block sets even when the parameters coincide.  Coverage:  the rows covered by general proofs are \textup{Theorem \ref{thm:psl2}} and its overgroups \textup{(Section \ref{sec:psl2}}\textup{)}, \textup{Theorem \ref{thm:psln}} and the collinear unitary row \textup{(Sections \ref{sec:psln}--\ref{sec:psu3}}\textup{)}, \textup{Theorem \ref{thm:sz}} for the socle $\Sz(q)$ \textup{(with the classical Suzuki-ovoid input of \cite{suzuki1962})}, \textup{Theorem \ref{thm:alt}}, the affine families of \textup{Section \ref{sec:affine}}, and the Ree orbit number of \textup{Theorem \ref{thm:ree-triples}};  the rows covered by finite computation are the sporadic rows $HS$, $Co_3$ and the Hering exceptions;  the rows checked only as a consistency test are the remaining rows inside the tested degree range;  the field-overgroup rows $\Ree(q)\cdot f'$ are covered by \textup{Theorem \ref{thm:ree-ext}} together with the fixed-orbit count of \textup{Lemma \ref{lem:ree-ext-cycles}} for every field quotient, and the case $q=27$ is independently verified;  the Suzuki field-overgroup row $\Sz(q)\cdot f'$ is covered by the closed form of \textup{Theorem \ref{thm:suzuki-field-final}} and is verified computationally for $q=8$ \textup{(Table \ref{tab:verify})} and $q=32$ \textup{(item (vi) above)};  the $\Sp(2n,2)$ rows for $n\ge4$ are proved family-wide by \textup{Theorem \ref{thm:sp2n-triples}} and lie outside the tested degree range.  The overgroup $HS.2$ is treated by the orbit--stabilizer argument of \textup{Proposition \ref{prop:hs2}} \textup{(Section \ref{sec:spor}}\textup{)}, whose arithmetic input is verified in GAP.

In addition, the census script was configured with $\texttt{TS\_MAX\_DEGREE}:=200$ and run in GAP \cite{gap}:  for every primitive group of degree $3\le v\le200$ with transitivity at least $2$ (taken from the GAP primitive-group library;  the final total is reported in the run summary), the orbits on $3$-subsets and, for each orbit, the order $|G_{\Ocal}|$, the kernel order and the stabilizer image in $S_3$ were computed;  the structural invariants (orbit size $\times|G_{\Ocal}|=|G|$;  the orbit sizes sum to $\binom{v}{3}$;  kernel $\times$ image $=|G_{\Ocal}|$;  image in $\{1,2,3,6\}$) hold with zero exceptions, and the design parameter $\lambda=6|G|/(|G_{\Ocal}|\,v(v-1))$ is integral on every orbit, as required by \textup{Lemma \ref{lem:lambda}} (a consistency check of the computation rather than a new classification fact).  For each explicitly computed representation, the output agrees with the orbit and stabilizer data displayed in the corresponding table.  These finite computations do not establish the asserted data for arbitrary field size and do not replace a proof of a family-wide classification.  Every entry marked as computed is supported by the corresponding finite computation;  there are no entries marked partial or open in the present version \textup{(Remark \ref{rem:open})}.  The GAP scripts, the full output of the census run, and the auxiliary computations ($HS$, $HS.2$, $Co_3$, $v=729$) are provided as supplementary material;  for reproducibility the scripts record the GAP version, the \textsc{primgrp} package version, the script file name and its checksum in the output files, and the acting point set is passed explicitly to the counting functions.

The results are collected in \textup{Table \ref{tab:verify}} (notation:  ``$a{:}k$ ($X$)'' means an orbit with $|G_{\Ocal}|=a$, occurring $k$ times, with stabilizer image $X$;  ``nc/col'' = non-collinear/collinear).

{\setlength{\tabcolsep}{3pt}%
\begin{longtable}{p{3.3cm}>{\raggedleft\arraybackslash}p{2.6cm}cp{4.6cm}}
\caption{Computational verification:  $|G|$, the number of orbits on three-point subsets, and the occurring stabilizer orders (with kernel and image) for the listed groups.}\label{tab:verify}\\
\toprule
$G$ (degree $v$) & $|G|$ & \#orbits & $|G_{\Ocal}|$: count (image) \\
\midrule
\endfirsthead
\multicolumn{4}{l}{\textup{(continued)}}\\
\toprule
$G$ (degree $v$) & $|G|$ & \#orbits & $|G_{\Ocal}|$: count (image) \\
\midrule
\endhead
\bottomrule
\endfoot
\endlastfoot
$\AGL(1,3)$ (3) & 6 & 1 & 6:1 ($S_3$) \\
$\AGL(1,5)$ (5) & 20 & 1 & 2:1 ($C_2$) \\
$\AGL(1,7)$ (7) & 42 & 2 & 2:1 ($C_2$), 3:1 ($C_3$) \\
$\AGL(1,4)$ (4) & 12 & 1 & 3:1 ($C_3$) \\
$\AGL(1,8)$ (8) & 56 & 1 & 1:1 \\
$\AGL(1,9)$ (9) & 72 & 2 & 1:1, 6:1 ($S_3$) \\
$\AGL(1,16)$ (16) & 240 & 3 & 1:2, 3:1 ($C_3$) \\
$\AGL(1,27)$ (27) & 702 & 5 & 1:4, 6:1 ($S_3$) \\
$\AaL(1,9)$ (9) & 144 & 2 & 2:1 ($C_2$), 12:1 ($S_3$) \\
$\AaL(1,8)$ (8) & 168 & 1 & 3:1 ($C_3$) \\
$\AaL(1,27)$ (27) & 2106 & 3 & 1:1, 3:1 ($C_3$), 18:1 ($S_3$) \\
$\ASL(2,2)=\AGL(2,2)$ (4) & 24 & 1 & 6:1 ($S_3$) \\
$\ASL(2,3)$ (9) & 216 & 2 & nc: 3:1 ($C_3$);  col: 18:1 ($S_3$) \\
$\AGL(2,3)$ (9) & 432 & 2 & nc: 6:1 ($S_3$);  col: 36:1 ($S_3$) \\
$\AGL(3,2)$ (8) & 1344 & 1 & 24:1 ($S_3$) \\
$\PSL(2,5)$ (6) & 60 & 2 & 6:2 ($S_3$) \\
$\PSL(2,7)$ (8) & 168 & 1 & 3:1 ($C_3$) \\
$\PSL(2,8)$ (9) & 504 & 1 & 6:1 ($S_3$) \\
$\PSL(2,9)$ (10) & 360 & 2 & 6:2 ($S_3$) \\
$\PSL(2,11)$ (12) & 660 & 1 & 3:1 ($C_3$) \\
$\PGL(2,5)$ (6) & 120 & 1 & 6:1 ($S_3$) \\
$\PGL(2,7)$ (8) & 336 & 1 & 6:1 ($S_3$) \\
$\PGL(2,8)$ (9) & 504 & 1 & 6:1 ($S_3$) \\
$\PGL(2,9)$ (10) & 720 & 1 & 6:1 ($S_3$) \\
$\PpGL(2,9)$ (10) & 1440 & 1 & 12:1 ($S_3$) \\
$\PpGL(2,25)$ (26) & 31200 & 1 & 12:1 ($S_3$) \\
$\PSL(2,9)\cdot2$ (10) & 720 & 2 & 12:2 ($S_3$) \\
$\PSL(2,25)\cdot2$ (26) & 15600 & 2 & 12:2 ($S_3$) \\
$\PSL(2,27)\cdot3$ (28) & 29484 & 1 & 9:1 ($C_3$) \\
$\PSU(3,2)$ (9) & 72 & 2 & nc: 1:1;  col: 6:1 ($S_3$) \\
$\Alt(4)$ (4) & 12 & 1 & 3:1 ($C_3$) \\
$\Alt(5)$ (5) & 60 & 1 & 6:1 ($S_3$) \\
$\Alt(6)$ (6) & 360 & 1 & 18:1 ($S_3$) \\
$\Alt(7)$ (7) & 2520 & 1 & 72:1 ($S_3$) \\
$\Sym(5)$ (5) & 120 & 1 & 12:1 ($S_3$) \\
$\Sym(6)$ (6) & 720 & 1 & 36:1 ($S_3$) \\
$\Sym(7)$ (7) & 5040 & 1 & 144:1 ($S_3$) \\
$2^4{:}\Alt(7)$ (16) & 40320 & 1 & 72:1 ($S_3$) \\
$\Alt(7)$ on 15 vectors (15) & 2520 & 2 & 6:1 ($S_3$), 72:1 ($S_3$) \\
$2^4{:}\Alt(6)$ (16) & 5760 & 2 & 18:1 ($S_3$), 24:1 ($S_3$) \\
$\AGL(2,5)$, linear part $\SL(2,3)$ (25) & 600 & 7 & 1:2, 2:1 ($C_2$), 3:4 ($C_3$) \\
$\AGL(2,11)$, linear part $\SL(2,5)$ (121) & 14520 & 27 & 1:16, 2:1 ($C_2$), 3:10 ($C_3$) \\
$\ASL(2,4)$ (16) & 960 & 4 & 6:3 ($S_3$), 12:1 ($C_3$) \\
$\AGL(2,4)$ (16) & 2880 & 2 & 6:1 ($S_3$), 36:1 ($C_3$) \\
$\ASL(2,5)$ (25) & 3000 & 3 & col: 10:1 ($C_2$);  nc: 3:2 ($C_3$) \\
$\AGL(2,5)$ (25) & 12000 & 2 & col: 40:1 ($C_2$);  nc: 6:1 ($S_3$) \\
$\ASL(2,7)$ (49) & 16464 & 5 & col: 21:1 ($C_3$), 14:1 ($C_2$);  nc: 3:3 ($C_3$) \\
$\AGL(2,7)$ (49) & 98784 & 3 & col: 126:1 ($C_3$), 84:1 ($C_2$);  nc: 6:1 ($S_3$) \\
$\ASL(2,8)$ (64) & 32256 & 8 & col: 8:1;  nc: 6:7 ($S_3$) \\
$\ASL(3,3)$ (27) & 151632 & 2 & col: 1296:1 ($S_3$);  nc: 54:1 ($S_3$) \\
$\AGL(3,3)$ (27) & 303264 & 2 & col: 2592:1 ($S_3$);  nc: 108:1 ($S_3$) \\
$\ASL(3,4)$ (64) & 3870720 & 2 & col: 2880:1 ($C_3$);  nc: 96:1 ($S_3$) \\
$\ASL(6,2)$ (64) & 1290157424640 & 1 & 30965760:1 ($S_3$) \\
$\ASp(4,2)$ (16) & 11520 & 2 & 36:1 ($S_3$), 48:1 ($S_3$) \\
$\ASp(6,2)=2^6{:}\Sp(6,2)\cong2^6{:}O(7,2)$ (64) & 92897280 & 2 & 4320:1 ($S_3$), 4608:1 ($S_3$) \\
$2^6{:}\PSU(3,3)$ (64) & 387072 & 3 & 18:1 ($C_3$), 24:1 ($S_3$), 96:1 ($S_3$) \\
$2^6{:}(\PSU(3,3){:}C_2)$ (64) & 774144 & 3 & 48:1 ($S_3$), 192:1 ($S_3$), 36:1 ($S_3$) \\
$3^4{:}\SL(2,9)$ (81) & 58320 & 6 & 54:1 ($S_3$), 3:4 ($C_3$), 9:1 \\
$3^4{:}(2.S_5)$ (81) & 19440 & 8 & 18:1 ($S_3$), 1:3, 3:3 ($C_3$), 3:1 \\
$3^4{:}\SL(4,3)$ (81) & 982575360 & 2 & col: 909792:1 ($S_3$);  nc: 11664:1 ($S_3$) \\
$3^4{:}\GL(4,3)$ (81) & 1965150720 & 2 & col: 1819584:1 ($S_3$);  nc: 23328:1 ($S_3$) \\
$\PSL(3,2)$ (7) & 168 & 2 & col: 24:1 ($S_3$);  nc: 6:1 ($S_3$) \\
$\PSL(3,3)$ (13) & 5616 & 2 & col: 108:1 ($S_3$);  nc: 24:1 ($S_3$) \\
$\PSL(3,4)$ (21) & 20160 & 2 & col: 96:1 ($S_3$);  nc: 18:1 ($S_3$) \\
$\PSL(3,5)$ (31) & 372000 & 2 & col: 600:1 ($S_3$);  nc: 96:1 ($S_3$) \\
$\PSL(3,7)$ (57) & 1876896 & 2 & col: 588:1 ($S_3$);  nc: 72:1 ($S_3$) \\
$\PSL(3,8)$ (73) & 16482816 & 2 & col: 2688:1 ($S_3$);  nc: 294:1 ($S_3$) \\
$\PSL(3,9)$ (91) & 42456960 & 2 & col: 3888:1 ($S_3$);  nc: 384:1 ($S_3$) \\
$\PSL(4,2)\cong\Alt(8)$ (15) & 20160 & 2 & col: 576:1 ($S_3$);  nc: 48:1 ($S_3$) \\
$\PSL(4,3)$ (40) & 6065280 & 2 & col: 11664:1 ($S_3$);  nc: 648:1 ($S_3$) \\
$\PSL(4,4)$ (85) & 987033600 & 2 & col: 276480:1 ($S_3$);  nc: 10368:1 ($S_3$) \\
$\PSL(5,2)$ (31) & 9999360 & 2 & col: 64512:1 ($S_3$);  nc: 2304:1 ($S_3$) \\
$\PSL(6,2)$ (63) & 20158709760 & 2 & col: 30965760:1 ($S_3$);  nc: 516096:1 ($S_3$) \\
$\Sp(6,2)\cong O(7,2)$ (28) & 1451520 & 2 & 1152:1 ($S_3$), 720:1 ($S_3$) \\
$\Sp(6,2)\cong O(7,2)$ (36) & 1451520 & 2 & 432:1 ($S_3$), 384:1 ($S_3$) \\
$\PSU(3,3)$ (28) & 6048 & 3 & col: 24:1 ($S_3$);  nc: 3:1 ($C_3$), 6:1 ($S_3$) \\
$\PSU(3,4)$ (65) & 62400 & 3 & col: 30:1 ($S_3$);  nc: 3:2 ($C_3$) \\
$\Sz(8)$ (65) & 29120 & 2 & 1:1, 2:1 ($C_2$) \\
$\Sz(8)\cdot3$ (65) & 87360 & 2 & 3:1 ($C_3$), 6:1 ($C_2$) \\
$\PpGL(2,8)\cong{}^2G_2(3)$ (28) & 1512 & 6 & 2:1 ($C_2$), 3:1 ($C_3$), 2:1 ($C_2$), 6:1 ($S_3$), 2:1 ($C_2$), 6:1 ($C_3$) \\
$M_{11}$ (11) & 7920 & 1 & 48:1 ($S_3$) \\
$M_{11}$ (12) & 7920 & 1 & 36:1 ($S_3$) \\
$M_{12}$ (12) & 95040 & 1 & 432:1 ($S_3$) \\
$M_{22}$ (22) & 443520 & 1 & 288:1 ($S_3$) \\
$M_{22}.2$ (22) & 887040 & 1 & 576:1 ($S_3$) \\
$M_{23}$ (23) & 10200960 & 1 & 5760:1 ($S_3$) \\
$M_{24}$ (24) & 244823040 & 1 & 120960:1 ($S_3$) \\
$HS$ (176) & 44352000 & 3 & 120:1 ($S_3$), 96:1 ($S_3$), 720:1 ($S_3$) \\
$Co_3$ (276) & 495766656000 & 2 & 241920:1 ($S_3$), 349920:1 ($S_3$) \\
$^2G_2(27)$ (19684) & 10073444472 & 150 & class count:  $\sum_g n(g)=150\,|G|$ \textup{(Remark \ref{rem:char};  Theorem \ref{thm:ree-triples});  the stabilizer distribution is verified separately by the suborbit computation discussed below \textup{(Table \ref{tab:verify-ree})}} \\
$2^{12}{:}G_2(4)$ (4096) & 1030540492800 & 6 & col: 184320:1 ($C_3$), $\lambda=2$;  hyp: 360:3 ($S_3$), $\lambda=1024$;  iso: 6144:1 ($S_3$), $\lambda=60$;  384:1 ($S_3$), $\lambda=960$ \\
\bottomrule
\end{longtable}}

These numbers agree, row by row, with the statements of \textup{Sections \ref{sec:psl2}--\ref{sec:affine}}:  e.g.\ the $\PSL(2,q)$-rows confirm \textup{Theorem \ref{thm:psl2}} (including the $q\equiv1\pmod4$ two-orbit $S_3$ case and the $q\equiv3\pmod4$ single-orbit $C_3$ case);  the $\PpGL(2,q)$-rows confirm the twisted-kernel statement $|G_{\Ocal}|=6f$ of \textup{Theorem \ref{thm:pgl}(2)} for \emph{every} triple;  the $2^4{:}\Alt(6)$, $2^4{:}\Alt(7)$ and $\Alt(7)$-on-$15$ rows confirm \textup{Theorem \ref{thm:affine-exc}};  the two Hering-exception rows ($\AGL(2,5)$ with $\SL(2,3)$, $\AGL(2,11)$ with $\SL(2,5)$) were obtained from explicit generator sets realizing the linear parts and confirm the orbit decompositions claimed there.  The $\ASp(6,2)$ row is the Hering class (c) group $2^6{:}\Sp(6,2)$ ($=2^6{:}O(7,2)$, $\mathrm{PrimitiveGroup}(64,55)$), confirming that no Hering family is missing from the classification;  the $\Sz(8)\cdot3$ row verifies the field-overgroup closed form of \textup{Section \ref{sec:sz}} \textup{(Theorem \ref{thm:suzuki-field-final})};  and $|\PpGL(2,25)|=15600\cdot2=31200$ as listed.  The unitary rows $\PSU(3,3)$ (degree $28$) and $\PSU(3,4)$ (degree $65$) confirm \textup{Theorem \ref{thm:main-psu}} and \textup{Table \ref{tab:magma}} within the census range;  the degrees $q^3+1=730,19684$ of $q=9,27$ lie outside that range and are verified by the dedicated suborbit computations of \textup{Table \ref{tab:psu-pgu}} \textup{(Section \ref{sec:psu3})}.

The $\Ree(27)$ row of \textup{Table \ref{tab:verify}} is computed without enumerating the $\binom{19684}{3}\approx1.27\times10^{12}$ triples.  From the Atlas representation of $\Ree(27)$ \textup{(\cite{conway1985,wilson1996})} a permutation representation of degree $19684$ is constructed, and the straight-line program for the class representatives of the $35$ conjugacy classes \textup{(CharacterTable("2G2(27)") in GAP)} produces one representative per class.  For each representative the cycle type on the $19684$ points is computed by a direct orbit walk \textup{($p:=p^g$ with a visited-marker array;  this avoids the memory blow-up of $O(v\log v)$ structure algorithms)}.  The resulting cycle types are:  $1^{19684}$ (identity);  $1^{28}2^{9828}$ ($J$);  $1^13^{6561}$ ($X,T,T^{-1}$);  $1^19^{2187}$ ($Y,YT,YT^{-1}$);  $1^13^{9}6^{3276}$ ($JT,JT^{-1}$);  $7^{2812}$ ($S^a$, $7\mid q+1$);  $1^213^{1514}$ ($R^a$, $13\mid q-1$);  $1^213^226^{756}$ ($JR^a$);  $7^414^{1404}$ ($JS^a$);  $19^{1036}$, $37^{532}$ (elements of the Hall subgroups, $19,37\mid q^2-q+1$).  Substituting these into the class sum of \textup{Remark \ref{rem:char}} gives $\sum_g n(g)=150\,|G|$, hence $\sigma_3(\Ree(27))=150$, matching \textup{Theorem \ref{thm:ree-triples}}.  The closed form of \textup{Theorem \ref{thm:ree-triples}} applied to the same class input gives $\sigma_3(\Ree(243))=10086$;  the value $6$ for $q=3$ is the direct enumeration of the $\PpGL(2,8)$ row.

The stabilizer-order distribution of \textup{Theorem \ref{thm:ree-structure}} is verified for $q=27$ by a direct suborbit computation that likewise avoids the enumeration of all triples:  from the degree-$19684$ Atlas representation of $\Ree(27)$ the point stabilizer $S_1$ $|S_1|=q^3(q-1)=511758$ and the two-point stabilizers $S_{1,2}$ $(|S_{1,2}|=26)$ and $S_{\{1,2\}}=\langle S_{1,2},\tau\rangle$ $(|S_{\{1,2\}}|=52$, dihedral$)$ are constructed, where $\tau$ is the transposition swapping $1$ and $2$;  the orbits of $S_{\{1,2\}}$ on the remaining $19682$ points are then $393=1+28+364$ in number, with lengths $26$ ($29$ orbits:  the collinear orbit and the $28$ invariant orbits of the block $L'$) and $52$ ($364$ paired orbits).  For one representative of each $S_{\{1,2\}}$-orbit the occurrence of the three transpositions $(12),(13),(23)$ and of a $3$-cycle inside $S_{\{1,2,x\}}$ is tested, giving the $m$-distribution $[1,324],[2,54],[3,14],[6,1]$;  by the fusion analysis of \textup{Theorem \ref{thm:ree-structure}} this yields $N_1=324/3=108$, $N_2=54/2=27$ \textup{(the paired transposition suborbits, fused two-to-one with the $27$ invariant transposition orbits among the $28$ invariant ones)}, $N_3=14-1=13$ \textup{(the paired $3$-cycle orbits;  the remaining $3$-cycle orbit is the invariant one)}, $N_6=2$, in agreement with $(108,27,13,2)$ of \textup{Theorem \ref{thm:ree-structure}}.  The arithmetic is collected in \textup{Table \ref{tab:verify-ree}}.

For the full automorphism group $A=\Ree(27){:}3$ the construction uses the Atlas generator of the split extension \textup{(AtlasGenerators("R(27).3",1);  $|A|=30220333416=3|S|$)}.  The field automorphism $\sigma$ (generator of order $3$) acts on the $19684$ points with cycle type $1^{28}3^{6552}$, fixing exactly the $28=q_0^3+1$ $\F_3$-rational points \textup{($q_0=3$)};  choosing two fixed points $a,b$, the setwise two-point stabilizer has order $|A_{\{a,b\}}|=156=2(q-1)f'$ and its $137$ orbits on the remaining points have lengths $26\times5$, $52\times4$, $78\times8$, $156\times120$ (summing to $19682$).  Fusing these suborbits by $\texttt{RepresentativeAction}$ in $A$ gives $\sigma_3(A)=56$ and the stabilizer distributions of \textup{Theorem \ref{thm:ree-ext}}:  setwise $1^{35},2^8,3^7,6^3,9^1,18^2$ and pointwise $1^{50},3^5,6^1$.  This independently re-verifies the value $56$ of the orbit count and, together with the socle computation, verifies the whole of \textup{Theorem \ref{thm:ree-ext}} for $q=27$;  the scripts and full outputs are provided as supplementary material (\texttt{GAP\_ree27\_suborbits2.g}, \texttt{GAP\_ree27\_final.g}, \texttt{GAP\_ree27\_ext\_fusion.g}).

\begin{table}[htbp]
\centering
\footnotesize
\renewcommand{\arraystretch}{0.95}
\caption{Direct verification of \textup{Theorem \ref{thm:ree-structure}} and \textup{Theorem \ref{thm:ree-ext}} for $q=27$, from the degree-$19684$ Atlas representations of $\Ree(27)$ and $\Ree(27){:}3$ \textup{(Chapter \ref{sec:verify})}.  All values are computed without enumerating the $\binom{19684}{3}$ triples.}\label{tab:verify-ree}
\resizebox{\textwidth}{!}{%
\begin{tabular}{@{}>{\raggedright\arraybackslash}p{5.4cm}>{\raggedright\arraybackslash}p{4.0cm}>{\raggedright\arraybackslash}p{4.4cm}@{}}
\toprule
quantity & value & comment \\
\midrule
$|S_{\{1,2\}}|$, $S_{\{1,2\}}$ & $52$, $D_{52}$ & dihedral two-point setwise stabilizer, $\tau$ swaps $1,2$ \\
$S_{\{1,2\}}$-suborbits on the remaining $19682$ points & $393=1+28+364$ & lengths $26$ ($29$ orbits) and $52$ ($364$ orbits) \\
$m$-distribution (setwise orders $1,2,3,6$ at suborbit level) & $[1,324],[2,54],[3,14],[6,1]$ & by the transposition and $3$-cycle tests of \textup{Theorem \ref{thm:ree-structure}} \\
$(N_1,N_2,N_3,N_6)$ & $(108,27,13,2)$ & fused;  matches the closed forms $q(q-3)/6,q,(q-1)/2,2$ \\
$|A_{\{a,b\}}|$ & $156=2(q-1)f'$ & $a,b$ the two fixed points of $\sigma$, $f'=3$ \\
$A_{\{a,b\}}$-suborbits & $137$ & lengths $26\times5,52\times4,78\times8,156\times120$ \\
$\sigma_3(\Ree(27){:}3)$ & $56$ & computed and independently re-verified in this monograph \\
setwise / pointwise stabilizer distributions of $A$ & $1^{35},2^8,3^7,6^3,9^1,18^2$ / $1^{50},3^5,6^1$ & on the $56$ orbits \\
\bottomrule
\end{tabular}}
\end{table}

The Mathieu rows above are machine-computed (GAP):  $M_{11}$ (degree $11$ and $12$), $M_{12}$, $M_{22}$, $M_{22}.2$, $M_{23}$, $M_{24}$, each a single orbit with the listed $|G_{\Ocal}|$ and image $S_3$, matching \textup{Theorem \ref{thm:mathieu}}.  For $HS$ the full suborbit computation for the representation $\mathrm{PrimitiveGroup}(176,4)$ on the $174$ remaining points, starting from the two-point stabilizer of order $1440$ (Atlas \cite{conway1985,wilson1996}), was carried out in GAP:  $HS$ has three orbits (\textup{Table \ref{tab:verify}}, $\lambda=72,90,12$).  The three $HS$-orbits have block counts $61600,369600,462000$, each dividing $|HS.2|=88704000$ with quotient $1440,240,192$, while no nontrivial union of these counts divides $|HS.2|$;  combined with the normalizer checks inside $\Aut(HS)$ this yields the $HS.2$ data of \textup{Proposition \ref{prop:hs2}} (\textup{Section \ref{sec:spor}}\textup{)}.  The $Co_3$ degree-$276$ representation was constructed directly from the Atlas generators \textup{(\texttt{Co3G1-p276aB0}, \cite{conway1985,wilson1996})} and the three-point orbits computed in GAP \textup{(supplementary script \texttt{co3\_orbits.g})}:  exactly two orbits, with setwise stabilizers of order $241920,349920$, pointwise stabilizers $40320,58320$, induced image $S_3$ in both cases, and $\lambda=162,112$ \textup{(Proposition \ref{prop:co3})};  the block counts $2049300,1416800$ sum to $\binom{276}{3}$, and the class-function count $\sum_{g\in G}n(g)=991533312000=2|G|$ gives $r=2$ independently \textup{(Remark \ref{rem:char})}.

% ============================================================
\begin{remark}[Proof status and remaining open cases]
\label{rem:open}
The general orbit-to-design correspondence is complete for every fixed $2$-transitive permutation group \textup{(Theorem \ref{thm:main}(1)--(4)}\textup{)}.  For the family-specific data we distinguish four classes of statements:
\begin{enumerate}
\item \textbf{Proved for all admissible parameters.}  The classical families \textup{(Sections \ref{sec:psl2}--\ref{sec:affine})}:  $\PSL(2,q)$ and its projective overgroups, $\PSL(n,q)$, $n\ge3$, on projective space, $\Sp(2n,2)$, $\Alt(n)$, the Mathieu groups (all $3$-transitive), $\AGL(1,q)$, $\AaL(1,q)$, the classical affine linear and symplectic rows, and the $G_2'(q)$ rows of \textup{Theorem \ref{thm:g2-classical}} \textup{(}$q=2^f\ge4$\textup{)};  together with the orbit numbers of $\Sz(q)$ \textup{(Theorem \ref{thm:sz-orbits})}, of $\Ree(q)$ \textup{(Theorem \ref{thm:ree-triples};  with the stabilizer distribution of \textup{Theorem \ref{thm:ree-structure}}}\textup{)}, of $\PSU(3,q)$ \textup{(Theorem \ref{thm:main-psu}(3), closed forms in the $d=1$, the odd-$d=3$ \textup{(Proposition \ref{prop:psu-odd})} and the even-$d=3$ \textup{(Proposition \ref{prop:psu-even})} rows, with the small-$q$ stabilizer refinement of \textup{Proposition \ref{prop:psu-computed}}}\textup{)}, the Suzuki field extensions \textup{(Theorem \ref{thm:suzuki-field-final})} and the Ree field extensions for every field quotient $f'\mid 2m+1$:  cubic degree $f'=3$ by \textup{Theorem \ref{thm:ree-ext}} ($a=(q_0^2+2q_0+3)/2$, Lemma \ref{lem:ree-ext-cycles}), and $f'\ge5$ by the descent argument of \textup{Proposition \ref{prop:ree-descent}} ($a=(q_0^2+2q_0+3)/2$ for $3\mid f'$, $a=(q_0+3)^2/6$ for $3
\nmid f'$, Proposition \ref{prop:ree-ext-noncubic}), the diagonal overgroup orbit number $\sigma_3(\PGU(3,q))$ \textup{(Proposition \ref{prop:psu-diag}, proved for every $q$)}\textup{)}, and the field conjugation of $\PpGU(3,q)$ (proved to cause no fusion for every $q\ge3$ by \textup{Proposition \ref{prop:psu-field-nofusion}}).
\item \textbf{Computed for an explicit range.}  The sporadic rows $HS$, $Co_3$ and $HS.2$ (degree $176$ resp.\ $276$), the Hering exceptions of degree $\le 729$ \textup{(including $3^6{:}\SL(2,13)$, \textup{Proposition \ref{prop:aff729}}\textup{)}}, the unitary stabilizer refinement of \textup{Proposition \ref{prop:psu-computed}} for $q\le49$ and the suborbit census of \textup{Table \ref{tab:psu-general}} for $q\le27$, the $\Ree(27){:}3$ data of \textup{Theorem \ref{thm:ree-ext}}, and the $\Sz(8){:}3$, $\Sz(32){:}5$ data of \textup{Theorem \ref{thm:suzuki-field-final}};  the full field part $C_{2f}$ of $\PpGU(3,q)$ \textup{(generated by the Frobenius $\sigma_p$)} is likewise computed for these parameters, fusing certain $\PGU(3,q)$-orbits for $q=8,9,25,27$ ($\sigma_3(\PpGU(3,q))=3,4,8,7$;  no closed form for general $q$ is proved, \textup{Remark~\ref{rem:psu-field}}).  Each such claim states its exact degree and group range \textup{(Chapter \ref{sec:verify})}.
\item \textbf{Conditional.}  The $\PpGU(3,q)$ full field part is computed for $q=8,9,25,27$ and is not claimed in closed form for every $q$ \textup{(Remark \ref{rem:psu-field};  it does not affect the rows of Table \ref{tab:main})};  the $\PpGU(3,q)$ conjugation is proved not to fuse for every $q$ by \textup{Proposition \ref{prop:psu-field-nofusion}}.  All other family statements of \textup{Table \ref{tab:main}}, including every Ree field extension, are unconditional.
\item \textbf{Not treated.}  The designs are not classified up to abstract isomorphism, and the specified group is not asserted to be the full automorphism group of each design.
\end{enumerate}
The finite computations of \textup{Chapter \ref{sec:verify}} verify the listed representations and do not replace a proof for arbitrary parameter values.
\end{remark}

\chapter{Conclusion}
\label{sec:concl}
% ============================================================

The complete result proved uniformly in this monograph is the orbit--design correspondence for a fixed $2$-transitive permutation group.  The family-specific statements are complete only in the parameter ranges for which an independent proof or an explicitly reproducible computation is supplied;  no unverified manuscript is used as a proof source, and every row of \textup{Table \ref{tab:main}} is either proved for all admissible parameters or explicitly marked as computed for an explicit range, including the Ree field extensions for every $f'\mid 2m+1$ \textup{(Proposition \ref{prop:ree-descent};  Proposition \ref{prop:ree-ext-noncubic})} \textup{(the proof status of every family is stratified as proved, computed or not treated in Remarks \ref{rem:proofstatus} and \ref{rem:open})}, so that every family is covered either by a general proof or by a reproducible computation within the stated ranges.

\begin{enumerate}
\item \textbf{The classification framework.}  Every $2$-transitive group is either of affine or of almost simple type (\textup{Theorem \ref{thm:2trans}});  \textup{Theorem \ref{thm:main}} (\textup{Section 1}) determines the stabilizer $G_{\Ocal}$ of every three-point subset for each treated family within the proved or computed parameter ranges, with the proof status of each family (proved for all parameters, computed for an explicit range, conditional, or not treated) recorded in \textup{Remark \ref{rem:open}}.  The key mechanism is uniform:  after translating a three-point set to $\{0,x,y\}$ (affine type) or normalizing it on a line/unital/ovoid (almost simple type), the stabilizer is described by a small number of linear conditions in the point stabilizer (\textup{Lemma \ref{lem:affine}}), and the answer depends only on elementary field arithmetic (squares in $\Fq$, primitive cube roots, $-1$ being a square, etc.).
\item \textbf{The largest part is the affine type.}  The sharply $2$-transitive groups $\AGL(1,q)$, the groups $\AaL(1,q)$, and the classical linear and symplectic (nondegenerate case) families are completely solved;  for the $G_2'(q)$ rows the full orbit and stabilizer classification is proved for every $q=2^f\ge4$ by \textup{Theorem \ref{thm:g2-classical}} \textup{($r=q+1+c(q)$, $c(q)=1+\frac{q-4}6$ if $3\mid q-1$ and $\frac{q-2}6$ otherwise;  collinear $\lambda=2$ on the $\{\omega,\omega^2\}$-orbit and $\lambda=6$ on the remaining collinear orbits, hyperbolic $\lambda=q^5$ on $q-1$ orbits, isotropic $\lambda=q(q^2-1)$ and $q^3(q^2-1)$;  verified for $q=4$ in \textup{Chapter \ref{sec:verify}}}\textup{)}, completing the pointwise-stabilizer data of the affine type \textup{(Remark \ref{rem:open})};  the totally isotropic symplectic case is reduced to a parabolic subgroup calculation and is not closed-form here \textup{(Remark \ref{rem:sp-isotropic}}\textup{)}.  Among Hering's exceptions, the nearfield cases ($\SL(2,3)\le\GL(2,p)$, $\SL(2,5)\le\GL(2,p)$, multiplicative groups of non-Desarguesian nearfields) and the small-rank cases ($2^4{:}\Alt(6)$, $2^4{:}\Alt(7)$, $\Alt(7)$ on $15$ points) are solved with exact stabilizer data;  the rows $3^4{:}(2.S_5)$, $2^6{:}\PSU(3,3)$ are now computed (\textup{Table \ref{tab:verify}}), the row $2_-^{1+4}\le\GL(4,3)$ is shown \emph{not} to be $2$-transitive, and $3^6{:}\SL(2,13)$ (degree $729$) is computed for the representation from a $6$-dimensional faithful irreducible $\F_3$-module \textup{(Proposition \ref{prop:aff729});  Remark \ref{rem:open})}.
\item \textbf{The key object.}  The setwise stabilizer $G_{\Ocal}$ of an arbitrary three-point subset, its kernel $G_{(\Ocal)}$ and its image in $S_3$, is the single invariant that determines the whole classification;  \textup{Theorem \ref{thm:main}} and \textup{Table \ref{tab:verify}} display it for every family.
\item \textbf{The two layers of the classification.}  For a fixed $2$-transitive group $G$ with orbit parameters $\lambda_1,\ldots,\lambda_r$, every nonempty union of the three-point orbits is a $G$-invariant triple system with parameter $\lambda_I=\sum_{i\in I}\lambda_i$, and conversely every nonempty $G$-invariant simple triple system arises in this way (\textup{Propositions \ref{prop:union-orbits} and \ref{cor:all-invariant}}).  Thus the orbit classification and the subset-sum rule together give the complete classification of invariant triple systems for the fixed group $G$.
\item \textbf{The Suzuki orbit number.}  For every $q=2^{2m+1}>2$ the number of three-subset orbits of the Suzuki group $\Sz(q)$ is proved to be $(q+4)/6$:  a single circle orbit with stabilizer $C_2$ and $\frac{q-2}{6}$ generic orbits with trivial stabilizer, parameters $3(q-1)$ and $6(q-1)$ \textup{(Theorem \ref{thm:sz};  Theorem \ref{thm:sz-orbits};  Corollary \ref{cor:sz-designs})}.  For every field extension $\Sz(q)\cdot f'$ the orbit number is given by the closed form of \textup{Theorem \ref{thm:suzuki-field-final}} (a Frobenius action on the parameter space $(\mathbb F_q\setminus\{0,1\})/H$;  Lemma \ref{lem:suzuki-bluher-count}), verified computationally for $q=8,32$;  in particular $\sigma_3(\Sz(8){:}3)=2$, $\sigma_3(\Sz(32){:}5)=2$ and $\sigma_3(\Sz(512){:}9)=12$.
\item \textbf{The Ree orbit number and structure.}  For every $q=3^{2m+1}$ the number of three-subset orbits of the Ree group ${}^2G_2(q)$ is proved to be $(q+3)^2/6$ from the classical class data \textup{(Theorem \ref{thm:ree-triples})}, and the value $150$ for $q=27$ is independently confirmed by the class computation of \textup{Chapter \ref{sec:verify}}.  The stabilizer-order distribution $N_1=\frac{q(q-3)}6,N_2=q,N_3=\frac{q-1}2,N_6=2$ and the collinear/non-collinear split are proved for every $q>3$ \textup{(Theorem \ref{thm:ree-structure})} and verified for $q=27$ by the suborbit computation \textup{(Table \ref{tab:verify-ree})};  for the field-automorphism overgroup $\Ree(27){:}3$ the value $\sigma_3=56$ and the full stabilizer distribution are computed and independently verified in \textup{Chapter \ref{sec:verify}} \textup{(Theorem \ref{thm:ree-ext};  Table \ref{tab:verify-ree})}, and the fixed-orbit number for every field quotient $f'\mid 2m+1$ is closed unconditionally:  $a=(q_0^2+2q_0+3)/2$ by \textup{Lemma \ref{lem:ree-ext-cycles}} ($f'=3$) and by the cubic reduction of \textup{Proposition \ref{prop:ree-descent}(ii)} ($3\mid f'\ge5$), and $a=(q_0+3)^2/6$ for $3
\nmid f'$ \textup{(Proposition \ref{prop:ree-descent}(ii);  Proposition \ref{prop:ree-ext-noncubic})}, so the Ree extension formula is closed for every field quotient, with e.g.\ $\sigma_3(\Ree(3^9)\cdot9)=7\,176\,734$ \textup{(Remark \ref{rem:ree-descent-status})}.
\item \textbf{The $G_2'$ affine rows.}  For every $q=2^f\ge4$ the affine group $G=q^6{:}G_2'(q)$ acting on the natural module $\F_q^6$ has $r=q+1+c(q)$ orbits on three-subsets, where $c(q)=1+\frac{q-4}{6}$ if $3\mid q-1$ and $c(q)=\frac{q-2}{6}$ otherwise:  the $c(q)$ collinear orbits with pointwise stabilizer $q^6(q^2-1)$, image $C_3$ and $\lambda=2$ on the single $\{\omega,\omega^2\}$-orbit \textup{(}$3\mid q-1$\textup{)} and image $1$ and $\lambda=6$ on the remaining collinear orbits;  the $q-1$ hyperbolic orbits ($h(u,w)
\ne 0$) with stabilizer $q(q^2-1)$, image $S_3$ and $\lambda=q^5$;  and the two isotropic orbits (long-root and short-root) with stabilizers $q^5$ and $q^3$, image $S_3$, and $\lambda=q(q^2-1)$ resp.\ $q^3(q^2-1)$ \textup{(Theorem \ref{thm:g2-classical};  verified for $q=4$ in \textup{Chapter \ref{sec:verify}}, where $c(4)=1$ and $r=6=q+2$}\textup{)}.  The case $q=2$ is the exceptional row $2^6{:}\PSU(3,3)$ with $r=3$ and $\lambda=32,24,6$ (\textup{Table \ref{tab:verify}}).  Together with \textup{Theorem \ref{thm:g2-classical}} the pointwise-stabilizer order and orbit structure of $G_2'(q)$ are proved for every $q=2^f\ge4$, completing the proved class of \textup{Remark \ref{rem:open}}.
\item \textbf{The unitary classification.}  For every $q\ge3$ the number of three-subset orbits of $\PSU(3,q)$ on the Hermitian unital is proved to be $\sigma_3=(q+3)/2$ for $q$ odd with $d=1$ and $1+q/2$ for $q$ even with $d=1$, and is proved as $(5q+11)/6$ for $q$ odd with $d=3$ \textup{(Proposition \ref{prop:psu-odd}, }$3\mid q+1$\textup{)} and as $(5q+2)/6$ for $q$ even with $d=3$ \textup{(Proposition \ref{prop:psu-even}, }$q=2^{2m+1}$\textup{)}, together with the complete non-collinear stabilizer distribution of \textup{Theorem \ref{thm:main-psu}(4)} \textup{(}$k=0$ for $q$ even and $k=d$ for $q$ odd;  the fixed orbits carry the full $S_3$ image by \textup{Lemma \ref{lem:psu-odd-fixed}} when $q$ is odd with $3\mid q+1$\textup{)} and the computed refinement for $q\le49$ of \textup{Proposition \ref{prop:psu-computed}};  the diagonal overgroup $\PGU(3,q)$ (when $d=3$) fuses with $\sigma_3(\PGU(3,q))=(q+3)/2$ ($q$ odd) or $(q+2)/2$ ($q$ even) for every $q$ by \textup{Proposition \ref{prop:psu-diag}}, and no fusion occurs for $d=1$;  the field conjugation of $\PpGU(3,q)$ is proved to cause no fusion for every $q\ge3$ \textup{(Proposition \ref{prop:psu-field-nofusion})}, while the full field part $C_{2f}$ (generated by the Frobenius $\sigma_p$) fuses certain $\PGU(3,q)$-orbits ($\sigma_3(\PpGU(3,q))=3$ for $q=8$, $4$ for $q=9$, $8$ for $q=25$ and $7$ for $q=27$;  Remark~\ref{rem:psu-field});  the closed forms are verified by GAP for every $q\le27$ \textup{(Table \ref{tab:psu-general})}, supplemented for the odd-$d=3$ row by $q=5,11,17,23$ \textup{(Chapter \ref{sec:verify})}, and the stabilizer census covers $q\le49$ \textup{(Table \ref{tab:magma})}.

\end{enumerate}

The same orbit classification also drives the applications developed in
\textup{Sections \ref{sec:graphs}--\ref{sec:assoc}}, completing the chain from triple
orbits to triple systems and then to graphs, codes and association schemes.
The incidence graph of every orbit is a connected biregular bipartite graph
with spectrum read off from $(\lambda,v)$; arc-transitivity is characterized
by a two-point stabilizer criterion and proved family-wide for the unitary,
Suzuki and Ree rows (\textup{Propositions \ref{prop:arc-image},
\ref{prop:arc-classification}}); the block graph is regular in general,
strongly regular in the Steiner case, and its $\lambda\ge2$ failure is
decided by the criterion of \textup{Proposition \ref{prop:blockgraph-srg}}.
The constant-weight code has distance $2$ or $4$ according as
$\lambda\ge2$ or $\lambda=1$, with the Steiner codes optimal (the Johnson
bound); the design code has rank, self-orthogonality and minimum-weight data
determined by $(r,\lambda)$ (with full $\mathbb{F}_2$-rank and minimum weight
$r$ verified for the unitary rows by
\textup{Proposition \ref{prop:rank-minweight}}), and is a $G$-module
isomorphic to the permutation module in the nonsingular case.  The Schurian coherent configuration on triples has
class number $m=\sum_{i,j}\rho_{ij}$ (\textup{Proposition \ref{prop:valencies}}), with the overlap classes giving $m\ge5$ for
$v\ge7$ and the computed values $m=32$ ($\AGL(1,7)$) and $m=1814$
($\PSU(3,3)$) \textup{(Proposition \ref{prop:astnum})}.

\medskip
This completes the transfer from three-point orbits to triple systems and, through the classified orbit data, to graphs, codes and association schemes:  arc-transitivity, the strong regularity of the block graph, the rank and minimum weight of the design codes, and the intersection tables of the triple schemes are all determined by the data of \textup{Chapter \ref{sec:prelim}} (see \textup{Propositions \ref{prop:arc-classification}, \ref{prop:blockgraph-srg}, \ref{prop:rank-minweight} and \ref{prop:ast-closed}}).

\section*{Acknowledgements}
This work begins with my master's advisor, Professor Hua Zhang, whose field
was permutation groups and directed graphs;  the very first inspiration of
this monograph comes from that early training, and the graph-theoretic
knowledge I rely on here was acquired in those years.  My doctoral advisor,
Professor Shenglin Zhou, whose field was groups and combinatorial designs,
gave me the knowledge of combinatorial designs on which the present work
rests, and it was under his guidance that this line of research took
shape.  My postdoctoral advisor,
Professor Pingzhi Yuan, whose field was combinatorial number theory, finite
fields and their applications, taught me finite fields, and it was learning
the finite-field viewpoint that finally opened a way into the central
difficulties of this monograph.  I am grateful to Professor Zanbo Zhang, the head of my research
group, whose field is directed graphs;  working in a team, I have come to
see, carries research much further than solitary effort.  A special word of
thanks goes to my colleague Shaohua Liu:  when this work had been stalled for
half a year, it was his GPT-based verification that unlocked the road, and
with the help of AI the pace of the project reached a height I had not
expected, bringing the whole work forward by three or four years.  I am deeply
indebted to Guangdong University of Finance and Economics for providing me a
stable position in which I could concentrate on this unusually difficult
problem.  And last, but far from least, I must thank from the bottom of my
heart my wife and my family, who have been with me all the way;  thank you for
being there.

Throughout this monograph I have made extensive use of the AI tool
GPT6-Astra, at many stages and for a variety of purposes:  drafting and
polishing the exposition, checking the internal consistency of the arguments,
and writing and debugging the GAP and Magma scripts on which the computational
verification of Chapter \ref{sec:verify} rests.  In every case the output was read, checked
and verified by myself before it entered the manuscript, and I take full
responsibility for all of the mathematics (definitions, theorems, proofs,
and computations) presented here.


\begin{thebibliography}{99}

\bibitem{alavi2019} S.~H. Alavi, M.~Bayat, A.~Daneshkhah, Flag-transitive block designs and unitary groups, \emph{Monatsh. Math.}~193 (2020), 535--553.

\bibitem{assmuskey1992} E.~F.~Assmus, J.~D.~Key, \emph{Designs and Their Codes}, Cambridge Tracts in Math. 103, Cambridge Univ. Press, 1992.

\bibitem{balmaceda2021} J.~M.~P.~Balmaceda, D.~V.~A.~Briones, Families of association schemes on triples from two-transitive groups, \emph{Ars Math. Contemp.}~25 (2025), \#P3.05.

\bibitem{bamberg2023} J.~Bamberg, S.~D.~Freedman, M.~Giudici, Spreading primitive groups of diagonal type do not exist, \emph{Proc. Roy. Soc. Edinburgh Sect. A Math.}~156 (2026), 188--198.

\bibitem{magma} W.~Bosma, J.~Cannon, C.~Playoust, \emph{The Magma algebra system I:  The user language}, J.\ Symbolic Comput.\ 24 (1997) 235--265.

\bibitem{bradley2014} P.~M. Bradley, \emph{Counting $G$-orbits on the induced action on $k$-subsets}, Ph.D.\ thesis, University of Manchester, 2014.

\bibitem{cameron1981} P.~J. Cameron, \emph{Finite permutation groups and finite simple groups}, Bull.\ London Math.\ Soc.\ 13 (1981) 1--22.

\bibitem{cameron1999} P.~J. Cameron, \emph{Permutation Groups}, London Math.\ Soc.\ Student Texts 45, Cambridge Univ.\ Press, 1999.

\bibitem{conway1985} J.~H. Conway, R.~T. Curtis, S.~P. Norton, R.~A. Parker, R.~A. Wilson, \emph{ATLAS of Finite Groups}, Clarendon Press, Oxford, 1985.

\bibitem{cooperstein1979} B.~N.~Cooperstein, \emph{The geometry of root subgroups in exceptional groups I}, Geom.\ Dedicata 8 (1979) 317--381.

\bibitem{curtis1976} C.~W. Curtis, W.~M. Kantor, G.~M. Seitz, \emph{The 2-transitive permutation representations of the finite groups of Lie type}, Trans.\ Amer.\ Math.\ Soc.\ 218 (1976) 1--59.

\bibitem{devillers2020} A.~Devillers, H.~Liang, C.~E.~Praeger, B.~Xia, On flag-transitive $2$-$(v,k,2)$ designs, \emph{J. Combin. Theory Ser. A}~177 (2021), 105309.

\bibitem{dixon1996} J.~D. Dixon, B.~Mortimer, \emph{Permutation Groups}, Graduate Texts in Math.\ 163, Springer, 1996.

\bibitem{feit1963} W.~Feit, J.~G.~Thompson, Solvability of groups of odd order, \emph{Pacific J. Math.}~13 (1963), 775--1029.

\bibitem{francot2024} E.~Francot, A.~Montinaro, A classification of the flag-transitive $2$-\mbox{$(v,3,\lambda)$} designs with $v\equiv1,3\pmod6$ and $v\equiv6\pmod\lambda$, \emph{J. Combin. Des.}~6 (2025), 15206610.

\bibitem{gap} The GAP Group, \emph{GAP --- Groups, Algorithms, and Programming}, Version 4.16, 2025;  https://www.gap-system.org.

\bibitem{hall1985} M.~Hall, Jr., \emph{Steiner triple systems with a doubly transitive automorphism group}, J.\ Combin.\ Theory Ser.\ A 38 (1985) 192--202.

\bibitem{hering1985} C.~Hering, \emph{Transitive linear groups and linear groups which contain irreducible subgroups of prime order.~II}, J.~Algebra 93 (1985) 151--164.

\bibitem{huber2005} M.~Huber, The classification of flag-transitive Steiner $3$-designs, \emph{Adv. Geom.}~5 (2005), 195--221.

\bibitem{huber2007} M.~Huber, The classification of flag-transitive Steiner $4$-designs, \emph{J. Algebraic Combin.}~26 (2007), 183--207.

\bibitem{huber2009book} M.~Huber, \emph{Flag-transitive Steiner Designs}, Frontiers in Mathematics, Birkh\"auser, Basel, 2009.

\bibitem{huppert1967} B.~Huppert, \emph{Endliche Gruppen I}, Springer, Berlin, 1967.

\bibitem{johnson1962} S.~M.~Johnson, A new upper bound for error-correcting codes, \emph{IRE Trans. Inform. Theory}~8 (1962), 203--207.

\bibitem{kantor1985} W.~M. Kantor, \emph{Homogeneous designs and geometric lattices}, J.\ Combin.\ Theory Ser.\ A 38 (1985) 66--74.

\bibitem{keymoori1992} J.~D.~Key, J.~Moori, Designs, codes and graphs from sporadic groups, \emph{Comm. Algebra}~20 (1992), 1571--1596.

\bibitem{kirkman1847} T.~P.~Kirkman, On a problem in combinations, \emph{Camb. Dublin Math. J.}~2 (1847), 191--204.

\bibitem{kleidmanliebeck} P.~Kleidman, M.~Liebeck, \emph{The subgroup structure of the finite classical groups}, London Math.\ Soc.\ Lecture Note Ser.\ 129, Cambridge University Press, Cambridge, 1990.

\bibitem{lang1956} S.~Lang, Algebraic groups over finite fields, \emph{Amer. J. Math.}~78 (1956), 555--563.

\bibitem{li2010} C.~H.~Li, \emph{Permutation Groups and Symmetrical Graphs}, lecture notes, 2010 (Sections~6--7: orbital graphs, Higman's theorem, arc-transitive graphs and their criteria).

\bibitem{ling2004} S.~Ling, C.~Xing, \emph{Coding Theory: A First Course}, Cambridge Univ. Press, 2004 (Chapter~5: Hamming, Singleton, Plotkin, Johnson bounds).

\bibitem{livingstone1965} D.~Livingstone, A.~Wagner, Transitivity of finite permutation groups on unordered sets, \emph{Math. Z.}~90 (1965), 393--403.

\bibitem{luneburg1966} H.~L\"uneburg, \emph{Some remarks concerning the Ree groups of type $G_2$}, J.\ Algebra 3 (1966) 256--259.

\bibitem{mesner1990} D.~M.~Mesner, P.~Bhattacharya, Association schemes on triples and a ternary algebra, \emph{J. Combin. Theory Ser. A}~55 (1990), 204--234.

\bibitem{mitchell1911} H.~H. Mitchell, \emph{Determination of the ordinary and modular ternary linear groups}, Trans.\ Amer.\ Math.\ Soc.\ 12 (1911) 207--242.

\bibitem{muller2009} J.~M\"uller, \emph{Association Schemes --- Group Theory Without Groups}, lecture notes, RWTH Aachen, SS 2009 (Chapter~1: Schurian schemes, valencies, generous transitivity).

\bibitem{ree1961} R.~Ree, \emph{A family of simple groups associated with the simple Lie algebra of type $G_2$}, Amer.\ J.\ Math.\ 83 (1961) 432--462.

\bibitem{steinberg1968} R.~Steinberg, \emph{Endomorphisms of linear algebraic groups}, Mem.\ Amer.\ Math.\ Soc.\ 80, Amer.\ Math.\ Soc., Providence, RI, 1968.

\bibitem{steiner1853} J.~Steiner, Combinatorische Aufgabe, \emph{J. Reine Angew. Math.}~45 (1853), 181--185.

\bibitem{suzuki1962} M.~Suzuki, \emph{On a class of doubly transitive groups}, Ann.\ of Math.\ (2) 75 (1962) 105--145.

\bibitem{tongviet2013} H.~P. Tong-Viet, \emph{Rank 3 permutation characters and maximal subgroups}, Forum Math.\ 25 (2013) 49--106.

\bibitem{wilson1996} R.~A. Wilson et al., \emph{ATLAS of finite group representations} (online database), 1996--;  http://brauer.maths.qmul.ac.uk/Atlas/.

\bibitem{witt1938} E.~Witt, \"Uber Steinersche Systeme, \emph{Abh. Math. Sem. Univ. Hamburg}~12 (1938), 265--275.
\bibitem{zieschang2005} P.-H.~Zieschang, \emph{Theory of Association Schemes}, Springer Monographs in Mathematics, Springer, Berlin, 2005 (283 pp.)

\end{thebibliography}
\end{document}